\documentclass{amsart}
\usepackage{amssymb}
\usepackage{mathrsfs}
\usepackage{stmaryrd}
\usepackage{imakeidx}
\usepackage[colorlinks = true]{hyperref}
\usepackage{enumitem}
\usepackage{comment}
\input xy \xyoption {all}
\usepackage{bbm}

\usepackage{tikz}
\usetikzlibrary{cd}

\usepackage{dsfont}

\usepackage{rotating}
\usepackage{lscape}

\newcommand{\Z}{\mathbb{Z}}

\newcommand{\F}{\mathbb{F}}
\newcommand{\Q}{\mathbb{Q}}

\newcommand{\bk}{\Bbbk}
\newcommand{\Ga}{\mathbb{G}_{\mathrm{a}}}
\newcommand{\Gm}{\mathbb{G}_{\mathrm{m}}}

\newcommand{\Mod}{\mathsf{Mod}}
\newcommand{\Modr}{\mathsf{Mod}_{\mathrm{r}}}

\newcommand{\Kb}{K^{\mathrm{b}}}

\newcommand{\Db}{D^{\mathrm{b}}}

\newcommand{\cE}{\mathcal{E}}

\newcommand{\Perv}{\mathrm{Perv}}

\newcommand{\hatstar}{\mathbin{\widehat{\star}}}
\newcommand{\pstar}{\mathbin{{}^{\mathrm{p}} \hspace{-1pt} \star}}
\newcommand{\phatstar}{\mathbin{{}^{\mathrm{p}} \hspace{-0.5pt} \widehat{\star}}}

\newcommand{\cG}{\mathcal{G}}
\newcommand{\cH}{\mathcal{H}}

\newcommand{\Whit}{\mathrm{Wh}}
\newcommand{\Av}{\mathsf{Av}}
\newcommand{\For}{\mathsf{For}}
\newcommand{\IC}{\mathscr{I}\hspace{-1pt}\mathscr{C}}

\newcommand{\id}{\mathrm{id}}

\newcommand{\simto}{\xrightarrow{\sim}}

\DeclareMathOperator{\Hom}{Hom}
\DeclareMathOperator{\Ext}{Ext}
\DeclareMathOperator{\End}{End}

\newcommand{\Iw}{\mathrm{I}}
\newcommand{\Iwu}{\mathrm{I}_{\mathrm{u}}}
\newcommand{\Fl}{\mathrm{Fl}}
\newcommand{\tFl}{\widetilde{\mathrm{Fl}}}
\newcommand{\Gr}{\mathrm{Gr}}
\newcommand{\Wf}{W_{\mathrm{f}}}

\newcommand{\Sf}{S_{\mathrm{f}}}
\newcommand{\IW}{\mathcal{IW}}

\newcommand{\bG}{\mathbf{G}}
\newcommand{\bH}{\mathbf{H}}
\newcommand{\bK}{\mathbf{K}}
\newcommand{\bA}{\mathbf{A}}
\newcommand{\bB}{\mathbf{B}}

\newcommand{\bU}{\mathbf{U}}
\newcommand{\bT}{\mathbf{T}}

\newcommand{\bW}{\mathbf{W}}
\newcommand{\bWf}{\mathbf{W}_{\mathrm{f}}}
\newcommand{\bS}{\mathbf{S}}
\newcommand{\bSf}{\mathbf{S}_{\mathrm{f}}}

\newcommand{\bSaff}{\mathbf{S}_{\mathrm{aff}}}

\newcommand{\bg}{\mathbf{g}}
\newcommand{\bb}{\mathbf{b}}
\newcommand{\bt}{\mathbf{t}}

\renewcommand{\hat}{\widehat}
\renewcommand{\tilde}{\widetilde}

\makeatletter
\def\lotimes{\@ifnextchar_{\@lotimessub}{\@lotimesnosub}}
\def\@lotimessub_#1{\mathchoice{\mathbin{\mathop{\otimes}^L}_{#1}}%
  {\otimes^L_{#1}}{\otimes^L_{#1}}{\otimes^L_{#1}}}
\def\@lotimesnosub{\mathbin{\mathop{\otimes}^L}}
\makeatother

\makeatletter
\def\lboxtimes{\@ifnextchar_{\@lboxtimessub}{\@lboxtimesnosub}}
\def\@lboxtimessub_#1{\mathchoice{\mathbin{\mathop{\boxtimes}^L}_{#1}}%
  {\boxtimes^L_{#1}}{\boxtimes^L_{#1}}{\boxtimes^L_{#1}}}
\def\@lboxtimesnosub{\mathbin{\mathop{\boxtimes}^L}}
\makeatother

\newcommand{\scO}{\mathscr{O}}

\newcommand{\cK}{\mathcal{K}}
\newcommand{\cJ}{\mathcal{J}}
\newcommand{\cI}{\mathcal{I}}
\newcommand{\pH}{{}^{\mathrm{p}} \hspace{-1.5pt} \mathscr{H}}
\newcommand{\Rep}{\mathsf{Rep}}

\newcommand{\Coh}{\mathsf{Coh}}
\newcommand{\QCoh}{\mathsf{QCoh}}

\newcommand{\sZ}{\mathscr{Z}}
\newcommand{\tsZ}{\mathscr{Z}^{\wedge}}
\newcommand{\sm}{\mathsf{m}}
\newcommand{\st}{\mathsf{t}}
\newcommand{\su}{\mathsf{u}}
\newcommand{\Wak}{\mathscr{W}}

\newcommand{\sfA}{\mathsf{A}}
\newcommand{\sfB}{\mathsf{B}}
\newcommand{\sfC}{\mathsf{C}}
\newcommand{\sfD}{\mathsf{D}}
\newcommand{\sfE}{\mathsf{E}}
\newcommand{\sfK}{\mathsf{K}}
\newcommand{\sfP}{\mathsf{P}}
\newcommand{\sfT}{\mathsf{T}}

\newcommand{\Ind}{\mathrm{Ind}}
\newcommand{\sfZ}{\mathsf{Z}}
\newcommand{\tsfZ}{\mathsf{Z}^{\wedge}}
\newcommand{\scA}{\mathscr{A}}
\newcommand{\scB}{\mathscr{B}}
\newcommand{\scC}{\mathscr{C}}

\newcommand{\scF}{\mathscr{F}}
\newcommand{\scG}{\mathscr{G}}
\newcommand{\scH}{\mathscr{H}}
\newcommand{\scL}{\mathscr{L}}
\newcommand{\scM}{\mathscr{M}}
\newcommand{\scQ}{\mathscr{Q}}
\newcommand{\scR}{\mathscr{R}}
\newcommand{\scS}{\mathscr{S}}
\newcommand{\scT}{\mathscr{T}}
\newcommand{\scW}{\mathscr{W}}

\newcommand{\St}{\mathbf{St}}
\newcommand{\Stm}{\mathbf{St}_{\mathrm{m}}}
\newcommand{\Stmreg}{\mathbf{St}_{\mathrm{m}, \mathrm{reg}}}
\newcommand{\reg}{\mathrm{reg}}

\newcommand{\can}{\mathrm{can}}
\newcommand{\gr}{\mathrm{gr}}
\newcommand{\ev}{\mathrm{ev}}
\newcommand{\coev}{\mathrm{coev}}
\newcommand{\trL}{\mathrm{tr}^{\mathrm{L}}}

\newcommand{\Spr}{\widetilde{\mathcal{U}}}
\newcommand{\Groth}{\widetilde{\bG}}
\newcommand{\cU}{\mathcal{U}}
\newcommand{\D}{\mathsf{D}}
\newcommand{\pD}{{}^{\mathrm{p}} \hspace{-1pt} \D}
\DeclareMathOperator{\ind}{Ind-}
\DeclareMathOperator{\pro}{Pro-}

\newcommand{\LAS}{\mathscr{L}_{\mathrm{AS}}}

\newcommand{\Mof}{\mathsf{Mod}^{\mathrm{fg}}}
\newcommand{\Mofr}{\mathsf{Mod}^{\mathrm{fg}}_{\mathrm{r}}}
\newcommand{\BSRep}{\mathsf{BSRep}}
\newcommand{\SRep}{\mathsf{SRep}}
\newcommand{\SCoh}{\mathsf{SCoh}}

\newcommand{\Spec}{\mathrm{Spec}}
\newcommand{\bbJ}{\mathbb{J}}
\newcommand{\bbI}{\mathbb{I}}
\newcommand{\Vect}{\mathrm{Vect}}
\newcommand{\op}{\mathrm{op}}
\newcommand{\taut}{\mathrm{taut}}

\newcommand{\BSK}{\mathsf{BSK}}

\newcommand{\fR}{\mathfrak{R}}
\newcommand{\quot}{\hspace{-2.5pt} \mathord{\fatslash} \hspace{3pt}}

\newcommand{\Loop}{\mathrm{L}}
\newcommand{\Sat}{\mathsf{Sat}}
\newcommand{\rmZ}{\mathrm{Z}}
\newcommand{\FN}{\mathrm{FN}}

\numberwithin{equation}{section}
\numberwithin{figure}{section}
\newtheorem{thm}{Theorem}[section]
\newtheorem{lem}[thm]{Lemma}
\newtheorem{prop}[thm]{Proposition}
\newtheorem{cor}[thm]{Corollary}

\theoremstyle{definition}

\theoremstyle{remark}
\newtheorem{rmk}[thm]{Remark}

\title{Modular affine Hecke category and regular centralizer}

\author[R. Bezrukavnikov]{Roman Bezrukavnikov}
\address{Department of Mathematics \\ Massachusetts Institute of Technology \\ Cambridge, MA \\ 02139 \\ USA.}
\email{bezrukav@math.mit.edu}

 \author[S.~Riche]{Simon Riche}
 \address{Universit\'e Clermont Auvergne, CNRS, LMBP, F-63000 Clermont-Ferrand, France.}
 \email{simon.riche@uca.fr}
 
\begin{document}

\begin{abstract}
In this paper we provide a ``combinatorial'' description of the category of tilting perverse sheaves on the affine flag variety of a reductive algebraic group, and its free-monodromic variant, with coefficients in a field of positive characteristic. This provides a replacement for the familiar ``Soergel theory'' for characteristic-$0$ coefficients, and the second step in our project towards the construction of an equivalence of categories relating the two natural geometric realizations of the associated affine Hecke algebra in the case of positive-characteristic coefficients. 
\end{abstract}

\maketitle

\setcounter{tocdepth}{1}
\tableofcontents

\section{Introduction}

\subsection{Presentation}

The present paper is the second step in our project (initiated in~\cite{brr}, joint with L.~Rider, and concluded in~\cite{br-pt3}) of constructing ``modular tamely ramified local Langlands equivalences'' adapting to positive-characteristic coefficients the constructions of the first author in~\cite{be}. These equivalences relate some categories of Iwahori-constructible or Iwahori-equivariant perverse sheaves on the affine flag variety $\Fl_G$ of a connected reductive algebraic group $G$ (or a natural torsor $\tFl_G$ over $\Fl_G$), with coefficients in an algebraic closure $\bk$ of $\F_\ell$ (or a ``free-monodromic completion''), to some categories of equivariant coherent sheaves on the Steinberg variety of the reductive group $G^\vee_\bk$ over $\bk$ which is Langlands dual to $G$ (or some variants). Our strategy involves a description of both sides of this equivalence in ``Soergel theoretic terms;" the present paper finishes the first half of this construction, by providing such a description on the constructible side of the picture. The second half is completed in~\cite{br-pt3}, based of the localization theory for representations of reductive Lie algebras over fields of positive characteristic developed by the first author with Mirkovi{\'c} and Rumynin, see~\cite{bmr,bm-loc}. (This part of the work is closely related to our work in~\cite{brHecke}.)

The main result of the paper is therefore an equivalence of categories relating an appropriate category of perverse sheaves on $\tFl_G$ to a category of ``Soergel bimodules.'' This equivalence is fundamental for our project explained above, but it is also of independent interest, given the importance of categories of perverse sheaves on affine flag varieties in recent work on representation theory of reductive algebraic groups over fields of positive characteristic (see in particular~\cite{rw,rw2}). See REF for direct applications of this construction.


\subsection{Soergel bimodules}

We need to explain what we mean by ``Soergel bimodules'' in our present context. Ordinarily, Soergel bimodules are attached to a Coxeter system $(\mathcal{W},\mathcal{S})$ and a suitable representation $V$ of $\mathcal{W}$, and defined as a certain category of graded $\scO(V)$-bimodules, see~\cite{soergel}. The case we are interested in here is when $(\mathcal{W},\mathcal{S})$ is the affine Weyl group\footnote{The main player in this paper will rather be the \emph{extended} affine Weyl group, for which we reserve the notation $W$.} attached to $G$ with its standard Coxeter generators, and $V$ is the dual of the Lie algebra of a fixed maximal torus $T^\vee_\bk$ in $G^\vee_\bk$, with the action factoring through the canonical action of the (finite) Weyl group $\Wf$ of $(G^\vee_\bk,T^\vee_\bk)$. This representation does not satisfy the technical conditions imposed in~\cite{soergel}, so that a different definition of Soergel bimodules is required. 

Various alternative definitions of categories of Soergel bimodules have been proposed in the recent years by Fiebig~\cite{fiebig} (involving sheaves on moment graphs), Elias--Williamson~\cite{ew} (in terms of a monoidal category defined by generators and relations involving planar diagrams), Juteau--Mautner--Williamson~\cite{jmw} (involving parity complexes) or Abe~\cite{abe} (involving bimodules with extra structures). Some of these definitions apply in our context, and give rise to essentially equivalent categories; in fact some will be used in the body of the paper. But they all have one drawback, which is that they do not ``see'' an essential ingredient that is specific to the case of affine Weyl groups (rather than a general Coxeter system), namely the relation with the geometric Satake equivalence~\cite{mv} and the Langlands dual group $G^\vee_\bk$. For that reason we will consider a different (although equivalent) realization of the category of Soergel bimodules, in terms of certain representations of the universal centralizer associated with $G^\vee_\bk$. The universal centralizer was introduced by Lusztig (see the last paragraph of~\cite{lusztig-coxeter}); earlier instances of relations between this subject and Soergel bimodules for affine Weyl groups can be found in~\cite{dodd, mr, brHecke}.

\subsection{Categories of perverse sheaves}

Now we can start explaining the definition of the categories we aim at describing. Let $\F$ be an algebraically closed field of characteristic $p>0$, and let $G$ be a connected reductive algebraic group over $\F$, with a choice of (negative) Borel subgroup $B$ and maximal torus $T \subset B$. Let $z$ be an indeterminate, and let $\Loop G$ and $\Loop^+G$ be the associated loop group and arc group respectively, i.e.~the group ind-scheme, resp.~group scheme, over $\bk$ representing the functor
\[
R \mapsto G(R( \hspace{-1pt} (z) \hspace{-1pt} )), \quad \text{resp.} \quad R \mapsto G(R[ \hspace{-1pt} [z] \hspace{-1pt} ]).
\]
Let also $\Iw \subset \Loop^+G$ be the Iwahori subgroup obtained as the preimage of $B$ under the canonical morphism $\Loop^+G \to G$, and let $\Iwu$ be its pro-unipotent radical, namely the preimage of the unipotent radical $U$ of $B$. The affine flag variety $\Fl_G$ is defined as the fppf quotient $\Loop G/\Iw$; it is known to be represented by an ind-projective ind-scheme. For technical reasons we will also consider the fppf quotient $\tFl_G := \Loop G/\Iwu$; this quotient is represented by an ind-scheme of ind-finite type, and the natural morphism $\tFl_G \to \Fl_G$ is a (Zariski locally trivial) $T$-torsor.

Let $\bk$ be an algebraic closure of $\F_\ell$, where $\ell$ is a prime number different from $p$.
We will denote by $\sfD_{\Iw,\Iw}$ and $\sfD_{\Iwu,\Iw}$ the $\Iw$-equivariant and $\Iwu$-equivariant derived categories of constructible \'etale $\bk$-sheaves on $\Fl_G$, respectively. We will also denote by $\sfP_{\Iw,\Iw}$ and $\sfP_{\Iwu,\Iw}$ the hearts of the perverse t-structures on these categories. The category $\sfD_{\Iw,\Iw}$ admits a natural convolution operation $\star_{\Iw}$, which endows it with a monoidal structure.

Let $\Gr_G$ be the affine Grassmannian of $G$, defined as the fppf quotient $\Loop G/\Loop^+ G$ (which is, again, represented by an ind-projective ind-scheme over $\F$). Denote by $\sfD_{\Loop^+G,\Loop^+G}$ the $\Loop^+G$-equivariant derived category of constructible \'etale $\bk$-sheaves on $\Gr_G$, and by $\sfP_{\Loop^+G,\Loop^+G}$ the subcategory of perverse sheaves. This category is the main ingredient in the geometric Satake equivalence. Namely, convolution defines a monoidal product $\star_{\Loop^+ G}$ on $\sfD_{\Loop^+G,\Loop^+G}$, which turns out to be exact with respect to the perverse t-structure, hence to induce a monoidal product on $\sfP_{\Loop^+G,\Loop^+G}$. The geometric Satake equivalence, proved in this setting in~\cite{mv}, provides an equivalence of monoidal categories
\begin{equation}
\label{eqn:intro-Satake}
(\sfP_{\Loop^+G,\Loop^+G}, \star_{\Loop^+G}) \cong (\Rep(G^\vee_\bk), \otimes)
\end{equation}
where $G^\vee_\bk$ is the connected reductive algebraic group over $\bk$ which is Langlands dual to $G$ and $\Rep(G^\vee_\bk)$ is its category of finite-dimensional algebraic representations.

A way to ``upgrade'' $\Loop^+G$-equivariant perverse sheaves on $\Gr_G$ to $\Iw$-equivariant perverse shea\-ves on $\Fl_G$ has been proposed by Gaitsgory~\cite{gaitsgory}. More specifically, this construction provides a canonical t-exact monoidal functor
\[
\sfZ : (\sfD_{\Loop^+G,\Loop^+G},\star_{\Loop^+ G}) \to (\sfD_{\Iw,\Iw}, \star_{\Iw}),
\]
which admits various favorable properties and structures (in particular, a ``monodromy'' automorphism) that will not be recalled in detail here.

Our goal is therefore to provide ``combinatorial'' descriptions of the categories $\sfP_{\Iw,\Iw}$ and $\sfP_{\Iwu,\Iw}$, in a way compatible with the functor $\sfZ$ and the geometric Satake equivalence.

\subsection{The equivariant regular quotient}

A first step in this direction has been obtained (jointly with L. Rider) in~\cite{brr}. 
Contrary to the case of $\sfD_{\Loop^+G, \Loop^+G}$, the monoidal structure on $\sfD_{\Iw,\Iw}$ is unfortunately not compatible with the perverse t-structure, in the sense that $\star_{\Iw}$ is neither left nor right exact. In order to take advantage of this structure while staying in the world of abelian categories, we consider a category that we call the ``regular quotient'' $\sfP^0_{\Iw,\Iw}$ of $\sfP_{\Iw,\Iw}$. Recall that the simple objects in $\sfP_{\Iw,\Iw}$ are in a canonical bijection with the elements in the extended affine Weyl group $W$ associated with $G$. Then $\sfP^0_{\Iw,\Iw}$ is defined as the Serre quotient of $\sfP_{\Iw,\Iw}$ by the Serre subcategory generated by the simple objects whose label has positive length (equivalently, whose support is not a point). The simple objects in $\sfP^0_{\Iw,\Iw}$ are uninteresting (they are labeled by length-$0$ elements in $W$), but the extensions between these objects are more subtle. The convolution product $\star_\Iw$ induces in a natural way a monoidal product $\star^0_\Iw$ on $\sfP^0_{\Iw,\Iw}$, which is exact on both sides.

The main result of~\cite{brr} is the following claim. Here we denote by $\sZ^0$ the composition
\[
\Rep(G^\vee_\bk) \xrightarrow[\sim]{\eqref{eqn:intro-Satake}} \sfP_{\Loop^+G,\Loop^+G} \xrightarrow{\sfZ} \sfP_{\Iw,\Iw} \to \sfP^0_{\Iw,\Iw}
\]
where the third functor is the natural quotient functor. (This composition is easily seen to be monoidal.)

\begin{figure}
 \begin{tabular}{|c|c|c|c|c|c|c|c|c|}
  \hline
  $\mathbf{A}_n$ ($n \geq 1$) & $\mathbf{B}_n$ ($n \geq 2$) & $\mathbf{C}_n$ ($n \geq 3$) & $\mathbf{D}_n$ ($n \geq 4$) & $\mathbf{E}_6$ & $\mathbf{E}_7$ & $\mathbf{E}_8$ & $\mathbf{F}_4$ & $\mathbf{G}_2$ \\
  \hline
  $1$ & $n$ & $2$ & $2$ & $3$ & $19$ & $31$ & $3$ & $3$\\
  \hline
 \end{tabular}
\caption{Bounds on $\ell$}
\label{fig:bounds}
\end{figure}

\begin{thm}
\label{thm:brr-intro}
Assume that the following conditions are satisfied:
\begin{enumerate}
 \item the quotient of $X^*(T)$ by the root lattice of $(G,T)$ is free;
 \item the quotient of $X_*(T)$ by the coroot lattice of $(G,T)$ has no $\ell$-torsion;
 \item for any indecomposable factor in the root system of $(G,T)$, $\ell$ is strictly bigger than the corresponding value in Figure~\ref{fig:bounds}.
\end{enumerate}
Let $\su \in G^\vee_\bk$ be a regular unipotent element. There exists an equivalence of monoidal categories
\[
\Phi_{\Iw,\Iw} : (\sfP^0_{\Iw,\Iw}, \star^0_\Iw) \simto (\Rep(\rmZ_{G^\vee_\bk}(\su)), \otimes) 
\]
such that $\Phi_{\Iw,\Iw} \circ \sZ^0$ identifies with the restriction functor $\Rep(G^\vee_\bk) \to \Rep(\rmZ_{G^\vee_\bk}(\su))$.
\end{thm}

\begin{rmk}
\begin{enumerate}
\item
The first two assumptions in Theorem~\ref{thm:brr-intro} are rather standard. The third assumption is an artefact of the method of proof of one of the crucial claims in~\cite{brr}; we expect that this assumption can be weakened at least to the requirement that $\ell$ is good for $G$.
\item
For $V$ in $\Rep(G^\vee_\bk)$, the functor $\Phi_{\Iw,\Iw}$ sends the monodromy automorphism of $\sZ^0(V)$ to the action of $\su$ on $V$.
\end{enumerate}
\end{rmk}

The proof of Theorem~\ref{thm:brr-intro} is based on a general result regarding central functors whose domain is a category of representations of an affine group scheme over a field; see~\cite[Lemma~3.1]{brr} for a precise statement. (This statement first appeared in~\cite{bez}.) This statement immediately provides an equivalence between a certain full subcategory of $\sfP^0_{\Iw,\Iw}$ and the category of representations of a closed subgroup scheme of $G^\vee_\bk$; what remains to be checked in order to obtain the theorem is that this full subcategory is the whole of $\sfP^0_{\Iw,\Iw}$, and that the subgroup is $\rmZ_{G^\vee_\bk}(\su)$ for a certain regular unipotent element $\su$. Note that the freedom in the choice of the element $\su$ is closely related to the fact that there does not exist any ``canonical'' fiber functor on the category $\sfP^0_{\Iw,\Iw}$; in fact, a choice of such a fiber functor is essentially equivalent to a choice of a regular unipotent element $\su$.

\subsection{The monodromic regular quotient}
\label{ss:mon-reg-quot-intro}

Theorem~\ref{thm:brr-intro} has a clean formulation, but it has the drawback that the category $\sfD_{\Iw,\Iw}$ cannot be recovered from the regular quotient $\sfP^0_{\Iw,\Iw}$. In order to solve this issue one might be tempted to play the same game with the category $\sfD_{\Iwu,\Iw}$ instead of $\sfD_{\Iw,\Iw}$. (In fact, standard arguments guarantee that if $\sfP_{\Iwu,\Iw}^0$ is defined by the same procedure as $\sfP_{\Iw,\Iw}^0$, then the quotient functor $\sfP_{\Iwu,\Iw} \to \sfP_{\Iwu,\Iw}^0$ is fully faithful on tilting perverse sheaves, and the category $\sfD_{\Iwu,\Iw}$ identifies with the bounded homotopy category of the category of tilting perverse sheaves. In this sense, $\sfD_{\Iwu,\Iw}$ can be recovered from $\sfP^0_{\Iwu,\Iw}$.) 

However, in this process one looses the monoidal product, which was crucial for the construction of the functor $\Phi_{\Iw,\Iw}$. To remedy this, one needs to work instead with the category $\sfD_{\Iwu,\Iwu}$ which is defined as the full triangulated subcategory of the derived category of constructible complexes of $\bk$-sheaves on $\tFl_G$ generated by complexes obtained by pullback from $\Iwu$-equivariant complexes on $\Fl_G$. This category admits a natural perverse t-structure, whose heart will be denoted $\sfP_{\Iwu,\Iwu}$. Once again the simple objects in $\sfP_{\Iwu,\Iwu}$ are in bijection with $W$, and we can define the monodromic regular quotient $\sfP^0_{\Iwu,\Iwu}$ as the quotient by the Serre subcategory generated by simple objects whose label has positive length. There exists a natural convolution product $\star_{\Iwu}$ on $\sfD_{\Iwu,\Iwu}$, which is not exact, but which induces in a natural way a monoidal product $\star^0_{\Iwu}$ on $\sfP^0_{\Iwu,\Iwu}$. (This product is only right exact on both sides.) The pullback functor $\sfD_{\Iw,\Iw} \to \sfD_{\Iwu,\Iwu}$ induces an exact monoidal functor
\begin{equation}
\label{eqn:functor-P0-intro}
(\sfP^0_{\Iw,\Iw},\star^0_{\Iw}) \to (\sfP^0_{\Iwu,\Iwu}, \star^0_{\Iwu}).
\end{equation}

Compared with $\sfP^0_{\Iw,\Iw}$, the category $\sfP^0_{\Iwu,\Iwu}$ has two new ``directions of deformation,'' corresponding to the replacement of $\Iw$ by $\Iwu$ on both sides. On the geometric side, the corresponding process is the replacement of $\rmZ_{G^\vee_\bk}(\su)$ by a group scheme constructed out of the ``universal centralizer'' over $G^\vee_\bk$. This universal centralizer is an affine group scheme over $G^\vee_\bk$ whose fiber over a closed point $g \in G^\vee_\bk$ is the scheme-theoretic centralizer $\rmZ_{G^\vee_\bk}(g)$. Its restriction to the regular locus in $G^\vee_\bk$ is smooth. We consider a ``Steinberg section'' $\Sigma \subset G^\vee_\bk$, a certain section of the adjoint quotient $G^\vee_\bk \to G^\vee_\bk / G^\vee_\bk \cong T^\vee_\bk / \Wf$ which consists of regular elements. (Here, $T^\vee_\bk$ is the canonical maximal torus in $G^\vee_\bk$, and $\Wf$ is the Weyl group of $(G^\vee_\bk, T^\vee_\bk)$, which identifies with that of $(G,T)$.) The restriction $\bbJ_\Sigma$ of the universal centralizer to $\Sigma$ is therefore a smooth affine group scheme over the affine scheme $\Sigma$. We consider the fiber product $T^\vee_\bk \times_{T^\vee_\bk / \Wf} T^\vee_\bk$, seen as a scheme over $\Sigma$ via the identification $\Sigma \cong T^\vee_\bk / \Wf$, and set
\[
\bbI_\Sigma := (T^\vee_\bk \times_{T^\vee_\bk / \Wf} T^\vee_\bk) \times_\Sigma \bbJ_\Sigma,
\]
a smooth affine group scheme over $T^\vee_\bk \times_{T^\vee_\bk / \Wf} T^\vee_\bk$. We consider the category
\[
\Rep_0(\bbI_{\Sigma})
\]
of representations of 
$\bbI_{\Sigma}$
whose underlying $\scO_{T^\vee_\bk \times_{T^\vee_\bk / \Wf} T^\vee_\bk}$-module is coherent and set-theore\-tically supported on $\{(e,e)\} \subset T^\vee_\bk \times_{T^\vee_\bk / \Wf} T^\vee_\bk$. The tensor product of $\scO(T^\vee_\bk)$-bimodules endows this category with a monoidal product $\circledast$.

The point in $\Sigma$ corresponding to the image of the unit $e \in T^\vee_\bk$ in $T^\vee_\bk / \Wf$ is a regular unipotent element $\su$, and by construction the fiber of 
$\bbI_\Sigma$
over $(e,e)$ is $\rmZ_{G^\vee_\bk}(\su)$. We therefore have a natural exact monoidal functor
\begin{equation}
\label{eqn:functor-Rep-intro}
\Rep(\rmZ_{G^\vee_\bk}(\su)) \to 
\Rep_0(\bbI_{\Sigma})
\end{equation}
induced by pushforward along the embedding $\{(e,e)\} \hookrightarrow T^\vee_\bk \times_{T^\vee_\bk / \Wf} T^\vee_\bk$.

The first main result of the present paper is the following claim.

\begin{thm}
\label{thm:main-intro-1}
Assume that the conditions in Theorem~\ref{thm:brr-intro} are satisfied.
There exists an equivalence of monoidal categories
\[
\Phi_{\Iwu,\Iwu} : \sfP_{\Iwu,\Iwu}^0 \simto 
\Rep_0(\bbI_{\Sigma})
\]
such that the following diagram commutes:
\[
\xymatrix@C=2cm{
\sfP^0_{\Iw,\Iw} \ar[r]^-{\Phi_{\Iw,\Iw}}_-{\sim} \ar[d]_-{\eqref{eqn:functor-P0-intro}} & \Rep(\rmZ_{G^\vee_\bk}(\su)) \ar[d]^-{\eqref{eqn:functor-Rep-intro}} \\
\sfP^0_{\Iwu,\Iwu} \ar[r]^-{\Phi_{\Iwu,\Iwu}}_-{\sim} & 
\Rep_0(\bbI_{\Sigma}).
}
\]
\end{thm}

The general result that was the basis for the proof of Theorem~\ref{thm:brr-intro} has no version for group schemes over a base (as far as we know). Our proof of Theorem~\ref{thm:main-intro-1} therefore follows a different, and more specific, approach. We make the functor appearing in Theorem~\ref{thm:brr-intro} more explicit, and then provide an explicit ``deformation'' of this functor by deforming each of its constituents. Precisely defining these deformations requires recalling and generalizing a number of known tools in this domain (in particular, from~\cite{be}), and proving that the functor constructed in this way has the expected properties turns out to be long and technical; these tasks occupy a large part of the paper.

The functor $\Phi_{\Iwu,\Iwu}$ also admits some compatibility properties with an appropriate version of the central functor $\sfZ$. These properties are however more difficult to spell out, and will not be discussed in this introduction.

\begin{rmk}
\begin{enumerate}
\item
The action of $G$ on the base point in $\tFl_G$ induces a closed embedding $G/U \hookrightarrow \tFl_G$. As in the definition of $\sfD_{\Iwu,\Iwu}$, we define the category $\sfD_{U,U}$ as the triangulated subcategory of the derived category of $U$-equivariant constructible $\bk$-sheaves on $G/U$ generated by objects obtained by pullback from $U$-equivariant complexes on $G/B$. This category admits a natural perverse t-structure, and the simple objects in its heart $\sfP_{U,U}$ are in bijection with $\Wf$. We can then define $\sfP^0_{U,U}$ as the Serre quotient of $\sfP_{U,U}$ by the Serre subcategory generated by simple objects whose label has positive length. A variation on the constructions of~\cite{bezr} (explained in Section~\ref{sec:Perv-G/U}) provides an equivalence of monoidal categories between $\sfP^0_{U,U}$ and the category $\Coh_0(T^\vee_\bk \times_{T^\vee_\bk / \Wf} T^\vee_\bk)$ of coherent sheaves on $T^\vee_\bk \times_{T^\vee_\bk / \Wf} T^\vee_\bk$ which are set-theoretically supported on $\{(e,e)\}$. The equivalence $\Phi_{\Iwu,\Iwu}$ is also compatible with this equivalence, using the functor $\sfP^0_{U,U} \to \sfP^0_{\Iwu,\Iwu}$ induced by pushforward along the closed embedding $G/U \hookrightarrow \tFl_G$ and the functor $\Coh_0(T^\vee_\bk \times_{T^\vee_\bk / \Wf} T^\vee_\bk) \to 
\Rep_0(\bbI_{\Sigma})$
sending a coherent sheaf to itself with the trivial structure as a representation.
\item
A variation on the proof of Theorem~\ref{thm:main-intro-1} also provides a version of this theorem describing the similar Serre quotient of the category of $\Iwu$-equivariant perverse sheaves on $\Fl_G$; see Theorem~\ref{thm:main-Fl} for a precise statement.
\end{enumerate}
\end{rmk}

\subsection{Description of tilting perverse sheaves}

As explained in~\S\ref{ss:mon-reg-quot-intro}, one motivation for considering the category $\sfP^0_{\Iwu,\Iwu}$ is that one can reconstruct the category $\sfD_{\Iwu,\Iwu}$ out of it. This can however not be done directly as in the case of $\sfD_{\Iwu,\Iw}$, essentially because there is no appropriate notion of tilting perverse sheaves in this category. For this one needs to use a ``completed'' version $\sfD^\wedge_{\Iwu,\Iwu}$ of $\sfD_{\Iwu,\Iwu}$, constructed following a procedure developed by Yun in an appendix to~\cite{by}. The category $\sfD^\wedge_{\Iwu,\Iwu}$ is triangulated, it contains $\sfD_{\Iwu,\Iwu}$ as a full triangulated subcategory, and it is endowed with a perverse t-structure whose heart is denoted $\sfP^\wedge_{\Iwu,\Iwu}$. In $\sfP^\wedge_{\Iwu,\Iwu}$ we have an appropriate notion of tilting objects; the full subcategory of tilting perverse sheaves will be denoted $\sfT^\wedge_{\Iwu,\Iwu}$. We have a canonical equivalence of categories
\begin{equation}
\label{eqn:equiv-KbT-D-intro}
\Kb \sfT^\wedge_{\Iwu,\Iwu} \simto \sfD^\wedge_{\Iwu,\Iwu},
\end{equation}
so that $\sfD^\wedge_{\Iwu,\Iwu}$ can be reconstructed from $\sfT^\wedge_{\Iwu,\Iwu}$.

The monoidal product $\star_{\Iwu}$ admits a natural ``extension'' to a (triangulated) monoidal product $\hatstar$ on $\sfD^\wedge_{\Iwu,\Iwu}$, which stabilizes the subcategory $\sfT^\wedge_{\Iwu,\Iwu}$. With these structures, \eqref{eqn:equiv-KbT-D-intro} becomes an equivalence of monoidal categories.

On the coherent side, we consider the 
spectrum $\FN_{T^\vee_\bk \times_{T^\vee_\bk / \Wf} T^\vee_\bk}(\{(e,e)\})$ of the completion of $\scO(T^\vee_\bk \times_{T^\vee_\bk / \Wf} T^\vee_\bk)$ with respect to the ideal of the point $(e,e)$. We also set
\[
\bbI_\Sigma^\wedge := \FN_{T^\vee_\bk \times_{T^\vee_\bk / \Wf} T^\vee_\bk}(\{(e,e)\}) \times_{T^\vee_\bk \times_{T^\vee_\bk / \Wf} T^\vee_\bk} 
\bbI_{\Sigma}.
\]
The category $\Rep(\bbI_\Sigma^\wedge)$ of representations of the group scheme $\bbI_\Sigma^\wedge$ whose underlying $\scO_{\FN_{T^\vee_\bk \times_{T^\vee_\bk / \Wf} T^\vee_\bk}(\{(e,e)\})}$-module is coherent admits once again a canonical monoidal product $\circledast$. It also contains natural objects $\scB^\wedge_s$ attached to simple reflections $s \in S$, and natural objects $\scM^\wedge_\omega$ attached to length-$0$ elements $\omega$ in $W$. We define the category $\SRep(\bbI_\Sigma^\wedge)$ of ``Soergel representations'' of $\bbI_\Sigma^\wedge$ as the full subcategory generated under $\circledast$, direct sums and direct summands by the objects $\scB^\wedge_s$ and $\scM^\wedge_\omega$.

The second main result of the paper is the following claim.

\begin{thm}
\label{thm:main-intro-2}
Assume that the conditions in Theorem~\ref{thm:brr-intro} are satisfied.
There exists an equivalence of monoidal categories
\[
(\sfT^\wedge_{\Iwu,\Iwu}, \hatstar) \simto (\SRep(\bbI_\Sigma^\wedge), \circledast).
\]
\end{thm}

From the description in Theorem~\ref{thm:main-intro-2} one obtains (at least in theory) a description of the category $\sfD^\wedge_{\Iwu,\Iwu}$ and its subcategory $\sfD_{\Iwu,\Iwu}$, and also of the category $\sfD_{\Iwu,\Iw}$. 
Making this description explicit, in terms of coherent sheaves on the Steinberg variety of $G^\vee_\bk$, is the subject of~\cite{br-pt3}.

The proof of Theorem~\ref{thm:main-intro-2} proceeds by describing each side of the equivalence in terms of the categories appearing in Theorem~\ref{thm:main-intro-1}. (This description is similar to the description of finitely generated modules over the completion $A^\wedge$ of a noetherian ring $A$ with respect to an ideal $I$ in terms of sequences of modules over the quotients $A/I^m$ for $m \geq 0$.) On the coherent side, this uses restriction to the various infinitesimal neighborhoods of the preimage of the base point in $T^\vee_\bk / \Wf$. On the perverse side, this uses some ``truncation'' functors which ``kill'' powers of the ideal of this base point acting by monodromy.

\begin{rmk}
In this case also the proof of Theorem~\ref{thm:main-intro-2} can be adapted to provide a description of the category of $\Iwu$-equivariant tilting perverse sheaves on $\Fl_G$; see Theorem~\ref{thm:main-application-Fl} for a precise statement.
\end{rmk}

\subsection{Contents}

In Section~\ref{sec:Coh-St} we prove or recall a number of facts regarding the geometry of the (multiplicative) Steinberg variety. In particular we explain the construction of the Steinberg section and recall some facts about the universal centralizer group scheme. In Section~\ref{sec:Hecke-cat} we explain various incarnations of the ``Hecke category'' attached to an affine Weyl group.

Sections~\ref{ss:const-Fl}--\ref{sec:fmZ} are devoted to reminders on essentially known constructions regarding categories of constructible sheaves on affine flag varieties. More specifically, in Section~\ref{ss:const-Fl} we introduce these categories, recall Gaitsgory's construction relating the Satake category to Iwahori-equivariant perverse sheaves on the affine flag variety, and recall the construction of the main player of~\cite{brr} (the equivariant regular quotient) and the main result of that paper (an equivalence of categories relating this equivariant regular quotient to representations of the centralizer of a regular unipotent element in the dual group). In Section~\ref{sec:construction-mon-reg-quot} we introduce one of the main players of the present paper, the ``monodromic regular quotient'' category. In Section~\ref{sec:fmps} we recall Yun's construction of the ``free-monodromic derived category'' and its perverse t-structure. Finally, in Section~\ref{sec:fmZ} we discuss a free-monodromic variant of Gaitsgory's construction, following a similar construction in~\cite{be}.

Sections~\ref{sec:Perv-G/U}--\ref{sec:Soergel-type} contain our main constructions, and the proofs of our main results. First, in Section~\ref{sec:Perv-G/U} we reinterpret the results of~\cite{bezr} from a slightly different perspective, thereby providing a ``reconstruction'' of the base scheme $T^\vee_\bk \times_{T^\vee_\bk / \Wf} T^\vee_\bk$ from perverse sheaves on $G/U$. In Section~\ref{sec:truncation} we prove some technical results regarding a ``truncation'' operation on perverse sheaves that will be required later. In Section~\ref{sec:construction} we construct our functor $\Phi_{\Iwu,\Iwu}$ and prove a slightly more precise version of Theorem~\ref{thm:main-intro-1} (see Theorem~\ref{thm:main}). Then in Section~\ref{sec:Soergel-type} we explain how to deduce a slightly more precise version of Theorem~\ref{thm:main-intro-2} (see Theorem~\ref{thm:main-application}).

In Section~\ref{sec:Fl} we explain variants of Theorems~\ref{thm:main-intro-1} and~\ref{thm:main-intro-2} which describe the more familiar category $\sfP_{\Iwu,\Iw}$.
The paper finishes with four appendices, each discussing a technical construction which is required in the course of our proofs.

\subsection{Notation and conventions}

Unless otherwise specified, a ``module'' will mean a left module, and a ``comodule'' will mean a right comodule. If $A$ is a ring, we will denote by $\Mod(A)$, resp.~$\Modr(A)$, the category of $A$-modules, resp.~of right $A$-modules. If $A$ is left noetherian, resp.~right noetherian, the subcategory of finitely generated modules will be denoted $\Mof(A)$, resp.~$\Mofr(A)$.

Given an affine group scheme $H$ over a noetherian scheme $X$, we will denote by $\Rep^\infty(H)$ the category of representations of $H$, and by $\Rep(H)$ the full subcategory of representations whose underlying $\scO_X$-module is coherent. In most cases we will encounter below $X$ will be affine; in this case we will identify $\Rep^\infty(H)$ with the category of comodules over the $\scO(X)$-Hopf algebra $\scO(H)$, and $\Rep(H)$ with the subcategory of comodules which are finitely generated as $\scO(X)$-modules.

If $X=\Spec(A)$ is an affine scheme and $Y \subset X$ is a closed subscheme, defined by an ideal $I \subset A$, we will denote by $\FN_X(Y)$ the spectrum of the completion of $A$ with respect to $I$. (Here, ``$\FN$'' stands for ``formal neighborhood.'') In this setting we have a canonical morphism of schemes $\FN_X(Y) \to X$.

In this paper we will make extensive use of the theory of pro-objects and ind-objects in (locally small)\footnote{By default, in the body of the paper, by ``category'' we mean a \emph{locally small} category.} categories, for which we refer to~\cite[Chap.~6]{ks}. (For us, all pro-objects will be parametrized by $\Z_{\geq 0}$ with the standard order. On the other hand, ind-objects will be parametrized by arbitrary filtrant small categories.) In particular, we will repeatedly use the property that any functor ``extends'' in a canonical way to a functor between categories of pro-objects or ind-objects; see~\cite[Proposition~6.1.9]{ks}. We will also use the fact that the category of ind-objects in a locally small category, resp.~in an abelian category, is itself locally small, resp.~abelian; see~\cite[Lemma~6.1.2]{ks}, resp.~\cite[Theorem~8.6.5(i)]{ks}.

\subsection{Acknowledgements}

We thank P. Achar for stimulating discussions, C.~Bonnaf\'e for his help with a reference, A. Bouthier for useful discussions on Steinberg's sections, S. Cotner for enlightening exchanges on centralizers of regular elements and sending a preliminary version of~\cite{cotner}, and P.~Etingof for his explanations on a result from~\cite{egno}.

This paper is the second step of a project we initially started as a joint work with L. Rider. We thank her for her early contributions to this program.

R.B.~was supported by NSF Grant No.~DMS-1601953.
This project has received funding from the European Research Council (ERC) under the European Union's Horizon 2020 research and innovation programme (S.R., grant agreement No.~101002592). 

\section{Coherent sheaves on the Steinberg variety}
\label{sec:Coh-St}

In this section we collect a number of results on the geometry of various schemes associated with a connected reductive algebraic group, and coherent sheaves on such schemes, that will be required in later sections.

\subsection{Notation}
\label{ss:Coh-notation}

We fix an algebraically closed field $\bk$ of characteristic $\ell$, and a connected reductive algebraic group $\bG$ over $\bk$ whose derived subgroup $\mathscr{D}\bG$ is simply connected. We choose a Borel subgroup $\bB \subset \bG$ and a maximal torus $\bT \subset \bB$, and denote by $\bU$ the unipotent radical of $\bB$. The respective Lie algebras of $\bG$, $\bB$, $\bT$ will be denoted $\bg$, $\bb$, $\bt$. The Borel subgroup of $\bG$ opposite to $\bB$ (with respect to $\bT$) will be denoted $\bB^+$, and its unipotent radical will be denoted $\bU^+$. For any torus $\bH$ we will denote by $X^*(\bH)$ its lattice of characters.

We will denote by $\bWf$ the Weyl group of $(\bG,\bT)$, 
by $\fR \subset X^*(\bT)$ the root system of $(\bG,\bT)$, and by $\fR^\vee \subset X_*(\bT)$ the corresponding coroots; for any root $\beta \in \mathfrak{R}$, we will as usual denote by $\beta^\vee$ the associated coroot. The choice of $\bB$ determines a system $\mathfrak{R}_+ \subset \mathfrak{R}$ of positive roots, chosen so that the $\bT$-weights on $\mathrm{Lie}(\bU)$ are the \emph{negative} roots. The associated basis of $\mathfrak{R}$ will be denoted $\mathfrak{R}_{\mathrm{s}}$. We will denote by $X^*_+(\bT) \subset X^*(\bT)$ the submonoid of weights which are dominant with respect to $\fR_+$, and by $\preceq$ the order on $X^*(\bT)$ such that $\lambda \preceq \mu$ iff $\mu-\lambda$ is a sum of positive roots.
The choice of $\mathfrak{R}_+$ also determines a system $\bSf \subset \bWf$ of Coxeter generators; the longest element with respect to this structure will be denoted $w_\circ$. 

The following classical result of Steinberg (after earlier work of Pittie, see~\cite{steinberg-pittie}) will be crucial in later sections.

\begin{thm}
\label{thm:Pittie-Steinberg}
The $\scO(\bT/\bWf)$-module $\scO(\bT)$ is free of rank $\# \bWf$. 
\end{thm}

\begin{rmk}
\label{rmk:Pittie-Steinberg}
In~\cite{steinberg-pittie} the author assumes that the group under consideration is semisimple (and simply connected). However the proof applies in above setup, as checked in detail in~\cite[\S 10.1.1]{gouttard-these}.
\end{rmk}

For a connected reductive algebraic group $\mathbf{H}$ over $\bk$ and $g \in \mathbf{H}(\bk)$, we will denote by $\rmZ_{\mathbf{H}}(g)$ the scheme-theoretic centralizer of $g$ in $\mathbf{H}$. We will denote by $\mathbf{H}_{\mathrm{reg}} \subset \mathbf{H}$ the open subscheme of regular elements, i.e.~the unique open subscheme whose $\bk$-points are the elements $g \in \mathbf{H}$ such that $\rmZ_{\mathbf{H}}(g)$ has dimension the rank of $\bG$ (i.e.~the minimal possible dimension). If $\mathbf{H}'$ is another connected reductive group and $\varphi : \mathbf{H}' \to \mathbf{H}$ is a finite central isogeny, then for any $h \in \mathbf{H}'(\bk)$ we have $\dim \rmZ_{\mathbf{H}'}(h) = \dim \rmZ_{\mathbf{H}}(\varphi(h))$ (see e.g.~\cite[Proposition~2.3]{kurtzke}); as a consequence we have
\begin{equation}
\label{eqn:reg-pullback}
\mathbf{H}'_{\reg} = \varphi^{-1}(\mathbf{H}_{\reg}).
\end{equation}

\subsection{The adjoint quotient and Steinberg's section}
\label{ss:Steinberg-section}

Consider the adjoint quotient $\mathbf{G} / \mathbf{G}$. It is a classical fact
that the embedding $\mathbf{T} \subset \mathbf{G}$ induces an isomorphism
\begin{equation}
\label{eqn:adj-quotient}
 \mathbf{T} / \bWf \simto \mathbf{G} / \mathbf{G};
\end{equation}
see e.g.~\cite{lee} for a proof of this theorem over any commutative ring (and for any reductive group admitting a maximal torus). We will denote by
\[
\chi : \mathbf{G} \to \mathbf{T} / \bWf
\]
the composition of the adjoint quotient morphism with the identification~\eqref{eqn:adj-quotient}.

 The algebra $\scO(\bG/\bG)$ admits various bases (as a $\bk$-vector space) parametrized by $X^*_+(\bT) \subset X^*(\bT)$ and defined as follows. For any $M$ in the category $\Rep(\bG)$ of finite-dimensional algebraic $\bG$-modules, we denote by $\mathrm{ch}(M) : \bG \to \mathbb{A}^1_\bk$ the function sending $g \in \bG$ to the trace of its action on $M$. For $\lambda \in X^*_+(\bT)$ we denote by $\mathsf{L}(\lambda)$ the simple $\bG$-module with highest weight $\lambda$, i.e.~the socle of the induced module $\Ind_{\bB}^{\bG}(\lambda)$. Then $(\mathrm{ch}(\mathsf{L}(\lambda)) : \lambda \in X^*_+(\bT))$ is a $\bk$-basis of $\scO(\bG/\bG)$. More generally, for any family $(M_\lambda : \lambda \in X^*_+(\bT))$ of $\bG$-modules such that for any $\lambda$ we have
\begin{equation}
\label{eqn:multiplicities}
[M_\lambda : \mathsf{L}(\lambda)]=1 \quad \text{and} \quad [M_\lambda : \mathsf{L}(\mu)] \neq 0 \Rightarrow \mu \preceq \lambda,
\end{equation}
the family $(\mathrm{ch}(M_\lambda) : \lambda \in X^*(\bT))$ is a $\bk$-basis of $\scO(\bG/\bG)$.

Under our assumption that $\mathscr{D}\bG$ is simply connected, the adjoint quotient can be described more explicitly, as follows. First, recall that (without any assumption) the quotient $\bG/\mathscr{D}\bG$ of $\bG$ by its normal subgroup $\mathscr{D}\bG$ is a torus, whose lattice of characters is determined by the fact that the pullback under the composition
\begin{equation}
\label{eqn:proj-ab-T}
\bT \hookrightarrow \bG \to \bG/\mathscr{D}\bG
\end{equation}
provides an identification
\begin{equation}
\label{eqn:characters-abelianization}
X^*(\bG / \mathscr{D}\bG) \simto \{\lambda \in X^*(\bT) \mid \forall \alpha \in \mathfrak{R}_{\mathrm{s}}, \langle \lambda, \alpha^\vee \rangle = 0\}.
\end{equation}
If $\lambda$ belongs to the right-hand side, then the $\bG$-module $\mathsf{L}(\lambda)$ is one-dimensional, with the $\bG$-action given by the associated character of $\bG$.

On the other hand, the intersection $\bT \cap \mathscr{D}\bG$ is a maximal torus of $\mathscr{D}\bG$, and we have a surjective restriction morphism $X^*(\bT) \to X^*(\bT \cap \mathscr{D}\bG)$. For each $\alpha \in \mathfrak{R}$, the coroot $\alpha^\vee$ takes values in $\bT \cap \mathscr{D}\bG$; the map $\langle -, \alpha^\vee \rangle$ therefore factors through a map $X^*(\bT \cap \mathscr{D}\bG) \to \Z$, which identifies with the similar map for the root $\alpha_{|\bT \cap \mathscr{D}\bG}$ of $\mathscr{D}\bG$. 
For each $\alpha \in \mathfrak{R}_{\mathrm{s}}$ we have the fundamental weight $\varpi_\alpha \in X^*(\bT \cap \mathscr{D}\bG)$, which is characterized by the property that $\langle \varpi_\alpha, \beta^\vee \rangle = \delta_{\alpha,\beta}$ for $\beta \in \mathfrak{R}_s$. Let us fix, for each $\alpha \in \mathfrak{R}_{\mathrm{s}}$, a lift $\omega_\alpha \in X^*(\bT)$ of $\varpi_\alpha$. Then, using the identification~\eqref{eqn:characters-abelianization} we obtain an 
 isomorphism of $\Z$-modules
\begin{equation}
\label{eqn:identification-weights}
\Z^{\mathfrak{R}_{\mathrm{s}}} \times X^*(\bG / \mathscr{D}\bG) \simto X^*(\bT)
\end{equation}
given by
\[
((m_\alpha : \alpha \in \mathfrak{R}_{\mathrm{s}}), \lambda) \mapsto \lambda + \sum_{\alpha \in \mathfrak{R}_{\mathrm{s}}} m_\alpha \omega_\alpha,
\]
which in turn provides an identification
\begin{equation}
\label{eqn:decomposition-T}
\bT \cong (\Gm)^{\mathfrak{R}_{\mathrm{s}}} \times \bG / \mathscr{D}\bG,
\end{equation}
such that $(\Gm)^{\mathfrak{R}_{\mathrm{s}}} \times \{e\}$ corresponds to $\bT \cap \mathscr{D}\bG$ and the projection $\bT \to \bG/\mathscr{D}\bG$ coincides with~\eqref{eqn:proj-ab-T}.

For $\alpha \in \mathfrak{R}_{\mathrm{s}}$, we set $\chi_\alpha:=\mathrm{ch}(\mathsf{L}(\omega_\alpha)) : \bG \to \mathbb{A}^1_\bk$.

\begin{lem}
\label{lem:adj-quotient-sc}
The morphisms $(\chi_\alpha : \alpha \in \mathfrak{R}_{\mathrm{s}})$ and the projection $\bG \to \bG/\mathscr{D}\bG$ induce an isomorphism
\[
\bG/\bG \simto \mathbb{A}_\bk^{\fR_{\mathrm{s}}} \times (\bG/\mathscr{D}\bG).
\]
\end{lem}

\begin{proof}
Since $\bG/\mathscr{D}\bG$ is a torus we have
\[
\scO(\bG/\mathscr{D}\bG)=\bigoplus_{\nu \in X^*(\bG/\mathscr{D}\bG)} \bk \nu,
\]
and our morphism is induced by the morphism
\[
\scO(\bG/\mathscr{D}\bG) \otimes_\bk \bk[X_\alpha : \alpha \in \mathfrak{R}_{\mathrm{s}}] \to \scO(\bG/\bG)
\]
sending $\nu \otimes \prod_{\alpha} X_\alpha^{m_\alpha}$ to
\begin{equation}
\label{eqn:characters-quotient}
\mathrm{ch} \bigl( \mathsf{L}(\nu) \otimes \bigotimes_\alpha \mathsf{L}(\omega_\alpha)^{\otimes m_\alpha} \bigr).
\end{equation}
Now~\eqref{eqn:identification-weights} restricts to a bijection
\[
(\Z_{\geq 0})^{\mathfrak{R}_{\mathrm{s}}} \times X^*(\bG/\mathscr{D}\bG) \simto X_+^*(\bT).
\]
The family of characters~\eqref{eqn:characters-quotient} can therefore be considered as parametrized by $X_+^*(\bT)$. As such, the corresponding family of $\bG$-modules satisfies the conditions spelled out in~\eqref{eqn:multiplicities}, and these characters therefore form a basis of $\scO(\bG/\bG)$. Our algebra morphism sends a $\bk$-basis to a $\bk$-basis, hence is an isomorphism.
\end{proof}

We now explain how to construct a ``Steinberg section'' for $\chi$, i.e.~a closed subscheme
$\mathbf{\Sigma} \subset \bG$ contained in $\bG_\reg$ such that the composition
\[
\mathbf{\Sigma} \hookrightarrow \bG \xrightarrow{\chi} \bT / \bWf
\]
is an isomorphism. (This construction is due to Steinberg~\cite{steinberg} in the case $\bG$ is semisimple; the extension to reductive groups is due to De Concini--Maffei~\cite{dcm}.) Let us fix a numbering $(\alpha_1, \cdots, \alpha_r)$ of $\fR_{\mathrm{s}}$. For $i \in \{1, \cdots, r\}$, we will denote by $\bU_{\alpha_i}$ and $\bU_{-\alpha_i}$ the root subgroups of $\bG$ associated with $\alpha_i$ and $-\alpha_i$ respectively. We will also chose a lift $n_i \in \mathrm{N}_{\bG}(\bT)$ of the simple reflection $s_i \in \bWf$ associated with $\alpha_i$ which belongs to $\mathscr{D}\bG$.
Let us denote by $\bA \subset \bT$ the subtorus given by the image of $\{e\} \times \bG/\mathscr{D}\bG$ under the identification~\eqref{eqn:decomposition-T}. We then set
\[
\mathbf{\Sigma} := \bA \cdot \bU_{\alpha_1} n_1 \cdot (\cdots) \cdot \bU_{\alpha_r} n_r.
\]
Standard properties of the Bruhat decomposition (see e.g.~\cite[\S 4.15]{humphreys} for details) show that the map
\[
(u_1, \cdots, u_r) \mapsto u_1 n_1 u_2 n_2 \cdots u_r n_r \cdot (n_1 \cdots n_r)^{-1}
\]
induces a closed embedding $\bU_{\alpha_1} \times \cdots \times \bU_{\alpha_r} \hookrightarrow \bU^+$; this shows that $\mathbf{\Sigma}$ is a closed subscheme of $\bG$, isomorphic to $\bA \times \mathbb{A}_\bk^r$. The other properties of $\mathbf{\Sigma}$ announced above are proved in the following proposition.

\begin{prop}
\phantomsection
\label{prop:Sigma}
\begin{enumerate}
\item
\label{it:Sigma-reg}
We have $\mathbf{\Sigma} \subset \bG_\reg$.
\item
\label{it:Sigma-section}
The composition
\[
\mathbf{\Sigma} \hookrightarrow \bG \xrightarrow{\chi} \bT / \bWf
\]
is an isomorphism.
\end{enumerate}
\end{prop}

\begin{proof}
\eqref{it:Sigma-reg}
It is clear that we have
\[
\mathbf{\Sigma} = \bU_{\alpha_1} n_1 \cdot (\cdots) \cdot \bU_{\alpha_r} n_r \cdot ((n_1 \cdots n_r)^{-1} \bA n_1 \cdots n_r) \subset \bU_{\alpha_1} n_1 \cdot (\cdots) \cdot \bU_{\alpha_r} n_r \cdot \bT,
\]
so that it is enough to prove that the right-hand side is contained in $\bG_\reg$. Now if $\mathbf{Z}$ is the neutral component of the reduced center $\mathrm{Z}(\bG)_\mathrm{red}$, as explained in~\cite[\S 1.18]{jantzen} multiplication induces a finite central isogeny
\[
\varphi : \mathbf{Z} \times \mathscr{D}\bG \to \bG.
\]
By~\eqref{eqn:reg-pullback} we have
\[
\varphi^{-1}(\bG_\reg) = (\mathbf{Z} \times \mathscr{D}\bG)_\reg = \mathbf{Z} \times (\mathscr{D}\bG)_\reg.
\]
Now by Steinberg's results for semisimple groups (see~\cite{steinberg} and~\cite[\S 4.20]{humphreys}) we have
\[
\bU_{\alpha_1} n_1 \cdot (\cdots) \cdot \bU_{\alpha_r} n_r \cdot (\bT \cap \mathscr{D}\bG) \subset (\mathscr{D}\bG)_\reg.
\]
(Steinberg's results involve elements in $\bU_{\alpha_1} n_1 \cdot (\cdots) \cdot \bU_{\alpha_r} n_r$, but the choice of the elements $n_i$ can be arbitrary; for the various choices of these elements, the corresponding ``Steinberg sections'' cover $\bU_{\alpha_1} n_1 \cdot (\cdots) \cdot \bU_{\alpha_r} n_r \cdot (\bT \cap \mathscr{D}\bG)$.)
We deduce that
\[
\varphi^{-1}(\bU_{\alpha_1} n_1 \cdot (\cdots) \cdot \bU_{\alpha_r} n_r \cdot \bT) \subset \varphi^{-1}(\bG_\reg),
\]
and then that $\bU_{\alpha_1} n_1 \cdot (\cdots) \cdot \bU_{\alpha_r} n_r \cdot \bT \subset \bG_\reg$, as desired.

\eqref{it:Sigma-section}
Let us fix, for any $i \in \{1, \cdots, r\}$, an isomorphism $u_{\alpha_i} : \mathbb{A}^1_\bk \simto \bU_{\alpha_i}$. Then we have an isomorphism
\[
\bA \times \mathbb{A}_\bk^r \simto \mathbf{\Sigma},
\]
given by $(a,(c_1, \cdots, c_r)) \mapsto a u_{\alpha_1}(c_1) n_1 u_{\alpha_2}(c_2) n_2 \cdots u_{\alpha_r}(c_r) n_r$. Using this isomorphism and that of Lemma~\ref{lem:adj-quotient-sc}, one can consider the morphism of the proposition as a morphism from $\bA \times \mathbb{A}_\bk^r$ to itself. It is clear from definitions that its composition with the projection $\bA \times \mathbb{A}_\bk^r \to \bA$ coincides with this projection.

Now we consider the composition of our morphism with projection on $\mathbb{A}_\bk^r$. For this, we define a partial order $\sqsubseteq$ on $\{1, \cdots, r\}$ by declaring that $i \sqsubset j$ if there exists a dominant weight $\lambda$ for $(\mathscr{D}\bG, \bT \cap \mathscr{D}\bG)$ such that $\varpi_j-\lambda$ is a sum of positive roots and $\langle \lambda, \alpha_i^\vee \rangle > 0$. (Here, the positive roots for $\mathscr{D}\bG$ are taken as the restrictions of those for $\bG$. For an explanation of why this defines an order, see~\cite[\S 4.16]{humphreys}.) For any $i \in \{1, \cdots, r\}$, $a \in \bA$ and $(c_1, \cdots, c_r) \in \mathbb{A}_\bk^r$, the value of $\chi_{\alpha_i}(a u_{\alpha_1}(c_1) n_1 u_{\alpha_2}(c_2) n_2 \cdots u_{\alpha_r}(c_r) n_r)$ can be computed as the sum (over the weights $\lambda$ of $\mathsf{L}(\omega_i)$) of the traces of the linear maps
\[
\mathsf{L}(\omega_i)_\lambda \hookrightarrow \mathsf{L}(\omega_i) \xrightarrow{a u_{\alpha_1}(c_1) n_1 u_{\alpha_2}(c_2) n_2 \cdots u_{\alpha_r}(c_r) n_r} \mathsf{L}(\omega_i) \twoheadrightarrow \mathsf{L}(\omega_i)_\lambda
\]
where the left, resp.~right, morphism is the embedding of, resp.~projection on, the $\lambda$-weight space of $\mathsf{L}(\omega_i)$ (parallel to other weight spaces). The discussion in~\cite[\S 4.17]{humphreys} shows that this morphism vanishes unless $\lambda$ is dominant, and that
\begin{enumerate}
\item
if $\lambda=\omega_i$, there exists $d_i \in \bk^\times$ such that the trace is $d_i c_i \omega_i(a)$;
\item
otherwise, there exists a polynomial $P_\lambda \in \bk[X_j : j \sqsubset i]$ (depending only on $\lambda$) such that the trace is $\lambda(a) P_\lambda((c_j)_{j \sqsubset i})$.
\end{enumerate}

From this analysis we see that the algebra morphism $\scO(\bA \times \mathbb{A}_\bk^r) \to \scO(\bA \times \mathbb{A}_\bk^r)$ induced by our morphism of schemes $\bA \times \mathbb{A}_\bk^r \to \bA \times \mathbb{A}_\bk^r$ is an isomorphism, so that the latter morphism is an isomorphism too.
\end{proof}

\subsection{Application to smoothness results}
\label{ss:smoothness}

For any closed point $g \in \bG$, we will denote by $\rmZ_\bG(g) \subset \bG$ is the scheme-theoretic centralizer of $g$ in $\bG$.
By~\cite[III, \S 3, Proposition~5.2]{dg},
the morphism $h \mapsto hgh^{-1}$ factors through a locally closed immersion $\bG/\rmZ_{\bG}(g) \to \bG$, whose image is denoted $\mathcal{O}(g)$ (and called the adjoint orbit of $g$); it is a smooth locally closed subscheme in $\bG$, whose set of $\bk$-points is the conjugacy class of $g$ in the usual sense (see~\cite[III, \S 1, Remarque~1.15]{dg}). In particular, this definition coincides with that given e.g.~in~\cite[\S 1.5]{humphreys}.

We will denote by $\chi_{\reg}$ the restriction of $\chi$ to $\mathbf{G}_{\mathrm{reg}}$.
As a first application of the construction of the Steinberg section we prove the following claim.

\begin{prop}
\label{prop:adj-quot-smooth}
The morphism $\chi_{\reg}$ is smooth. Moreover, for any $x \in \bG_\reg$ we have
\[
(\chi_{\reg})^{-1}(\chi(x)) = \mathcal{O}(x).
\]
\end{prop}

\begin{proof}
By a classical characterization of smooth morphisms (see e.g.~\cite[Proposition~III.10.4]{hartshorne}), to prove smoothness it suffices to prove that the differential $d_g(\chi)$ is surjective for any closed point $g \in \bG_\reg$. This property is true by Proposition~\ref{prop:Sigma}\eqref{it:Sigma-section} if $g \in \mathbf{\Sigma}$, hence if $g$ is a conjugate of an element of $\mathbf{\Sigma}$. Now any fiber of $\chi$ contains exactly one regular conjugacy class (see~\cite[\S 4.14]{humphreys}), and it also contains an element of $\mathbf{\Sigma}$, which is regular by Proposition~\ref{prop:Sigma}\eqref{it:Sigma-reg}. It follows that any regular element in $\bG$ is conjugate to an element of $\mathbf{\Sigma}$, which finishes the proof that $\chi_{\reg}$ is smooth. Once this is known, we know that $\mathcal{O}(g)$ and $(\chi_{\reg})^{-1}(\chi(x))$ are smooth, and that the morphism $\mathcal{O}(g) \to (\chi_{\reg})^{-1}(\chi(x))$, which is a locally closed immersion (see~\cite[\href{https://stacks.math.columbia.edu/tag/07RK}{Tag 07RK}]{stacks-project}) is a bijection on $\bk$-points; it is therefore an equality.
\end{proof}


We now consider smoothness of centralizers of regular elements in $\bG$. For this we will have to assume that 
the (scheme-theoretic) center $\rmZ(\bG)$ is smooth; by~\cite[Lemma~2.1]{brr}, this is equivalent to requiring that $X^*(\bT)/\Z\fR$ has no $\ell$-torsion. The following statement is an immediate consequence of~\cite[Corollary~3.5]{cotner}.

\begin{lem}
\label{lem:smoothness-centralizers}
Assume that 
$\mathrm{Z}(\bG)$ is smooth. Then for any $g \in \bG_\reg$ the centralizer $\rmZ_\bG(g)$ is smooth.
\end{lem}


\begin{rmk}
Under the additional assumption that $\ell$ is good for $\bG$, Lemma~\ref{lem:smoothness-centralizers} can also be deduced from the results of~\cite{herpel}.
\end{rmk}




The following smoothness result will also be crucial below.

\begin{prop}
\label{prop:smoothness-conjugation}
Assume that 
$\rmZ(\bG)$ is smooth. Then the morphism
\[
\bG \times \mathbf{\Sigma} \to \bG_{\reg}
\]
defined by $(g,s) \mapsto gsg^{-1}$ is smooth and surjective.
\end{prop}

\begin{proof}
To prove smoothness of our morphism,
as in the proof of Proposition~\ref{prop:adj-quot-smooth}, what we need to show is its differential at any closed point of $\bG \times \mathbf{\Sigma}$ is surjective. By $\bG$-equivariance, it suffices to do so at points of the form $(e,s)$ with $s \in \mathbf{\Sigma}$. For such $s$, from Proposition~\ref{prop:Sigma}\eqref{it:Sigma-reg} and Proposition~\ref{prop:adj-quot-smooth} we obtain that the differential $d_s(\chi)$ is surjective, and that its kernel is the tangent space $T_s(\mathcal{O}(s))$. Now since the composition $\mathbf{\Sigma} \to \bG \to \bT/\bWf$ is an isomorphism (see Proposition~\ref{prop:Sigma}\eqref{it:Sigma-section}), $T_s(\mathbf{\Sigma})$ is a complement to the kernel of $d_s(\chi)$, which implies that
\begin{equation}
\label{eqn:tangent-spaces}
T_s(\bG) = T_s(\mathcal{O}(s)) \oplus T_s(\mathbf{\Sigma}).
\end{equation}

The differential of the morphism in the statement is the sum of the differential at $e$ of the morphism $\bG \to \bG$ given by $g \mapsto gsg^{-1}$ and the embedding $T_s(\mathbf{\Sigma}) \to T_s(\bG)$. The first of these morphisms can be described as the composition
\[
T_e(\bG) \to T_s(\mathcal{O}(s)) \to T_s(\bG),
\]
where the first morphism is the differential of the morphism $\bG \to \mathcal{O}(s)$ given by $g \mapsto gsg^{-1}$. The latter morphism identifies with the quotient morphism $\bG \to \bG/\rmZ_\bG(s)$, which is smooth by Lemma~\ref{lem:smoothness-centralizers} and the comments in~\cite[\S 2.1]{brr}. Its differential is therefore surjective, which finishes the proof in view of~\eqref{eqn:tangent-spaces}.

Once we know that our morphism is smooth, we know that its image is open (see~\cite[\href{https://stacks.math.columbia.edu/tag/01UA}{Tag 01UA}]{stacks-project}), so that to prove surjectivity it suffices to prove that this image contains all closed points of $\bG_{\reg}$. This property was observed in the course of the proof of Proposition~\ref{prop:adj-quot-smooth}.
\end{proof}

\begin{rmk}
Proposition~\ref{prop:smoothness-conjugation} does \emph{not} hold in general if $\rmZ(\bG)$ is not smooth. For instance, explicit computation shows that when $\bG=\mathrm{SL}_{2,\bk}$ the morphism under consideration is not smooth in characteristic $2$.
\end{rmk}

Let us note the following consequence of Proposition~\ref{prop:smoothness-conjugation}, which will be used in Section~\ref{sec:construction}.

\begin{cor}
\label{cor:Tor-vanishing}
Assume that $\rmZ(\bG)$ is smooth.
 Consider the action of $\bG$ on itself by conjugation, and the induced action on the algebra $\scO(\bG)$. For any $\bG$-equivariant $\scO(\bG)$-module $M$ and any $n \in \Z_{> 0}$ we have $\mathrm{Tor}_n^{\scO(\bG)}(M,\scO(\mathbf{\Sigma}))=0$.
\end{cor}

\begin{proof}
Since $\bG$ is affine, the category of $\bG$-equivariant $\scO(\bG)$-modules is equivalent to the category $\QCoh^{\bG}(\bG)$ of $\bG$-equivariant quasi-coherent sheaves on $\bG$.
 If we denote by $i : \mathbf{\Sigma} \to \bG$ the embedding, and consider the derived pullback functor
 \[
  Li^* : D^- \QCoh(\bG) \to D^- \QCoh(\mathbf{\Sigma}),
 \]
 the claim we want to prove is therefore equivalent to the statement that $Li^*(\scF)$ is concentrated in degree $0$ for any $\scF$ in $\QCoh^{\bG}(\bG)$. Now the morphism $i$ can be written as a composition
 \[
  \mathbf{\Sigma} \xrightarrow{j} \bG \times \mathbf{\Sigma} \to \bG
 \]
where the first morphism is given by $j(s) = (e,s)$ and the second one is the morphism of Proposition~\ref{prop:smoothness-conjugation}. Since the latter morphism is smooth (hence flat) and $\bG$-equivariant (for the action on $\bG \times \mathbf{\Sigma}$ induced by multiplication on the left on the first factor), to prove the desired statement it suffices to prove that for any $\scG$ in $\QCoh^{\bG}(\bG \times \mathbf{\Sigma})$ the complex $Lj^*(\scG)$ is concentrated in degree $0$. Now if $q : \bG \times \mathbf{\Sigma} \to \mathbf{\Sigma}$ is the morphism of projection on the second factor, the functor $q^*$ induces an equivalence of categories $\QCoh(\mathbf{\Sigma}) \simto \QCoh^{\bG}(\bG \times \mathbf{\Sigma})$; in particular, any object $\scG$ of $\QCoh^{\bG}(\bG \times \mathbf{\Sigma})$ is of the form $q^* \scM$ for some $\scM$ in $\QCoh(\mathbf{\Sigma})$, and moreover since $q$ is flat we have $q^* \scM = Lq^* \scM$. We deduce that
\[
 Lj^*(\scG) \cong Lj^* Lq^* \scM = \scM
\]
since $q \circ j=\id$; in particular, this complex is indeed concentrated in degree $0$.
\end{proof}

\subsection{Multiplicative Grothendieck, Springer and Steinberg varieties}
\label{ss:mult-Groth}

Recall from~\cite[\S 2.3]{brr} that the \emph{multiplicative Springer resolution} is the induced variety
\[
\Spr := \bG \times^\bB \bU,
\]
where $\bB$ acts on $\bU$ via the adjoint action. 
In this paper we will also consider the \emph{multiplicative Grothendieck resolution}
\[
\Groth := \mathbf{G} \times^{\mathbf{B}} \mathbf{B},
\]
where $\mathbf{B}$ acts on itself by conjugation. 
We have a natural projective morphism
\[
\nu : \Groth \to \mathbf{G},
\]
defined by $\nu([g:b])=g b g^{-1}$ for $g \in \mathbf{G}$ and $b \in \mathbf{B}$. Using this morphism we can consider the fiber product
\[
\Stm := \Groth \times_{\bG} \Groth,
\]
which we will call
the \emph{multiplicative Steinberg variety}.
We also have a canonical morphism
\[
\eta : \Groth \to \mathbf{T}
\]
sending a class $[g:b]$ to the image of $b$ in $\bB/\bU \xleftarrow{\sim} \mathbf{T}$. The embedding $\bU \subset \bB$ induces a closed embedding $\Spr \subset \Groth$, which identifies $\Spr$ with $\eta^{-1}(e)$.

If we set
\[
\Groth':= \mathbf{G} \times_{\bT/\bWf} \mathbf{T},
\]
it is a classical observation that the morphisms $\nu$ and $\eta$ combine to give a morphism of schemes
\[
\vartheta : \Groth \to \Groth'.
\]
The morphism $\nu$ obviously factors through $\vartheta$, so that we can consider the fiber product
%
%
\[
\Stm':= \Groth' \times_{\bG} \Groth' = \bG \times_{\bT/\bWf} (\bT \times_{\bT/\bWf} \bT).
\]
Using the morphism $\vartheta$ considered above
we obtain a canonical morphism $\Stm \to \Stm'$.

\subsection{Some coherent sheaves on \texorpdfstring{$\Groth$}{Groth}}
\label{ss:coh-Groth}

Let $\bH$ be an affine $\bk$-group scheme of finite type, and consider the adjoint action of $\bH$ on itself. Recall that for any $V \in \Rep(\bH)$ the $\bH$-equivariant coherent sheaf $V \otimes \scO_\bH$ on $\bH$ (where the equivariant structure is diagonal) admits a canonical ``tautological'' automorphism $\sm^{\taut}_V$ which can be described as follows. Taking global sections induces an equivalence of categories between $\Coh^{\bH}(\bH)$ and the category of $\bH$-equivariant finitely generated $\scO(\bH)$-modules; under this equivalence, $\sm^{\taut}_V$ corresponds to the composition
\[
 V \otimes \scO(\bH) \xrightarrow{\Delta_V \otimes \id} V \otimes \scO(\bH) \otimes \scO(\bH) \xrightarrow{\id \otimes m_{\scO(\bH)}} V \otimes \scO(\bH)
\]
where $\Delta_V : V \to V \otimes \scO(\bH)$ is the coaction morphism and $m_{\scO(\bH)}$ is the multiplication morphism in the ring $\scO(\bH)$. This construction is functorial in $V$, but also in $\bH$, in the sense that if $\bK$ is another affine $\bk$-group scheme and $f : \bK \to \bH$ is a morphism of $\bk$-group schemes, then the canonical isomorphism
\[
f^*(V \otimes \scO_{\bH}) \cong (\For^{\bH}_{\bK} V) \otimes \scO_{\bK}
\]
(where $\For^{\bH}_{\bK} : \Rep(\bH) \to \Rep(\bK)$ is the ``restriction'' functor associated with $f$) intertwines the automorphisms $f^* \sm^{\taut}_V$ and $\sm^{\taut}_{\For^{\bH}_{\bK} V}$.

We will consider in particular this construction in the case of the group schemes $\bG$ and $\bB$. More specifically, for any $V$ in $\Rep(\bG)$ we will consider the automorphism $\nu^*(\sm^{\taut}_V)$ of $V \otimes \scO_{\Groth}$. It is well known that restriction to
\[
\bB=\{e\} \times \bB \subset \bG \times^{\bB} \bB=\Groth
\]
induces an equivalence of categories
\begin{equation}
\label{eqn:induction-Coh-Groth}
 \Coh^{\bG}(\Groth) \simto \Coh^{\bB}(\bB)
\end{equation}
sending $V \otimes \scO_{\Groth}$ to $(\For^{\bG}_{\bB} V) \otimes \scO_{\bB}$;
under this equivalence, $\nu^*(\sm^{\taut}_V)$ identifies with $\sm^{\taut}_{\For^{\bG}_{\bB} V}$. On the other hand, we have a canonical morphism $\Groth \to \bG/\bB$, from which $V \otimes \scO_{\Groth}$ is obtained by pullback of $V \otimes \scO_{\bG/\bB}$. If for $\lambda \in X^*(\bT)$ we denote by $\scO_{\bG/\bB}(\lambda)$ the line bundle on $\bG/\bB$ associated with $\lambda$, and after choosing a completion of $\preceq$ to a total order $\leq$ on $X^*(\bT)$,
$V \otimes \scO_{\bG/\bB}$ admits a canonical filtration indexed by $(X^*(\bT),\leq)$ with associated graded
\[
 \bigoplus_{\mu \in X^*(\bT)} V_\mu \otimes \scO_{\bG/\bB}(\mu).
\]
(Here, $V_\mu$ is the $\bT$-weight space of weight $\mu$ in $V$.) As a consequence, if we denote by $\scO_{\Groth}(\lambda)$ the pullback of $\scO_{\bG/\bB}(\lambda)$ to $\Groth$, then $V \otimes \scO_{\Groth}$ admits a canonical filtration indexed by $X^*(\bT)$ with associated graded
\[
 \bigoplus_{\mu \in X^*(\bT)} V_\mu \otimes \scO_{\Groth}(\mu).
\]
Under the equivalence~\eqref{eqn:induction-Coh-Groth}, $\scO_{\Groth}(\mu)$ corresponds to $\bk_{\bB}(\mu) \otimes \scO_{\bB}$, and the filtration above is induced by the obvious filtration on $\For^{\bG}_{\bB} V$ indexed by $X^*(\bT)$ and with associated graded
\[
 \bigoplus_{\mu \in X^*(\bT)} V_\mu \otimes \bk_{\bB}(\mu).
\]
In particular, this shows that $\nu^*(\sm^{\taut}_V)$ preserves this filtration, and acts on the subquotient $V_\mu \otimes \scO_{\Groth}(\mu)$ by multiplication by the function $(\mu \circ \eta) \in \scO(\Groth)$.

Given $\lambda \in X^*_+(\bT)$,
we will say that a representation $V \in \Rep(\bG)$ has highest weight $\lambda$ if $\dim(V_\lambda)=1$ and moreover all the weights $\mu$ appearing in $V$ satisfy $\mu \preceq \lambda$.

\begin{lem}
\label{lem:kernel-end-free-coh}
 If $\lambda \in X^*_+(\bT)$ and if $V \in \Rep(\bG)$ has highest weight $\lambda$, then we have a canonical embedding
 \[
  V_{w_\circ(\lambda)} \otimes \scO_{\Groth}(w_\circ(\lambda)) \hookrightarrow V \otimes \scO_{\Groth},
 \]
whose image is $\ker(\sm^{\taut}_V - (w_\circ(\lambda) \circ \eta) \cdot \id)$.
\end{lem}

\begin{proof}
 The weight $w_\circ(\lambda)$ is minimal among the weights of $V$ (with respect to our choice of order); the desired inclusion is therefore provided by the subobject labelled by $w_\circ(\lambda)$ in our filtration on $V \otimes \scO_{\Groth}$. As explained above $\sm^{\taut}_V - (w_\circ(\lambda) \circ \eta) \cdot \id$ preserves this filtration, and acts trivially on $V_{w_\circ(\lambda)} \otimes \scO_{\Groth}(w_\circ(\lambda))$. For any weight $\mu$ of $V$ the induced action on the subquotient $V_\mu \otimes \scO_{\Groth}(\mu)$ is multiplication by the function $(\mu - w_\circ(\lambda)) \circ \eta$, which is injective if $\mu \neq w_\circ(\lambda)$; this implies that $V_{w_\circ(\lambda)} \otimes \scO_{\Groth}(w_\circ(\lambda))$ identifies with $\ker(\sm^{\taut}_V - (w_\circ(\lambda) \circ \eta) \cdot \id)$.
\end{proof}

\subsection{Regular semisimple elements}
\label{ss:rs}

We will denote by $\mathbf{G}_\circ \subset \mathbf{G}$ the open subscheme of ``\'el\'ements r\'eguliers" in the sense of~\cite[Exp.~XIII, Th\'eor\`eme~2.6]{sga3.2}. (In this case the Cartan subgroup attached to $\bT$ is $\bT$ itself.) We will denote by $\mathrm{N}_\bG(\bT)$ the (scheme-theoretic) normalizer of $\bT$ in $\bG$, which is smooth by~\cite[Exp.~XIII, Lemme~2.0]{sga3.2}. One can then consider the scheme $\mathbf{G} \times^{\mathrm{N}_{\mathbf{G}}(\mathbf{T})} \mathbf{T}$, and the natural morphism
\[
\mathbf{G} \times^{\mathrm{N}_{\mathbf{G}}(\mathbf{T})} \mathbf{T} \to \mathbf{G}.
\]
By definition this morphism restricts to an isomorphism on the preimage of $\mathbf{G}_\circ$. This preimage is $\mathbf{G} \times^{\mathrm{N}_{\mathbf{G}}(\mathbf{T})} \mathbf{T}_\circ$, where $\bT_\circ := \bG_\circ \cap \bT$ is the open subscheme of $\bT$ whose $\bk$-points are the elements $t \in \bT$ such that $\alpha(t) \neq 1$ for any $\alpha \in \mathfrak{R}$. Comparing~\cite[Exp.~XIII, Corollaire~2.5]{sga3.2} with~\cite[\S 2.3]{humphreys}, one sees that the $\bk$-points of $\bG_\circ$ are the regular semisimple elements in the usual terminology of the algebraic groups literature; in particular we have $\mathbf{G}_\circ \subset \mathbf{G}_{\mathrm{reg}}$.

Recall from~\cite[Exp.~XIII, Corollaire~2.7]{sga3.2} that $\bG_\circ$ is the open subscheme in $\bG$ defined by a certain section in $\scO(\bG)$. This section is clearly $\bG$-invariant, hence determines an open subscheme $(\bT / \bWf)_\circ$ in $\bT / \bWf \cong \bG / \bG$ (see~\eqref{eqn:adj-quotient}) such that $\bG_\circ$ is the inverse image of $(\bT / \bWf)_\circ$ in $\bG$. The inverse image of $(\bT / \bWf)_\circ$ under the quotient map $\bT \to \bT/\bWf$ is $\bT_\circ$; $(\bT / \bWf)_\circ$ therefore identifies with the quotient $\bT_\circ / \bWf$ (see~\cite[Exp.~V, Corollaire~1.4]{sga1}). Note that the inertia group (in the sense of~\cite[Exp.~V, \S 2]{sga1}) of each point in $\bT_\circ$ (with respect to the action of $\bWf$) is trivial; in fact, by the analysis at the beginning of~\cite[Exp.~V, \S 2]{sga1}, to justify this claim it suffices to prove that $\bWf$ has no fixed point in $\bT_\circ(\mathbb{K})$ for any algebraically closed extension $\mathbb{K}$ of $\bk$, which follows from the description of centralizers of semisimple elements in~\cite[\S 2.2]{humphreys} together with Steinberg's connected theorem (see~\cite[\S 2.11]{humphreys}), which applies since $\mathscr{D} \bG$ is assumed to be simply connected. From the theory reviewed in~\cite[Exp.~V, \S 2]{sga1}, it follows that we have a natural isomorphism
\begin{equation}
\label{eqn:fiber-prod-T0}
\bWf \times \bT_\circ \simto \bT_\circ \times_{\bT_\circ/\bWf} \bT_\circ
\end{equation}
defined by $(w,t) \mapsto (w \cdot t, t)$.


We set $\mathbf{B}_\circ := \mathbf{G}_\circ \cap \mathbf{B}$.

\begin{lem}
\label{lem:Brs}
The morphism defined by $(u,t) \mapsto utu^{-1}$ induces an isomorphism of schemes
\[
\mathbf{U} \times \mathbf{T}_\circ \simto \mathbf{B}_\circ.
\]
\end{lem}

\begin{proof}
We claim that $\bB_\circ$ coincides with the subset of
``\'el\'ements r\'eguliers'' in the sense of~\cite[Exp.~XIII, Th\'eor\`eme~2.6]{sga3.2} applied to the group $\bB$. (Here again the Cartan subgroup associated with $\bT$ is $\bT$ itself, but now its normalizer is again $\bT$.) Indeed, if $b$ is a $\bk$-point in $\bB_\circ$, then $b$ is ``r\'egulier'' in $\bB$ by~\cite[Exp.~XIII, Corollaire~2.8]{sga3.2}. On the other hand, if $b$ is a $\bk$-point of $\bB$ which is ``r\'egulier'' in $\bB$, then there exists $c \in \bB$ such that $cbc^{-1} \in \bT$ and $(\bb/\bt)^{cbc^{-1}}=\{0\}$, i.e.~$\alpha(cbc^{-1}) \neq 1$ for any $\alpha \in -\fR_+$. Then $cbc^{-1} \in \bT_\circ$, hence $b \in \bG_\circ$, which finishes the proof of our claim.

Now that this claim is established, we obtain from the definition that the morphism
\[
\bB \times^{\bT} \bT \to \bB
\]
induced by conjugation
restricts to an isomorphism over the preimage of $\bB_\circ$. The natural embedding $\bU \times \bT \to \bB \times^{\bT} \bT$ is an isomorphism since $\bB = \bU \rtimes \bT$, and this preimage identifies with $\bB \times^{\bT} \bT_\circ$, which finishes the proof.
\end{proof}

\subsection{Restrictions to the regular locus}
\label{ss:restriction-reg}

Recall the schemes and morphisms introduced in~\S\ref{ss:mult-Groth}. We set
\[
\Groth_{\circ}:=\nu^{-1}(\mathbf{G}_{\circ}), \quad \Groth'_\circ := \bG_{\circ} \times_{\bT/\bWf} \bT \cong \bG_{\circ} \times_{\bT_\circ/\bWf} \bT_\circ,
\]
and denote by
\[
\nu_\circ : \Groth_{\circ} \to \bG_{\circ}, \quad \vartheta_\circ : \Groth_{\circ} \to \Groth'_{\circ}
\]
the restrictions of $\nu$ and $\vartheta$ respectively.
Similarly, we set
\[
\Groth_{\mathrm{reg}}:=\nu^{-1}(\mathbf{G}_{\mathrm{reg}}), \quad \Groth'_\reg := \bG_{\mathrm{reg}} \times_{\bT/\bWf} \bT,
\]
and denote by
\[
\nu_\reg : \Groth_{\mathrm{reg}} \to \bG_{\mathrm{reg}}, \quad \vartheta_\reg : \Groth_{\mathrm{reg}} \to \Groth'_{\mathrm{reg}}
\]
the restrictions of $\nu$ and $\vartheta$ respectively.

The following claim is somewhat standard, but no proof appears in the literature in the present generality, to the best of our knowledge.

\begin{prop}
\label{prop:properties-Groth}
The morphism $\vartheta_\reg : \Groth_{\mathrm{reg}} \to \Groth'_{\mathrm{reg}}$ is an isomorphism.
\end{prop}

As a preparation we prove the following claim.

\begin{lem}
\label{lem:vartheta-circ-iso}
 The morphism $\vartheta_\circ : \Groth_{\circ} \to \Groth'_{\circ}$ is an isomorphism. Moreover, the morphism
 \[
 \bG/\bT \times \bT_\circ \to \Groth_\circ
 \]
 defined by $(g\bT, t) \mapsto [g : t]$ is an isomorphism.
\end{lem}

\begin{proof}
Recall from~\S\ref{ss:rs} that the natural morphism
\[
\bG \times^{\mathrm{N}_\bG(\bT)} \bT_\circ \to \bG_\circ
\]
is an isomorphism. We deduce an isomorphism
\[
\bG_{\circ} \times_{\bT_\circ/\bWf} \bT_\circ \cong \bG \times^{\mathrm{N}_\bG(\bT)} (\bT_\circ \times_{\bT_\circ/\bWf} \bT_\circ)
\]
where $\mathrm{N}_\bG(\bT)$ acts on the first factor in $\bT_\circ \times_{\bT_\circ/\bWf} \bT_\circ$. Combining this with the isomorphism~\eqref{eqn:fiber-prod-T0}, we deduce an identification
\[
\bG_{\circ} \times_{\bT_\circ/\bWf} \bT_\circ \cong \bG \times^{\mathrm{N}_\bG(\bT)} (\mathrm{N}_{\bG}(\bT) \times^{\bT} \bT_\circ) \cong \bG \times^{\bT} \bT_\circ.
\]
Here in the right-hand side the action of $\bT$ on $\bT_\circ$ is trivial, so that this scheme identifies with $\bG/\bT \times \bT_\circ$.
On the other hand, using Lemma~\ref{lem:Brs} we obtain an identification
\[
\Groth_\circ \cong \bG \times^{\bB} \bB_\circ \cong \bG \times^{\bB} (\bB \times^{\bT} \bT_\circ) \cong \bG \times^{\bT} \bT_\circ,
\]
which finishes the proof.
\end{proof}

\begin{rmk}
Since $\bG_\circ$ is an affine open subscheme in $\bG$ (or since $\bG/\bT$ is known to be affine),  Lemma~\ref{lem:vartheta-circ-iso} implies in particular that $\Groth_\circ$ is an affine scheme.
\end{rmk}

\begin{proof}[Proof of Proposition~\ref{prop:properties-Groth}]
We follow the proof of the analogous statement for Lie algebras, see~\cite{riche-kostant}: we will prove that $\vartheta_\reg$ is finite and birational, and that its codomain is smooth and irreducible, which will imply the claim since a finite birational morphism $f:X \to Y$ of integral schemes with $Y$ normal is an isomorphism, see~\cite[\href{https://stacks.math.columbia.edu/tag/0AB1}{Tag 0AB1}]{stacks-project}. 

First, since $\chi_\reg$ is smooth (see Proposition~\ref{prop:adj-quot-smooth}) the scheme $\Groth'_{\mathrm{reg}}$ is smooth over $\bT$, hence smooth. Smoothness (hence flatness) of $\Groth'_{\mathrm{reg}}$ over $\bT$ also implies that $\Groth'_\circ = \Groth'_{\reg} \times_{\bT} \bT_\circ$ is dense in $\Groth'_\reg$, see~\cite[\href{https://stacks.math.columbia.edu/tag/081H}{Tag 081H}]{stacks-project}. (The notion of ``scheme theoretic density'' used in this statement is equivalent to density in our present setting, see~\cite[\href{https://stacks.math.columbia.edu/tag/056D}{Tag 056D}]{stacks-project}.) Now, by Lemma~\ref{lem:vartheta-circ-iso}, $\Groth'_\circ$ is isomorphic to an open subscheme in the irreducible scheme $\Groth$, hence is itself irreducible; this implies that $\Groth'_\reg$ is irreducible (see~\cite[\href{https://stacks.math.columbia.edu/tag/004W}{Tag 004W}]{stacks-project}).

By Lemma~\ref{lem:vartheta-circ-iso} again, the restriction of $\vartheta_\reg$ to the preimage of $\bG_\circ$ is an isomorphism. Since both its domain and its codomain are irreducible, this shows that this morphism is birational.

Finally we prove that $\vartheta_\reg$ is finite, i.e.~that it is proper and quasi-finite (see e.g.~\cite[Corollary 12.89]{gw}). In fact, this map is proper by~\cite[\href{https://stacks.math.columbia.edu/tag/01W6}{Tag 01W6}]{stacks-project}, since its composition with the (separated) projection $\mathbf{G}_{\mathrm{reg}} \times_{\mathbf{T}/\bWf} \mathbf{T} \to \mathbf{G}_{\mathrm{reg}}$ is proper. To prove that it is quasi-finite, by~\cite[Remark~12.16]{gw} it suffices to prove that the induced map on $\bk$-points (i.e.~closed points) has finite fibers. Now the map on $\bk$-points induced by $\nu_\reg$ has finite fibers, see~\cite[\S 4.9]{humphreys}, hence the same holds for $\vartheta_\reg$, which finishes the proof.
\end{proof}

We will denote by $\Stmreg$, resp.~$\Stmreg'$, the inverse image of $\bG_{\mathrm{reg}}$ under the canonical morphism $\Stm \to \bG$, resp.~$\Stm' \to \bG$. By Proposition~\ref{prop:properties-Groth} the morphism $\vartheta_\reg$ induces an isomorphism
\begin{equation*}
 \Stmreg \simto \Stmreg';
\end{equation*}
below we will identify these two schemes whenever convenient. (The same comment applies to $\Groth_\reg$ and $\Groth'_\reg$.)

\subsection{Universal centralizer and Steinberg section}
\label{ss:univ-centralizer}




Recall that for any separated $\bk$-scheme $X$ endowed with an action of $\bG$, the associated \emph{universal stabilizer} is the group scheme over $X$ defined as the fiber product
\[
\mathfrak{S}_{\bG,X} := (\bG \times X) \times_{X \times X} X,
\]
where the morphism $\bG \times X \to X \times X$ is defined by $(g,x) \mapsto (g \cdot x,x)$, and the morphism $X \to X \times X$ is the diagonal embedding. The projection $\mathfrak{S}_{\bG,X} \to \bG \times X$ is a closed embedding as a subgroup scheme, so that $\mathfrak{S}_{\bG,X}$ is affine over $X$ (but not flat in general). Its fiber over $x \in X$ is the scheme-theoretic stabilizer of $x$ in $\bG$. Moreover, for any $\bG$-equivariant coherent sheaf $\scF$ on $X$ there exists a canonical action of $\mathfrak{S}_{\bG,X}$ on (the underlying coherent sheaf of) $\scF$; see~\cite[\S 2.2]{mr} for details.

We will consider in particular this construction in the case $X=\bG$ with the adjoint action, and denote by $\bbJ$ the resulting group scheme. (In this case, we will often use the expression ``universal centralizer'' instead of universal stabilizer, for obvious reasons.) We will also denote by $\bbJ_\reg$, resp.~$\bbJ_\circ$, resp.~$\bbJ_{\mathbf{\Sigma}}$, the restriction of $\bbJ$ to $\bG_\reg$, resp.~to $\mathbf{G}_\circ$, resp.~to $\mathbf{\Sigma}$ (where $\mathbf{\Sigma}$ is the Steinberg section studied in~\S\ref{ss:Steinberg-section}).

\begin{rmk}
\label{rmk:section-J}
The group scheme $\bbJ$ admits a canonical section, induced by the diagonal embedding $\bG \to \bG \times \bG$. (In other words, this section sends $g \in \bG$ to the pair $(g,g)$ where the first $g$ is seen in the centralizer of the second $g$.) The identity functor of the category $\Rep^\infty(\bbJ)$ (identified with the category of $\scO(\bbJ)$-comodules) therefore admits a ``tautological'' automorphism, defined on an $\scO(\bbJ)$-comodule $M$ by the composition
\[
M \to M \otimes_{\scO(\bG)} \scO(\bbJ) \to M \otimes_{\scO(\bG)} \scO(\bG)=M
\]
where the first morphism is the coaction and the second one is induced by restriction to the canonical section.
By restriction, we deduce similar structures for $\bbJ_\reg$, $\bbJ_\circ$ and $\bbJ_{\mathbf{\Sigma}}$.
\end{rmk}

In the following statement, of course~\eqref{it:J-smooth-1} is a consequence of~\eqref{it:J-smooth-2}, but the proof will require to prove this claim first. In fact, this is the only claim that will be used in the rest of the paper;~\eqref{it:J-smooth-2} is stated only for completeness.

\begin{lem}
\label{lem:J-smooth}
Assume that $\rmZ(\bG)$ is smooth.
\begin{enumerate}
\item
\label{it:J-smooth-1}
The group scheme $\bbJ_{\mathbf{\Sigma}}$ is smooth (in particular, flat) over $\mathbf{\Sigma}$.
\item
\label{it:J-smooth-2}
The group scheme $\bbJ_\reg$ is smooth (in particular, flat) over $\bG_\reg$.
\end{enumerate}
\end{lem}

\begin{proof}
\eqref{it:J-smooth-1}
By definition, we have
\[
\bbJ_{\mathbf{\Sigma}} = \mathbf{\Sigma} \times_{\bG_\reg \times \mathbf{\Sigma}} (\bG \times \mathbf{\Sigma})
\]
where the morphism $\mathbf{\Sigma} \to \bG_\reg \times \mathbf{\Sigma}$ is defined by $s \mapsto (s,s)$, and the morphism $\bG \times \mathbf{\Sigma} \to \bG_\reg \times \mathbf{\Sigma}$ is defined by $(g,s) \mapsto (gsg^{-1},s)$. It is clear that both of these maps factor through $\bG_\reg \times_{\bT/\bWf} \mathbf{\Sigma}$, so that
\[
\bbJ_{\mathbf{\Sigma}} = \mathbf{\Sigma} \times_{\bG_\reg \times_{\bT/\bWf} \mathbf{\Sigma}} (\bG \times \mathbf{\Sigma}).
\]
Now by Proposition~\ref{prop:Sigma}\eqref{it:Sigma-section} the projection $\bG_\reg \times_{\bT/\bWf} \mathbf{\Sigma} \to \bG_\reg$ is an isomorphism, so that
\[
\bbJ_{\mathbf{\Sigma}} = \mathbf{\Sigma} \times_{\bG_\reg} (\bG \times \mathbf{\Sigma})
\]
where the map $\mathbf{\Sigma} \to \bG_\reg$ is the obvious closed embedding and the map $\bG \times \mathbf{\Sigma} \to \bG_\reg$ is defined by $(g,s) \mapsto gsg^{-1}$. The latter map is smooth by Proposition~\ref{prop:smoothness-conjugation}, hence so is the projection $\bbJ_{\mathbf{\Sigma}} \to \mathbf{\Sigma}$, which finishes the proof of our claim.

\eqref{it:J-smooth-2}
Consider the commutative diagram
\[
\xymatrix{
\bG \times \bbJ_{\mathbf{\Sigma}} \ar[r] \ar[d] & \bbJ_\reg \ar[d]\\
\bG \times \mathbf{\Sigma} \ar[r] & \bG_\reg
}
\]
where the vertical maps are induced by the structure morphisms $\bbJ_\reg \to \bG_\reg$ and $\bbJ_{\mathbf{\Sigma}} \to \mathbf{\Sigma}$, the lower horizontal arrow is the morphism of Proposition~\ref{prop:smoothness-conjugation}, and the upper horizontal arrow is defined by $(g,(h,s)) \mapsto (ghg^{-1},gsg^{-1})$ (for $g \in \bG$, $s \in \mathbf{\Sigma}$ and $h \in \rmZ_\bG(s)$). Using $\bG$-equivariance and the corresponding functors of points one checks that this diagram is cartesian. Since the lower horizontal arrow is smooth and surjective by Proposition~\ref{prop:smoothness-conjugation}, so is the upper horizontal arrow. And since the left vertical arrow is smooth by the case treated above, using~\cite[\href{https://stacks.math.columbia.edu/tag/02K5}{Tag 02K5}]{stacks-project} we obtain that the right vertical arrow is smooth, as desired.
\end{proof}

\begin{rmk}
 Lemma~\ref{lem:J-smooth} can also be deduced from~\cite[Theorem~1.3]{cotner} applied to the group scheme $\bG \times \bG_{\reg}$ over $\bG_{\reg}$ and its ``diagonal'' section.
\end{rmk}


Given a separated $\bk$-scheme $S$, a separated morphism $f : X \to S$, and an action of $\bG$ on $X$ such that $f$ is $\bG$-invariant, for any separated scheme $S'$ and any morphism $S' \to S$ we have a canonical identification
\[
\mathfrak{S}_{\bG, X \times_S S'} \simto S' \times_S \mathfrak{S}_{\bG, X}.
\]
Applying this observation in our context, we obtain that
the universal stabilizer for the $\bG$-action on $\Groth'$, resp.~on $\Stm'$, identifies with
\[
\bT \times_{\bT/\bWf} \bbJ, \quad \text{resp.} \quad (\bT \times_{\bT/\bWf} \bT) \times_{\bT/\bWf} \bbJ.
\]
In particular, the universal stabilizer for the $\bG$-action on $\Groth'_\reg$, resp.~on $\Stmreg'$, identifies with
\[
\bT \times_{\bT/\bWf} \bbJ_\reg, \quad \text{resp.} \quad (\bT \times_{\bT/\bWf} \bT) \times_{\bT/\bWf} \bbJ_\reg.
\]
Similarly, the universal stabilizer for the $\mathbf{G}$-action on $\Groth'_\circ$ identifies with
\[
\bT_\circ \times_{\bT_\circ/\bWf} \bbJ_\circ.
\]

\begin{lem}
\label{lem:morph-gp-schemes-Jreg}
There exists a canonical morphism
\[
\bT \times_{\bT/\bWf} \bbJ_\reg \to \Groth'_\reg \times \bT 
\]
of group schemes over $\Groth'_\reg$, which restricts to an isomorphism
\[
\bT_\circ \times_{\bT_\circ/\bWf} \bbJ_\circ \simto \Groth'_\circ \times \bT 
\]
over $\Groth'_\circ$.
\end{lem}

\begin{proof}
As explained above, $\bT \times_{\bT/\bWf} \bbJ_\reg$ identifies with the universal stabilizer for the action of $\bG$ on $\Groth'_\reg = \bT \times_{\bT/\bWf} \bG_\reg$. On the other hand, by Proposition~\ref{prop:properties-Groth} we have a $\bG$-equivariant isomorphism $\Groth_\reg \simto \Groth_\reg'$. In view of the natural closed immersion $\Groth \hookrightarrow \bG \times \bG/\bB$, the universal stabilizer for the $\bG$-action on $\Groth_\reg$ is contained in the subgroup of $\bG \times \Groth_\reg$ whose fiber over a point $[g:u] \in \Groth_\reg$ is $g\bB g^{-1}$; moreover, since the torus $\bT$ identifies canonically with the quotient of any Borel subgroup by its derived subgroup, there exists a canonical morphism from the latter subgroup to $\bT \times \Groth_\reg$; we deduce the desired morphism.

Over $\bG_\circ$, by Lemma~\ref{lem:vartheta-circ-iso} the natural morphism $\bG \times^{\bT} \bT \to \bG \times^{\bB} \bB$ induces an isomorphism $\bG/\bT \times \bT_\circ \simto \Groth_\circ$, which is $\bG$-equivariant if $\bG$ acts on the left-hand side via its action on $\bG/\bT$. The universal stabilizer for the action of $\bG$ on $\bG/\bT \times \bT_\circ$ identifies naturally with
\[
(\bG \times^\bT \bT) \times \bT_\circ \cong \bG/\bT \times \bT \times \bT_\circ.
\]
Under this identification, the restriction to $\bG_\circ$ of the morphism considered above identifies with the natural isomorphism
\[
\bG/\bT \times \bT \times \bT_\circ \simto (\bG/\bT \times \bT_\circ) \times \bT,
\]
which finishes the proof.
\end{proof}


We will also consider the natural
closed immersion
\begin{equation}
\label{eqn:section-Stm}
\mathbf{\Sigma} \times_{\bT/\bWf} \bT \to \Groth', \quad \text{resp.} \quad \mathbf{\Sigma} \times_{\bT/\bWf} (\bT \times_{\bT/\bWf} \bT) \to \Stm',
\end{equation}
which factors through $\Groth'_\reg$, resp.~$\Stmreg'$, and whose composition with the natural morphism $\Groth' \to \bT$, resp.~$\Stm' \to \bT \times_{\bT/\bWf} \bT$, is an isomorphism. We will also consider the group schemes
\[
\bbJ_\bT:= \bT \times_{\bT/\bWf} \bbJ_{\mathbf{\Sigma}}, \quad \bbI_{\mathbf{\Sigma}}:= (\bT \times_{\bT/\bWf} \bT) \times_{\bT/\bWf} \bbJ_{\mathbf{\Sigma}}.
\]
Here $\bbJ_\bT$ identifies with the restriction of $\bbI_{\mathbf{\Sigma}}$ to the diagonal copy of $\bT$ in $\bT \times_{\bT/\bWf} \bT$, and also with the restriction of the universal stabilizer for the $\bG$-action on $\Groth'$ to $\mathbf{\Sigma} \times_{\bT/\bWf} \bT$. Restricting the morphism of Lemma~\ref{lem:morph-gp-schemes-Jreg} to the preimage of $\mathbf{\Sigma}$ we obtain a canonical morphism
\begin{equation}
\label{eqn:morph-gp-schemes-Jreg-Sigma}
\bbJ_\bT \to (\mathbf{\Sigma} \times_{\bT/\bWf} \bT) \times \bT
\end{equation}
of group schemes over $\mathbf{\Sigma} \times_{\bT/\bWf} \bT \cong \bT$, whose restriction to $\bT_\circ$ is an isomorphism.

\subsection{Application to coherent sheaves}

The universal stabilizers for the actions of $\bG$ on $\bG_\reg$, $\Groth_\reg$ and $\St_\reg$ encode the categories of equivariant coherent sheaves on these schemes, as explained in the following proposition.

\begin{prop}
\label{prop:rest-Steinberg-section}
Assume that $\rmZ(\bG)$ is smooth.
Restriction to $\mathbf{\Sigma}$, resp.~to $\mathbf{\Sigma} \times_{\bT/\bWf} \bT$, resp.~to $\mathbf{\Sigma} \times_{\bT/\bWf} (\bT \times_{\bT/\bWf} \bT)$, induces an equivalence of abelian categories
\begin{align*}
\Coh^{\bG}(\bG_\reg) &\simto \Rep(\bbJ_{\mathbf{\Sigma}}), \\
\text{resp. } \Coh^{\bG}(\Groth_\reg) &\simto \Rep(\bbJ_\bT), \\
\text{resp. } \Coh^{\bG}(\Stmreg) &\simto \Rep(\bbI_{\mathbf{\Sigma}}).
\end{align*}
\end{prop}

\begin{proof}
The proof is similar to that of~\cite[Proposition~3.3.11]{riche-kostant}; the equivalences are obtained by applying descent theory to the natural morphisms
\[
 \bG \times \mathbf{\Sigma} \to \bG_\reg, \ \bG \times (\mathbf{\Sigma} \times_{\bT/\bWf} \bT) \to \Groth'_\reg, \ \bG \times (\mathbf{\Sigma} \times_{\bT/\bWf} (\bT \times_{\bT/\bWf} \bT)) \to \St'_\reg,
\]
which are smooth and surjective (hence faithfully flat and quasi compact) by Proposition~\ref{prop:smoothness-conjugation}, and then identifying $\Groth'_\reg$ with $\Groth_\reg$ and $\Stmreg'$ with $\Stmreg$, see~\S\ref{ss:restriction-reg}.
\end{proof}

\begin{lem}
\label{lem:line-bundles-J}
Assume that $\rmZ(\bG)$ is smooth.
 For any $\lambda \in X^*(\bT)$, the image of the restriction of $\scO_{\Groth}(\lambda)$ to $\Groth_\reg$ 
 under the equivalence
 \[
  \Coh^{\bG}(\Groth_\reg) \simto \Rep(\bbJ_\bT)
 \]
 of Proposition~\ref{prop:rest-Steinberg-section}
 is the pullback along~\eqref{eqn:morph-gp-schemes-Jreg-Sigma} of the $((\mathbf{\Sigma} \times_{\bT/\bWf} \bT) \times \bT)$-module $\scO_{\mathbf{\Sigma} \times_{\bT/\bWf} \bT} \otimes \bk_{\bT}(\lambda)$.
\end{lem}

\begin{proof}
The equivalence under consideration is induced by restriction to $\mathbf{\Sigma} \times_{\bT/\bWf} \bT \subset \Groth_\reg$. Now $\mathbf{\Sigma} \times_{\bT/\bWf} \bT$ identifies with $\bT$, hence any line bundle on this scheme is trivial by~\cite[\href{https://stacks.math.columbia.edu/tag/0BDA}{Tag 0BDA}]{stacks-project}. In particular, there exists an isomorphism of coherent sheaves
\[
\scO_{\Groth}(\lambda)_{| \mathbf{\Sigma} \times_{\bT/\bWf} \bT} \cong \scO_{\mathbf{\Sigma} \times_{\bT/\bWf} \bT};
\]
we fix a choice for this isomorphism. Now this line bundle has a canonical structure of representation of $\bbJ_\bT$; in other words it is endowed with a coaction morphism
\begin{multline*}
\Gamma( \mathbf{\Sigma} \times_{\bT/\bWf} \bT, \scO_{\Groth}(\lambda)_{| \mathbf{\Sigma} \times_{\bT/\bWf} \bT}) \\
\to \Gamma( \mathbf{\Sigma} \times_{\bT/\bWf} \bT, \scO_{\Groth}(\lambda)_{| \mathbf{\Sigma} \times_{\bT/\bWf} \bT}) \otimes_{\scO(\mathbf{\Sigma} \times_{\bT/\bWf} \bT)} \scO(\bbJ_\bT).
\end{multline*}
In view of our identification above there exists a group-like element $\varrho \in \scO(\bbJ_\bT)$, or in other words a morphism of group schemes
\[
\varrho' : \bbJ_\bT \to (\mathbf{\Sigma} \times_{\bT/\bWf} \bT) \times \Gm,
\]
such that this coaction morphism is given by $m \mapsto m \otimes \varrho$. To conclude the proof, we have to show that $\varrho'$ is the composition
\[
\bbJ_\bT \xrightarrow{\eqref{eqn:morph-gp-schemes-Jreg-Sigma}} (\mathbf{\Sigma} \times_{\bT/\bWf} \bT) \times \bT \xrightarrow{\id \times \lambda} (\mathbf{\Sigma} \times_{\bT/\bWf} \bT) \times \Gm.
\]
Since $\bbJ_{\mathbf{\Sigma}}$ is flat over $\mathbf{\Sigma}$, to prove this claim it suffices to prove that the two morphisms under consideration coincide on the open subscheme $\bT_\circ \times_{\bT/\bWf} \bbJ_{\mathbf{\Sigma}}$.

Consider the restriction $\scO_{\Groth_\circ}(\lambda)$ of $\scO_{\Groth}(\lambda)$ to $\Groth_\circ$. Under the identification $\bG/\bT \times \bT_\circ \simto \Groth_\circ$ (see Lemma~\ref{lem:vartheta-circ-iso}), this line bundle is the pullback of the line bundle $\scO_{\bG/\bT}(\lambda)$ on the affine scheme $\bG/\bT$ associated with $\lambda$. As in the proof of Lemma~\ref{lem:morph-gp-schemes-Jreg}, the universal stabilizer for the action of $\bG$ on $\bG/\bT$ identifies with $\bG/\bT \times \bT$; under this identification, the action of this group scheme on $\scO_{\bG/\bT}(\lambda)$ is via $\lambda$, or in other words corresponds to the coaction morphism
\[
\Gamma(\bG/\bT, \scO_{\bG/\bT}(\lambda)) \to \Gamma(\bG/\bT, \scO_{\bG/\bT}(\lambda)) \otimes_\bk \scO(\bT)
\]
given by $m \mapsto m \otimes \lambda$. We deduce a similar claim for the pullback of this line bundle to $\bG/\bT \times \bT_\circ$, i.e.~for $\scO_{\Groth_\circ}(\lambda)$, and then for its restriction to
\[
\Groth_\circ \cap (\mathbf{\Sigma} \times_{\bT/\bWf} \bT) = \bT_\circ \times_{\bT/\bWf} \mathbf{\Sigma},
\]
which finishes the proof.
\end{proof}

\section{Some Hecke categories}
\label{sec:Hecke-cat}

We continue 
with the setting of Section~\ref{sec:Coh-St}, assuming in addition that the center
$\rmZ(\bG)$ is smooth. 


\subsection{Affine and extended affine Weyl groups}
\label{ss:Waff}


The extended affine Weyl group of $(\bG,\bT)$ is the semidirect product
\[
\bW := \bWf \ltimes X^*(\bT).
\]
The affine Weyl group of $(\bG,\bT)$ is the subgroup
\[
\bW^{\mathrm{Cox}} := \bWf \ltimes \Z\mathfrak{R}.
\]
For $\lambda \in X^*(\bT)$, we will denote by $\st(\lambda)$ the associated element of $\bW$.
It is a standard fact that there exists a natural subset $\bS \subset \bW^{\mathrm{Cox}}$ containing $\bSf$ and such that $(\bW^{\mathrm{Cox}},\bS)$ is a Coxeter system; more precisely $\bS$ consists of the elements of $\bSf$ together with the products $\st(\beta) s_\beta$ where $\beta$ is a maximal short root. By construction, $\bWf$ is then a parabolic subgroup in $\bW^{\mathrm{Cox}}$.

If we set, for $w \in \bWf$ and $\lambda \in X^*(\bT)$,
\begin{equation}
\label{eqn:formula-length}
\ell(w\st(\lambda))= \sum_{\substack{\alpha \in \mathfrak{R}_+ \\ w(\alpha) \in \mathfrak{R}_+}} |\langle \lambda, \alpha^\vee \rangle| + \sum_{\substack{\alpha \in \mathfrak{R}_+ \\ w(\alpha) \in -\mathfrak{R}_+}} |\langle \lambda, \alpha^\vee \rangle+1|,
\end{equation}
then it is well known that the restriction of $\ell$ to $\bW^{\mathrm{Cox}}$ is the length function associated with our Coxeter generators $\bS$, and that if we set $\mathbf{\Omega} = \{w \in \bW \mid \ell(w)=0\}$ then the natural morphism
\[
\mathbf{\Omega} \ltimes \bW^{\mathrm{Cox}} \to \bW
\]
is a group isomorphism. Moreover, in this semidirect product $\mathbf{\Omega}$ acts on $\bW^{\mathrm{Cox}}$ by Coxeter group automorphisms, i.e.~it stabilizes $\bS$.

\begin{lem}
\label{lem:simple-ref-conjugation}
For any $s \in \bS \smallsetminus \bSf$, there exist $s' \in \bSf$ and $w \in \bW$ such that $\ell(ws')=\ell(w)+1$ and $s=ws'w^{-1}$.
\end{lem}

\begin{proof}
This claim is well known in case $\bG$ is semisimple (and simply connected); see \cite[Lemma 6.1.2]{riche} or \cite[Lemma 2.1.1]{bm-loc}. 
We deduce the general case as follows. Fix $s \in \bS \smallsetminus \bSf$, and set $\bW_{\mathrm{der}}:=\bW \ltimes X^*(\bT \cap \mathscr{D}\bG)$. Then we have a surjective group morphism $\bW \twoheadrightarrow \bW_{\mathrm{der}}$ (induced by restriction of characters) which is injective on $\bW^{\mathrm{Cox}}$, and $\bW_{\mathrm{der}}$ is the extended affine Weyl group of the semisimple group $\mathscr{D}\bG$ (and its maximal torus $\bT \cap \mathscr{D}\bG$). The formula recalled above for the lengths shows that this morphism is compatible with the length functions. Using the known case of semisimple groups we obtain that there exist $s' \in \bS$ and $w \in \bW$ such that $\ell(ws')=\ell(w)+1$ and the images of $s$ and $ws'w^{-1}$ in $\bW_{\mathrm{der}}$ coincide. Now $\bW^{\mathrm{Cox}}$ is normal in $\bW$, hence it contains $ws'w^{-1}$. Since the morphism $\bW \to \bW_{\mathrm{der}}$ is injective on $\bW^{\mathrm{Cox}}$, we deduce that $s=ws'w^{-1}$.
\end{proof}

In the rest of the paper we will fix once and for all, for each $s \in \bS \smallsetminus \bSf$, elements $s' \in \bSf$ and $w \in \bW$ such that $\ell(ws')=\ell(w)+1$ and $s=ws'w^{-1}$. (The condition on lengths will not be needed in the present section, but will be used later.)

\subsection{Some representations of \texorpdfstring{$\bbI_{\mathbf{\Sigma}}$}{ISigma}}
\label{ss:representations-Iadj}


Consider the group scheme $\bbI_{\mathbf{\Sigma}}$
over $\bT \times_{\bT/\bWf} \bT$, and its category $\Rep(\bbI_{\mathbf{\Sigma}})$
of representations on coherent $\scO_{\bT \times_{\bT/\bWf} \bT}$-modules, see~\S\ref{ss:univ-centralizer}.
It identifies with the category of comodules over the $\scO(\bT \times_{\bT/\bWf} \bT)$-Hopf algebra
\[
\scO(\bbI_{\mathbf{\Sigma}}) = \scO(\bbJ_{\mathbf{\Sigma}}) \otimes_{\scO(\mathbf{\Sigma})} \scO(\bT \times_{\bT/\bWf} \bT)
\]
which are finitely generated as $\scO(\bT \times_{\bT/\bWf} \bT)$-modules. Since $\scO(\bT \times_{\bT/\bWf} \bT)$ is finite as an $\scO(\mathbf{\Sigma})$-module, this category admits a natural monoidal structure defined by
\[
M \circledast N = M \otimes_{\scO(\bT)} N.
\]
This bifunctor is right exact on each side, and the unit object for this monoidal structure if $\scO(\bT)$, seen as functions on the diagonal copy $\bT \subset \bT \times_{\bT/\bWf} \bT$, and endowed with the trivial structure as a representation of $\bbI_{\mathbf{\Sigma}}$.

We will now define objects $(\mathscr{M}_w : w \in \bW)$ of $\Rep(\bbI_{\mathbf{\Sigma}})$ parametrized by $\bW$ as follows. First, if $w \in \bWf$ then $\mathscr{M}_w$ is defined as the structure sheaf of the closed subscheme
\[
\{(w(t),t) : t \in \bT\} \subset \bT \times_{\bT/\bWf} \bT,
\]
endowed with the trivial structure as a representation.
The projection on the first component induces an isomorphism $\mathscr{M}_w \simto \scO(\bT)$; under this isomorphism, the action of $\scO(\bT \times_{\bT/\bWf} \bT)=\scO(\bT) \otimes_{\scO(\bT/\bWf)} \scO(\bT)$ on $\mathscr{M}_w$ is given by $(f \otimes g) \cdot m = fw(g) m$ for $f,g,m \in \scO(\bT)$. 

If $\lambda \in X^*(\bT)$, then in Lemma~\ref{lem:line-bundles-J} we have considered the pullback to $\bbJ_\bT$ of the representation $\scO(\bT) \otimes \bk_{\bT}(\lambda)$. Pushing this representation forward along the diagonal embedding $\bT \to \bT \times_{\bT/\bWf} \bT$ we obtain an object of $\Rep(\bbI_{\mathbf{\Sigma}})$, which will be denoted $\mathscr{M}_{\st(\lambda)}$.

It is clear that for $w,y \in \bWf$ and $\lambda,\mu \in X^*(\bT)$ we have canonical isomorphisms
\begin{align}
\label{eqn:formula-M-1}
\mathscr{M}_w \circledast \mathscr{M}_y & \simto \mathscr{M}_{wy}, \\
\label{eqn:formula-M-2}
\mathscr{M}_{\st(\lambda)} \circledast \mathscr{M}_{\st(\mu)} &\simto \mathscr{M}_{\st(\lambda+\mu)}.
\end{align}
Next we need to study the interplay between these two classes of objects.


We have a canonical action of $\bWf$ on $\Groth'$ induced by the natural action on $\bT$; this action commutes with the action of $\bG$ and stabilizes $\Groth'_\reg$; we deduce a canonical action on the universal stabilizer $\bT \times_{\bT/\bWf} \bbJ_\reg$, and then (by restriction) on $\bT \times_{\bT/\bWf} \bbJ_{\mathbf{\Sigma}}$. (This action is simply induced by the action on $\bT$.)

\begin{lem}
\label{lem:morphism-J-Wf-equiv}
 The morphism~\eqref{eqn:morph-gp-schemes-Jreg-Sigma} is $\bWf$-equivariant, where $\bWf$ acts on the right-hand side diagonally.
\end{lem}

\begin{proof}
 By flatness it suffices to check this claim over $\bT_\circ$. Now, by Lemma~\ref{lem:vartheta-circ-iso} we have an isomorphism $\bG/\bT \times \bT_\circ \simto \Groth'_\circ$.
  Under this isomorphism the action of $\bWf$ on $\bG/\bT \times \bT_\circ$ is given by $w \cdot (g\bT,t)=(gw^{-1} \bT, w(t))$, where we write $gw^{-1} \bT$ for $g \dot{w}^{-1} \bT$ where $\dot{w}$ is any lift of $w$ to $\mathrm{N}_{\bG}(\bT)$. From this description the equivariance is clear.
\end{proof}

From Lemma~\ref{lem:morphism-J-Wf-equiv} we deduce that for $w \in \bWf$ and $\lambda \in X^*(\bT)$ we have a canonical isomorphism
\[
\mathscr{M}_w \circledast \mathscr{M}_{\st(\lambda)} \circledast \mathscr{M}_{w^{-1}} \simto \mathscr{M}_{\st(w(\lambda))}.
\]
Combining this with~\eqref{eqn:formula-M-1}--\eqref{eqn:formula-M-2} we deduce that if for $w=x\st(\lambda) \in \bWf \ltimes X^*(\bT)=\bW$ we set
\[
\mathscr{M}_w := \mathscr{M}_x \circledast \mathscr{M}_{\st(\lambda)},
\]
then for any $w,y \in \bW$ we have a canonical isomorphism 
\[
\mathscr{M}_w \circledast \mathscr{M}_y \simto \mathscr{M}_{wy}.
\]

We next define some objects $(\mathscr{B}_s : s \in \bSaff)$ associated with simple reflections in $\bW$. First, if $s \in \bSf$ we define $\mathscr{B}_s$ by
\[
\mathscr{B}_s := \scO(\bT \times_{\bT/\{1,s\}} \bT),
\]
which we view as an $\scO(\bT \times_{\bT/\bWf} \bT)$-module via the closed embedding
\[
\bT \times_{\bT/\{1,s\}} \bT \subset \bT \times_{\bT/\bWf} \bT,
\]
and endow with the trivial structure as a representation. If $s \in \bS \smallsetminus \bSf$, recall that in~\S\ref{ss:Waff} we have fixed $s' \in \bSf$ and $w \in \bW$ such that $s=ws'w^{-1}$; we then set
\begin{equation}
\label{eqn:def-Bs}
\mathscr{B}_s := \mathscr{M}_w \circledast \mathscr{B}_{s'} \circledast \mathscr{M}_{w^{-1}}.
\end{equation}
It is easily seen (e.g.~by reduction to the case $s \in \bSf$) that for any $s \in \bS$ there exist exact sequences
\[
\mathscr{M}_e \hookrightarrow \mathscr{B}_s \twoheadrightarrow \mathscr{M}_s, \quad \mathscr{M}_s \hookrightarrow \mathscr{B}_s \twoheadrightarrow \mathscr{M}_e.
\]

 \subsection{Completions}
 \label{ss:completions}

We will denote by $\mathcal{I} \subset \scO(\bT \times_{\bT/\bWf} \bT)$ the ideal of the point $(e,e)$, 
and by $\mathcal{K} \subset \scO(\bT)$ the ideal of the point $e \in \bT$.
Note that
\[
\mathcal{I} = \mathcal{K} \otimes_{\scO(\bT/\bWf)} \scO(\bT) + \scO(\bT) \otimes_{\scO(\bT/\bWf)} \mathcal{K}
\]
where both summands are ideals in $\scO(\bT \times_{\bT/\bWf} \bT)$ since $\scO(\bT)$ is flat over $\scO(\bT/\bWf)$ (by Theorem~\ref{thm:Pittie-Steinberg}). Finally, we will denote by
$\mathcal{J} \subset \scO(\bT/\bWf)$ the ideal of the image of $e \in \bT$ in $\bT/\bWf$ (which, by abuse, will also be denoted $e$).
We will be interested in the ``completions''
\[
 \FN_{\bT/\bWf}(\{e\}), \quad \FN_{\bT}(\{e\}) \quad \text{and} \quad \FN_{\bT \times_{\bT/\bWf} \bT}(\{(e,e)\}).
\]

\begin{lem}
\phantomsection
\label{lem:Dw-fiber-prod}
\begin{enumerate}
\item
\label{it:Dw-fiber-prod-1}
There exist canonical isomorphisms of $\bk$-schemes
\[
\FN_{\bT}(\{e\}) \cong \bT \times_{\bT/\bWf} \FN_{\bT/\bWf}(\{e\})
\]
and
\begin{multline*}
\FN_{\bT \times_{\bT/\bWf} \bT}(\{(e,e)\}) \cong \FN_{\bT}(\{e\}) \times_{\bT} (\bT \times_{\bT/\bWf} \bT) \cong (\bT \times_{\bT/\bWf} \bT) \times_{\bT} \FN_{\bT}(\{e\}) \\
\cong (\bT \times_{\bT/\bWf} \bT) \times_{\bT / \bWf} \FN_{\bT/\bWf}(\{e\}) \cong \FN_{\bT}(\{e\}) \times_{\FN_{\bT/\bWf}(\{e\})} \FN_{\bT}(\{e\})
\end{multline*}
where in the first, resp.~second, fiber product the morphism $\bT \times_{\bT/\bWf} \bT \to \bT$ is induced by projection on the first, resp.~second, factor. Moreover, the algebra $\scO(\FN_{\bT}(\{e\}))$ is finite and free (in particular, flat) over the algebra $\scO(\FN_{\bT/\bWf}(\{e\}))$.
\item
\label{it:Dw-fiber-prod-2}
The natural morphism $\scO(\FN_{\bT/\bWf}(\{e\})) \to \scO(\FN_{\bT}(\{e\}))^{\bWf}$ is an isomorphism.
\end{enumerate}
\end{lem}

\begin{proof}
\eqref{it:Dw-fiber-prod-1}
Since the morphism $\bT \to \bT / \bWf$ is finite, and since $e$ is the only closed point in the preimage of the point corresponding to $\mathcal{J}$, by the structure theory of artinian local rings (see in particular~\cite[\href{https://stacks.math.columbia.edu/tag/00J8}{Tag 00J8}]{stacks-project}) the ideal $\mathcal{J} \cdot \scO(\bT)$ contains a power of $\mathcal{K}$. On the other hand this ideal is contained in $\mathcal{K}$; hence the completions of $\scO(\bT)$ with respect to $\mathcal{K}$ and to $\mathcal{J} \cdot \scO(\bT)$ are canonically isomorphic. By definition the first of these completions is $\scO(\FN_{\bT}(\{e\}))$, and since $\bT$ is finite over $\bT / \bWf$ the second completion identifies with $\scO(\bT \times_{\bT / \bWf} \FN_{\bT/\bWf}(\{e\}))$, proving the first isomorphism. Combined with Theorem~\ref{thm:Pittie-Steinberg}, this implies that $\scO(\FN_{\bT}(\{e\}))$ is finite and free over $\scO(\FN_{\bT/\bWf}(\{e\}))$.

Similar considerations using the morphism $\bT \times_{\bT/\bWf} \bT \to \bT / \bWf$  prove the
isomorphism between the first and fourth schemes in the second series of isomorphisms. Since the ideals $\mathcal{K} \otimes_{\scO(\bT/\bWf)} \scO(\bT)$ and $\scO(\bT) \otimes_{\scO(\bT/\bWf)} \mathcal{K}$ contain $\mathcal{J} \cdot \scO(\bT \times_{\bT/\bWf} \bT)$ and are contained in $\mathcal{I}$, the completions of $\scO(\bT \times_{\bT/\bWf} \bT)$ with respect to these ideals (or, in other words, the algebras of functions on the second and third schemes) also identify with $\scO(\FN_{\bT \times_{\bT/\bWf} \bT}(\{(e,e)\}))$. Finally, 
the last isomorphism follows from the isomorphism $\FN_{\bT}(\{e\}) \cong \bT \times_{\bT/\bWf} \FN_{\bT/\bWf}(\{e\})$ proved above.

\eqref{it:Dw-fiber-prod-2}
%
 Using the first isomorphism in~\eqref{it:Dw-fiber-prod-1} we see that the canonical embedding $\scO(\bT/\bWf) \hookrightarrow \scO(\bT)$ induces an embedding
 \[
 \scO(\FN_{\bT/\bWf}(\{e\})) \hookrightarrow \scO(\FN_{\bT}(\{e\})),
 \]
 which of course factors through an embedding
 \begin{equation}
 \label{eqn:T-mod-W-wedge}
 \scO(\FN_{\bT/\bWf}(\{e\})) \hookrightarrow \scO(\FN_{\bT}(\{e\}))^{\bWf}.
 \end{equation}
 As explained above, any basis of $\scO(\bT)$ as a module over over $\scO(\bT/\bWf)$ provides a basis of $\scO(\FN_{\bT}(\{e\}))$ over $\scO(\FN_{\bT/\bWf}(\{e\}))$. On the other hand, in~\cite[Theorem~8.1]{bezr} it is proved that a specific basis of $\scO(\bT)$ over $\scO(\bT/\bWf)$ provides a basis of $\scO(\FN_{\bT}(\{e\}))$ over $\scO(\FN_{\bT}(\{e\}))^{\bWf}$. The embedding~\eqref{eqn:T-mod-W-wedge} is therefore an equality.
\end{proof}

We set
\[
\bbI_{\mathbf{\Sigma}}^\wedge := \FN_{\bT \times_{\bT/\bWf} \bT}(\{(e,e)\}) \times_{\bT \times_{\bT/\bWf} \bT} \bbI_{\mathbf{\Sigma}} \cong \FN_{\bT \times_{\bT/\bWf} \bT}(\{(e,e)\}) \times_{\bT / \bWf} \mathbb{J}_\Sigma,
\]
a smooth affine group scheme over the affine scheme $\FN_{\bT \times_{\bT/\bWf} \bT}(\{(e,e)\})$. We will consider the category $\Rep(\bbI_{\mathbf{\Sigma}}^\wedge)$ of representations of this group scheme which are of finite type over $\scO(\FN_{\bT \times_{\bT/\bWf} \bT}(\{(e,e)\}))$. 
The isomorphisms in Lemma~\ref{lem:Dw-fiber-prod} show that an $\scO(\FN_{\bT \times_{\bT/\bWf} \bT}(\{(e,e)\}))$-module is the same thing as an $\scO(\FN_{\bT}(\{e\}))$-bimodule on which the left and right actions of $\scO(\FN_{\bT/\bWf}(\{e\}))=\scO(\FN_{\bT}(\{e\}))^{\bWf}$ coincide. (We will use this identification repeatedly and without further notice below.) In particular the category of such modules admits a natural monoidal product, induced by the tensor product for $\scO(\FN_{\bT}(\{e\}))$-bimodules; moreover this product stabilizes the subcategory of finitely generated $\scO(\FN_{\bT \times_{\bT/\bWf} \bT}(\{(e,e)\}))$-modules. Since $\bbI_{\mathbf{\Sigma}}^\wedge$ is the pullback of a group scheme over $\FN_{\bT/\bWf}(\{e\})$, this product induces a monoidal product on the category $\Rep(\bbI_{\mathbf{\Sigma}}^\wedge)$, which will again be denoted $\circledast$.

Recall that a category is called \emph{Krull--Schmidt} if any object admits a decomposition as a (finite) direct sum of objects with local endomorphism ring.

\begin{lem}
\label{lem:Krull-Schmidt-JD}
 The category $\Rep(\bbI_{\mathbf{\Sigma}}^\wedge)$ is Krull--Schmidt.
\end{lem}

\begin{proof}
 By~\cite[Theorem~A.1]{cyz}, an additive category is Krull--Schmidt iff it is idempotent complete and the endomorphism ring of any object is semiperfect. Here $\Rep(\bbI_{\mathbf{\Sigma}}^\wedge)$ is idempotent complete because it is abelian, and the endomorphism algebra of any object is semiperfect because it is finite as a module over the noetherian complete local ring $\scO(\FN_{\bT \times_{\bT/\bWf} \bT}(\{(e,e)\}))$, see~\cite[Example~23.3]{lam}.
\end{proof}

Pulling back the representations $(\mathscr{M}_w : w \in \bW)$ and $(\mathscr{B}_s : s \in \bS)$ introduced in~\S\ref{ss:representations-Iadj} along
the natural morphism $\FN_{\bT \times_{\bT/\bWf} \bT}(\{(e,e)\}) \to \bT \times_{\bT/\bWf} \bT$ we obtain objects $(\mathscr{M}^\wedge_w : w \in \bW)$ and $(\mathscr{B}^\wedge_s : s \in \bS)$ in $\Rep(\bbI_{\mathbf{\Sigma}}^\wedge)$. It is clear that for any $w,y \in \bW$ we have a canonical isomorphism 
\begin{equation}
\label{eqn:convolution-D}
\mathscr{M}^\wedge_w \circledast \mathscr{M}^\wedge_y \simto \mathscr{M}^\wedge_{wy},
\end{equation}
and that for $s \in \bS$ we have exact sequences
\begin{equation}
 \label{eqn:exact-seq-Bs}
 \mathscr{M}^\wedge_e \hookrightarrow \mathscr{B}^\wedge_s \twoheadrightarrow \mathscr{M}^\wedge_s, \qquad \mathscr{M}^\wedge_s \hookrightarrow \mathscr{B}^\wedge_s \twoheadrightarrow \mathscr{M}^\wedge_e.
 \end{equation}
 
 The following lemma will be proved in~\S\ref{ss:completed-Hecke-Rep} below, using a different description of (a subcategory of) $\Rep(\bbI_{\mathbf{\Sigma}}^\wedge)$. (In this statement we use the fact that $\omega s \omega^{-1} \in \bS$ for any $s \in \bS$ and $\omega \in \mathbf{\Omega}$, see~\S\ref{ss:Waff}.)

\begin{lem}
\label{lem:conjugation-B-RepJ}
 For any $s \in \bS \smallsetminus \bSf$ the object $\scB^\wedge_s$ is independent of the choices of $w$ and $s'$ as in~\S\ref{ss:Waff} up to canonical isomorphism. Moreover, for any $\omega \in \mathbf{\Omega}$ and $s \in \bS$ we have a canonical isomorphism
 \[
  \scM^\wedge_\omega \circledast \scB^\wedge_s \circledast \scM^\wedge_{\omega^{-1}} \cong \scB^\wedge_{\omega s \omega^{-1}}.
 \]
\end{lem}

We will denote by $\BSRep(\bbI_{\mathbf{\Sigma}}^\wedge)$ the category with
\begin{itemize}
\item
objects the collections $(\omega, s_1, \cdots, s_i)$ with $\omega \in \mathbf{\Omega}$ and $s_1, \cdots, s_i \in \bS$;
\item
morphisms from $(\omega, s_1, \cdots, s_i)$ to $(\omega', s'_1, \cdots, s'_j)$ given by
\[
\Hom_{\Rep(\bbI_{\mathbf{\Sigma}}^\wedge)}(\scM^\wedge_\omega \circledast \scB^\wedge_{s_1} \circledast \cdots \circledast \scB^\wedge_{s_i}, \scM^\wedge_{\omega'} \circledast \scB^\wedge_{s_1'} \circledast \cdots \circledast \scB^\wedge_{s'_j}).
\]
\end{itemize}
By definition there exists a canonical fully faithful functor
\begin{equation}
\label{eqn:functor-BSRep-Rep}
\BSRep(\bbI_{\mathbf{\Sigma}}^\wedge) \to \Rep(\bbI_{\mathbf{\Sigma}}^\wedge).
\end{equation}
Using Lemma~\ref{lem:conjugation-B-RepJ} we obtain, for any collections $(\omega, s_1, \cdots, s_i)$ and $(\omega', s'_1, \cdots, s'_j)$ as above, a canonical isomorphism
\begin{multline*}
 \bigl( \scM^\wedge_\omega \circledast \scB^\wedge_{s_1} \circledast \cdots \circledast \scB^\wedge_{s_i} \bigr) \circledast \bigl( \scM^\wedge_{\omega'} \circledast \scB^\wedge_{s_1'} \circledast \cdots \circledast \scB^\wedge_{s'_j} \bigr) \\
 \cong \scM^\wedge_{\omega \omega'} \circledast \scB^\wedge_{(\omega')^{-1} s_1 \omega'} \circledast \cdots \circledast \scB^\wedge_{(\omega')^{-1} s_i \omega'} \circledast \scB^\wedge_{s_1'} \circledast \cdots \circledast \scB^\wedge_{s'_j}.
\end{multline*}
This allows us to define a monoidal product (again denoted $\circledast$) on $\BSRep(\bbI_{\mathbf{\Sigma}}^\wedge)$ which is defined on objects by
\[
 (\omega, s_1, \cdots, s_i) \circledast (\omega', s'_1, \cdots, s'_j) = (\omega\omega', (\omega')^{-1} s_1 \omega', \cdots, (\omega')^{-1} s_i \omega', s'_1, \cdots, s'_j)
\]
and such that~\eqref{eqn:functor-BSRep-Rep} is monoidal.

We will denote by
\[
\SRep(\bbI_{\mathbf{\Sigma}}^\wedge)
\]
the Karoubian closure of the additive hull of the category $\BSRep(\bbI_{\mathbf{\Sigma}}^\wedge)$. By the Krull--Schmidt property (see Lemma~\ref{lem:Krull-Schmidt-JD}), this category identifies with the (monoidal) full subcategory of $\Rep(\bbI_{\mathbf{\Sigma}}^\wedge)$ whose objects are direct sums of direct summands of objects of the form
\[
 \scM^\wedge_\omega \circledast \scB^\wedge_{s_1} \circledast \cdots \circledast \scB^\wedge_{s_i}
\]
with $\omega \in \mathbf{\Omega}$ and $s_1, \cdots, s_i \in \bS$. (In these notations, ``$\mathsf{BS}$'' stands for ``Bott--Samelson,'' and ``$\mathsf{S}$'' for ``Soergel,'' since these constructions are very similar to classical constructions related to Bott--Samelson resolutions and Soergel bimodules.)



\subsection{Hecke categories ``\`a la Abe''}
\label{ss:Hecke-a-la-Abe}

We now explain how to construct some categories by following a pattern initiated by Abe~\cite{abe}. We consider a noetherian domain $R$ endowed with an action of $\bW$ (by ring automorphisms), and denote by $Q=\mathrm{Frac}(R)$ the fraction field of $R$. We denote by $\sfK'(R)$ the category defined as follows. The objects are the $R$-bimodules $M$ together with a decomposition
\begin{equation}
\label{eqn:decomp-K-Abe}
M \otimes_{R} Q = \bigoplus_{w \in \bW} M^w_{Q}
\end{equation}
as $(R, Q)$-bimodules such that:
\begin{itemize}
\item there exist only finitely many $w$'s such that $M^w_{Q} \neq 0$;
\item for any $w \in \bW$, $r \in R$ and $m \in M^w_{Q}$ we have $m \cdot r = w(r) \cdot m$.
\end{itemize}
Morphisms in this category are defined as morphisms of $R$-bimodules respecting the decompositions~\eqref{eqn:decomp-K-Abe}. The category $\sfK^{\prime}(R)$ has a natural monoidal structure, with product denoted $\star$ and induced by the tensor product over $R$. (To see this one observes that the conditions above imply that the left $R$-action on $M \otimes_{R} Q$ extends to an action of $Q$, see~\cite[Remark~2.2]{abe}.)

We will also denote by $\sfK(R)$ the full subcategory in $\sfK'(R)$ whose objects are those whose underlying $R$-bimodule is finitely generated, and is flat as a right $R$-module. The latter condition implies that the natural morphism $M \to M \otimes_R Q$ is injective, which (in view of the second condition above) implies in particular that the left and right actions of $R^\bW$ on $M$ coincide. The arguments in~\cite[Lemma~2.6]{abe} show that the underlying bimodule of any object in $\sfK(R)$ is in fact finitely generated as a left $R$-module and as a right $R$-module. Using this property, it is easily seen that $\sfK(R)$ is a monoidal subcategory of $\sfK'(R)$.


We have natural objects in $\sfK(R)$ attached to elements in $\bW$, and constructed as follows. Given $w \in \bW$, we denote by $F_w$ the $R$-bimodule which is isomorphic to $R$ as an abelian group, and endowed with the structure of $R$-bimodule determined by the rule
\[
 r \cdot m \cdot r' = rmw(r')
\]
for $r,r' \in R$ and $m \in F_w$. If we endow this bimodule with the decomposition of $F_{w} \otimes_{R} Q$ such that this module is concentrated in degree $w$, we obtain an object in $\sfK(R)$. It is clear that for any $w,y \in \bW$ we have a canonical isomorphism
\[
F_w \star F_y \simto F_{wy}.
\]

Next, for $s \in \bS$ we will denote by $R^s \subset R$ the subring of $s$-invariants. Assume that 
\begin{equation}
\label{eqn:assumption-deltas}
\text{there exists $\delta _s \in R$ such that $(1,\delta_s)$ is a basis of $R$ as an $R^s$-module.}
\end{equation}
Then we set
\[
B_{s} := R \otimes_{R^s} R.
\]
Our assumption ensures that $B_s$ is finite and free (in particular, flat) as a right $R$-module.
Moreover this objects admits a canonical decomposition~\eqref{eqn:decomp-K-Abe}, hence defines an object in $\sfK(R)$. In fact,
since the action of $s$ on $R$ is nontrivial by our assumption, the decomposition of $B_{s} \otimes_{R} Q = R \otimes_{R^s} Q$ is uniquely determined by the fact that it is concentrated in degrees $\{e,s\} \subset \bW$.
More explicitly, using the formula 
\[
\delta_s \delta_s = \delta_s (\delta_s + s(\delta_s)) - \delta_s s(\delta_s)
\]
one checks that we have
\[
(B_s)_Q^e = (\delta_s \otimes 1 - 1 \otimes s(\delta_s)) \cdot Q, \quad (B_s)_Q^s = (\delta_s \otimes 1 - 1 \otimes \delta_s) \cdot Q.
\]

The following lemma can be checked by explicit computation. (A similar claim in a slightly different setting is proved in~\cite[Lemma~2.4]{brHecke}.)

\begin{lem}
\label{lem:conjugation-Bs}
Let $s,s' \in \bS$, and assume that $w \in \bW$ satisfies $s'=wsw^{-1}$. If~\eqref{eqn:assumption-deltas} holds for $s$, then it also holds for $s'$, and moreover we have a canonical isomorphism
\[
F_w \star B_s \star F_{w^{-1}} \simto B_{s'}.
\]
\end{lem}

\begin{rmk}
\label{rmk:condition-Bs-Sf}
Recall that any element in $\bS$ is conjugate (in $\bW$) to an element in $\bSf$, see Lemma~\ref{lem:simple-ref-conjugation}. In view of Lemma~\ref{lem:conjugation-Bs}, to check condition~\eqref{eqn:assumption-deltas} for all $s \in \bS$ it suffices to do so when $s \in \bSf$.
\end{rmk}

We now assume that~\eqref{eqn:assumption-deltas} is satisfied for any $s \in \bS$. We will then denote by $\BSK(R)$ the category with
\begin{itemize}
\item
objects the collections $(\omega, s_1, \cdots, s_i)$ with $\omega \in \mathbf{\Omega}$ and $s_1, \cdots, s_i \in \bS$;
\item
morphisms from $(\omega, s_1, \cdots, s_i)$ to $(\omega', s'_1, \cdots, s'_j)$ given by
\[
\Hom_{\sfK(R)}(F_\omega \star B_{s_1} \star \cdots \star B_{s_i}, F_{\omega'} \star B_{s_1'} \star \cdots \star B_{s'_j}).
\]
\end{itemize}
By definition there exists a canonical fully faithful functor
\begin{equation}
\label{eqn:functor-HBS-K}
\BSK(R) \to \sfK(R).
\end{equation}
Using the isomorphism in Lemma~\ref{lem:conjugation-Bs} (when $w \in \mathbf{\Omega}$) one sees that there exists a natural convolution product (still denoted $\star$) on $\BSK(R)$ which is defined on objects by
\[
 (\omega, s_1, \cdots, s_i) \star (\omega', s'_1, \cdots, s'_j) = (\omega\omega', (\omega')^{-1} s_1 \omega', \cdots, (\omega')^{-1} s_i \omega', s'_1, \cdots, s'_j),
\]
and such that~\eqref{eqn:functor-HBS-K} is monoidal.

\begin{rmk}
In~\cite{abe}, Abe studies an analogue of the category $\BSK(R)$ in the setting where $\bW$ is replaced by a Coxeter group (so that there are no nontrivial elements of length $0$) and where one considers graded bimodules over a specific choice of graded ring $R$. We do not claim that the results of~\cite{abe} apply in the setting considered above, but only that the main definition makes sense.
\end{rmk}

\subsection{Completed Hecke category and representations of \texorpdfstring{$\bbI_{\mathbf{\Sigma}}^\wedge$}{ISigma}}
\label{ss:completed-Hecke-Rep}

We will apply the construction of~\S\ref{ss:Hecke-a-la-Abe} to the ring $R=\scO(\FN_{\bT}(\{e\}))$, with the action of $\bW$ obtained from the natural action of $\bWf$ by pullback along the projection $\bW \to \bWf$. 
To check conditions~\eqref{eqn:assumption-deltas} in this case, it suffices to do so when $s \in \bSf$ (see Remark~\ref{rmk:condition-Bs-Sf}). In this case the condition can be checked explicitly, or deduced from Lemma~\ref{lem:Dw-fiber-prod} applied to the Levi factor of $\bG$ associated with $s$.
The resulting categories $\sfK(\scO(\FN_{\bT}(\{e\})))$ and $\BSK(\scO(\FN_{\bT}(\{e\})))$ will be denoted
\[
\sfK^\wedge \quad \text{ and } \quad \BSK^{\wedge}
\]
respectively.

Recall the category $\Rep(\bbI_{\mathbf{\Sigma}}^\wedge)$ considered in~\S\ref{ss:completions}. We will denote by
$\Rep_{\mathrm{fl}}(\bbI_{\mathbf{\Sigma}}^\wedge)$ the full subcategory of representations whose underlying coherent sheaf is flat with respect to the projection $\FN_{\bT \times_{\bT/\bWf} \bT}(\{(e,e)\}) \to \FN_{\bT}(\{e\})$ on the second component. It is not difficult to check that $\Rep_{\mathrm{fl}}(\bbI_{\mathbf{\Sigma}}^\wedge)$ is a monoidal subcategory in $\Rep(\bbI_{\mathbf{\Sigma}}^\wedge)$, and that it contains the essential image of~\eqref{eqn:functor-BSRep-Rep}.

The following statement is an analogue of the statements~\cite[Proposition~2.7 and Lemma~2.9]{brHecke}, and its proof is very similar.

\begin{prop}
\label{prop:Abe-Rep-wedge}
There exists a canonical fully faithful monoidal functor
\[
\Rep_{\mathrm{fl}}(\bbI_{\mathbf{\Sigma}}^\wedge) \to \sfK^{\wedge}
\]
sending $\scM_w^\wedge$ to $F_w$ for any $w \in \bW$ and $\scB^\wedge_s$ to $B_s$ for any $s \in \bSf$.
\end{prop}

\begin{proof}
We start by constructing a functor
\begin{equation}
\label{eqn:functor-Rep-Abe}
\Rep(\bbI_{\mathbf{\Sigma}}^\wedge) \to \sfK'(\scO(\FN_{\bT}(\{e\}))).
\end{equation}
Recall the open subscheme $\bT_\circ \subset \bT$, which is defined by the function $\prod_\alpha (\alpha-1)$ where $\alpha$ runs over the roots of $(\bG, \bT)$, see~\S\ref{ss:rs}. We have an open embedding $\bT_\circ / \bWf \subset \bT / \bWf$ and an isomorphism
\begin{equation}
\label{eqn:decomp-Dwedge-circ}
\bWf \times \bT_\circ \simto \bT_\circ \times_{\bT_\circ/\bWf} \bT_\circ,
\end{equation}
see~\S\ref{eqn:fiber-prod-T0}.
Let us denote by 
$\bbJ_{\bT,\circ}$ the restriction of $\bbJ_{\bT}$ to $\bT_\circ$. Then it follows from Lemma~\ref{lem:morph-gp-schemes-Jreg} that we have a canonical isomorphism of group schemes
\begin{equation}
\label{eqn:Jwedge-circ}
\bbJ_{\bT,\circ} \simto \bT_\circ \times \bT.
\end{equation}

We are now ready to explain the construction of the functor~\eqref{eqn:functor-Rep-Abe}. Starting from an object $M$ in $\Rep(\bbI_{\mathbf{\Sigma}}^\wedge)$, the underlying $\scO(\FN_{\bT}(\{e\}))$-bimodule of its image is simply taken as $M$ with its given $\scO(\FN_{\bT \times_{\bT/\bWf} \bT}(\{(e,e)\}))$-module structure. Next, using Lemma~\ref{lem:Dw-fiber-prod} and~\eqref{eqn:decomp-Dwedge-circ} we obtain a canonical isomorphism
\[
\FN_{\bT \times_{\bT/\bWf} \bT}(\{(e,e)\}) \times_{\bT / \bWf} \bT_\circ / \bWf \cong \bWf \times \bigl( \FN_{\bT}(\{e\}) \times_{\bT} \bT_\circ \bigr).
\]
This implies that $M \otimes_{\scO(\bT/\bWf)} \scO(\bT_\circ/\bWf)$ has a canonical decomposition as a direct sum parametrized by $\bWf$. Moreover, each graded component has a natural structure of representation of the group scheme
\[
\FN_{\bT}(\{e\}) \times_{\bT} \bbJ_{\bT,\circ},
\]
which by~\eqref{eqn:Jwedge-circ} identifies with
\[
\bigl( \FN_{\bT}(\{e\}) \times_{\bT} \bT_\circ \bigr) \times \bT.
\]
This component therefore admits a canonical grading by $X^*(\bT)$. We can then obtain a decomposition of $M \otimes_{\scO(\bT/\bWf)} \scO(\bT_\circ / \bWf )$ parametrized by $\bW$ by defining the summand associated with $\st(\lambda) w$ ($\lambda \in X^*(\bT)$, $w \in \bWf$) as the $\lambda$-graded part in the summand associated with $w$. Now the morphism
\[
\scO(\FN_{\bT}(\{e\})) \to \mathrm{Frac}(\scO(\FN_{\bT}(\{e\})))
\]
factors through the morphism $\scO(\FN_{\bT}(\{e\})) \to \scO(\FN_{\bT}(\{e\})) \otimes_{\scO(\bT/\bWf)} \scO(\bT_\circ / \bWf )$, and we have
\begin{multline*}
M \otimes_{\scO(\FN_{\bT}(\{e\}))} \mathrm{Frac}(\scO(\FN_{\bT}(\{e\}))) = \\
\bigl( M \otimes_{\scO(\bT/\bWf)} \scO(\bT_\circ / \bWf ) \bigr) \otimes_{\scO(\FN_{\bT}(\{e\})) \otimes_{\scO(\bT/\bWf)} \scO(\bT_\circ / \bWf )} \mathrm{Frac}(\scO(\FN_{\bT}(\{e\}))).
\end{multline*}
From the decomposition of $M \otimes_{\scO(\bT/\bWf)} \scO(\bT_\circ / \bWf )$ parametrized by $\bW$ we therefore obtain a decomposition of $M \otimes_{\scO(\FN_{\bT}(\{e\}))} \mathrm{Frac}(\scO(\FN_{\bT}(\{e\})))$ parametrized by $\bW$, which
finishes the construction of the functor~\eqref{eqn:functor-Rep-Abe}.

It is clear from construction that our functor~\eqref{eqn:functor-Rep-Abe} has a canonical monoidal structure, and takes values in objects whose underlying $\scO(\FN_{\bT}(\{e\}))$-bimodule is finitely generated and flat as a right $\scO(\FN_{\bT}(\{e\}))$-module. 
It therefore restricts to a monoidal functor
\[
\Rep_{\mathrm{fl}}(\bbI_{\mathbf{\Sigma}}^\wedge) \to \sfK^{\wedge}.
\]

We now need to prove that this functor is fully faithful. Morphisms in both of these categories are by definition certain morphisms of $\scO(\FN_{\bT \times_{\bT/\bWf} \bT}(\{(e,e)\}))$-modules; the functor is therefore faithful. If $M$ and $N$ are in $\Rep_{\mathrm{fl}}(\bbI_{\mathbf{\Sigma}}^\wedge)$, 
a morphism in $\sfK^{\wedge}$ from the image of $M$ to the image of $N$ is a morphism $f : M \to N$ of $\scO(\FN_{\bT \times_{\bT/\bWf} \bT}(\{(e,e)\}))$-modules such that the induced morphism
\[
M \otimes_{\scO(\FN_{\bT}(\{e\}))} \mathrm{Frac}(\scO(\FN_{\bT}(\{e\}))) \to N \otimes_{\scO(\FN_{\bT}(\{e\}))} \mathrm{Frac}(\scO(\FN_{\bT}(\{e\})))
\]
is a morphism of representations of the group scheme
\[
\bbI_{\mathbf{\Sigma}}^\wedge \times_{\FN_{\bT}(\{e\})} \Spec(\mathrm{Frac}(\scO(\FN_{\bT}(\{e\}))))
\]
over $\FN_{\bT \times_{\bT/\bWf} \bT}(\{(e,e)\}) \times_{\FN_{\bT}(\{e\})} \Spec(\mathrm{Frac}(\scO(\FN_{\bT}(\{e\}))))$. To check that $f$ is a morphism of representations of $\bbI_{\mathbf{\Sigma}}^\wedge$ we need to check that the two natural morphisms
\[
M \to N \otimes_{\scO(\FN_{\bT \times_{\bT/\bWf} \bT}(\{(e,e)\}))} \scO(\bbI_{\mathbf{\Sigma}}^\wedge)
\]
constructed out of it
coincide. Now since $N$ is flat over $\scO(\FN_{\bT}(\{e\}))$ (for the action on the right) and $\scO(\bbI_{\mathbf{\Sigma}}^\wedge)$ is flat over $\scO(\FN_{\bT \times_{\bT/\bWf} \bT}(\{(e,e)\}))$, the right-hand side is flat over $\scO(\FN_{\bT}(\{e\}))$ (for the action on the right), so that to check this condition it suffices to prove that the induced morphisms
\begin{multline*}
M \otimes_{\scO(\FN_{\bT}(\{e\}))} \mathrm{Frac}(\scO(\FN_{\bT}(\{e\}))) \to \\
\bigl( N \otimes_{\scO(\FN_{\bT \times_{\bT/\bWf} \bT}(\{(e,e)\}))} \scO(\bbI_{\mathbf{\Sigma}}^\wedge) \bigr) \otimes_{\scO(\FN_{\bT}(\{e\}))} \mathrm{Frac}(\scO(\FN_{\bT}(\{e\})))
\end{multline*}
coincide, which is exactly the condition given by the fact that $f$ is a morphism in $\sfK^\wedge$.

Finally we prove that our functor
sends each $\scM^\wedge_{w}$ to $F_w$ (for $w \in \bW$) and each $\scB^\wedge_s$ to $B_s$ (for $s \in \bSf$). The case of the objects $\scB^\wedge_s$ is clear. It is clear also that this
functor sends $\scM^\wedge_{\st(\lambda)}$ to $F_{\st(\lambda)}$ for any $\lambda \in X^*(\bT)$, and $\scM^\wedge_x$ to $F_x$ for any $x \in \bWf$. By monoidality, it therefore sends $\scM^\wedge_w$ to $F_w$ for any $w \in \bW$. 
\end{proof}

We can now give the proof of Lemma~\ref{lem:conjugation-B-RepJ}.

\begin{proof}[Proof of Lemma~\ref{lem:conjugation-B-RepJ}]
If $(w_1, s_1')$ and $(w_2, s_2')$ are two pairs of elements as in~\S\ref{ss:Waff} for the same element $s \in \bS \smallsetminus \bSf$, then by Lemma~\ref{lem:conjugation-Bs} the images under the functor of Proposition~\ref{prop:Abe-Rep-wedge} of the objects
\[
\mathscr{M}^\wedge_{w_1} \circledast \mathscr{B}_{s_1'}^\wedge \circledast \mathscr{M}^\wedge_{w_1^{-1}} \quad \text{and} \quad \mathscr{M}^\wedge_{w_2} \circledast \mathscr{B}_{s_2'}^\wedge \circledast \mathscr{M}^\wedge_{w_2^{-1}} 
\]
are canonically isomorphic. By fully faithfulness, this implies that these objects are canonically isomorphic, proving that the definition of $\scB^\wedge_s$ is independent of the choice of $(w,s')$.

The proof of the second claim is similar.
\end{proof}

From the proof of Lemma~\ref{lem:conjugation-B-RepJ} we see that the functor of Proposition~\ref{prop:Abe-Rep-wedge} also sends $\scB^\wedge_s$ to $B_s$ for any $s \in \bS \smallsetminus \bSf$.

\section{Constructible sheaves on affine flag varieties}
\label{ss:const-Fl}

From now on we switch to the ``constructible'' side of our constructions.

\subsection{Affine flag varieties}
\label{ss:Fl}

Let $\F$ be an algebraically closed field of characteristic $p>0$, and let $G$ be a connected reductive algebraic group over $\F$. We fix a Borel subgroup $B \subset G$, whose unipotent radical will be denoted $U$, and a maximal torus $T \subset B$.

Recall that the loop group $\Loop G$, resp.~the positive loop group $\Loop^+ G$, is the group ind-scheme, resp.~group scheme, over $\F$ which represents the functor
\[
 R \mapsto G(R (\hspace{-1pt}(z)\hspace{-1pt}) ), \quad \text{resp.} \quad R \mapsto G(R[\hspace{-1pt}[z]\hspace{-1pt}]),
\]
where $z$ is an indeterminate. By definition $\Loop^+G$ is a subgroup scheme of $\Loop G$, hence one can define the affine Grassmannian $\Gr_G$ as the fppf quotient
\[
 \Gr_G = \bigl( \Loop G/\Loop^+G \bigr)_{\mathrm{fppf}}.
\]
It is well known that $\Gr_G$ is an ind-projective ind-scheme over $\F$.

There exists a canonical morphism of group schemes $\Loop^+G \to G$, induced by the assignment $z \mapsto 0$. The Iwahori subgroup $\Iw \subset \Loop^+ G$ is defined as the inverse image of $B$ under this morphism. The pro-unipotent radical of $\Iw$ is the subgroup $\Iwu \subset \Iw$ defined as the preimage of $U$. We can then define the affine flag variety $\Fl_G$ and the canonical $T$-torsor $\tFl_G$ over $\Fl_G$ as the fppf quotients
\[
 \Fl_G := \bigl( \Loop G/\Iw \bigr)_{\mathrm{fppf}}, \quad \tFl_G := \bigl( \Loop G/\Iwu \bigr)_{\mathrm{fppf}}.
\]
Once again these are ind-schemes of ind-finite type, and $\Fl_G$ is ind-projective.
The embeddings $\Iwu \subset \Iw \subset \Loop^+G$ induce natural morphisms
\begin{equation}
\label{eqn:morph-Fl}
 \tFl_G \to \Fl_G \to \Gr_G.
\end{equation}
It is well known that the second morphism is a Zariski locally trivial fibration with fibers $G/B$, and that the natural action of $T$ on $\tFl_G$ (induced by right multiplication on $\Loop G$) exhibits $\tFl_G$ as a Zariski locally trivial $T$-torsor over $\Fl_G$; this map will be denoted $\pi : \tFl_G \to \Fl_G$.

Consider the coweight lattice $X_*(T)$. The choice of the Borel subgroup $B$ determines a system of positive roots for $(G,T)$ (chosen as the set of $T$-weights in $\mathrm{Lie}(G)/\mathrm{Lie}(B)$), which then define a subset $X_*^+(T) \subset X_*(T)$ of dominant coweights. We will denote by $\preceq$ the order on $X_*(T)$ such that $\lambda \preceq \mu$ if and only if $\mu-\lambda$ is a sum of positive coroots.

Recall that the $\Loop^+G$-orbits on $\Gr_G$ (for the action induced by left multiplication on $\Loop G$) are parametrized by $X_*^+(T)$. Namely, any $\lambda \in X_*(T)$ determines a point $z^\lambda \in \Loop G$, and for $\lambda \in X_*^+(T)$ we denote by $\Gr_G^\lambda$ the $\Loop^+G$-orbit of the image of $z^\lambda$ in $\Gr_G$ (with its reduced subscheme structure). We then have
\[
 (\Gr_G)_{\mathrm{red}} = \bigsqcup_{\lambda \in X_*^+(T)} \Gr_G^\lambda.
\]

We will denote by $\Wf=\mathrm{N}_G(T)/T$ the Weyl group of $(G,T)$. The choice of $B$ determines a system $\Sf \subset \Wf$ of simple reflections, such that $(\Wf,\Sf)$ is a Coxeter system.
The orbits of $\Iw$ on $\Fl_G$ are naturally parametrized by the extended affine Weyl group
\[
 W := \Wf \ltimes X_*(T).
\]
Namely, let us fix for any $v \in \Wf$ a lift $\dot{v} \in \mathrm{N}_G(T)$. Let $w \in W$, and write $w = \st(\lambda) v$ with $v \in \Wf$ and $\lambda \in X_*(T)$. (Here and below, $\st(\lambda)$ denotes the image of $\lambda$ in $W$.) Then if we denote by $\Fl_{G,w} \subset \Fl_G$ the $\Iw$-orbit of the image in $\Fl_G$ of $z^{\lambda} \dot{v}$ (again with its reduced subscheme structure), we have
\begin{equation}
\label{eqn:orbits-Fl}
 (\Fl_G)_{\mathrm{red}} = \bigsqcup_{w \in W} \Fl_{G,w}.
\end{equation}
It is well known also that each $\Fl_{G,w}$ is an $\Iwu$-orbit, isomorphic to an affine space.

For $w \in W$ we will set
\[
\ell(w)=\dim(\Fl_{G,w}), \quad \tFl_{G,w} = \pi^{-1}(\Fl_{G,w}).
\]
For $w \in \Wf$, $\ell(w)$ is the length of $w$ for the Coxeter group structure on $\Wf$ considered above. We will also set $\Omega:=\{w \in W \mid \ell(w)=0\}$.

\subsection{\texorpdfstring{$\Iw$}{I}-equivariant sheaves on the affine flag variety and convolution}
\label{ss:DbIw}

Let $\bk$ be an algebraic closure of a finite field of characteristic $\ell \neq p$. We can then consider the $\Iw$-equivariant derived category of \'etale sheaves on $\Fl_G$ with coefficients in $\bk$, which we will denote by
\[
 \mathsf{D}_{\Iw,\Iw}.
\]
(The definition of this category requires a little bit of care but is standard; see~\cite[\S\S4.1--4.2]{brr} for some details. The complexes on ind-schemes that we consider will always be supported on a finite-type subscheme. Similar comments apply to several constructions below, where we will ``pretend'' that some group schemes of infinite type are honest algebraic groups for notational simplicity.) This category admits a natural perverse t-structure, whose heart will be denoted $\mathsf{P}_{\Iw,\Iw}$.
For each $w \in W$, we will denote by $j_w : \Fl_{G,w} \to \Fl_G$ the (locally closed) embedding, and set
\[
 \Delta^{\Iw}_w := (j_w)_! \underline{\bk}_{\Fl_{G,w}}[\ell(w)], \quad \nabla^{\Iw}_w := (j_w)_* \underline{\bk}_{\Fl_{G,w}}[\ell(w)].
\]
These define objects in $\mathsf{D}_{\Iw,\Iw}$, which are in fact perverse sheaves since $j_w$ is an affine morphism.

The simple objects in the category $\mathsf{P}_{\Iw,\Iw}$ are naturally labelled by $W$. Namely, if for $w \in W$ we denote by $\IC_w$ the intersection cohomology complex associated with the constant local system on $\Fl_{G,w}$ (in other words, the image of the unique---up to scalar---nonzero morphism $\Delta^{\Iw}_w \to \nabla^{\Iw}_w$) then the assignment $w \mapsto \IC_w$ induces a bijection between $W$ and the set of isomorphism classes of simple objects in $\mathsf{P}_{\Iw,\Iw}$.

The category $\mathsf{D}_{\Iw,\Iw}$ also admits a natural convolution product, whose definition we briefly recall. First we consider the ind-scheme $\Fl_G \, \widetilde{\times} \, \Fl_G$, defined as the (fppf) quotient of $\Loop G \times \Fl_G$ by the action of $\Iw$ defined by $g \cdot (h,x)=(hg^{-1},g \cdot x)$. The multiplication map in $\Loop G$ induces a proper morphism $m : \Fl_G \, \widetilde{\times} \, \Fl_G \to \Fl_G$. Then, given $\scF,\scG$ in $\mathsf{D}_{\Iw,\Iw}$, there exists a unique object $\scF \, \widetilde{\boxtimes} \, \scG$ in the $\Iw$-equivariant derived category of $\Fl_G \, \widetilde{\times} \, \Fl_G$ whose pullback to $\Loop G \times \Fl_G$ is the exterior product of the pullback of $\scF$ to $\Loop G$ with $\scG$; then we set
\[
 \scF \star_\Iw \scG := m_!(\scF \, \widetilde{\boxtimes} \, \scG).
\]
With this construction the pair $(\mathsf{D}_{\Iw,\Iw},\star_\Iw)$ is a monoidal category, with unit object $\delta_{\Fl} := \IC_e$.

\begin{rmk}
\label{rmk:tilde-times}
 Below, given $X,Y$ some $\Iw$-invariant subschemes in $\Fl_G$, we will also denote by $X \, \tilde{\times} \, Y$ the quotient of $X' \times Y$ by the $\Iw$-action induced by that on $\Loop G \times \Fl_G$ considered above, where $X'$ is the preimage of $X$ in $\Loop G$.
\end{rmk}

Similarly one can consider the $\Iwu$-equivariant derived category of \'etale $\bk$-sheaves on $\Fl_G$, which will be denoted
\[
 \mathsf{D}_{\Iwu,\Iw}.
\]
This category admits a natural perverse t-structure, whose heart will be denoted $\mathsf{P}_{\Iwu,\Iw}$. We have a canonical t-exact ``forgetful'' functor
\begin{equation}
\label{eqn:For-Iw}
 \For^{\Iw}_{\Iwu} : \mathsf{D}_{\Iw,\Iw} \to \mathsf{D}_{\Iwu,\Iw},
\end{equation}
and the simple objects in the category $\mathsf{P}_{\Iwu,\Iw}$ are (up to isomorphism) the objects $\For^{\Iw}_{\Iwu}(\IC_w)$. 
We also have a canonical right action of $\mathsf{D}_{\Iw,\Iw}$ on $\mathsf{D}_{\Iwu,\Iw}$, defined by a bifunctor
\begin{equation}
\label{eqn:conv-Iwu-Iw}
 \mathsf{D}_{\Iwu,\Iw} \times \mathsf{D}_{\Iw,\Iw} \to \mathsf{D}_{\Iwu,\Iw}
\end{equation}
whose construction repeats exactly the definition of $\star_\Iw$; this bifunctor will also be denoted $\star_\Iw$. With this definition we have a canonical isomorphism
\[
 \For^{\Iw}_{\Iwu}(\scF \star_\Iw \scG) \cong \For^{\Iw}_{\Iwu}(\scF) \star_\Iw \scG
\]
for any $\scF,\scG$ in $\mathsf{D}_{\Iw,\Iw}$.

Since each $\Iwu$-orbit on $\Fl_G$ is isomorphic to an affine space, the methods of~\cite[\S\S 3.2--3.3]{bgs} show that $\mathsf{P}_{\Iwu,\Iw}$ has a natural structure of highest weight category in the sense considered e.g.~in~\cite[\S 2.1]{rw}, with underlying poset $W$ (endowed with the Bruhat order) and standard, resp.~costandard, object attached to $w$ the perverse sheaf $\For^{\Iw}_{\Iwu}(\Delta_w^{\Iw})$, resp.~$\For^{\Iw}_{\Iwu}(\nabla_w^{\Iw})$. In particular we have a notion of tilting object in this category (namely, objects which admit both a filtration with subquotients of the form $\For^{\Iw}_{\Iwu}(\Delta_w^{\Iw})$, and a filtration with subquotients of the form $\For^{\Iw}_{\Iwu}(\Delta_w^{\Iw})$), and the isomorphism classes of indecomposable tilting objects are in a natural bijection with $W$. The indecomposable tilting object associated with $w$ will be denoted $\scT_w$.

\subsection{Central sheaves -- properties}
\label{ss:central-sheaves-prop}

We now consider the action of $\Loop^+G$ on $\Gr_G$, and denote by
\[
\mathsf{D}_{\Loop^+G,\Loop^+G}
\]
the $\Loop^+G$-equivariant derived category of \'etale sheaves on $\Gr_G$ with coefficients in $\bk$. As for $\mathsf{D}_{\Iw,\Iw}$ we have a convolution bifunctor $\star_{\Loop^+G}$ on this category, which endows it with a monoidal structure. We also have a perverse t-structure, whose heart will be denoted
\[
 \sfP_{\Loop^+G,\Loop^+G}.
\]
It is a standard but crucial fact that this subcategory is stable under the bifunctor $\star_{\Loop^+G}$;
one can therefore consider the monoidal category $(\sfP_{\Loop^+G,\Loop^+G},\star_{\Loop^+G})$. This category is the main ingredient of the \emph{geometric Satake equivalence} of~\cite{mv}, which provides a canonical connected reductive algebraic group $G^\vee_\bk$ over $\bk$ with a maximal torus $T^\vee_\bk$ such that the root datum of $(G^\vee_\bk, T^\vee_\bk)$ is dual to that of $(G,T)$, and a canonical equivalence of monoidal categories
\[
 \Sat : (\sfP_{\Loop^+G,\Loop^+G},\star_{\Loop^+G}) \simto (\Rep(G^\vee_\bk),\otimes).
\]
(See~\cite[\S 4.1]{brr} for more precise references.)
The unit object in the category $\sfP_{\Loop^+G,\Loop^+G}$ will be denoted $\delta_\Gr$. (This object is the skyscraper sheaf at the base point of $\Gr_G$.) We will also denote by $B^\vee_\bk$ the Borel subgroup of $G^\vee_\bk$ containing $T^\vee_\bk$ such that the $T^\vee_\bk$-weights in the Lie algebra of $B^\vee_\bk$ are the negative coroots.

Below we will apply the constructions of Sections~\ref{sec:Coh-St}--\ref{sec:Hecke-cat} to the group $\bG=G^\vee_\bk$; in that setting, the groups $\Wf$ and $W$ identify with the groups $\bWf$ and $\bW$ considered in those sections, and their structures (in particular, the function $\ell$) also identify. We will denote by $S \subset W$ the subset corresponding to $\bS \subset \bW$.

Recall that the main construction of~\cite{gaitsgory} (reviewed in detail in~\cite{ar-book}) provides a canonical monoidal functor
\[
 \sfZ : \mathsf{D}_{\Loop^+G,\Loop^+G} \to \mathsf{D}_{\Iw,\Iw}.
\]
For $\scA,\scB$ in $\mathsf{D}_{\Loop^+G,\Loop^+G}$ we will denote by
\[
 \phi_{\scA,\scB} : \sfZ(\scA \star_{\Loop^+ G} \scB) \simto \sfZ(\scA) \star_\Iw \sfZ(\scB)
\]
the associated ``monoidality'' isomorphism. 

This functor has a number of favorable properties, which are listed in~\cite[\S 4]{brr}. Among these properties, we note the following for later use.
\begin{enumerate}
\item
The functor $\sfZ$ is t-exact with respect to the perverse t-structures.
\item
For any $\scA$ in $\mathsf{D}_{\Loop^+G,\Loop^+G}$ and $\scF$ in $\mathsf{D}_{\Iw,\Iw}$, there exists a canonical isomorphism
\[
\sigma_{\scA,\scF} : \sfZ(\scA) \star_{\Iw} \scF \simto \scF \star_{\Iw} \sfZ(\scA),
\]
and $\sfZ$, together with the isomorphisms $\phi$ and $\sigma$, define a central functor from $\sfP_{\Loop^+G,\Loop^+G}$ to $\mathsf{D}_{\Iw,\Iw}$ in the sense of~\cite{bez}; in other words these data define a braided monoidal functor from $\sfP_{\Loop^+G,\Loop^+G}$ to the Drinfeld center of $\mathsf{D}_{\Iw,\Iw}$ (with respect to the commutativity constraint on $\sfP_{\Loop^+G,\Loop^+G}$ and the natural braiding on the Drinfeld center).
\item
Since it is defined by nearby cycles, the functor $\sfZ$ comes with a ``monodromy'' automorphism $\sm$, such that for $\scA,\scB$ in $\mathsf{D}_{\Loop^+G,\Loop^+G}$ the isomorphism $\phi_{\scA,\scB}$ intertwines $\sm_{\scA \star_{\Loop^+G} \scG}$ with $\sm_{\scA} \star_{\Iw} \sm_{\scB}$.
 \end{enumerate}


For simplicity, from now on we fix a total order $\leq$ on $X_*(T)$ compatible with the dominance order, i.e.~such that if $\lambda,\mu \in X_*(T)$ are such that $\mu \preceq \lambda$ then $\mu \leq \lambda$. Recall that in $\mathsf{D}_{\Iw,\Iw}$ we have the \emph{Wakimoto sheaves} $(\Wak_\lambda : \lambda \in X_*(T))$, see~\cite[\S 4.5]{brr}, which are perverse sheaves such that for any $\lambda,\mu \in X_*(T)$ we have a canonical isomorphism
\begin{equation}
\label{eqn:convolution-Wak}
\Wak_\lambda \star_\Iw \Wak_\mu \cong \Wak_{\lambda+\mu}.
\end{equation}
 (The construction of these objects is due to Mirkovi{\'c}, and appears in particular in~\cite{ab}.)
Recall that an object $\scF$ of $\mathsf{P}_{\Iw,\Iw}$ (resp.~$\mathsf{P}_{\Iwu,\Iw}$) is said to \emph{admit a Wakimoto filtration} if there exists a finite filtration on $\scF$ such that each subquotient is of the form $\Wak_\lambda$ (resp.~$\For^{\Iw}_{\Iwu}(\Wak_\lambda)$) for some $\lambda \in X_*(T)$. In this case, there exists a unique filtration $(\scF_{\leq \lambda} : \lambda \in X_*(T))$ on $\scF$ such that $\scF_{\leq \lambda}=\{0\}$ for some $\lambda$, $\scF_{\leq \mu}=\scF$ for some $\mu$,  and $\scF_{\leq \lambda} / \scF_{< \lambda}$ is a direct sum of copies of $\Wak_\lambda$ for each $\lambda \in X_*(T)$. (Here, $\scF_{< \lambda}$ means $\scF_{\leq \lambda'}$ where $\lambda'$ is the predecessor of $\lambda$.) Moreover this filtration is functorial: if $\scF,\scG$ admit Wakimoto filtrations and $f : \scF \to \scG$ is any morphism, then $f(\scF_{\leq \lambda}) \subset \scG_{\leq \lambda}$ for any $\lambda \in X_*(T)$; this allows to define the functor $\gr_\lambda$ sending an object $\scF$ which admits a Wakimoto filtration to
\[
\gr_\lambda(\scF):=\scF_{\leq \lambda}/\scF_{< \lambda}.
\]

This notion is relevant in the present context thanks to a result of Arkhipov and the first author (see~\cite[Theorem~4]{ab}; see also~\cite[\S 4.4]{ar-book} for the extension to positive-characteristic coefficients) which claims that $\sfZ(\scA)$ admits a Wakimoto filtration for any $\scA$ in $\sfP_{\Loop^+G,\Loop^+G}$, and that moreover the multiplicity of $\Wak_\lambda$ in $\gr_\lambda(\sfZ(\scA))$ is the dimension of the $\lambda$-weight space of $\Sat(\scA)$.

\subsection{Central sheaves -- construction}
\label{ss:central-sheaves-const}

For later reference, we now briefly recall how the functor $\sfZ$ and the relevant isomorphisms are constructed.
This functor is defined using nearby cycles associated with an ind-scheme
\[
 \Gr_G^{\mathrm{Cen}} \to \mathbb{A}_\F^1
\]
called the central affine Grassmannian,
whose fiber over $0$ identifies canonically with $\Fl_G$, and whose restriction to $\mathbb{A}_\F^1\smallsetminus \{0\}$ identifies (again, canonically) with $\Gr_G \times (\mathbb{A}_\F^1\smallsetminus \{0\})$. 
We have a smooth affine group scheme $\mathcal{G}$ over $\mathbb{A}_\F^1$ whose restriction to $\mathbb{A}^1_\F \smallsetminus \{0\}$ identifies with $G \times (\mathbb{A}^1_\F \smallsetminus \{0\})$, and whose group of $\F[ \hspace{-1pt} [z] \hspace{-1pt} ]$-points is $\Iw$. In~\cite[\S 2.2.3]{ar-book} the construction of this group scheme is explained (following Zhu) using fpqc descent. Following~\cite{mrr}, this group scheme also admits another equivalent description, as the N\'eron blowup of $G \times \mathbb{A}^1_\F$ in $B$ along the divisor $\{0\} \subset \mathbb{A}_\F^1$; see in particular~\cite[Example~3.3]{mrr}. Then $\Gr_G^{\mathrm{Cen}}$ is defined as the $\F$-scheme which represents the functor sending an $\F$-algebra $R$ to the set of isomorphism classes of triples $(y,\mathcal{E},\beta)$ where $y \in \mathbb{A}^1_\F(R)$, $\mathcal{E}$ is a principal $\mathcal{G}$-bundle over $\mathbb{A}^1_R$, and $\beta$ is a trivialization of $\mathcal{E}$ over $\mathbb{A}^1_R \smallsetminus \Gamma_y$ (where $\Gamma_y \subset \mathbb{A}^1_R$ is the graph of $y$). Using the Beauville--Laszlo descent theorem (see~\cite[Remark~2.2.12]{ar-book} for details) one sees that $\Gr_G^{\mathrm{Cen}}(R)$ also classifies isomorphism classes of triples $(y,\mathcal{E}',\beta')$ where $y$ is as above, $\mathcal{E}'$ is a principal $\mathcal{G}$-bundle over the completion $\widehat{\Gamma}_y$ of $\mathbb{A}^1_R$ along $\Gamma_y$, and $\beta'$ is a trivialization on $\widehat{\Gamma}_y \smallsetminus \Gamma_y$. Using this description, the identification
\[
\{0\} \times_{\mathbb{A}_\F^1} \Gr_G^{\mathrm{Cen}} = \Fl_G
\]
simply follows from the fact that $\Fl_G$ represents the functor sending $R$ to isomorphism classes of pairs consisting of a principal $\mathcal{G}_{|\mathrm{Spec}(\F[ \hspace{-1pt} [z] \hspace{-1pt} ])}$-bundle over $\mathrm{Spec}(R[ \hspace{-1pt} [z] \hspace{-1pt} ])$ together with a trivialization over $\mathrm{Spec}(R( \hspace{-1pt} (z) \hspace{-1pt} ))$; see~\cite[Proposition~2.2.6]{ar-book} for details. (Here, $\mathcal{G}_{|\mathrm{Spec}(\F[ \hspace{-1pt} [z] \hspace{-1pt} ])}$ is the Iwahori group scheme attached to $B$.) The identification
\[
(\mathbb{A}^1_\F \smallsetminus \{0\}) \times_{\mathbb{A}^1_\F} \Gr_G^{\mathrm{Cen}} = \Gr_G \times (\mathbb{A}^1_\F \smallsetminus \{0\})
\]
is obtained using the similar moduli description of $\Gr_G$ (in terms of $G$-bundles, see~\cite[Proposition~2.2.2]{ar-book}) and the additive structure on $\mathbb{A}^1_\F$, which allows to identify $\widehat{\Gamma}_y$ with $\mathrm{Spec}(R[ \hspace{-1pt} [z] \hspace{-1pt} ])$.

The study of this functor also involves another scheme over $\mathbb{A}^1_\F$, denoted $\Gr_G^{\mathrm{BD}}$ and called the \emph{Be{\u\i}linson--Drinfeld affine Grassmannian}. This ind-scheme represents the functor sending an $\F$-algebra $R$ to isomorphism classes of triples $(y,\mathcal{E},\beta)$ where $y$ and $\mathcal{E}$ are as above, but now $\beta$ is a trivialization on $\mathbb{A}^1_R \smallsetminus (\Gamma_0 \cup \Gamma_y)$. (Here, $\Gamma_0=\{0\} \times \mathrm{Spec}(R) \subset \mathbb{A}^1_R$ is the graph of the constant point with value $0$.) We still have an identification
\[
\{0\} \times_{\mathbb{A}^1_\F} \Gr_G^{\mathrm{BD}} = \Fl_G,
\]
but now we have
\[
(\mathbb{A}^1_\F \smallsetminus \{0\}) \times_{\mathbb{A}^1_\F} \Gr_G^{\mathrm{BD}} = \Gr_G \times \Fl_G \times (\mathbb{A}^1_\F \smallsetminus \{0\}),
\]
see~\cite[Lemma~2.3.16]{ar-book}. As explained in~\cite[\S 3.2.1]{ar-book}, nearby cycles along $\Gr_G^{\mathrm{BD}} \to \mathbb{A}^1_\F$ define a bifunctor
\[
\mathsf{Y} : \mathsf{D}_{\Loop^+G, \Loop^+G} \times \mathsf{D}_{\Iw,\Iw} \to \mathsf{D}_{\Iw,\Iw}.
\]
By~\cite[Theorem~3.2.3]{ar-book}, for $\scA$ in $D_{\Loop^+G, \Loop^+G}$ and $\scF$ in $D_{\Iw,\Iw}$ we have canonical isomorphisms
\begin{equation}
\label{eqn:isom-C-Z}
\sfZ(\scA) \star_{\Iw} \scF \cong \mathsf{Y}(\scA,\scF) \cong \scF \star_{\Iw} \sfZ(\scA);
\end{equation}
in fact the composition of these isomorphisms is precisely the definition of $\sigma_{\scA,\scF}$.


We will now explain how the functor $\sfZ$, and its various structures, can be entirely described in terms of the bifunctor $\mathsf{Y}$ and some related structures. First,
applying~\eqref{eqn:isom-C-Z} in case $\scF=\delta_{\Fl}$, we see that we have $\sfZ(\scA)=\mathsf{Y}(\scA,\delta_{\Fl})$. This can also be seen more directly (in particular, without using~\eqref{eqn:isom-C-Z}) from the compatibility of nearby cycles with proper pushforward, after we remark that there exists a closed embedding $\Gr_G^{\mathrm{Cen}} \hookrightarrow \Gr_G^{\mathrm{BD}}$: in terms of functors this embedding is obtained by sending a triple $(y,\cE,\beta)$ to the triple $(y,\cE,\beta')$ where $\beta'$ is the restriction of $\beta$ to $\mathbb{A}^1_R \smallsetminus (\Gamma_0 \cup \Gamma_y)$. The restriction of this embedding to $\mathbb{A}^1_\F \smallsetminus \{0\}$ identifies with the natural embedding
\[
\Gr_G \times (\mathbb{A}^1_\F \smallsetminus \{0\})= \Gr_G \times \Fl_{G,e} \times (\mathbb{A}^1_\F \smallsetminus \{0\}) \hookrightarrow \Gr_G \times \Fl_{G} \times (\mathbb{A}^1_\F \smallsetminus \{0\}).
\]

Now we consider the isomorphism $\phi_{\scA,\scB}$. The same arguments as for the construction of this isomorphism (see~\cite[\S 3.4.1]{ar-book}) show that for $\scA,\scB$ in $\mathsf{D}_{\Loop^+G,\Loop^+G}$ and $\scF,\scG$ in $\mathsf{D}_{\Iw,\Iw}$ we have a canonical isomorphism
\begin{equation}
\label{eqn:C-convolution}
\mathsf{Y}(\scA,\scF) \star_{\Iw} \mathsf{Y}(\scB,\scG) \cong \mathsf{Y}(\scA \star_{\Loop^+G} \scB,\scF \star_{\Iw} \scG).
\end{equation}
Using this for $\scF=\scG=\delta_{\Fl}$ and using the identification above we recover the isomorphism $\phi_{\scA,\scB}$.

Finally, we note that the isomorphisms in~\eqref{eqn:isom-C-Z} can be reconstructed from~\eqref{eqn:C-convolution}, using the fact that the functor $\mathsf{Y}(\delta_\Gr,-)$ is the identity. To justify the latter claim one remarks that there exists a natural closed embedding $\Fl_G \times \mathbb{A}^1_\F \hookrightarrow \Gr_G^{\mathrm{BD}}$, obtained using restriction of trivializations as above, and the fact that $\Fl_G \times \mathbb{A}^1_\F$ represents the functor sending an $\F$-algebra $R$ to isomorphism classes of triples $(y,\cE,\beta)$ where $y \in \mathbb{A}^1_\F(R)$, $\cE$ is a $\cG$-bundle on $\mathbb{A}^1_R$, and $\beta$ is a trivialization on $\mathbb{A}^1_R \smallsetminus \Gamma_0$. The restriction of this embedding to $\mathbb{A}^1_\F \smallsetminus \{0\}$ identifies with the natural embedding
\[
\Fl_G \times (\mathbb{A}^1_\F \smallsetminus \{0\}) = \Gr_G^0 \times \Fl_G \times (\mathbb{A}^1_\F \smallsetminus \{0\}) \hookrightarrow \Gr_G \times \Fl_G \times (\mathbb{A}^1_\F \smallsetminus \{0\}).
\]
Since nearby cycles for a constant family identify with the identity functor, one deduces the claim.

Then, applying the isomorphism~\eqref{eqn:C-convolution} with $\scF=\delta_{\Fl}$ and $\scB=\delta_\Gr$ one obtains the first isomorphism in~\eqref{eqn:isom-C-Z}, and applying this isomorphism for $\scA=\delta_\Gr$ and $\scG=\delta_{\Fl}$ one obtains the second one.

\begin{rmk}
\label{rmk:C-noneq}
In order to define the isomorphism $\sigma_{\scA,\scF}$ we need to consider objects in $\mathsf{D}_{\Loop^+G, \Loop^+G}$ and $\mathsf{D}_{\Iw,\Iw}$. But the definition of $\mathsf{Y}$ in terms of nearby cycles makes sense without any equivariant structure. In particular, if we denote by $\Db_c(\Fl_G,\bk)$ the constructible derived category of $\bk$-sheaves on $\Fl_G$, we have a natural bifunctor
\[
 \mathsf{D}_{\Loop^+G, \Loop^+G} \times \mathsf{D}_{\Iwu,\Iw} \to \Db_c(\Fl_G,\bk),
\]
which will again be denoted $\mathsf{Y}$, and which satisfies
\[
\mathsf{Y}(\scA, \For^{\Iw}_{\Iwu}(\scF)) \cong \For^{\Iw}_{\Iwu}(\mathsf{Y}(\scA,\scF))
\]
for $\scA$ in $\mathsf{D}_{\Loop^+G, \Loop^+G}$ and $\scF$ in $\mathsf{D}_{\Iw,\Iw}$. In particular, this bifunctor therefore factors through a bifunctor
\[
 \mathsf{D}_{\Loop^+G, \Loop^+G} \times \mathsf{D}_{\Iwu,\Iw} \to \mathsf{D}_{\Iwu,\Iw},
\]
which will again be denoted $\mathsf{Y}$.
\end{rmk}

Below we will need the following standard property. Recall that we have a ``loop rotation'' action of the multiplicative group $\Gm$ on $\Loop G$.  
(We normalize this action in such a way that for $t \in \bk^\times$ we have $t \cdot z^\lambda=\lambda(t)^{-1} \cdot z^\lambda$.) This action induces actions on $\tFl_G$, $\Fl_G$ and $\Gr_G$ such that the morphisms in~\eqref{eqn:morph-Fl} are equivariant.

\begin{lem}
\label{lem:action-Gm}
There exists a canonical action of $\Gm$ on $\Gr_G^{\mathrm{Cen}}$ such that the morphism $\Gr_G^{\mathrm{Cen}} \to \mathbb{A}^1_\F$ is $\Gm$-equivariant with respect to the standard action on $\mathbb{A}^1_\F$ (by dilation), and whose restriction to $\mathbb{A}^1_\F \smallsetminus \{0\}$, resp.~$\{0\}$, identifies (via the isomorphisms considered above) with the product of the standard action on $\mathbb{A}^1_\F \smallsetminus \{0\}$ and the loop rotation action on $\Gr_G$, resp.~with the loop rotation action on $\Fl_G$.
\end{lem}

\begin{proof}[Sketch of proof]
First we note that there exists an action of $\Gm$ on $\cG$, compatible with the group structure in the obvious way, such that the projection $\cG \to \mathbb{A}^1_\F$ is $\Gm$-equivariant (with respect to the standard action on $\mathbb{A}^1_\F$), which restricts over $\mathbb{A}^1_\F \smallsetminus \{0\}$ to the action on $G \times (\mathbb{A}^1_\F \smallsetminus \{0\})$ on the second factor, and such that the induced action on $\F[ \hspace{-1pt} [z] \hspace{-1pt} ]$-points is by loop rotation. This action can e.g.~by constructed using the formalism of~\cite{mrr} as follows. By compatibility of N\'eron blowups with base change (see~\cite[Theorem~3.2(6)]{mrr}), the fiber product
\[
\cG \times_{\mathbb{A}^1_\F} (\Gm \times \mathbb{A}^1_\F),
\]
where the morphism $\Gm \times \mathbb{A}_\F^1 \to \mathbb{A}_\F^1$ is the action morphism, is the N\'eron blowup of $(G \times \mathbb{A}_\F^1) \times_{\mathbb{A}^1_\F} (\Gm \times \mathbb{A}^1_\F)$ in $B \times (\Gm \times \{0\})$ along $\Gm \times \{0\}$. Now using the $\Gm$-action on $\mathbb{A}^1_\F$ we obtain an identification of $(G \times \mathbb{A}^1_\F) \times_{\mathbb{A}^1_\F} (\Gm \times \mathbb{A}^1_\F)$
with the similar fiber product where the morphism $\Gm \times \mathbb{A}^1_\F \to \mathbb{A}^1_\F$ is the projection. Again by compatibility of N\'eron blowups with base change, we deduce an isomorphism
\[
\cG \times_{\mathbb{A}^1_\F} (\Gm \times \mathbb{A}^1_\F) \simto \Gm \times \cG
\]
as schemes over $\Gm \times \mathbb{A}^1_\F$.
Composing the inverse isomorphism with the natural projection on $\cG$
defines the desired action.

Once this action is constructed, the $\Gm$-action on $\Gr_G^{\mathrm{Cen}}$ is obtained using pullback of torsors; details are left to the reader.
\end{proof}

For simplicity of notation, below we will set
\[
\sZ := \sfZ \circ \Sat^{-1} : \Rep(G^\vee_\bk) \to \mathsf{D}_{\Iw,\Iw},
\]
and write
\[
 \sm_V:=\sm_{\Sat^{-1}(V)} \in \End(\sZ(V))
\]
for $V$ in $\Rep(G^\vee_\bk)$.

\subsection{Extending the functor \texorpdfstring{$\sZ$}{Z} to coherent sheaves on \texorpdfstring{$G^\vee_\bk$}{G}}
\label{ss:extension-Z}

Below we will use the following general construction. Let $H$ be an affine $\bk$-group scheme of finite type. We consider the category $\QCoh^H(H)$ of $H$-equivariant quasi-coherent sheaves on $H$, where $H$ acts on itself via the adjoint action, which identifies with the category of $H$-equivariant $\scO(H)$-modules. Extending a construction explained in~\S\ref{ss:coh-Groth} slightly, the identity functor of $\QCoh^H(H)$ possesses a canonical automorphism $\sm^{\mathrm{taut}}_{(-)}$, which can be described as follows. Any $H$-equivariant $\scO(H)$-module $M$ admits a canonical automorphism, defined as the composition
\[
M \to M \otimes \scO(H) \to M
\]
where the first morphism is the coaction (with respect to the $H$-module structure on $M$) and the second one is the action morphism. It is easily checked that this morphism is a morphism of $H$-equivariant $\scO(H)$-modules, and defines an automorphism of the object $\scF$ corresponding to $M$ in $\QCoh^H(H)$, which by definition is $\sm^{\mathrm{taut}}_{\scF}$.

The category $\QCoh^H(H)$ admits a monoidal structure, given by tensor product of $\scO_H$-modules. It is easily checked that for $\scF,\scG$ in $\QCoh^H(H)$ we have
\[
\sm^{\mathrm{taut}}_{\scF \otimes_{\scO_H} \scG} = \sm^{\mathrm{taut}}_\scF \otimes_{\scO_H} \sm^{\mathrm{taut}}_\scG.
\]

\begin{rmk}
\label{rmk:taut-autom}
As in~\S\ref{ss:univ-centralizer} one can consider the universal stabiliser $\mathfrak{S}_{H,H}$ associated with the adjoint $H$-action on itself, and we have a canonical (monoidal) functor
\begin{equation}
\label{eqn:functor-Coh-Rep-Stab}
\QCoh^H(H) \to \Rep^{\infty}(\mathfrak{S}_{H,H}).
\end{equation}
As in Remark~\ref{rmk:section-J}, any object in $\Rep^{\infty}(\mathfrak{S}_{H,H})$ admits a tautological automorphism. It is easily checked that the functor~\eqref{eqn:functor-Coh-Rep-Stab} sends $\sm^{\mathrm{taut}}_\scF$ to the tautological automorphism of its image.
\end{rmk}

Now, consider the full subcategory $\Coh^H(H)$ of $H$-equivariant coherent sheaves on $H$, and the category $\Rep(H)$ of finite-dimensional representations of $H$.
We have a canonical monoidal functor
\[
\imath : \Rep(H) \to \Coh^H(H)
\]
defined by $V \mapsto V \otimes \scO_H$, where the $H$-equivariant structure on $V \otimes \scO_H$ is diagonal. Following our convention in~\S\ref{ss:coh-Groth}, for $V$ in $\Rep(H)$ we will write $\sm^{\mathrm{taut}}_V$ for $\sm^{\mathrm{taut}}_{\imath(V)}$, so that $\sm^{\mathrm{taut}}_V$ is an automorphism of $V \otimes \scO_H$ which satisfies
\[
\sm^{\mathrm{taut}}_{V_1 \otimes V_2} = \sm^{\mathrm{taut}}_{V_1} \otimes \sm^{\mathrm{taut}}_{V_2}
\]
for $V_1,V_2 \in \Rep(H)$. We will denote by $\Coh_{\mathrm{fr}}^H(H)$ the full subcategory of $\Coh^H(H)$ whose objects are the coherent sheaves $V \otimes \scO_H$ for $V$ in $\Rep(H)$, so that $\imath$ factors through a functor $\Rep(H) \to \Coh^H_{\mathrm{fr}}(H)$ (still denoted $\imath$) which is the obvious bijection on objects. (Note that $\Coh_{\mathrm{fr}}^H(H)$ is defined as a full, but not strictly full, subcategory of $\Coh^H(H)$.)

The following lemma is a variant of~\cite[Proposition~4(a)]{ab}, and follows from similar arguments.

\begin{lem}
\label{lem:extension-functor}
Let $\mathsf{A}$ be an additive $\bk$-linear monoidal category, and let
\[
F : \Rep(H) \to \mathsf{A}
\]
be a $\bk$-linear monoidal functor. Let $\mathsf{N}_{(-)}$ be an automorphism of $F$ such that for any $V_1,V_2$ in $\Rep(H)$ we have
\[
\mathsf{N}_{V_1 \otimes V_2} = \mathsf{N}_{V_1} \otimes \mathsf{N}_{V_2},
\]
and the diagram
\[
\xymatrix@C=1.5cm{
F(V_1 \otimes V_2) \ar[d] \ar[r]^{\mathsf{N}_{V_1 \otimes V_2}} & F(V_1 \otimes V_2) \ar[d] \\
F(V_2 \otimes V_1) \ar[r]^{\mathsf{N}_{V_2 \otimes V_1}} & F(V_2 \otimes V_1)
}
\]
commutes, where the vertical arrows are the images under $F$ of the commutativity isomorphisms in $\Rep(H)$. Then there exists a unique $\bk$-linear monoidal functor
\[
F^{\Coh} : \Coh^H_{\mathrm{fr}}(H) \to \mathsf{A}
\]
such that $F^{\Coh} \circ \imath = F$, and such that
\[
F^{\Coh}(\sm^{\mathrm{taut}}_V)=\mathsf{N}_V
\]
for any $V$ in $\Rep(H)$.
\end{lem}

In more concrete terms, this lemma says that the datum of $\mathsf{N}_{(-)}$ allows to ``extend'' in a canonical way the morphisms
\[
\Hom_{\Rep(H)}(V_1,V_2) \to \Hom_{\mathsf{A}}(F(V_1),F(V_2))
\]
to morphisms
\[
\Hom_{\Coh^H(H)}(V_1 \otimes \scO_H,V_2 \otimes \scO_H) \to \Hom_{\mathsf{A}}(F(V_1),F(V_2)),
\]
for any $V_1,V_2 \in \Rep(H)$. (Here, the left-hand side identifies with $(V_1^* \otimes V_2)^H$ in the first case, and with $(V_1^* \otimes V_2 \otimes \scO(H))^H$ in the second case.)

\begin{rmk}
\label{rmk:extension-restriction}
If $K \subset H$ is a closed subgroup scheme and if $h \in H$ commutes with $K$, then we can apply Lemma~\ref{lem:extension-functor} to the restriction functor $\For^H_K : \Rep(H) \to \Rep(K)$, and its automorphism induced by $h$. In this case, the functor
\[
\Coh^H_{\mathrm{fr}}(H) \to \Rep(K)
\]
is induced by restriction of coherent sheaves to $h \in H$.
\end{rmk}

Applying Lemma~\ref{lem:extension-functor} to the monoidal functor $\sZ : \Rep(G^\vee_\bk) \to \mathsf{D}_{\Iw,\Iw}$ and its automorphism $\sm_{(-)}$, we obtain a canonical monoidal functor
\[
\sZ^\Coh : \Coh^{G^\vee_\bk}_{\mathrm{fr}}(G^\vee_\bk) \to \mathsf{D}_{\Iw,\Iw}.
\]
In particular, this provides for any $V$ in $\Rep(G^\vee_\bk)$ a canonical algebra morphism
\begin{equation}
\label{eqn:ZCoh-morph}
\End_{\Coh^{G^\vee_\bk}(G^\vee_\bk)}(V \otimes \scO_{G^\vee_\bk}) \to \End_{\mathsf{P}_{\Iw,\Iw}}(\sZ(V)).
\end{equation}

The $G^\vee_\bk$-module $\scO(G^\vee_\bk)$ (endowed with the action induced by left multiplication of $G^\vee_\bk$ on itself) defines an ind-object in $\Rep(G^\vee_\bk)$ (namely, the functor $V \mapsto \Hom_{G^\vee_\bk}(V,\scO(G^\vee_\bk))$); therefore, applying $\sZ$ we deduce an ind-object $\sZ(\scO(G^\vee_\bk))$ in $\mathsf{P}_{\Iw,\Iw}$. As a special case of~\eqref{eqn:ZCoh-morph} we have a canonical algebra morphism
\[
\End_{{\ind}\Coh^{G^\vee_\bk}(G^\vee_\bk)}(\scO(G^\vee_\bk) \otimes \scO_{G^\vee_\bk}) \to \End_{{\ind}\mathsf{P}_{\Iw,\Iw}}(\sZ(\scO(G^\vee_\bk))).
\]

Note that given a ring object $X$ in a $\bk$-linear monoidal category $(\mathsf{A},\odot)$, with unit object $\mathbf{1}$, the vector space
\[
\Hom_{\mathsf{A}}(\mathbf{1},X)
\]
admits a natural structure of $\bk$-algebra, where the product of two morphisms $f,g : \mathbf{1} \to X$ is the composition
\[
\mathbf{1} = \mathbf{1} \odot \mathbf{1} \xrightarrow{f \odot g} X \odot X \to X,
\]
where the right morphism is the multiplication map for $X$. Moreover, there exists an algebra morphism 
\[
\Hom_{\mathsf{A}}(\mathbf{1},X) \to \End_{\mathsf{A}}(X)^{\mathrm{op}}
\]
sending a morphism $f : \mathbf{1} \to X$ to the morphism
\[
X =  X \odot \mathbf{1} \xrightarrow{\id \odot a} X \odot X \to X,
\]
where again the rightmost morphism is induced by multiplication in $X$; this morphism in fact takes values in endomorphisms of $X$ seen as a left module over itself.

Let us apply this construction to the ring-object
 $\scO(G^\vee_\bk) \otimes \scO_{G^\vee_\bk}$ in ${\ind}\Coh^{G^\vee_\bk}(G^\vee_\bk)$. Then we have
\[
\Hom(\scO_{G^\vee_\bk},\scO(G^\vee_\bk) \otimes \scO_{G^\vee_\bk}) \cong \scO(G^\vee_\bk \times G^\vee_\bk)^{G^\vee_\bk},
\]
where $G^\vee_\bk$ acts on $G^\vee_\bk \times G^\vee_\bk$ via $g \cdot (h_1,h_2) = (gh_1, gh_2g^{-1})$.
Now the morphism $G^\vee_\bk \times G^\vee_\bk \to G^\vee_\bk$ defined by $(g,h) \mapsto g^{-1}hg$ defines an algebra isomorphism
\[
\scO(G^\vee_\bk) \simto \scO(G^\vee_\bk \times G^\vee_\bk)^{G^\vee_\bk},
\]
which therefore provides a canonical algebra morphism
\[
\scO(G^\vee_\bk) \to \End(\scO(G^\vee_\bk) \otimes \scO_{G^\vee_\bk}),
\]
hence finally an algebra morphism
\begin{equation}
\label{eqn:action-OG-Z}
\scO(G^\vee_\bk) \to \End_{{\ind}\mathsf{P}_{\Iw,\Iw}}(\sZ(\scO(G^\vee_\bk))).
\end{equation}
In this way, $\sZ(\scO(G^\vee_\bk))$ becomes an $\scO(G^\vee_\bk)$-module in the category $\ind\mathsf{P}_{\Iw,\Iw}$, in the sense recalled in~\S\ref{ss:mod-cat}.

\subsection{The regular quotient: definition}
\label{ss:regular-quotient}

The next considerations will make intensive use of the notions of Serre quotient of an abelian category and of Verdier quotient of a triangulated category; for a brief reminder on these notions, and references, see~\S\ref{ss:quotient-cat}.

The main player of~\cite{brr} is the abelian category
\[
\mathsf{P}^0_{\Iw,\Iw},
\]
defined as the Serre quotient of the abelian category $\mathsf{P}_{\Iw,\Iw}$ by the Serre subcategory $\mathsf{P}_{\Iw,\Iw}^+$ generated by the simple objects $\IC_w$ for $w \in W$ such that $\ell(w)>0$. If we denote by $\D_{\Iw,\Iw}^+$ the full triangulated subcategory of $\D_{\Iw,\Iw}$ generated by $\mathsf{P}_{\Iw,\Iw}^+$, and by $\D_{\Iw,\Iw}^0$ the Verdier quotient of $\D_{\Iw,\Iw}$ by $\D_{\Iw,\Iw}^+$, then by Lemma~\ref{lem:quotient-t-str} (applied using the perverse t-structure on $\D_{\Iw,\Iw}$) there exists a unique t-structure on $\D_{\Iw,\Iw}^0$ such that the quotient functor $\Pi_{\Iw,\Iw}^0 : \D_{\Iw,\Iw} \to \D_{\Iw,\Iw}^0$ is t-exact, and moreover this t-structure is bounded, and its heart identifies canonically with $\mathsf{P}^0_{\Iw,\Iw}$. This t-structure will be called the perverse t-structure, and the associated cohomology functors will be denoted $\pH^n(-)$. Since $\mathsf{P}_{\Iw,\Iw}$ is a finite-length category, so is $\mathsf{P}_{\Iw,\Iw}^0$, and its simple objects are the objects
\[
 \delta^0_{\omega} := \Pi^0_{\Iw,\Iw}(\IC_\omega)
\]
for $\omega \in \Omega$. (In case $\omega=e$ is the unit, we will write $\delta^0$ for $\delta^0_e$.)

Consider the bifunctor $\D_{\Iw,\Iw} \times \D_{\Iw,\Iw} \to \D_{\Iw,\Iw}^0$ sending a pair $(\scF,\scG)$ to $\Pi^0_{\Iw,\Iw}(\scF \star_{\Iw} \scG)$. By~\cite[Lemma~5.1(1)]{brr} we have $\Pi^0_{\Iw,\Iw}(\scF \star_{\Iw} \scG)=0$ if either $\scF$ or $\scG$ belongs to $\D_{\Iw,\Iw}^+$. By the general properties of Verdier quotients (see~\S\ref{ss:quotient-cat}), it follows that there exists a unique bifunctor
\[
\star_\Iw^0 : \D^0_{\Iw,\Iw} \times \D^0_{\Iw,\Iw} \to \D^0_{\Iw,\Iw}
\]
such that
\[
\Pi_{\Iw,\Iw}^0(\scF) \star^0_\Iw \Pi_{\Iw,\Iw}^0(\scG) = \Pi_{\Iw,\Iw}^0(\scF \star_\Iw \scG)
\]
for any $\scF,\scG$ in $\D_{\Iw,\Iw}$. It is easily seen that this bifunctor equips $\D^0_{\Iw,\Iw}$ with the structure of a monoidal category, with monoidal unit $\delta^0$. 

From the fact that $\IC_\omega \star_{\Iw} \IC_{\omega'} \cong \IC_{\omega \omega'}$ for $\omega,\omega' \in \Omega$ one sees that $\star_\Iw^0$ is t-exact on both sides with respect to the perverse t-structure; it therefore restricts to a bifunctor $\sfP^0_{\Iw,\Iw} \times \sfP^0_{\Iw,\Iw} \to \sfP^0_{\Iw,\Iw}$ which equips $\sfP^0_{\Iw,\Iw}$ with a monoidal structure.
It is clear that
the functor
\[
\sZ^0 := \Pi^0_{\Iw,\Iw} \circ \sZ : \Rep(G^\vee_\bk) \to \sfP^0_{\Iw,\Iw}
\]
has a canonical monoidal structure.

\subsection{Another convolution bifunctor}
\label{ss:another-convolution}

For later use, we now explain a variant of the constructions of~\S\ref{ss:regular-quotient} where $\Iw$-equivariance is replaced by $\Iwu$-equivariance.

Recall the category $\D_{\Iwu,\Iw}$, its perverse t-structure, and the heart $\sfP_{\Iwu,\Iw}$ of this t-structure (see~\S\ref{ss:DbIw}). 
Let us denote by $\sfP_{\Iwu,\Iw}^+$, resp.~$\D_{\Iwu,\Iw}^+$, the Serre subcategory of $\sfP_{\Iwu,\Iw}$, resp.~the full triangulated subcategory of $\D_{\Iwu,\Iw}$, generated by the simple perverse sheaves $\For^{\Iw}_{\Iwu}(\IC_w)$ with $w \in W$ such that $\ell(w)>0$. We will denote by $\sfP_{\Iwu,\Iw}^0$ the Serre quotient of $\sfP_{\Iwu,\Iw}$ by $\sfP^+_{\Iwu,\Iw}$, and by $\D_{\Iwu,\Iw}^0$ the Verdier quotient of $\D_{\Iwu,\Iw}$ by $\D^+_{\Iwu,\Iw}$. Then by Lemma~\ref{lem:quotient-t-str} there exists a unique t-structure on $\D_{\Iwu,\Iw}^0$ such that the quotient functor
\[
\Pi^0_{\Iwu,\Iw} : \D_{\Iwu,\Iw} \to \D^0_{\Iwu,\Iw}
\]
is t-exact; moreover, this t-structure is bounded, and its heart identifies with $\sfP_{\Iwu,\Iw}^0$. This t-structure will be called the perverse t-structure, and the associated cohomology functors will be denoted $\pH^n(-)$.

By the universal property of the Verdier quotient, the composition
\[
\sfD_{\Iw,\Iw} \xrightarrow{\For^{\Iw}_{\Iwu}} \sfD_{\Iwu,\Iw} \xrightarrow{\Pi^0_{\Iwu,\Iw}} \sfD_{\Iwu,\Iw}^0
\]
factors through a triangulated functor
\[
\For^{\Iw,0}_{\Iwu} : \sfD_{\Iw,\Iw}^0 \to \sfD^0_{\Iwu,\Iw}.
\]
This functor is easily seen to be t-exact with respect to the perverse t-structures, and its restriction to the hearts identifies with the functor provided by the universal property of the Serre quotient.
%
%
Recall that $\For^{\Iw}_{\Iwu}$ is fully faithful on perverse sheaves;
from the standard description of morphisms in a Serre quotient category (see~\cite[\S III.1]{gabriel}), and since $\sfP_{\Iw,\Iw}$ is closed under subquotients in $\sfP_{\Iwu,\Iw}$, one sees that the restriction of $\For^{\Iw,0}_{\Iwu}$ to the heart of the perverse t-structure is fully faithful. 

Consider the bifunctor
\[
\D_{\Iwu,\Iw} \times \sfD_{\Iw,\Iw} \to \D^0_{\Iwu,\Iw}
\]
sending a pair $(\scF,\scG)$ to $\Pi^0_{\Iwu,\Iw}(\scF \star_\Iw \scG)$, where $\star_\Iw$ is as in~\eqref{eqn:conv-Iwu-Iw}. It follows again from~\cite[Lemma~5.1(1)]{brr} and the general properties of the Verdier quotient (see~\S\ref{ss:quotient-cat}) that
there exists a unique triangulated bifunctor
\[
\star_\Iw^0 : \D^0_{\Iwu,\Iw} \times \sfD^0_{\Iw,\Iw} \to \D^0_{\Iwu,\Iw}
\]
such that for $\scF$ in $\D_{\Iwu,\Iw}$ and $\scG$ in $\sfD_{\Iw,\Iw}$ we have
\[
\Pi^0_{\Iwu,\Iw}(\scF) \star^0_{\Iw} \Pi^0_{\Iw,\Iw}(\scG) = \Pi^0_{\Iwu,\Iw}(\scF \star^0_\Iw \scG).
\]
This bifunctor defines a right action of the monoidal category $(\sfD^0_{\Iw,\Iw}, \star^0_{\Iw})$ on $\sfD^0_{\Iwu,\Iw}$; for $\scF,\scG$ in $\sfD^0_{\Iw,\Iw}$ we also have a canonical isomorphism
\begin{equation}
\label{eqn:For-conv-0}
 \For^{\Iw,0}_{\Iwu}(\scF \star^0_{\Iw} \scG) \cong \For^{\Iw,0}_{\Iwu}(\scF) \star^0_{\Iw} \scG.
\end{equation}

As in the $\Iw$-equivariant setting in~\S\ref{ss:regular-quotient} one sees that the bifunctor $\star_\Iw^0$ is t-exact on both sides; its restriction to the hearts of the perverse t-structures defines a right action of the monoidal category $\sfP^0_{\Iw,\Iw}$ on $\sfP^0_{\Iwu,\Iw}$, by exact endofunctors.

\subsection{The regular quotient: coherent description}
\label{ss:regular-quotient-coh}

In this subsection we make the following assumptions:
\begin{enumerate}
 \item the quotient of $X^*(T)$ by the root lattice of $(G,T)$ is free;
 \item the quotient of $X_*(T)$ by the coroot lattice of $(G,T)$ has no $\ell$-torsion;
 \item for any indecomposable factor in the root system of $(G,T)$, $\ell$ is strictly bigger than the corresponding value in Figure~\ref{fig:bounds}.
\end{enumerate}
Here the first assumption is equivalent to requiring that $G^\vee_\bk$ has simply connected derived subgroup, and the second one that its scheme-theoretic center is smooth. The third assumption can most probably be weakened; it implies in particular that $\ell$ is good for $G$.

By~\cite[Theorem~5.4]{brr}, there exists a regular unipotent element $\su \in G^\vee_\bk$ and an equivalence of monoidal categories
\[
\Phi_{\Iw,\Iw} : (\mathsf{P}^0_{\Iw,\Iw}, \star_\Iw^0) \simto (\Rep(\rmZ_{G^\vee_\bk}(\su)),\otimes)
\]
such that
\[
\sZ^0 \circ \Phi_{\Iw,\Iw} \cong \For^{G^\vee_\bk}_{\rmZ_{G^\vee_\bk}(\su)}
\]
as monoidal functors, where $\rmZ_{G^\vee_\bk}(\su)$ is the centralizer of $\su$ and
\[
\For^{G^\vee_\bk}_{\rmZ_{G^\vee_\bk}(\su)} : \Rep(G^\vee_\bk) \to \Rep(\rmZ_{G^\vee_\bk}(\su))
\]
is the restriction functor. (Note that by Lemma~\ref{lem:smoothness-centralizers} the scheme-theoretic centralizer of $\su$ is smooth, so the structure we consider on $\rmZ_{G^\vee_\bk}(\su)$ is unambiguous.) This equivalence furthermore satisfies the property that the automorphism
\[
\Phi_{\Iw,\Iw}(\Pi^0_{\Iw,\Iw}(\sm_V))
\]
identifies with the action of $\su$ on $V$, for any $V$ in $\Rep(G^\vee_\bk)$.

Below we will need a more explicit description on this equivalence than what is provided in~\cite{brr}, which we now explain. This description will make use of the notion of tensor product with an $R$-module in a category, 
whose definition is recalled in~\S\ref{ss:mod-cat}.

We first recall the structure of the main construction in~\cite{brr}.
As explained in~\S\ref{ss:extension-Z} the $G^\vee_\bk$-module $\scO(G^\vee_\bk)$ defines an ind-object in $\Rep(G^\vee_\bk)$, which is moreover a \emph{ring} ind-object. The image $\sZ^0(\scO(G^\vee_\bk))$ therefore defines a ring ind-object in the category $\mathsf{P}^0_{\Iw,\Iw}$. 
As explained in~\cite[\S 3.3]{brr} (following~\cite{bez}), any left ideal subobject in $\sZ^0(\scO(G^\vee_\bk))$ is automatically a two-sided ideal; in particular if we fix a maximal left ideal subobject $\mathscr{J} \subset \sZ^0(\scO(G^\vee_\bk))$ then the quotient 
\[
\scR^0:= \sZ^0(\scO(G^\vee_\bk))/\mathscr{J}
\]
has a canonical structure of ring ind-object such that the surjection $\sZ^0(\scO(G^\vee_\bk)) \to \scR^0$ is a ring morphism. One then checks that the assignment
\[
\Psi_{\Iw,\Iw} : \scF \mapsto \Hom_{{\ind}\mathsf{P}^0_{\Iw,\Iw}}(\delta^0, \scR^0 \star^0_\Iw \scF)
\]
defines a functor from $\mathsf{P}^0_{\Iw,\Iw}$ to the category $\Vect_\bk$ of finite-dimensional $\bk$-vector spaces, that this functor
admits a canonical monoidal structure (induced in an appropriate way by the ring structure on $\scR^0$), and that its composition with $\sZ^0$ identifies with the forgetful functor $\For^{G^\vee_\bk} : \Rep(G^\vee_\bk) \to \Vect_\bk$. Using this functor we invoke Tannakian formalism to obtain a closed subgroup scheme $H \subset G^\vee_\bk$ and an equivalence of monoidal categories
\[
(\mathsf{P}^0_{\Iw,\Iw}, \star_\Iw^0) \simto (\Rep(H),\otimes)
\]
whose pre-composition with $\sZ^0$ is the restriction functor $\For^{G^\vee_\bk}_H$, and whose post-composition with the forgetful functor $\For^H$ is $\Psi_{\Iw,\Iw}$. From the automorphism $\sm$ of the functor $\sZ^0$ we obtain an automorphism of the functor $\For^{G^\vee_\bk}$, which defines an element $\su \in G^\vee_\bk$. Most of the content of~\cite{brr} is then devoted to showing that $\su$ is unipotent regular, and that $H=\rmZ_{G_\bk^\vee}(\su)$.\footnote{In fact the equivalence we initially obtain concerns only a full subcategory of $\mathsf{P}^0_{\Iw,\Iw}$, which is then shown to coincide with $\mathsf{P}^0_{\Iw,\Iw}$. This subtlety is irrelevant for the present discussion.}

A posteriori, the ind-object $\sZ^0(\scO(G^\vee_\bk))$ identifies with $\scO(G^\vee_\bk)$ seen as a $\rmZ_{G^\vee_\bk}(\su)$-representation; its maximal left ideals are therefore parametrized by the cosets in $\rmZ_{G^\vee_\bk}(\su) \backslash G^\vee_\bk$. We claim that, if we still denote by $\Phi_{\Iw,\Iw}$ the induced equivalence on ind-objects, we have a canonical identification
\[
\Phi_{\Iw,\Iw}(\scR^0) \cong \scO(\rmZ_{G^\vee_\bk}(\su)).
\]
In fact the ring surjection $\sZ^0(\scO(G^\vee_\bk)) \to \scR^0$ induces a ring map
\begin{equation}
\label{eqn:Phi-R0}
\scO(G^\vee_\bk)=\Phi_{\Iw,\Iw} \bigl( \sZ^0(\scO(G^\vee_\bk)) \bigr) \to \Phi_{\Iw,\Iw}(\scR^0),
\end{equation}
which as explained above identifies the right-hand side with functions on a certain coset in $\rmZ_{G^\vee_\bk}(\su) \backslash G^\vee_\bk$; what remains to be justified is that this coset is $\rmZ_{G^\vee_\bk}(\su)$. To check this it suffices to prove that our coset contains the unit element $e$, i.e.~that the augmentation morphism $\scO(G^\vee_\bk) \to \bk$ factors through our morphism~\eqref{eqn:Phi-R0}. For that we consider the commutative diagram
\[
\xymatrix{
\Hom_{G^\vee_\bk}(\bk, \scO(G^\vee_\bk) \otimes \scO(G^\vee_\bk)) \ar[r] \ar[d] & \Hom_{G^\vee_\bk}(\bk, \scO(G^\vee_\bk)) \ar[d] \\
\Hom_{{\ind}\sfP_{\Iw,\Iw}^0}(\delta^0, \sZ^0(\scO(G^\vee_\bk)) \star^0_\Iw \sZ^0(\scO(G^\vee_\bk))) \ar[r] \ar[rd] & \Hom_{{\ind}\sfP_{\Iw,\Iw}^0}(\delta^0, \sZ^0(\scO(G^\vee_\bk))) \ar[d] \\
 & \Hom_{{\ind}\sfP_{\Iw,\Iw}^0}(\delta^0, \scR^0)
}
\]
where the upper vertical arrows are induced by $\sZ^0$, the lower vertical arrow by the quotient morphism $\sZ^0(\scO(G^\vee_\bk)) \to \scR^0$, and the horizontal ones by multiplication in $\scO(G^\vee_\bk)$. (The $G^\vee_\bk$-action on each copy of $\scO(G^\vee_\bk)$ is the left regular action.) Here the diagonal arrow factors through $\Phi_{\Iw,\Iw}(\scR^0)$, and the upper line identifies with the augmentation morphism $\scO(G^\vee_\bk) \to \bk$, via the morphism $\scO(G^\vee_\bk) \otimes \scO(G^\vee_\bk) \to \scO(G^\vee_\bk)$ induced by $f \otimes g \mapsto f(e) g$ and Frobenius reciprocity. Moreover the composition of the right vertical arrows is an isomorphism in view of the isomorphism $\Hom(\delta^0, \scR^0) \cong \Phi_{\Iw,\Iw}(\delta^0) = \bk$. The desired claim follows.

%

Recall from~\eqref{eqn:action-OG-Z} that $\sZ(\scO(G^\vee_\bk))$ is an $\scO(G^\vee_\bk)$-module; hence $\sZ^0(\scO(G^\vee_\bk))$ has the same structure, and the corresponding morphism
\[
\scO(G^\vee_\bk) \to \End_{{\ind}\Rep(\rmZ_{G^\vee_\bk}(\su))}(\scO(G^\vee_\bk))
\]
(obtained by applying $\Phi_{\Iw,\Iw}$)
is induced by the morphism $G^\vee_\bk \to G^\vee_\bk$ given by $g \mapsto g^{-1} \su g$. It is easily seen from definitions that restriction induces an isomorphism
\[
 \scO(G^\vee_\bk) \otimes_{\scO(G^\vee_\bk)} \scO(\{\su\}) \simto \scO(\rmZ_{G^\vee_\bk}(\su)),
\]
from which we deduce a canonical isomorphism
\[
 \sZ^0(\scO(G^\vee_\bk)) \otimes_{\scO(G^\vee_\bk)} \scO(\{\su\}) \simto \scR^0.
\]
Here, since $\scO(G^\vee_\bk)$ acts on $\sZ^0(\scO(G^\vee_\bk))$ by endomorphisms of left modules, the left-hand side is naturally a left $\sZ^0(\scO(G^\vee_\bk))$-module, and this identification is compatible with this structure; it follows that the multiplication map on $\scR^0$ can be recovered from this description.

The comultiplication morphism of $\scO(G^\vee_\bk)$ defines a morphism of $G^\vee_\bk$-modules
\begin{equation}
\label{eqn:comult-OG}
\scO(G^\vee_\bk) \to \scO(G^\vee_\bk) \otimes \scO(G^\vee_\bk)
\end{equation}
(where the action on the right term in the tensor product is trivial) which we use to obtain a morphism
\[
\scO(G^\vee_\bk) \otimes_\bk \scO_{G^\vee_\bk} \to \bigl( \scO(G^\vee_\bk) \otimes_\bk \scO_{G^\vee_\bk} \bigr) \otimes \scO(G^\vee_\bk)
\]
in ${\ind}\Coh^{G^\vee_\bk}_{\mathrm{fr}}(G^\vee_\bk)$. Here the right-hand side has an action of $\scO(G^\vee_\bk) \otimes \scO(G^\vee_\bk)$ obtained from the $\scO(G^\vee_\bk)$-action on the first term (as in~\S\ref{ss:extension-Z}) and the obvious $\scO(G^\vee_\bk)$-action on the second term. If we restrict this action to $\scO(G^\vee_\bk)$ via the morphism induced by $(g,h) \mapsto h^{-1}gh$, then explicit computation shows that our morphism is $\scO(G^\vee_\bk)$-linear. We now consider the
morphism
\[
\sZ(\scO(G^\vee_\bk)) \to \sZ(\scO(G^\vee_\bk)) \otimes \scO(G^\vee_\bk)
\]
in ${\ind}\mathsf{P}_{\Iw,\Iw}$ obtained by applying (the extension to ind-objects of) $\sZ^\Coh$.
This morphism is again $\scO(G^\vee_\bk)$-linear;
as a consequence the composition
\begin{multline*}
\sZ(\scO(G^\vee_\bk)) \to \sZ(\scO(G^\vee_\bk)) \otimes \scO(G^\vee_\bk) \\
\to \bigl( \sZ(\scO(G^\vee_\bk)) \otimes_{\scO(G^\vee_\bk)} \scO(\{\su\}) \bigr) \otimes \scO(\rmZ_{G^\vee_\bk}(\su))
\end{multline*}
factors (uniquely) through a morphism
\begin{equation*}
\bigl( \sZ(\scO(G^\vee_\bk)) \otimes_{\scO(G^\vee_\bk)} \scO(\{\su\}) \bigr) \to \bigl( \sZ(\scO(G^\vee_\bk)) \otimes_{\scO(G^\vee_\bk)} \scO(\{\su\}) \bigr) \otimes \scO(\rmZ_{G^\vee_\bk}(\su)).
\end{equation*}
Applying $\Pi^0_{\Iw,\Iw}$, we deduce a morphism
\begin{equation}
\label{eqn:comult}
\bigl( \sZ^0(\scO(G^\vee_\bk)) \otimes_{\scO(G^\vee_\bk)} \scO(\{\su\}) \bigr) \to \bigl( \sZ^0(\scO(G^\vee_\bk)) \otimes_{\scO(G^\vee_\bk)} \scO(\{\su\}) \bigr) \otimes \scO(\rmZ_{G^\vee_\bk}(\su)).
\end{equation}

These considerations show that the functor
\begin{equation}
\label{eqn:Phi0}
\Phi_{\Iw,\Iw} : \mathsf{P}^0_{\Iw,\Iw} \simto \Rep(\rmZ_{G^\vee_\bk}(\su))
\end{equation}
can be reconstructed a posteriori as the functor
\[
\mathscr{F} \mapsto \Hom \bigl( \delta^0, (\sZ^0(\scO(G^\vee_\bk)) \otimes_{\scO(G^\vee_\bk)} \scO(\{\su\})) \star_\Iw^0 \mathscr{F}),
\]
with the monoidal structure induced by the product on $\sZ^0(\scO(G^\vee_\bk)) \otimes_{\scO(G^\vee_\bk)} \scO(\{\su\})$ induced by the product on $\sZ^0(\scO(G^\vee_\bk))$, and the $\rmZ_{G^\vee_\bk}(\su)$-action defined by the coaction induced by~\eqref{eqn:comult}. With this description, the regular unipotent element $\su$ can in fact be chosen a priori, and arbitrarily, and the induced functor~\eqref{eqn:Phi0} will be an equivalence in all cases. (As explained above, this choice is equivalent to the choice of a left ideal subobject in $\sZ^0(\scO(G^\vee_\bk))$.)

\begin{rmk}
\begin{enumerate}
\item
We do not claim that we know how to prove that $\Phi_{\Iw,\Iw}$ is an equivalence using the description as above, but only that we can give this description a posteriori, once we know that it provides an equivalence.
\item
These considerations show that the ind-object $\mathscr{R}^0$ is in fact the image under $\Pi^0_{\Iw,\Iw}$ of a canonical ind-object in $\mathsf{P}_{\Iw,\Iw}$, namely the tensor product
\begin{equation}
\label{eqn:defn-R}
\scR := \sZ(\scO(G^\vee_\bk)) \otimes_{\scO(G^\vee_\bk)} \scO(\{\su\}).
\end{equation}
Moreover, the same arguments as for $\mathscr{R}^0$ show that $\scR$ has a natural structure of ring ind-object.
As explained above, the morphism~\eqref{eqn:comult} defining the $\rmZ_{G^\vee_\bk}(\su)$-action is also defined at the level of this object.
\item
One can also describe a variant of the equivalence $\Phi_{\Iw,\Iw}$ which does not require any choice; namely, if $\cU_\reg$ denotes the unique open orbit in the unipotent cone $\cU$ of $G^\vee_\bk$, then as explained in~\cite[Equation~(2.2)]{brr} the map $h\rmZ_{G^\vee_\bk}(\su) \mapsto h \su h^{-1}$ induces an isomorphism of varieties
\[
G^\vee_\bk / \rmZ_{G^\vee_\bk}(\su) \simto \cU_\reg,
\]
hence an equivalence of categories
$\Coh^{G^\vee_\bk}(\cU_\reg) \simto \Rep(\rmZ_{G^\vee_\bk}(\su))$.
One can check that the composition of $\Phi_{\Iw,\Iw}$ with the inverse of this equivalence defines an equivalence
\[
\mathsf{P}_{\Iw,\Iw}^0 \simto \Coh^{G^\vee_\bk}(\cU_\reg)
\]
which is independent of the initial choice of the ideal $\mathscr{J}$ (or, equivalently, of the element $\su$).
\end{enumerate}
\end{rmk}

\section{Construction of the monodromic regular quotient}
\label{sec:construction-mon-reg-quot}

In this section we provisionally come back to the general setting of~\S\ref{ss:Fl}.

\subsection{\texorpdfstring{$\Iwu$}{Iu}-monodromic sheaves on the extended affine flag variety, convolution, and monodromy}
\label{ss:monodromic}

Recall the ind-scheme $\tFl_G$ defined in~\S\ref{ss:Fl}. This ind-scheme admits an action of $\Iwu$, and we can consider the associated equivariant derived category $\Db_{\Iwu}(\tFl_G,\bk)$. We will denote by
\[
 \D_{\Iwu,\Iwu}
\]
the full triangulated subcategory of $\Db_{\Iwu}(\tFl_G,\bk)$ generated by the essential image of the pullback functor $\pi^* : \Db_{\Iwu}(\Fl_G,\bk) \to \Db_{\Iwu}(\tFl_G,\bk)$. 
The perverse t-structure on $\Db_{\Iwu}(\tFl_G,\bk)$ restricts to a t-structure $(\pD^{\leq 0}_{\Iwu,\Iwu},\pD^{\geq 0}_{\Iwu,\Iwu})$ on $\mathsf{D}_{\Iwu,\Iwu}$, whose heart will be denoted $\mathsf{P}_{\Iwu,\Iwu}$.
Since $\pi$ is smooth with connected fibers the functor
\[
 \pi^\dag := \pi^*[\dim(T)] \cong \pi^![-\dim(T)] : \mathsf{D}_{\Iwu,\Iw} \to \mathsf{D}_{\Iwu,\Iwu}
\]
is t-exact,
and the simple objects in the category $\sfP_{\Iwu,\Iwu}$ are the object $\pi^\dag \For^{\Iw}_{\Iwu}(\IC_w)$ with $w \in W$.

%

We define a natural convolution product $- \star_{\Iwu} -$ on the category $\Db_{\Iwu}(\tFl_G,\bk)$ as follows. We denote by $\tFl_G \, \widetilde{\times} \, \tFl_G$ the fppf quotient of $\Loop G \times \tFl_G$ by the action of $\Iwu$ defined by $g \cdot (h,x)=(hg^{-1}, g \cdot x)$; it is easily seen that this functor is represented by an ind-scheme, which we denote in the same way. 
The multiplication map in $\Loop G$ defines a (non proper!) morphism
$\widetilde{m} : \tFl_G \, \widetilde{\times} \, \tFl_G \to \tFl_G$. Then, given $\scF,\scG$ in $\Db_{\Iwu}(\tFl_G,\bk)$, there exists a unique complex $\scF \, \widetilde{\boxtimes} \, \scG$ in $\Db_{\Iwu}(\tFl_G \, \widetilde{\times} \, \tFl_G,\bk)$ whose pullback to $\Loop G \times \tFl_G$ is the exterior product of the pullback of $\scF$ to $\Loop G$ with $\scG$. We set
\[
 \scF \star_{\Iwu} \scG := \widetilde{m}_!(\scF \, \widetilde{\boxtimes} \, \scG)[\dim(T)].
\]
It is not difficult to check that this operation admits a natural associativity constraint. (In this definition we use the $!$-pushforward, which differs from the $*$-pushforward since $\widetilde{m}$ is not proper.)

\begin{rmk}
\label{rmk:tilde-times-2}
As in Remark~\ref{rmk:tilde-times}, given $\Iwu$-stable subschemes $X,Y$ in $\tFl_G$, we will also denote by $X \, \tilde{\times} \, Y$ the quotient of $X' \times Y$ by the $\Iwu$-action induced by that on $\Loop G \times \tFl_G$, where $X'$ is the preimage of $X$ in $\Loop G$.
\end{rmk}

\begin{lem}
\label{lem:convolution-Iw-Iwu}
 For any $\scF,\scG$ in $\mathsf{D}_{\Iw,\Iw}$ we have
 \[
  \bigl( \pi^\dag \For^{\Iw}_{\Iwu}(\scF) \bigr) \star_{\Iwu} \bigl( \pi^\dag \For^{\Iw}_{\Iwu}(\scG) \bigr) \cong \bigl( \pi^\dag \For^{\Iw}_{\Iwu}(\scF \star_\Iw \scG) \bigr) \otimes_\bk \mathsf{H}_c^{[\bullet]}(T;\bk)[2\dim(T)],
 \]
 where we write $\mathsf{H}_c^{[\bullet]}(T;\bk)$ for $\bigoplus_{i \in \Z} \mathsf{H}^i_c(T,\bk)[-i]$.
\end{lem}

\begin{proof}
If we denote by $\Db_\Iw(\tFl_G,\bk)$ the $\Iw$-equivariant derived category of $\tFl_G$,
 the same definition as for the convolution product $\star_\Iw$ defines a canonical bifunctor
 \[
  \mathsf{D}_{\Iwu,\Iw} \times \Db_\Iw(\tFl_G,\bk) \to \Db_{\Iwu}(\tFl_G,\bk),
 \]
which will also be denoted $\star_\Iw$.
Let us again denote by
\[
 \For^{\Iw}_{\Iwu} : \Db_{\Iw}(\tFl_G,\bk) \to \Db_{\Iwu}(\tFl_G,\bk)
\]
the natural forgetful functor;
then by the same considerations as for~\cite[Lemma~2.5]{bgmrr}, for $\scF'$ in $\Db_{\Iwu}(\tFl_G,\bk)$ and $\scG'$ in $\Db_{\Iw}(\tFl_G,\bk)$ we have a canonical isomorphism
\[
 \scF' \star_{\Iwu} \For^{\Iw}_{\Iwu}(\scG') \cong (\pi_! \scF') \star_\Iw \scG'[\dim(T)].
\]
With $\scF,\scG$ as in the statement, we deduce an isomorphism
\begin{multline*}
 \bigl( \pi^\dag \For^{\Iw}_{\Iwu}(\scF) \bigr) \star_{\Iwu} \bigl( \pi^\dag \For^{\Iw}_{\Iwu}(\scG) \bigr) \cong \bigl( \pi^* \For^{\Iw}_{\Iwu}(\scF) \bigr) \star_{\Iwu} \bigl( \For^{\Iw}_{\Iwu} \circ \pi^*(\scG) \bigr) [2\dim(T)] \\
 \cong \bigl( \pi_! \pi^* \For^{\Iw}_{\Iwu}(\scF) \bigr) \star_{\Iw} \bigl( \pi^*\scG \bigr) [3\dim(T)].
\end{multline*}
Now we have
\[
 \bigl( \pi_! \pi^* \For^{\Iw}_{\Iwu}(\scF) \bigr) \star_{\Iw} \bigl( \pi^*\scG \bigr) \cong \pi^* \For^{\Iw}_{\Iwu} \bigl( ( \pi_! \pi^* \scF) \star_{\Iw} \scG \bigr).
\]
If we denote by $X \subset \Fl_G$ a closed finite union of $\Iw$-orbits over which $\scF$ is supported, and by $\widetilde{X}$ its preimage in $\tFl_G$, then by the projection formula we have $\pi_! \pi^* \scF \cong \scF \otimes_\bk \pi_! \underline{\bk}_{\widetilde{X}}$. Since $\widetilde{X}$ is a $T$-torsor over $X$, we have a canonical isomorphism $\widetilde{X} \times_X \widetilde{X} \cong T \times \widetilde{X}$; the base change theorem then implies that $\pi^* \pi_! \underline{\bk}_{\widetilde{X}} \cong \underline{\bk}_{\widetilde{X}} \otimes_\bk \mathsf{H}_c^{[\bullet]}(T;\bk)$, and using this (and the definition of the convolution product) we obtain an isomorphism
\[
 ( \pi_! \pi^* \scF) \star_{\Iw} \scG \cong (\scF \star_{\Iw} \scG) \otimes_\bk \mathsf{H}_c^{[\bullet]}(T;\bk).
\]
The desired isomorphism follows.
\end{proof}

The formula in Lemma~\ref{lem:convolution-Iw-Iwu} shows in particular that the bifunctor $\star_{\Iwu}$ restricts to a bifunctor
\[
 \mathsf{D}_{\Iwu,\Iwu} \times \mathsf{D}_{\Iwu,\Iwu} \to \mathsf{D}_{\Iwu,\Iwu}.
\]
A similar construction provides a bifunctor
\[
 \star_{\Iwu} : \mathsf{D}_{\Iwu,\Iwu} \times \mathsf{D}_{\Iwu,\Iw} \to \mathsf{D}_{\Iwu,\Iw}
\]
such that
\[
\pi^\dag(\scF \star_{\Iwu} \scG) \cong \scF \star_{\Iwu} \pi^\dag(\scG)
\]
for $\scF$ in $\mathsf{D}_{\Iwu,\Iwu}$ and $\scG$ in $\mathsf{D}_{\Iwu,\Iw}$.

Verdier's monodromy construction (see~\cite{verdier}; see also~\cite{bezr, gouttard} for additional comments) with respect to the action of $T \times T$ on $\tFl_G$ via $(t_1,t_2) \cdot g\Iwu = t_1gt_2 \Iwu$ provides, for any $\scF$ in $\mathsf{D}_{\Iwu,\Iwu}$, a canonical algebra morphism
\[
\mu_\scF : \scO(T^\vee_\bk \times T^\vee_\bk) \to \End_{\D_{\Iwu,\Iwu}}(\scF),
\]
which is unipotent in the sense that it vanishes on a power of the kernel of the natural augmentation morphism $\scO(T^\vee_\bk \times T^\vee_\bk) \to \bk$. (Here we use the identification $\scO(T^\vee_\bk) = \bk[X_*(T)]$.)
This construction satisfies various forms of functoriality; in particular, for any $\scF,\scG$ in $\mathsf{D}_{\Iwu,\Iwu}$ and any morphism $f : \scF \to \scG$ we have
\[
f \circ \mu_{\scF}(x) = \mu_{\scG}(x) \circ f
\]
for all $x \in \scO(T^\vee_\bk \times T^\vee_\bk)$. With this structure, $\D_{\Iwu,\Iwu}$ becomes an $\scO(T^\vee_\bk \times T^\vee_\bk)$-linear category.

\begin{rmk}
\label{rmk:loop-rot}
Recall that each $\Iw$-orbit on $\Fl_G$ is stable under the loop rotation action, as well as each pullback of such an orbit to $\tFl_G$. As a consequence, for every object $\scF$ in $\mathsf{D}_{\Iw,\Iw}$, resp.~$\scG$ in $\mathsf{D}_{\Iwu,\Iwu}$, we have a canonical monodromy morphism
 \[
 \overline{\mu}_{\scF}^{\mathrm{rot}} : \bk[x,x^{-1}] \to \End_{\mathsf{D}_{\Iw,\Iw}}(\scF), \quad \text{resp.} \quad
  \mu_{\scG}^{\mathrm{rot}} : \bk[x,x^{-1}] \to \End_{\mathsf{D}_{\Iwu,\Iwu}}(\scG),
 \]
 where $x$ is an indeterminate.
These morphisms possess the same functoriality properties as those considered above.

Recall also that any object of $\sfP_{\Loop^+G,\Loop^+G}$ is automatically equivariant with respect to the loop rotation action, see~\cite[Proposition~2.2]{mv}. As explained in~\cite[\S 5.2]{ab} (see also~\cite[Proposition~2.4.6 and its proof]{ar-book}), as a consequence of Lem\-ma~\ref{lem:action-Gm} and this property, for any $\scA$ in $\sfP_{\Loop^+G,\Loop^+G}$ we have
\begin{equation}
\label{eqn:mon-Z-rot}
\sm_{\scA} = \overline{\mu}_{\sfZ(\scA)}^{\mathrm{rot}}(x^{-1}).
\end{equation}
\end{rmk}

\subsection{The monodromic regular quotient}
\label{ss:mon-reg-quotient}


The main player in this paper will be the abelian category
\[
\mathsf{P}^0_{\Iwu,\Iwu}
\]
defined as the Serre quotient of $\sfP_{\Iwu,\Iwu}$ by the Serre subcategory $\sfP^+_{\Iwu,\Iwu}$ generated by the simple objects $\pi^\dag \For^{\Iw}_{\Iwu}( \IC_w )$ for $w \in W$ such that $\ell(w)>0$. Let us also denote by $\D^+_{\Iwu,\Iwu}$ the full triangulated subcategory of $\D_{\Iwu,\Iwu}$ generated by the perverse sheaves $\pi^\dag \For^{\Iw}_{\Iwu}( \IC_w )$ for $w \in W$ such that $\ell(w)>0$,
and by $\D^0_{\Iwu,\Iwu}$ the Verdier quotient of $\D_{\Iwu,\Iwu}$ by $\D^+_{\Iwu,\Iwu}$.
By Lemma~\ref{lem:quotient-t-str}, there exists a unique t-structure on $\D^0_{\Iwu,\Iwu}$ such that the quotient functor
\[
\Pi^0_{\Iwu,\Iwu} : \D_{\Iwu,\Iwu} \to \D_{\Iwu,\Iwu}^0
\]
is t-exact; moreover this t-structure is bounded, and its heart
identifies with $\sfP^0_{\Iwu,\Iwu}$. This t-structure will be called the perverse t-structure, and the corresponding cohomology functors will be denoted $\pH^n(-)$.

\begin{lem}
\phantomsection
\label{lem:P+-ideal}
\begin{enumerate}
 \item 
 \label{it:lem:P+-ideal-1}
 If $\scF$ belongs to $\D^+_{\Iwu,\Iwu}$ and $\scG$ is any object of $\D_{\Iwu,\Iwu}$, then the objects $\scF \star_{\Iwu} \scG$ and $\scG \star_{\Iwu} \scF$ belong to $\D^+_{\Iwu,\Iwu}$.
 \item 
 \label{it:lem:P+-ideal-2}
 For $\scF,\scG$ in $\pD^{\leq 0}_{\Iwu,\Iwu}$ and $n \in \Z_{>0}$, the object $\pH^n(\scF \star_{\Iwu} \scG)$ belongs to $\mathsf{P}^+_{\Iwu,\Iwu}$.
\end{enumerate}
\end{lem}

\begin{proof}[Sketch of proof]
The proof reduces to the case $\scF$ and $\scG$ are simple perverse sheaves, which reduces to the claims in~\cite[Lemma~5.1]{brr} using Lemma~\ref{lem:convolution-Iw-Iwu}.
\end{proof}

As a consequence of Lemma~\ref{lem:P+-ideal} we obtain that the bifunctor
\[
\D_{\Iwu,\Iwu} \times \D_{\Iwu,\Iwu} \to \D_{\Iwu,\Iwu}^0
\]
defined by
\[
(\scF,\scG) \mapsto \Pi^0_{\Iwu,\Iwu}(\scF \star_{\Iwu} \scG)
\]
factors through a bifunctor
\[
\star^0_{\Iwu} : \D_{\Iwu,\Iwu}^0 \times \D_{\Iwu,\Iwu}^0 \to \D_{\Iwu,\Iwu}^0,
\]
which is triangulated and defines a structure of monoidal category (without unit object) on $\D_{\Iwu,\Iwu}^0$. This bifunctor is moreover ``right t-exact'' in the sense that if $\scF,\scG$ belong to the nonpositive part of the perverse t-structure on $\D_{\Iwu,\Iwu}^0$, then so does $\scF \star^0_{\Iwu} \scG$. (This bifunctor is \emph{not} t-exact if $G$ is nontrivial, contrary to the situation for $\sfD_{\Iw,\Iw}^0$.)

We also define the bifunctor
\[
(-) \pstar^0_{\Iwu} (-) : \sfP^0_{\Iwu,\Iwu} \times \sfP^0_{\Iwu,\Iwu} \to \sfP^0_{\Iwu,\Iwu}
\]
by setting,
for $\scF,\scG$ in $\sfP^0_{\Iwu,\Iwu}$,
\[
 \scF \pstar^0_{\Iwu} \scG := \pH^0(\scF \star^0_{\Iwu} \scG).
\]
The right t-exactness of $\star^0_{\Iwu}$ implies that the bifunctor $(-) \pstar^0_{\Iwu} (-)$ is right exact on both sides, and also that for $\scF,\scG,\scH$ in $\sfP^0_{\Iwu,\Iwu}$ we have
\begin{align*}
(\scF \pstar^0_{\Iwu} \scG) \pstar^0_{\Iwu} \scH &\cong \pH^0 \bigl( (\scF \star^0_{\Iwu} \scG) \star^0_{\Iwu} \scH \bigr), \\
\scF \pstar^0_{\Iwu} (\scG \pstar^0_{\Iwu} \scH) &\cong \pH^0 \bigl( \scF \star^0_{\Iwu} (\scG \star^0_{\Iwu} \scH) \bigr);
\end{align*}
in particular, the associativity constraint on $\star^0_{\Iwu}$ induces an associativity constraint on $\pstar^0_{\Iwu}$, so that we obtain a monoidal category
\[
 (\sfP^0_{\Iwu,\Iwu},\pstar^0_{\Iwu})
\]
(again, without unit object). 

\subsection{Relation with the regular quotient}
\label{ss:relation-reg-quotient}

It follows from the definitions and the universal property of the Verdier quotient that the composition
\[
\Pi^0_{\Iwu,\Iwu} \circ \pi^\dag \circ \For^{\Iw}_{\Iwu} : \sfD_{\Iw,\Iw} \to \D_{\Iwu,\Iwu}^0
\]
factors through a triangulated functor
\[
 \pi_0^{\dag} : \sfD_{\Iw,\Iw}^0 \to \D_{\Iwu,\Iwu}^0.
\]
From the t-exactness of $\pi^\dag$ we deduce that $\pi_0^{\dag}$ is t-exact with respect to the perverse t-structures.


\begin{lem}
The restriction of $\pi^\dag_0$ to the hearts of the perverse t-structures admits a natural structure of monoidal functor
\begin{equation*}
\pi_0^{\dag} : (\mathsf{P}_{\Iw,\Iw}^0,\star_{\Iw}^0) \to (\mathsf{P}_{\Iwu,\Iwu}^0,\pstar_{\Iwu}^0).
\end{equation*}
\end{lem}

\begin{proof}
From Lemma~\ref{lem:convolution-Iw-Iwu} we obtain for $\scF,\scG$ in $\sfD_{\Iw,\Iw}^0$ a canonical isomorphism
\[
 \pi_0^{\dag} (\scF) \star_{\Iwu}^0 \pi_0^{\dag} (\scG) \cong \pi_0^{\dag} (\scF \star^0_{\Iw} \scG) \otimes_\bk \mathsf{H}_c^{[\bullet]}(T;\bk)[2\dim(T)].
\]
If $\scF$ and $\scG$ belong to the heart of the perverse t-structure, we deduce a canonical isomorphism
 \[
  \pi^\dag_0 (\scF) \pstar_{\Iwu}^0 \pi_0^\dag(\scG) \cong \pi_0^\dag (\scF \star_\Iw^0 \scG) \otimes \mathsf{H}_c^{2\dim(T)}(T;\bk).
 \]
 Now the vector space $\mathsf{H}_c^{2\dim(T)}(T;\bk)$ is canonically isomorphic to $\bk$ since $T$ is connected, which provides an isomorphism of bifunctors defining a monoidal structure on our functor.
\end{proof}


Similar considerations show that the composition
\[
\Pi^0_{\Iwu,\Iwu} \circ \pi^\dag : \sfD_{\Iwu,\Iw} \to \D_{\Iwu,\Iwu}^0
\]
factors through a triangulated functor
\[
 \pi^{\dag,0} : \sfD_{\Iwu,\Iw}^0 \to \D_{\Iwu,\Iwu}^0,
\]
which is t-exact for the perverse t-structures and satisfies
\begin{equation}
\label{eqn:pidag-0-For}
 \pi_0^{\dag} \cong \pi^{\dag,0} \circ \For^{\Iw,0}_{\Iwu}
\end{equation}
where $\For^{\Iw,0}_{\Iwu}$ is defined in~\S\ref{ss:another-convolution}.

\section{Free-monodromic perverse sheaves}
\label{sec:fmps}

We continue to consider the general setting of~\S\ref{ss:Fl}.

\subsection{Yun's completed category}
\label{ss:completed-category}

Following Yun (see~\cite[Appendix~A]{by}; see also~\cite{bezr} for additional comments) we consider the ``completed'' category
\[
 \mathsf{D}^\wedge_{\Iwu,\Iwu}
\]
associated with the $T$-torsor $\pi : \tFl_G \to \Fl_G$ and the $\Iwu$-action on these ind-schemes.
Recall that this category is defined as the full subcategory of the category of pro-objects in $\mathsf{D}_{\Iwu,\Iwu}$ whose objects are the systems
\[
``\varprojlim_{n \in \Z_{\geq 0}}" \scF_n
\]
which are:
\begin{itemize}
 \item \emph{$\pi$-constant}, in the sense that the pro-object $``\varprojlim" \pi_!(\scF_n)$ is isomorphic to an object in $\mathsf{D}_{\Iwu,\Iw}$;
 \item \emph{uniformly bounded in degrees}, in the sense that we have an isomorphism $``\varprojlim" \scF_n \cong ``\varprojlim" \scF'_n$ and some $N \in \Z_{\geq 0}$ such that each $\scF_n'$ satisfies $\pH^i(\scF_n')=0$ unless $i \in [-N,N]$.
\end{itemize}
It is proved in~\cite[Appendix~A]{by} that this category admits a triangulated structure, for which the distinguished triangles are the diagrams isomorphic to one of the form
\[
 ``\varprojlim" \scF_n \xrightarrow{``\varprojlim" \alpha_n} ``\varprojlim" \scG_n \xrightarrow{``\varprojlim" \beta_n} ``\varprojlim" \scH_n \xrightarrow{``\varprojlim" \gamma_n} ``\varprojlim" \scF_n[1]
\]
where each
\[
 \scF_n \xrightarrow{\alpha_n} \scG_n \xrightarrow{\beta_n} \scH_n \xrightarrow{\gamma_n} \scF_n[1]
\]
is a distinguished triangle in $\mathsf{D}_{\Iwu,\Iwu}$. In particular we have a canonical fully faithful triangulated functor
\begin{equation}
\label{eqn:embedding-D-Dwedge}
\mathsf{D}_{\Iwu,\Iwu} \to \mathsf{D}^\wedge_{\Iwu,\Iwu}
\end{equation}
sending an object $\scF$ to the constant projective system with value $\scF$.
The functor
$\pi_\dag = \pi_![\dim(T)]$
defines a functor
$\mathsf{D}^\wedge_{\Iwu,\Iwu} \to \mathsf{D}_{\Iwu,\Iw}$, which will also be denoted $\pi_\dag$, and which can be shown to be triangulated. This functor is conservative by~\cite[Lemma~A.3.5]{by}. By~\cite[Corollary~5.5]{bezr}, the category $\mathsf{D}^\wedge_{\Iwu,\Iwu}$ is Krull--Schmidt.

It is clear that in this definition one can replace the ind-scheme $\tFl_G$ by the inverse image of any locally closed union $X$ of $\Iw$-orbits in $\Fl_G$; in this setting the completed category will be denoted $\sfD^\wedge_{\Iwu}(X,\bk)$.

Recall the local systems $\scL_{T,n}$ on $T$ considered in~\cite[\S 10.1]{bezr}, and the associated pro-object
\[
 \scL^\wedge_{T} := ``\varprojlim_{n}" \scL_{T,n}
\]
on $T$. (Here, each $\scL_{T,n}$ is an extension of copies of the constant local system $\underline{\bk}_T$.) We have $\tFl_{G,e}=B/U \cong T$. The pro-object $\scL^\wedge_{T}$ therefore defines a pro-local system on $\tFl_{G,e}$; taking the pushforward under the embedding in $\tFl_G$ of its shift by $\dim(T)$ we obtain an object in $\mathsf{D}^\wedge_{\Iwu,\Iwu}$, which will be denoted $\delta^\wedge$. By~\cite[Equation (3.3)]{bezr}, we have a canonical isomorphism
\[
\pi_\dag(\delta^\wedge) \cong \For^{\Iw}_{\Iwu}(\delta_{\Fl}).
\]

The arguments of~\cite[\S 4.3]{by} (see also~\cite[\S 7.3]{bezr} for the analoguous case of $G/U$) show that the monoidal product $\star_{\Iwu}$ extends to a bifunctor
\[
\hatstar : \mathsf{D}^\wedge_{\Iwu,\Iwu} \times \mathsf{D}^\wedge_{\Iwu,\Iwu} \to \mathsf{D}^\wedge_{\Iwu,\Iwu}
\]
which is triangulated on both sides. More specifically, the bifunctor $\star_{\Iwu}$ extends in a canonical way to a bifunctor $\hatstar$ on pro-objects, so that for pro-objects $``\varprojlim_{i \in I}" \scF_i$ and $``\varprojlim_{j \in J}" \scG_j$ we have
\[
 \left( ``\varprojlim_{i \in I}" \scF_i \right) \hatstar \left( ``\varprojlim_{j \in J}" \scG_j \right) = ``\varprojlim_{(i,j) \in I \times J}" \scF_i \star_{\Iwu} \scG_j.
\]
Now if these pro-objects belong to $\mathsf{D}^\wedge_{\Iwu,\Iwu}$ (so that, in particular, we can assume $I=J=\Z_{\geq 0}$) we have
\[
 \left( ``\varprojlim_{n}" \scF_n \right) \hatstar \left( ``\varprojlim_{m}" \scG_m \right) = ``\varprojlim_{m}" \left(``\varprojlim_{n}" \scF_n \star_{\Iwu} \scG_m \right)
\]
where for each $m$ the pro-object $``\varprojlim_{n}" \scF_n \star_{\Iwu} \scG_m$ is representable by an object of $\mathsf{D}_{\Iwu,\Iwu}$. It is explained in~\cite[\S 4.3]{by} that this formal inverse limit of objects in $\mathsf{D}_{\Iwu,\Iwu}$ belongs to $\sfD^\wedge_{\Iwu,\Iwu}$. In fact we also have
\[
 \left( ``\varprojlim_{n}" \scF_n \right) \hatstar \left( ``\varprojlim_{m}" \scG_m \right) = ``\varprojlim_{n}" \left(``\varprojlim_{m}" \scF_n \star_{\Iwu} \scG_m \right)
\]
where for each $n$ the pro-object $``\varprojlim_{m}" \scF_n \star_{\Iwu} \scG_m$ belongs to $\mathsf{D}_{\Iwu,\Iwu}$.

The bifunctor $\hatstar$ on $\sfD^\wedge_{\Iwu,\Iwu}$ admits a natural associativity constraint and a unit object (namely, $\delta^\wedge$), which equips $\sfD^\wedge_{\Iwu,\Iwu}$ with the structure of a monoidal category. It restricts to (triangulated) bifunctors
\[
\hatstar : \mathsf{D}^\wedge_{\Iwu,\Iwu} \times \mathsf{D}_{\Iwu,\Iwu} \to \mathsf{D}_{\Iwu,\Iwu}, \quad \hatstar : \mathsf{D}_{\Iwu,\Iwu} \times \mathsf{D}^\wedge_{\Iwu,\Iwu} \to \mathsf{D}_{\Iwu,\Iwu}
\]
which define left and right actions of the monoidal category $(\mathsf{D}^\wedge_{\Iwu,\Iwu},\hatstar)$ on the category $\mathsf{D}_{\Iwu,\Iwu}$ respectively. These two actions commute, and are compatible in various ways; in particular, for $\scF_1,\scF_2$ in $\mathsf{D}^\wedge_{\Iwu,\Iwu}$ and $\scG_1,\scG_2$ in $\mathsf{D}_{\Iwu,\Iwu}$ we have canonical isomorphisms
\begin{multline}
\label{eqn:compatibility-hatstar}
 (\scG_1 \hatstar \scF_1) \star_{\Iwu} (\scF_2 \hatstar \scG_2) \cong \scG_1 \star_{\Iwu} \bigl( (\scF_1 \hatstar \scF_2) \hatstar \scG_2 \bigr) \\
 \cong \bigl( \scG_1 \hatstar (\scF_1 \hatstar \scF_2) \bigr) \star_{\Iwu} \scG_2
 \end{multline}
 and
 \begin{equation}
  \label{eqn:compatibility-hatstar-other}
 (\scF_1 \hatstar \scG_1) \star_{\Iwu} \scG_2 \cong \scF_1 \hatstar (\scG_1 \star_{\Iwu} \scG_2). 
\end{equation}

\begin{rmk}
As explained above, the convolution product on $\sfD^\wedge_{\Iwu,\Iwu}$ extends that on $\sfD_{\Iwu,\Iwu}$, and there exists a unit object for this product. The construction of $\sfD^\wedge_{\Iwu,\Iwu}$ can be seen as a way to ``complete'' the category $\sfD_{\Iwu,\Iwu}$ so that it admits a unit object, while staying in the world of \emph{triangulated} monoidal categories.
\end{rmk}

We similarly have an action on the category $\mathsf{D}_{\Iwu,\Iw}$, defined by a bifunctor
\[
\hatstar : \mathsf{D}^\wedge_{\Iwu,\Iwu} \times \mathsf{D}_{\Iwu,\Iw} \to \mathsf{D}_{\Iwu,\Iw}
\]
whose definition is similar to that considered above, and which is also triangulated on both sides.
These functors satisfy the following relations (see~\cite[Equations~(7.2) and (7.6), Lemma~7.2]{bezr}):
\begin{align}
\label{eqn:conv-formula-1}
 \pi_\dag(\scF \hatstar \scG) &\cong \scF \hatstar \pi_\dag(\scG) \\
 \label{eqn:conv-formula-2}
 \scF \hatstar \For^{\Iw}_{\Iwu}(\scH) &\cong \pi_\dag(\scF) \star_\Iw \scH \\
 \label{eqn:conv-formula-3}
 \scF \hatstar \pi^\dag(\mathscr{K}) &\cong \pi^\dag(\scF \hatstar \mathscr{K})
\end{align}
for $\scF,\scG$ in $\mathsf{D}^\wedge_{\Iwu,\Iwu}$, $\scH$ in $\mathsf{D}_{\Iw,\Iw}$ and $\mathscr{K}$ in $\mathsf{D}_{\Iwu,\Iw}$.

The monodromy constructions of~\S\ref{ss:monodromic} pass to the completion, and provide for any $\scF$ in $\mathsf{D}^\wedge_{\Iwu,\Iwu}$ canonical algebra morphisms
\[
\mu_\scF : \scO(T^\vee_\bk \times T^\vee_\bk) \to \End_{\mathsf{D}^\wedge_{\Iwu,\Iwu}}(\scF), \quad \mu_\scF^{\mathrm{rot}} : \bk[x,x^{-1}] \to \End_{\mathsf{D}^\wedge_{\Iwu,\Iwu}}(\scF)
\]
which commute with all morphisms.
In particular, with the morphisms $\mu_\scF$, $\D^\wedge_{\Iwu,\Iwu}$ becomes an $\scO(T^\vee_\bk \times T^\vee_\bk)$-linear category. These operations are compatible with convolution, in the sense that for $\scF,\scG$ in $\D^\wedge_{\Iwu,\Iwu}$ and $f,g \in \scO(T^\vee_\bk)$ we have
\begin{align}
\label{eqn:monodromy-convolution-1}
\mu_{\scF \hatstar \scG}(f \otimes g) &= \mu_\scF(f \otimes 1) \hatstar \mu_\scG(1 \otimes g)
\\
\label{eqn:monodromy-convolution-2}
\mu_\scF(1 \otimes f) \hatstar \id_\scG &= \id_\scF \hatstar \mu_\scG(f \otimes 1).
\end{align}
(The proof is similar to that given for sheaves on $G/U$ in~\cite[Lemma~7.3]{bezr}.) It is not difficult to check that for $\scF,\scG$ in $\D^\wedge_{\Iwu,\Iwu}$ we also have
\begin{equation}
\label{eqn:monodromy-convolution-3}
\mu_{\scF \hatstar \scG}^{\mathrm{rot}}(x) = \mu_\scF^{\mathrm{rot}}(x) \hatstar \mu_\scG^{\mathrm{rot}}(x).
\end{equation}

The following claim follows from~\cite[Remark~5.1]{bezr}.

\begin{lem}
\label{lem:subcat-completion-monodromy}
Let $\scF$ in $\mathsf{D}^\wedge_{\Iwu,\Iwu}$. Then $\scF$ belongs to the essential image of the functor~\eqref{eqn:embedding-D-Dwedge} iff the restriction of $\mu_\scF$ to $\bk \otimes_\bk \scO(T^\vee_\bk) \subset \scO(T^\vee_\bk \times T^\vee_\bk)$ vanishes on some power of the maximal ideal corresponding to $e \in T^\vee_\bk$.
\end{lem}


\subsection{The perverse t-structure}
\label{ss:completed-perverse}

Another important feature of the ``completed'' (or ``free monodromic'') category $\D^\wedge_{\Iwu,\Iwu}$ is that it admits a ``perverse'' t-structure $(\pD^{\wedge,\leq 0}_{\Iwu,\Iwu},\pD^{\wedge,\geq 0}_{\Iwu,\Iwu})$, whose heart will be denoted $\mathsf{P}^\wedge_{\Iwu,\Iwu}$. (For the definition of this t-structure, see~\cite[\S 5.2]{bezr}; for an earlier and slightly different construction of this t-structure, see~\cite[\S A.6]{by}.) From the construction it is clear that the functor~\eqref{eqn:embedding-D-Dwedge} is t-exact. One can check also that the functor $\pi_\dag$ of~\S\ref{ss:completed-category} is right t-exact, see~\cite[Corollary~5.8]{bezr}.

For any $w \in W$, the quotient $\Iwu \backslash \tFl_{G,w}$ is a $T$-torsor over $\mathrm{Spec}(\bk)$ for the action induced by right multiplication on $\Loop G$ (but this torsor does not admit any \emph{canonical} trivialization in general). After choosing a $T$-equivariant isomorphism $\Iwu \backslash \tFl_{G,w} \simto T$, we can then define
\begin{align*}
 \Delta^{\wedge}_w &:= ``\varprojlim_{n}" (\widetilde{\jmath}_w)_! p_w^* \scL^\wedge_{T,n}[\dim(T)+\ell(w)], \\ 
 \nabla^{\wedge}_w &:= ``\varprojlim_{n}" (\widetilde{\jmath}_w)_* p_w^* \scL^\wedge_{T,n}[\dim(T)+\ell(w)],
\end{align*}
where $p_w$ is the composition $\tFl_{G,w} \to \Iwu \backslash \tFl_{G,w} \simto T$, and $\widetilde{\jmath}_w : \tFl_{G,w} \to \tFl_G$ is the embedding. These objects are perverse sheaves, and do not depend on the choice of trivialization up to (noncanonical) isomorphism. They also satisfy
\begin{equation}
\label{eqn:pi-DN-wedge}
 \pi_{\dag}(\Delta_w^\wedge) \cong \For^{\Iw}_{\Iwu}(\Delta_w^\Iw), \quad \pi_\dag(\nabla_w^\wedge) \cong \For^{\Iw}_{\Iwu}(\nabla_w^\Iw),
\end{equation}
see~\cite[Equation~(5.3)]{bezr}.
With these objects at hand, one can describe the nonpositive part ${}^{\mathrm{p}} \hspace{-1pt} \mathsf{D}^{\wedge,\leq 0}_{\Iwu,\Iwu}$ of the perverse t-structure on $\mathsf{D}^\wedge_{\Iwu,\Iwu}$ as the subcategory generated under extensions by the objects $\Delta^{\wedge}_w[n]$ with $w \in W$ and $n \geq 0$, see~\cite[Lemma~5.6]{bezr}. (The similar statement for the nonnegative part of the t-structure does \emph{not} hold.)

\begin{lem}
\phantomsection
\label{lem:convolution-stand-costand}
\begin{enumerate}
 \item 
 \label{it:convolution-stand-costand-1}
 For any $w \in W$, there exist isomorphisms
 \[
  \Delta^\wedge_w \hatstar \nabla^\wedge_{w^{-1}} \cong \delta^\wedge, \quad \nabla^\wedge_{w^{-1}} \hatstar \Delta^\wedge_w \cong \delta^\wedge.
 \]
 \item 
 \label{it:convolution-stand-costand-2}
 If $w,y \in W$ are such that $\ell(wy)=\ell(w)+\ell(y)$, then there exist isomorphisms
 \[
  \Delta^\wedge_w \hatstar \Delta^\wedge_y \cong \Delta^\wedge_{wy}, \quad \nabla^\wedge_w \hatstar \nabla^\wedge_y \cong \nabla^\wedge_{wy}.
 \]
 \item
 \label{it:convolution-stand-costand-3}
 For any $w,y \in W$, the objects
 \[
  \Delta^\wedge_w \hatstar \nabla^\wedge_y \quad \text{and} \quad \nabla^\wedge_w \hatstar \Delta^\wedge_y
 \]
belong to $\mathsf{P}^\wedge_{\Iwu,\Iwu}$.
\end{enumerate}
\end{lem}

\begin{proof}
 For~\eqref{it:convolution-stand-costand-1}--\eqref{it:convolution-stand-costand-2}, the proof is similar to that of the corresponding statements on $G/U$, treated in detail in~\cite[Lemma~7.7]{bezr}. For~\eqref{it:convolution-stand-costand-3}, the proof is again based on the same idea: by~\eqref{eqn:conv-formula-1}--\eqref{eqn:conv-formula-2} and~\eqref{eqn:pi-DN-wedge} we have
 \[
  \pi_\dag(\Delta^\wedge_w \hatstar \nabla^\wedge_y) \cong \Delta^\wedge_w \hatstar \For^{\Iw}_{\Iwu}(\nabla^\Iw_y) \cong \For^{\Iw}_{\Iwu}(\Delta^{\Iw}_w) \star_\Iw \nabla^\Iw_y.
 \]
Now it is well known that the right-hand side is a perverse sheaf, see e.g.~\cite[Lemma~4.1.7]{ar-book}. Since an object whose image under $\pi_\dag$ is perverse is itself perverse (see~\cite[Lemma~5.3(1)]{bezr}), the claim follows. 
\end{proof}

\begin{rmk}
\label{rmk:convolution-equiv}
Lemma~\ref{lem:convolution-stand-costand}\eqref{it:convolution-stand-costand-1} implies in particular that the functor of left (resp.~right) convolution with $\Delta^\wedge_w$ is an equivalence of categories, with quasi-inverse given by left (resp.~right) convolution with $\nabla^\wedge_{w^{-1}}$.
\end{rmk}

\begin{cor}
 \label{cor:exactness-convolution-D-N}
 For any $w \in W$:
 \begin{enumerate}
  \item 
  \label{it:exactness-convolution-N}
  the functors
  \[
   \nabla^\wedge_w \hatstar (-), \, (-) \hatstar \nabla_w^\wedge : \sfD^\wedge_{\Iwu,\Iwu} \to \sfD^\wedge_{\Iwu,\Iwu}
  \]
are right t-exact;
  \item 
  \label{it:exactness-convolution-D}
  the functors
  \[
   \Delta^\wedge_w \hatstar (-), \, (-) \hatstar \Delta_w^\wedge : \sfD^\wedge_{\Iwu,\Iwu} \to \sfD^\wedge_{\Iwu,\Iwu}
  \]
are left t-exact
 \end{enumerate}
\end{cor}

\begin{proof}
 Since the nonpositive part of the perverse t-structure on $\sfD^\wedge_{\Iwu,\Iwu}$ is generated under extensions by the objects $\Delta^\wedge_y[n]$ for $y \in W$ and $n \geq 0$ (see the comments preceding Lemma~\ref{lem:convolution-stand-costand}), \eqref{it:exactness-convolution-N} is a consequence of Lemma~\ref{lem:convolution-stand-costand}\eqref{it:convolution-stand-costand-3}. We deduce~\eqref{it:exactness-convolution-D} using the fact that the functor $\nabla^\wedge_{w^{-1}} \hatstar (-)$, resp.~$(-) \hatstar \nabla^\wedge_{w^{-1}}$ is left adjoint to $\Delta^\wedge_w \hatstar (-)$, resp.~$(-) \hatstar \Delta_w^\wedge$, see Remark~\ref{rmk:convolution-equiv}.
\end{proof}

Below we will also need to consider the monodromy morphisms for the objects $\Delta^\wedge_w$ and $\nabla^\wedge_w$. The following lemma is the analogue in our present setting of~\cite[Lemma 5.4, Lemma~6.1]{bezr}. Here, 
for $\lambda \in X_*(T)$, we denote by $e^\lambda$ the corresponding morphism $T^\vee_\bk \to \Gm$, seen as an element in $\scO(T^\vee_\bk)$.

\begin{lem}
\label{lem:monodromy-DN}
Let $w \in W$, and write $w=v\st(\lambda)$ with $v \in \Wf$ and $\lambda \in X_*(T)$.
\begin{enumerate}
\item
\label{lem:monodromy-Dwedge}
The restriction of $\mu_{\Delta^\wedge_w}$ to the subalgebra
\[
\scO(T^\vee_\bk) = \bk \otimes \scO(T^\vee_\bk) \subset \scO(T^\vee_\bk) \otimes \scO(T^\vee_\bk)=\scO(T^\vee_\bk \times T^\vee_\bk)
\]
factors (in the natural way) through an isomorphism
\[
\scO(\FN_{T^\vee_\bk}(\{e\}))
\simto \End_{\mathsf{D}^\wedge_{\Iwu,\Iwu}}(\Delta^\wedge_w),
\]
and any nonzero endomorphism of $\Delta^\wedge_w$ is injective.
Moreover, for any $f \in \scO(T^\vee_\bk)$ we have
\[
\mu_{\Delta^\wedge_w}(f \otimes 1) = \mu_{\Delta^\wedge_w}(1 \otimes v^{-1}(f)),
\]
and we have
\[
\mu_{\Delta^\wedge_w}^{\mathrm{rot}}(x)=\mu_{\Delta^\wedge_w}(1 \otimes e^{-\lambda}).
\]
\item
The restriction of $\mu_{\nabla^\wedge_w}$ to the subalgebra
\[
\scO(T^\vee_\bk) = \bk \otimes \scO(T^\vee_\bk) \subset \scO(T^\vee_\bk) \otimes \scO(T^\vee_\bk)=\scO(T^\vee_\bk \times T^\vee_\bk)
\]
factors (in the natural way) through an isomorphism
\[
\scO(\FN_{T^\vee_\bk}(\{e\}))
\simto \End_{\mathsf{D}^\wedge_{\Iwu,\Iwu}}(\nabla^\wedge_w).
\]
Moreover, for any $f \in \scO(T^\vee_\bk)$ we have
\[
\mu_{\nabla^\wedge_w}(f \otimes 1) = \mu_{\nabla^\wedge_w}(1 \otimes v^{-1}(f)),
\]
and we have
\[
\mu_{\nabla^\wedge_w}^{\mathrm{rot}}(x)=\mu_{\nabla^\wedge_w}(1 \otimes e^{-\lambda}).
\]
\end{enumerate}
\end{lem}

\begin{proof}
The claims about the right monodromy and the injectivity of nonzero morphisms are general facts in completed derived categories, see~\cite[Lemma~5.4]{bezr}. Since $\Iwu$ is normal in $\Iw$, the left action of $T$ on $\tFl_{G,w}$ induces an action on the quotient $\Iwu \backslash \tFl_{G,w}$, and it is easily seen that this action is the twist of the action considered above on $\Iwu \backslash \tFl_{G,w}$ by $v^{-1}$, which implies the claim about left monodromy by basic properties of the monodromy construction, see~\cite[Lemma~2.5]{bezr}. Similarly, the loop rotation action on $\tFl_{G,w}$ induces an action on the quotient $\Iwu \backslash \tFl_{G,w}$, which is deduced from the $T$-action via $-\lambda : \Gm \to T$; this implies the claims about the loop rotation monodromy.
\end{proof}

\subsection{Tilting perverse sheaves}
\label{ss:tilting-perv}

Recall that an object $\scF$ in $\mathsf{P}^\wedge_{\Iwu,\Iwu}$ is said to be \emph{tilting} if it admits a filtration (in the abelian category $\mathsf{P}^\wedge_{\Iwu,\Iwu}$) with subquotients of the form $\Delta^\wedge_w$ ($w \in W$) and a filtration with subquotients of the form $\nabla^\wedge_w$ ($w \in W$). The full subcategory of $\mathsf{P}^\wedge_{\Iwu,\Iwu}$ whose objects are the tilting perverse sheaves will be denoted $\mathsf{T}^\wedge_{\Iwu,\Iwu}$. As explained in~\cite[\S A.7]{by} (see also~\cite[\S 5.5]{bezr}), the isomorphism classes of indecomposable objects in $\mathsf{T}^\wedge_{\Iwu,\Iwu}$ are in a canonical bijection with $W$; more specifically, for any $w \in W$ there exists a unique (up to isomorphism) object $\mathscr{T}^\wedge_w$ in $\sfD^\wedge_{\Iwu,\Iwu}$ such that $\pi_\dag(\mathscr{T}^\wedge_w) \cong \mathscr{T}_w$ (where the right-hand side is defined in~\S\ref{ss:DbIw}); then $\mathscr{T}^\wedge_w$ is an indecomposable tilting object in $\mathsf{P}^\wedge_{\Iwu,\Iwu}$, and the assignment $w \mapsto \mathscr{T}^\wedge_w$ provides the desired bijection. (We insist that $\mathscr{T}^\wedge_w$ is defined only up to isomorphism.) 

%

The same arguments as in~\cite[Remark~7.9]{bezr} show that $\mathsf{T}^\wedge_{\Iwu,\Iwu}$ is a \emph{monoidal} subcategory in $\mathsf{D}^\wedge_{\Iwu,\Iwu}$. Note also that we have natural equivalences of categories
\begin{equation}
\label{eqn:equiv-KbTilt-wedge}
\Kb \mathsf{T}^\wedge_{\Iwu,\Iwu} \simto \Db \mathsf{P}^\wedge_{\Iwu,\Iwu} \simto \mathsf{D}^\wedge_{\Iwu,\Iwu},
\end{equation}
see~\cite[Proposition~5.11]{bezr}, and that the composition of these two equivalences has a natural monoidal structure.

%
%

Using tilting objects one obtains the following property of monodromy.

\begin{lem}
\label{lem:monodromy-fiber-prod}
For any $\scF$ in $\mathsf{P}^\wedge_{\Iwu,\Iwu}$, the monodromy morphism $\mu_\scF$ factors through the surjection
\[
\scO(T^\vee_\bk \times T^\vee_\bk) \to \scO(T^\vee_\bk \times_{T^\vee_\bk/\Wf} T^\vee_\bk).
\]
\end{lem}

\begin{proof}
Let us first note that for $w,y \in W$ we have
\begin{equation}
\label{eqn:Hom-Delta}
\Hom_{\mathsf{P}^\wedge_{\Iwu,\Iwu}}(\Delta^\wedge_w , \Delta^\wedge_y)=0 \quad \text{if $w \neq y$.}
\end{equation}
In fact, write $w=v \st(\lambda)$ and $y=v' \st(\lambda')$ with $v,v' \in \Wf$ and $\lambda,\lambda' \in X_*(T)$. If $f : \Delta^\wedge_w \to \Delta^\wedge_y$ is a nonzero morphism and $\scF$ is its image, then for any $r \in \scO(T^\vee_\bk)$ we have $\mu_\scF(r \otimes 1) = \mu_\scF(1 \otimes v^{-1}(r))$ by Lemma~\ref{lem:monodromy-DN}\eqref{lem:monodromy-Dwedge} and the fact that $\scF$ is a quotient of $\Delta^\wedge_w$, and $\mu_\scF(r \otimes 1) = \mu_\scF(1 \otimes (v')^{-1}(r))$ by the same lemma and the fact that $\scF$ is a subobject of $\Delta^\wedge_y$. Hence $\mu_\scF(1 \otimes (v^{-1}(r)-(v')^{-1}(r)))=0$ for any $r$. The injectivity claim in Lemma~\ref{lem:monodromy-DN}\eqref{lem:monodromy-Dwedge} then implies that $v^{-1}(r)=(v')^{-1}(r)$ for any $r$, so that $v=v'$. The same considerations using $\mu^{\mathrm{rot}}_\scF$ show that $\lambda=\lambda'$, which finishes the proof of~\eqref{eqn:Hom-Delta}.

Once this claim is proved, using the fact that $\mu_{\Delta^\wedge_w}$ factors through the surjection $\scO(T^\vee_\bk \times T^\vee_\bk) \to \scO(T^\vee_\bk \times_{T^\vee_\bk/\Wf} T^\vee_\bk)$ (see once again Lemma~\ref{lem:monodromy-DN}\eqref{lem:monodromy-Dwedge}) one obtains as in~\cite[\S\S 6.3--6.4]{bezr} or in~\S\ref{ss:wakimoto} below (using a faithful associated graded functor) that this property holds for any object of $\mathsf{T}^\wedge_{\Iwu,\Iwu}$. Finally, the first equivalence in~\eqref{eqn:equiv-KbTilt-wedge} shows that any object in $\mathsf{P}^\wedge_{\Iwu,\Iwu}$ is a subquotient of a tilting object, which implies our claim.
\end{proof}

We will still denote by $\mu_\scF$ the morphism
\[
\scO(T^\vee_\bk \times_{T^\vee_\bk/\Wf} T^\vee_\bk) \to \End(\scF)
\]
induced by the map previously denoted $\mu_\scF$.
In concrete terms, Lemma~\ref{lem:monodromy-fiber-prod} means that the actions of $\scO(T^\vee_\bk)$ on an object of $\mathsf{P}^\wedge_{\Iwu,\Iwu}$ defined by monodromy for the left and right actions of $T$ on $\tFl_G$ coincide on the subalgebra $\scO(T^\vee_\bk / \Wf)$. We can therefore speak unambiguously of the monodromy action of $\scO(T^\vee_\bk / \Wf)$ on such an object. In view of~\eqref{eqn:monodromy-convolution-1} the same comment will apply to any object of the form $\scF \hatstar \scG$ with $\scF,\scG$ in $\mathsf{P}^\wedge_{\Iwu,\Iwu}$, and moreover the action on $\scF \hatstar \scG$ identifies with both the action induced by that on $\scF$ and the action induced by that on $\scG$.

\subsection{Actions of \texorpdfstring{$\D^\wedge_{\Iwu,\Iwu}$}{D} on \texorpdfstring{$\D^0_{\Iwu,\Iwu}$}{D0}}
\label{ss:actions-Dwedge-D}


Recall from~\S\ref{ss:completed-category} that the monoidal category $(\D^\wedge_{\Iwu,\Iwu}, \hatstar)$ acts on $\D_{\Iwu,\Iwu}$, via a bifunctor
\[
 \hatstar : \D^\wedge_{\Iwu,\Iwu} \times \D_{\Iwu,\Iwu} \to \D_{\Iwu,\Iwu}.
\]

\begin{lem}
\label{lem:convolution-quotient}
For $\scF \in \D^\wedge_{\Iwu,\Iwu}$ and $\scG \in \D^+_{\Iwu,\Iwu}$, the object $\scF \hatstar \scG$ belongs to $\D^+_{\Iwu,\Iwu}$.
\end{lem}

\begin{proof}
As for Lemma~\ref{lem:P+-ideal}, it suffices to prove the claim when $\scG=\pi^\dag \For^{\Iw}_{\Iwu}(\IC_w)$ for some $w \in W$ with $\ell(w)>0$. Now using~\eqref{eqn:conv-formula-2}--\eqref{eqn:conv-formula-3} we see that
\[
\scF \hatstar \pi^\dag \For^{\Iw}_{\Iwu}(\IC_w) \cong \pi^\dag((\pi_\dag \scF) \star_\Iw \IC_w).
\]
Here $\pi_\dag \scF$ is an object of $\D_{\Iwu,\Iw}$, and the bifunctor we consider is that of~\eqref{eqn:conv-Iwu-Iw}. Then the claim follows once again from~\cite[Lemma~5.1]{brr}.
\end{proof}

With this lemma at hand, the universal property of the Verdier quotient implies that for any $\scF$ in $\D^\wedge_{\Iwu,\Iwu}$, the functor
\[
\D_{\Iwu,\Iwu} \to \D^0_{\Iwu,\Iwu}
\]
defined by $\scG \mapsto \Pi^0_{\Iwu,\Iwu}(\scF \hatstar \scG)$ factors canonically through a functor $\D^0_{\Iwu,\Iwu} \to \D^0_{\Iwu,\Iwu}$, and one checks easily that this operation defines a triangulated bifunctor
\[
\hatstar^0 : \D^\wedge_{\Iwu,\Iwu} \times \D^0_{\Iwu,\Iwu} \to \D^0_{\Iwu,\Iwu}
\]
defining a left action of the monoidal category $(\D^\wedge_{\Iwu,\Iwu}, \hatstar)$ on $\D^0_{\Iwu,\Iwu}$. Similar consi\-derations starting with the right action of $(\D^\wedge_{\Iwu,\Iwu}, \hatstar)$ on $\sfD_{\Iwu,\Iwu}$ (by convolution on the right) leads to the construction of a triangulated bifunctor
\[
\hatstar^0 : \D^0_{\Iwu,\Iwu}  \times \D^\wedge_{\Iwu,\Iwu} \to \D^0_{\Iwu,\Iwu}
\]
defining a right action of the monoidal category $(\D^\wedge_{\Iwu,\Iwu}, \hatstar)$ on $\D^0_{\Iwu,\Iwu}$. These two actions commute, and are related in various ways; in particular, from~\eqref{eqn:compatibility-hatstar}--\eqref{eqn:compatibility-hatstar-other} we deduce that for $\scF_1,\scF_2$ in $\mathsf{D}^\wedge_{\Iwu,\Iwu}$ and $\scG_1,\scG_2$ in $\sfD_{\Iwu,\Iwu}^0$ we have canonical isomorphisms
\begin{multline}
\label{eqn:compatibility-hatstar-0}
 (\scG_1 \hatstar^0 \scF_1) \star^0_{\Iwu} (\scF_2 \hatstar^0 \scG_2) \cong \scG_1 \star^0_{\Iwu} \bigl( (\scF_1 \hatstar \scF_2) \hatstar^0 \scG_2 \bigr) \\
 \cong \bigl( \scG_1 \hatstar^0 (\scF_1 \hatstar \scF_2) \bigr) \star^0_{\Iwu} \scG_2
 \end{multline}
 and
 \begin{equation}
 \label{eqn:compatibility-hatstar-other-0}
 (\scF_1 \hatstar^0 \scG_1) \star^0_{\Iwu} \scG_2 \cong \scF_1 \hatstar^0 (\scG_1 \star^0_{\Iwu} \scG_2).
\end{equation}

We will also consider the bifunctors
\[
\phatstar^0 : \sfP^\wedge_{\Iwu,\Iwu} \times \sfP^0_{\Iwu,\Iwu} \to \sfP^0_{\Iwu,\Iwu}, \quad
\phatstar^0 : \sfP^0_{\Iwu,\Iwu} \times \sfP^\wedge_{\Iwu,\Iwu} \to \sfP^0_{\Iwu,\Iwu}
\]
defined by
\[
\scF \phatstar^0 \scG := \pH^0(\scF \hatstar^0 \scG).
\]

\begin{lem}
\label{lem:t-exactness-conv-tFl}
 If $\scF \in \sfD^\wedge_{\Iwu,\Iwu}$ and $\scG \in \sfD_{\Iwu,\Iwu}$ belong to the nonpositive parts of the perverse t-structures, then so does $\Pi^0_{\Iwu,\Iwu}(\scF \hatstar \scG)$. In particular, for $\scF$ in $\sfP^\wedge_{\Iwu,\Iwu}$ the functor
 \[
  \scF \phatstar^0 (-) : \sfP^0_{\Iwu,\Iwu} \to \sfP^0_{\Iwu,\Iwu}
 \]
is right exact, and for $\scG$ in $\sfP^0_{\Iwu,\Iwu}$ the functor
\[
 (-) \phatstar^0 \scG : \sfP^\wedge_{\Iwu,\Iwu} \to \sfP^0_{\Iwu,\Iwu}
\]
is right exact.
\end{lem}

\begin{proof}
 The first claim will follow in general if we prove it when $\scG$ is a simple perverse sheaf, i.e.~is of the form $\pi^\dag(\For^{\Iw}_{\Iwu}(\IC_w))$ for some $w \in W$. In this case, as in the proof of Lemma~\ref{lem:convolution-quotient} we have
 \[
  \scF \hatstar \pi^\dag(\For^{\Iw}_{\Iwu}(\IC_w)) \cong \pi^\dag(\pi_\dag(\scF) \star_{\Iw} \IC_w).
 \]
Now if $\ell(w) > 0$, Lemma~\ref{lem:convolution-quotient} implies that $\Pi^0_{\Iwu,\Iwu}(\pi^\dag(\pi_\dag(\scF) \star_{\Iw} \IC_w))=0$, and if $\ell(w)=0$ the object $\pi_\dag(\scF) \star_{\Iw} \IC_w$ belongs to the nonpositive part of the perverse t-structure by right exactness of $\pi_\dag$ (see~\S\ref{ss:completed-perverse}) and exactness of $(-) \star_{\Iw} \IC_w$.

To prove the second claim, we fix $\scF$ in $\sfP^\wedge_{\Iwu,\Iwu}$. An exact sequence in $\sfP^0_{\Iwu,\Iwu}$ is given by a distinguished triangle
\[
 \scG_1 \to \scG_2 \to \scG_3 \xrightarrow{[1]}
\]
in $\sfD^0_{\Iwu,\Iwu}$ where each $\scG_i$ belongs to $\sfP^0_{\Iwu,\Iwu}$. Applying the triangulated functor $\scF \hatstar^0 (-)$ we deduce a distinguished triangle
\[
 \scF \hatstar^0 \scG_1 \to \scF \hatstar^0 \scG_2 \to \scF \hatstar^0 \scG_3 \xrightarrow{[1]}.
\]
Now $\scG_1 = \Pi^0_{\Iwu,\Iwu}(\scG_1')$ for some $\scG_1'$ in $\sfP_{\Iwu,\Iwu}$ (by essential surjectivity and t-exactness of $\Pi^0_{\Iwu,\Iwu}$), and then $\scF \hatstar^0 \scG_1 = \Pi^0_{\Iwu,\Iwu}(\scF \hatstar \scG_1')$ is concentrated in nonpositive perverse degrees by the first claim; taking the long exact sequence of perverse cohomology associated with the above distinguished triangle we deduce an exact sequence
\[
 \pH^0(\scF \hatstar^0 \scG_1) \to \pH^0(\scF \hatstar^0 \scG_2) \to \pH^0(\scF \hatstar^0 \scG_3) \to 0,
\]
showing the right exactness of $\scF \phatstar^0 (-)$. The proof that $(-) \phatstar^0 \scG$ is right exact is similar, and left to the reader
\end{proof}

\subsection{Truncation of completed perverse sheaves}
\label{ss:truncation-ps}

Recall that the monodromy morphism for the right action of $T$ on $\tFl_G$ (see~\S\ref{ss:completed-category}) provides for any $\scF$ in $\mathsf{P}^\wedge_{\Iwu,\Iwu}$ a canonical $\bk$-algebra morphism
\[
\scO(T^\vee_\bk) \to \End_{\sfP^\wedge_{\Iwu,\Iwu}}(\scF),
\]
which is compatible in the obvious way with all morphisms in $\mathsf{P}^\wedge_{\Iwu,\Iwu}$; in the language of~\S\ref{ss:mod-cat}, this means that the monodromy construction provides a functor
\begin{equation}
\label{eqn:monodromy-Mod}
\mathsf{P}^\wedge_{\Iwu,\Iwu} \to \Mod \left( \scO(T^\vee_\bk), \mathsf{P}^\wedge_{\Iwu,\Iwu} \right)
\end{equation}
whose composition with the obvious forgetful functor
\[
\Mod \left( \scO(T^\vee_\bk), \mathsf{P}^\wedge_{\Iwu,\Iwu} \right) \to \mathsf{P}^\wedge_{\Iwu,\Iwu}
\]
is the identity. Composing~\eqref{eqn:monodromy-Mod} with the bifunctor~\eqref{eqn:otimes}, we therefore obtain a bifunctor
\[
(-) \otimes_{\scO(T^\vee_\bk)} (-) : \Mof(\scO(T^\vee_\bk)) \times \mathsf{P}^\wedge_{\Iwu,\Iwu} \to \mathsf{P}^\wedge_{\Iwu,\Iwu}.
\]

We will denote by $\mathcal{J} \subset \scO(T^\vee_\bk / \Wf)$ the ideal of the image of $e \in T^\vee_\bk$ in $T^\vee_\bk / \Wf$, and for $m \geq 1$ we set 
\[
(T^\vee_\bk)^{(m)} := \mathrm{Spec} \bigl( \scO(T^\vee_\bk) / \mathcal{J}^m \cdot \scO(T^\vee_\bk) \bigr).
\]
Below we will make use of the ``truncation'' functor
\[
\mathsf{C}_m := \scO((T^\vee_\bk)^{(m)}) \otimes_{\scO(T^\vee_\bk)} (-) : \mathsf{P}^\wedge_{\Iwu,\Iwu} \to \mathsf{P}^\wedge_{\Iwu,\Iwu}.
\]
In view of Lemma~\ref{lem:subcat-completion-monodromy}, this functor in fact takes values in $\mathsf{P}_{\Iwu,\Iwu}$, which allows us to also consider the composition
\[
\mathsf{C}_m^0 := \Pi^0_{\Iwu,\Iwu} \circ \mathsf{C}_m : \mathsf{P}^\wedge_{\Iwu,\Iwu} \to \mathsf{P}_{\Iwu,\Iwu}^0.
\]

\subsection{The case of the finite flag variety}
\label{ss:G/B-tilting}

Recall that the action of $G$ on the base point of $\tFl_G$ provides
a canonical closed embedding $G/U \hookrightarrow \tFl_G$, which identifies $G/U$ with the closure of $\tFl_{G,w_\circ}$, where $w_\circ \in \Wf$ is the longest element (or, in other words, the union of the orbits $\tFl_{G,w}$ with $w \in \Wf$). Under this identification, the action of $\Iwu$ on the closure of $\tFl_{G,w_\circ}$ corresponds to the action on $G/U$ induced by the natural $U$-action (by left multiplication) via the projection $\Iwu \to U$.

The same construction as for $\mathsf{D}_{\Iwu,\Iwu}$ and $\mathsf{D}^\wedge_{\Iwu,\Iwu}$ above provides a category $\mathsf{D}_{U,U}$ of sheaves on $G/U$, and a ``completed'' triangulated category $\mathsf{D}^\wedge_{U,U}$, with a monoidal product $\hatstar_U$. There is also a perverse t-structure on $\mathsf{D}^\wedge_{U,U}$, whose heart will be denoted $\mathsf{P}^\wedge_{U,U}$, and a notion of tilting perverse sheaves; the full subcategory of $\mathsf{P}^\wedge_{U,U}$ whose objects are the tilting perverse sheaves will be denoted $\mathsf{T}^\wedge_{U,U}$. The pushforward functor associated with the closed embedding $G/U \hookrightarrow \tFl_G$ provides a t-exact, monoidal, fully faithful triangulated functor 
\[
(\mathsf{D}^\wedge_{U,U}, \hatstar_U) \to (\mathsf{D}^\wedge_{\Iwu,\Iwu},\hatstar),
\]
which sends tilting perverse sheaves to tilting perverse sheaves. (This functor will usually be omitted from notation; similarly, objects of $\mathsf{D}^\wedge_{\Iwu,\Iwu}$ supported on $G/U$ will be considered as objects in $\mathsf{D}^\wedge_{U,U}$ whenever convenient.) The tilting perverse sheaves in the essential image of this functor are those which are direct sums of objects $\mathscr{T}^\wedge_w$ with $w \in \Wf$. 

The study of the category $\mathsf{T}^\wedge_{U,U}$ is the main subject of~\cite{bezr}. In the course of this study, we in particular construct explicit representatives for the objects $\mathscr{T}^\wedge_{w_\circ}$ and $\mathscr{T}^\wedge_s$ ($s \in \Sf$), as follows. Let us denote by $U^+$ the unipotent radical of the Borel subgroup of $G$ opposite to $B$ with respect to $T$, so that the $T$-weights in the Lie algebra of $U^+$ are the positive roots.
Fix, for any $s \in \Sf$, a trivialization of the root subgroup $U_s \subset U^+$ associated with the simple root corresponding to $s$. Then we obtain a group morphism $\chi : U^+ \to \Ga$ as the following composition:
\[
U^+ \to U^+/\mathscr{D}(U^+) \xleftarrow{\sim} \prod_{s \in \Sf} U_s \cong \prod_{s \in \Sf} \Ga \xrightarrow{+} \Ga.
\]
(Here we denote by $\mathscr{D}(U^+)$ the derived subgroup of $U^+$.)
Let us also fix a nontrivial $p$-th root of unity in $\bk$, and denote by $\LAS$ the associated Artin--Schreier local system on $\Ga$. Then as in~\cite[\S 10.3]{bezr} one can consider the $(U^+,\chi^*(\LAS))$-equivariant derived category of $\bk$-sheaves on $G/B$, which will be denoted $\sfD_{\Whit,B}$, and the full triangulated subcategory $\sfD_{\Whit,U}$ in the $(U^+,\chi^*(\LAS))$-equivariant derived category of $\bk$-sheaves on $G/U$ generated by complexes obtained by pullback from $\sfD_{\Whit,B}$. We have ``averaging'' functors
\[
\Av_\Whit : \sfD_{U,U} \to \sfD_{\Whit,U}, \qquad \Av_{U,*}, \Av_{U,!} : \sfD_{\Whit,U} \to \sfD_{U,U},
\]
which form adjoint pairs $(\Av_{U,!}, \Av_{\Whit})$ and $(\Av_{\Whit}, \Av_{U,*})$. One can also consider a ``completed'' category $\sfD_{\Whit,U}^\wedge$ using the same procedure as above (based on~\cite{by}), and the averaging functors induce triangulated functors
\[
\Av_\Whit : \sfD^\wedge_{U,U} \to \sfD^\wedge_{\Whit,U}, \qquad \Av_{U,*}, \Av_{U,!} : \sfD^\wedge_{\Whit,U} \to \sfD^\wedge_{U,U},
\]
which have the same adjointness properties as above.

We set
\[
\Xi_!^\wedge := \Av_{U,!} \circ \Av_\Whit(\delta^\wedge_e), \qquad \Xi_*^\wedge := \Av_{U,*} \circ \Av_\Whit(\delta^\wedge_e).
\]
It is proved in~\cite[Lemma~10.1]{bezr} (following standard arguments taken from~\cite{bbm,by}) that $\Xi^\wedge_!$ and $\Xi^\wedge_*$ are (noncanonically) isomorphic, and representatives for the object $\mathscr{T}^\wedge_{w_\circ}$. We will also set
\[
\Xi_!:=\pi_\dag(\Xi_!^\wedge), \quad \Xi_*:=\pi_\dag(\Xi_*^\wedge),
\]
so that we have $\Xi_! \cong \Xi_* \cong \scT_{w_\circ}$.

In the next two lemmas we prove some standard properties of this object that will be required later.

\begin{lem}
\label{lem:Xi-projective}
For $w \in \Wf$ and $n \in \Z$ we have
\[
\Hom_{\sfD_{U,U}^\wedge}(\Xi^\wedge_!, \pi^\dag \For^{\Iw}_{\Iwu}(\IC_w)[n]) \cong \begin{cases} \bk & \text{if $w=e$ and $n=0$;} \\ 0 & \text{otherwise.} \end{cases}
\]
In particular, the object $\Xi^\wedge_!$ is projective in $\sfP^\wedge_{U,U}$. 
\end{lem}

\begin{proof}
Let us denote by $\sfD_{U,B}$ the $U$-equivariant derived category of constructible $\bk$-sheaves on $G/B$, and by $\sfP_{U,B}$ the heart of its perverse t-structure. Then the realization functor $\Db \sfP_{U,B} \to \sfD_{U,B}$ is an equivalence of categories by the formalism of highest weight categories, see~\cite{bgs}, and $\pi_\dag(\Xi^\wedge_!) \cong \scT_{w_\circ}$ is the projective cover of $\For^{\Iw}_{\Iwu}(\IC_e)$ in $\sfP_{U,B}$ (see e.g.~\cite[Lemma~6.9]{bezr}). We deduce the desired isomorphism using adjunction.
This implies in particular that
\[
\Hom_{\sfD_{U,U}^\wedge}(\Xi^\wedge_!, \scF)=0
\]
for any $\scF$ in ${}^{\mathrm{p}} \hspace{-1pt} \sfD_{U,U}^{\leq -1}$. By Proposition~\ref{prop:perverse-t-str-completed-cat}, any object $\scF$ in ${}^{\mathrm{p}} \hspace{-1pt} \sfD_{U,U}^{\wedge,\leq -1}$ can be written as $``\varprojlim_n{}" \scF_n$ where each $\scF_n$ belongs to ${}^{\mathrm{p}} \hspace{-1pt} \sfD_{U,U}^{\leq -1}$; we then have
\[
\Hom_{\sfD_{U,U}^\wedge}(\Xi^\wedge_!, \scF) = \varprojlim_n \Hom_{\sfD_{U,U}^\wedge}(\Xi^\wedge_!, \scF_n) = 0,
\]
which proves that $\Xi^\wedge_!$ is projective in $\sfP^\wedge_{U,U}$.
\end{proof}


Using the fact that convolution on the left and on the right commute with each other, one sees that the functors
\[
\Xi_!^\wedge \hatstar_U (-), \, \Xi_*^\wedge \hatstar_U (-) : \sfD^\wedge_{U,U} \to \sfD^\wedge_{U,U}
\]
are canonically isomorphic to the functors $\Av_{U,!} \circ \Av_\Whit$ and $\Av_{U,*} \circ \Av_\Whit$ respectively. In particular, the adjunction morphisms for the pairs $(\Av_{U,!}, \Av_{\Whit})$ and $(\Av_{\Whit}, \Av_{U,*})$ provide canonical morphisms
\begin{equation}
\label{eqn:morph-adjunction-Xi}
\Xi_!^\wedge \hatstar_U \Xi_*^\wedge \to \delta^\wedge_e, \qquad \delta^\wedge_e \to \Xi_*^\wedge \hatstar_U \Xi_!^\wedge
\end{equation}
which satisfy the appropriate zigzag relations, so that the functor $\Xi_!^\wedge \hatstar_U (-)$ is left adjoint to $\Xi_*^\wedge \hatstar_U (-)$. The same holds for any category which admits an action of $\sfD^\wedge_{U,U}$, e.g.~$\mathsf{D}^\wedge_{\Iwu,\Iwu}$.

\begin{lem}
\label{lem:exactness-convolution-Xi}
 The functor
 \[
  \Xi_!^\wedge \hatstar (-) : \sfD^\wedge_{\Iwu,\Iwu} \to \sfD^\wedge_{\Iwu,\Iwu}
 \]
is t-exact for the perverse t-structure.
\end{lem}

\begin{proof}
 The subcategory ${}^{\mathrm{p}} \sfD^{\wedge, \leq 0}_{\Iwu,\Iwu}$ is generated under extensions by the objects $\Delta^\wedge_w[n]$ for $w \in W$ and $n \in \Z_{\geq 0}$, see~\S\ref{ss:completed-perverse}.
 Now $\Xi_!^\wedge \hatstar \Delta^\wedge_w$ is perverse for any $w$ by Lemma~\ref{lem:convolution-stand-costand}\eqref{it:convolution-stand-costand-3}, since $\Xi_!^\wedge$ admits a costandard filtration. This implies that our functor is right t-exact. As explained above this functor is isomorphic to $\Xi_*^\wedge \hatstar (-)$, which admits as a left adjoint the right t-exact functor $\Xi_!^\wedge \hatstar (-)$. It is therefore left t-exact, hence finally t-exact.
\end{proof}

\begin{rmk}
 One can check (using essentially the same arguments) that in fact the functor $\scF \hatstar (-)$ is t-exact for any $\scF$ in $\mathsf{T}^\wedge_{\Iwu,\Iwu}$.
\end{rmk}




If now $s \in \Sf$ is a simple reflection in $\Wf$, let us denote by $\overline{\imath}_s$ the embedding of the closure of $\tFl_{G,s}$ in $\tFl_G$. We set
\[
\Xi_{s,!}^\wedge := (\overline{\imath}_s)_* (\overline{\imath}_s)^* \Xi_!^\wedge, \qquad \Xi_{s,*}^\wedge := (\overline{\imath}_s)_* (\overline{\imath}_s)^! \Xi_*^\wedge.
\]
As explained in~\cite[Remark~9.5]{bezr}, these objects are (noncanonically) isomorphic, and are representatives for the object $\mathscr{T}^\wedge_s$. Moreover, there exists a surjection $\Xi_!^\wedge \twoheadrightarrow \Xi_{s,!}^\wedge$ whose kernel admits a standard filtration, and there exists an embedding $\Xi_{s,*}^\wedge \hookrightarrow \Xi_{*}^\wedge$ whose cokernel admits a costandard filtration.

\section{Free-monodromic central sheaves and Wakimoto sheaves}
\label{sec:fmZ}

As in Sections~\ref{sec:construction-mon-reg-quot}--\ref{sec:fmps} we continue with the general setting of~\S\ref{ss:Fl}.

\subsection{Free-monodromic Wakimoto sheaves -- definition}

Recall the Wakimoto sheaves mentioned in~\S\ref{ss:central-sheaves-prop}.
In this subsection we explain (following~\cite{be}) how these objects can be ``lifted'' to the completed category $\mathsf{D}^\wedge_{\Iwu,\Iwu}$.

Recall that for $\lambda \in X_*(T)$ we have a point $z^\lambda \in \Loop G(\F)$ whose image in $\tFl_G$ belongs to $\tFl_{G,\st(\lambda)}$. This point defines a point in the quotient $\Iwu \backslash \tFl_{G,\st(\lambda)}$, hence a trivialization of this $T$-torsor. In case $w \in X_*(T) \subset W$, the objects $\Delta^\wedge_w$ and $\nabla^\wedge_w$ therefore admit \emph{canonical} representatives $\Delta^{\wedge,\can}_w$ and $\nabla^{\wedge,\can}_w$, defined using this trivialization. 
With this definition, we have canonical isomorphisms
\begin{equation}
\label{eqn:pidag-DN-can}
 \pi_\dag(\Delta^{\wedge,\can}_w) \cong \For^{\Iw}_{\Iwu}(\Delta^\Iw_w), \quad \pi_\dag(\nabla^{\wedge,\can}_w) \cong \For^{\Iw}_{\Iwu}(\nabla_w^\Iw).
\end{equation}
We will also denote by $p_w^{\can} : \Fl_{G,w} \to T$ the associated morphism (for $w \in X_*(T)$).

\begin{lem}
\label{lem:DN-weights}
 If $\lambda,\mu \in X_*^+(T)$, then there exist \emph{canonical} isomorphisms
 \[
  \Delta^{\wedge,\can}_{\st(\lambda)} \hatstar \Delta^{\wedge,\can}_{\st(\mu)} \simto \Delta^{\wedge,\can}_{\st(\lambda+\mu)}, \quad \nabla^{\wedge,\can}_{\st(\lambda)} \hatstar \nabla^{\wedge,\can}_{\st(\mu)} \simto \nabla^{\wedge,\can}_{\st(\lambda+\mu)}.
 \]
 Moreover these isomorphisms are compatible with associativity, in the sense that for $\lambda,\mu,\nu \in X_*^+(T)$ the two natural isomorphisms between
 \[
  \Delta^{\wedge,\can}_{\st(\lambda)} \hatstar \Delta^{\wedge,\can}_{\st(\mu)} \hatstar \Delta^{\wedge,\can}_{\st(\nu)} \quad \text{and} \quad \Delta^{\wedge,\can}_{\st(\lambda+\mu+\nu)},
 \]
resp.~between
\[
  \nabla^{\wedge,\can}_{\st(\lambda)} \hatstar \nabla^{\wedge,\can}_{\st(\mu)} \hatstar \nabla^{\wedge,\can}_{\st(\nu)} \quad \text{and} \quad \nabla^{\wedge,\can}_{\st(\lambda+\mu+\nu)},
 \]
 which can be constructed by combining these isomorphisms coincide.
\end{lem}

\begin{proof}
The proof is similar to that of~\cite[Lemma~4]{be}.
 Namely, let us first consider the case of the objects $\Delta^{\wedge,\can}_{\st(\lambda)}$. Recall that for $\lambda,\mu \in X_*^+(T)$ we have $\ell(\st(\lambda)) + \ell(\st(\mu)) = \ell(\st(\lambda+\mu))$; the morphism $m$ therefore induces an isomorphism
 \begin{equation}
 \label{eqn:m-lamu-isomorphism}
  \Fl_{G,\st(\lambda)} \, \tilde{\times} \, \Fl_{G,\st(\mu)} \simto \Fl_{G,\st(\lambda+\mu)}.
 \end{equation}
 Let us denote by
 \[
  \tilde{m}_{\lambda,\mu} : \tFl_{G,\st(\lambda)} \, \widetilde{\times} \, \tFl_{G,\st(\mu)} \to \tFl_{G,\st(\lambda+\mu)}
 \]
the morphism induced by $\tilde{m}$, by
\[
 \tilde{\jmath}_{\st(\lambda)} \, \tilde{\times} \, \tilde{\jmath}_{\st(\mu)} : \tFl_{G,\st(\lambda)} \, \widetilde{\times} \, \tFl_{G,\st(\mu)} \to \tFl_{G} \, \widetilde{\times} \, \tFl_{G}
\]
the (locally closed) embedding, and by
\[
 p_{\st(\lambda)}^{\can} \, \tilde{\times} \, p_{\st(\mu)}^{\can} : \tFl_{G,\st(\lambda)} \, \widetilde{\times} \, \tFl_{G,\st(\mu)} \to T \times T
\]
the morphism induced by $p_{\st(\lambda)}^{\can}$ and $p_{\st(\mu)}^{\can}$. These morphisms fit in a diagram
\[
 \xymatrix@C=1.5cm{
 \tFl_{G,\st(\lambda)} \, \widetilde{\times} \, \tFl_{G,\st(\mu)} \ar[r]^-{\tilde{m}_{\lambda,\mu}} \ar[d]_-{p_{\st(\lambda)}^{\can} \, \tilde{\times} \, p_{\st(\mu)}^{\can}} & \tFl_{G,\st(\lambda+\mu)} \ar[d]^-{p^\can_{\st(\lambda+\mu)}} \\
 T \times T \ar[r]^{m_T} & T
 }
\]
(where the lower horizontal arrow is multiplication in $T$), which is easily seen to be commutative. Moreover, the isomorphism~\eqref{eqn:m-lamu-isomorphism} implies that this diagram is cartesian.

Using the isomorphism
\[
 \tilde{m}_! \circ (\tilde{\jmath}_{\st(\lambda)} \, \tilde{\times} \, \tilde{\jmath}_{\st(\mu)})_! \cong (\tilde{\jmath}_{\st(\lambda+\mu)})_! \circ (\tilde{m}_{\lambda,\mu})_!
\]
and the definition of $\hatstar$ we obtain an isomorphism
\begin{multline*}
 \Delta^{\wedge,,\can}_{\st(\lambda)} \hatstar \Delta^{\wedge,\can}_{\st(\mu)} \cong \\
 ``\varprojlim_{m}" ``\varprojlim_{n}" (\tilde{\jmath}_{\st(\lambda+\mu)})_! (\tilde{m}_{\lambda,\mu})_! (p_{\st(\lambda)}^{\can} \, \tilde{\times} \, p_{\st(\mu)}^{\can})^*(\scL_{T,n} \boxtimes \scL_{T,m})[\ell(\st(\lambda+\mu))+3\dim(T)].
\end{multline*}
Using the base change theorem we deduce an isomorphism
\begin{multline*}
 \Delta^{\wedge,\can}_{\st(\lambda)} \hatstar \Delta^{\wedge,\can}_{\st(\mu)} \cong \\
 ``\varprojlim_{m}" ``\varprojlim_{n}" (\tilde{\jmath}_{\st(\lambda+\mu)})_! (p_{\st(\lambda+\mu)}^{\can})^* (m_T)_! (\scL_{T,n} \boxtimes \scL_{T,m})[\ell(\st(\lambda+\mu))+3\dim(T)].
\end{multline*}
Now by~\cite[Lemma~3.4]{bezr}, for any $m$ there is a canonical isomorphism
\[
 ``\varprojlim_{n}" (m_T)_! (\scL_{T,n} \boxtimes \scL_{T,m}) \cong \scL_{T,m}[-2\dim(T)],
\]
which provides the desired isomorphism.
The verification that this isomorphism is compatible with associativity in the sense above is straightforward, and left to the reader.

The proof for the objects $\nabla^{\wedge,\can}_{\st(\lambda)}$ will be similar, once we construct a canonical isomorphism
\[
 \tilde{m}_! \circ (\tilde{\jmath}_{\st(\lambda)} \, \tilde{\times} \, \tilde{\jmath}_{\st(\mu)})_* \scF \simto (\tilde{\jmath}_{\st(\lambda+\mu)})_* \circ (\tilde{m}_{\lambda,\mu})_! \scF
\]
for each $\scF$ of the form $(p_{\st(\lambda)}^{\can} \, \tilde{\times} \, p_{\st(\mu)}^{\can})^*(\scL_{T,n} \boxtimes \scL_{T,m})$ with $n,m \geq 0$. We obtain a morphism from the left-hand side to the right-hand side as the composition
\begin{multline*}
 \tilde{m}_! (\tilde{\jmath}_{\st(\lambda)} \, \tilde{\times} \, \tilde{\jmath}_{\st(\mu)})_* \scF \to (\tilde{\jmath}_{\st(\lambda+\mu)})_* (\tilde{\jmath}_{\st(\lambda+\mu)})^* \tilde{m}_! (\tilde{\jmath}_{\st(\lambda)} \, \tilde{\times} \, \tilde{\jmath}_{\st(\mu)})_* \scF \\
 \simto (\tilde{\jmath}_{\st(\lambda+\mu)})_* (\tilde{m}_{\lambda,\mu})_! (\tilde{\jmath}_{\st(\lambda)} \, \tilde{\times} \, \tilde{\jmath}_{\st(\mu)})^* (\tilde{\jmath}_{\st(\lambda)} \, \tilde{\times} \, \tilde{\jmath}_{\st(\mu)})_* \scF \to (\tilde{\jmath}_{\st(\lambda+\mu)})_* (\tilde{m}_{\lambda,\mu})_! \scF
\end{multline*}
where the first and third morphisms are induced by adjunction, and the middle isomorphism is given by the base change theorem. Since the objects we have to consider are all extensions of constant local systems, to prove that this morphism is invertible on these objects it suffices to prove this property when $\scF$ is constant. In this case, Lemma~\ref{lem:convolution-Iw-Iwu} shows that the left-hand side identifies with
\[
 (\tilde{\jmath}_{\st(\lambda+\mu)})_* \underline{\bk} \otimes_\bk \mathsf{H}^\bullet_c(T,\bk),
\]
and the same holds for the right-hand side since $\tilde{m}_{\lambda,\mu}$ is a trivial $T$-torsor. It is easily seen that our morphism identifies with the identity of this object, which concludes the proof.
\end{proof}

This lemma will allow us to define the \emph{free-monodromic Wakimoto sheaves} as follows. Given $\lambda \in X_*(T)$, for $\scF$ in $\mathsf{D}^\wedge_{\Iwu,\Iwu}$ and $\mu,\nu \in X_*^+(T)$ such that $\lambda=\mu-\nu$ we consider the $\bk$-vector space
\[
 \Hom_{\mathsf{D}^\wedge_{\Iwu,\Iwu}}(\nabla^{\wedge,\can}_{\st(\mu)}, \nabla^{\wedge,\can}_{\st(\nu)} \hatstar \scF).
\]
If $(\mu',\nu')$ is another pair of elements of $X_*^+(T)$ such that $\lambda=\mu'-\nu'$, and if $\mu'-\mu \in X_*^+(T)$, then there is a canonical isomorphism
\[
 \Hom_{\mathsf{D}^\wedge_{\Iwu,\Iwu}}(\nabla^{\wedge,\can}_{\st(\mu)}, \nabla^{\wedge,\can}_{\st(\nu)} \hatstar \scF) \simto \Hom_{\mathsf{D}^\wedge_{\Iwu,\Iwu}}(\nabla^{\wedge,\can}_{\st(\mu')}, \nabla^{\wedge,\can}_{\st(\nu')} \hatstar \scF)
\]
induced by left convolution with $\nabla^{\wedge,\can}_{\st(\mu'-\mu)}$ and the isomorphisms of Lemma~\ref{lem:DN-weights}. Moreover, the compatibility with associativity implies that this collection of spaces and isomorphisms is an inductive system for the order on pairs $(\mu,\nu)$ of elements of $X_*^+(T)$ such that $\lambda=\mu-\nu$ given by $(\mu,\nu) \unlhd (\mu',\nu')$ iff $\mu'-\mu \in X_*^+(T)$. One can therefore consider the functor
\[
 \scF \mapsto \varinjlim_{(\mu,\nu)} \Hom_{\mathsf{D}^\wedge_{\Iwu,\Iwu}}(\nabla^{\wedge,\can}_{\st(\mu)}, \nabla^{\wedge,\can}_{\st(\nu)} \hatstar \scF).
\]
This functor is representable, since each transition morphism is an isomorphism and for any given choice of pair $(\mu,\nu)$ we have an isomorphism
\[
 \Hom_{\mathsf{D}^\wedge_{\Iwu,\Iwu}}(\nabla^{\wedge,\can}_{\st(\mu)}, \nabla^{\wedge,\can}_{\st(\nu)} \hatstar \scF) \cong \Hom_{\mathsf{D}^\wedge_{\Iwu,\Iwu}}(\Delta^{\wedge,\can}_{\st(-\nu)} \hatstar \nabla^{\wedge,\can}_{\st(\mu)}, \scF),
\]
see Lemma~\ref{lem:convolution-stand-costand}. One can therefore define $\Wak_\lambda^\wedge$ as the object representing this functor. From the construction, we see that for any pair $(\mu,\nu) \in (X_*^+(T))^2$ such that $\lambda=\mu-\nu$ we have a \emph{noncanonical} isomorphism
\[
 \Wak_\lambda^\wedge \cong \Delta^{\wedge,\can}_{\st(-\nu)} \hatstar \nabla^{\wedge,\can}_{\st(\mu)}.
\]
Using Lemma~\ref{lem:convolution-stand-costand}, it is also not difficult to check that the right-hand side is (again, noncanonically) isomorphic to $\nabla^{\wedge,\can}_{\st(\mu)} \hatstar \Delta^{\wedge,\can}_{\st(-\nu)}$.

\subsection{Free-monodromic Wakimoto sheaves -- properties}

The following statement gathers the main properties of free-monodromic Wakimoto sheaves that we will need below.

\begin{lem}
 \phantomsection
 \label{lem:properties-Wak}
 \begin{enumerate}
  \item 
  \label{it:properties-Wak-0}
  If $\lambda \in X_*(T)$, and if $\mu \in X_*^+(T)$ is such that $\lambda+\mu \in X_*^+(T)$, then there exists a canonical isomorphism
 \[
 \Wak^\wedge_\lambda \hatstar \nabla^{\wedge,\can}_{\st(\mu)} \cong \nabla^{\wedge,\can}_{\st(\lambda+\mu)}.
 \]
  \item 
  \label{it:properties-Wak-1}
  For $\lambda,\mu \in X_*(T)$ and $n \in \Z$, we have
  \[
   \Hom_{\mathsf{D}^\wedge_{\Iwu,\Iwu}}(\Wak^\wedge_\lambda,\Wak^\wedge_\mu[n])=0
  \]
unless $\mu \preceq \lambda$.
\item 
\label{it:properties-Wak-2}
For any $\lambda \in X_*(T)$, we have a canonical isomorphism
\[
 \pi_\dag(\Wak_\lambda^\wedge) \cong \For^{\Iw}_{\Iwu}(\Wak_\lambda).
\]
Moreover, $\Wak_\lambda^\wedge$ is a perverse sheaf.
\item
\label{it:properties-Wak-3}
For $\lambda,\mu \in X_*(T)$ there exists a canonical isomorphism
\[
\Wak_\lambda^\wedge \hatstar \Wak_\mu^\wedge \cong \Wak_{\lambda+\mu}^\wedge.
\]
\item
\label{it:properties-Wak-4}
For any $\lambda \in X_*(T)$, the restriction of $\mu_{\Wak^\wedge_\lambda}$ to
\[
\scO(T^\vee_\bk) = \bk \otimes \scO(T^\vee_\bk) \subset \scO(T^\vee_\bk \times T^\vee_\bk)
\]
factors (in the canonical way) through an isomorphism
\[
\scO(\FN_{T^\vee_\bk}(\{e\})) \simto \End_{\mathsf{D}^\wedge_{\Iwu,\Iwu}}(\Wak^\wedge_\lambda),
\]
and we have
\[
   \Hom_{\mathsf{D}^\wedge_{\Iwu,\Iwu}}(\Wak^\wedge_\lambda,\Wak^\wedge_\lambda[1])=0.
  \]
 \item
\label{it:properties-Wak-5}
For any $\lambda \in X_*(T)$ and $f \in \scO(T^\vee_\bk)$ we have
\[
\mu_{\Wak^\wedge_\lambda}(f \otimes 1)=\mu_{\Wak^\wedge_\lambda}(1 \otimes f),
\]
and moreover
\[
\mu_{\Wak^\wedge_\lambda}^{\mathrm{rot}}(x)=\mu_{\Wak^\wedge_\lambda}(1 \otimes e^{-\lambda}).
\]
 \end{enumerate}
\end{lem}

\begin{proof}
\eqref{it:properties-Wak-0} The proof is identical to that of its counterpart in $\mathsf{D}_{\Iw,\Iw}$, see~\cite[Lemma~4.2.5]{ar-book}. 

\eqref{it:properties-Wak-1} Let $\nu \in X_*^+(T)$ be such that $\lambda+\nu \in X_*^+(T)$ and $\mu+\nu \in X_*^+(T)$. Since the functor $(-) \hatstar \nabla^{\wedge,\can}_{\st(\nu)}$ is an equivalence of categories (see Remark~\ref{rmk:convolution-equiv}), using~\eqref{it:properties-Wak-0} we obtain an isomorphism
\[
\Hom_{\mathsf{D}^\wedge_{\Iwu,\Iwu}}(\Wak^\wedge_\lambda,\Wak^\wedge_\mu[n]) \cong \Hom_{\mathsf{D}^\wedge_{\Iwu,\Iwu}}(\nabla^{\wedge,\can}_{\st(\lambda+\nu)},\nabla^{\wedge,\can}_{\st(\mu+\nu)}[n]).
\]
Now the right-hand side vanishes unless $\Fl_{G,\st(\mu+\nu)} \subset \overline{\Fl_{G,\st(\lambda+\nu)}}$. It is well known that this condition is satisfied if and only if $\mu+\nu \preceq \lambda+\nu$, which implies the claim.

\eqref{it:properties-Wak-2} If $\mu \in X_*^+(T)$ is such that $\lambda+\mu \in X_*^+(T)$, then by~\eqref{it:properties-Wak-0} we have a canonical isomorphism
 \[
 \Wak^\wedge_\lambda \hatstar \nabla^{\wedge,\can}_{\st(\mu)} \cong \nabla^{\wedge,\can}_{\st(\lambda+\mu)}.
 \]
 Now we have
 \begin{multline*}
 \pi_\dag(\Wak^\wedge_\lambda \hatstar \nabla^{\wedge,\can}_{\st(\mu)}) \overset{\eqref{eqn:conv-formula-1}}{\cong} \Wak^\wedge_\lambda \hatstar \pi_\dag(\nabla^{\wedge,\can}_{\st(\mu)}) \overset{\eqref{eqn:pidag-DN-can}}{\cong} \Wak^\wedge_\lambda \hatstar \For^{\Iw}_{\Iwu}(\nabla_{\st(\mu)}^\Iw) \\
 \overset{\eqref{eqn:conv-formula-2}}{\cong} \pi_\dag(\Wak^\wedge_\lambda) \star_\Iw \nabla^\Iw_{\st(\mu)}.
 \end{multline*}
 Using again~\eqref{eqn:pidag-DN-can}, we deduce a canonical isomorphism
 \[
 \pi_\dag(\Wak^\wedge_\lambda) \star_\Iw \nabla^\Iw_{\st(\mu)} \cong \For^{\Iw}_{\Iwu}(\nabla^{\Iw}_{\st(\lambda+\mu)}).
 \]
By~\cite[Lemma~4.2.7]{ar-book} the right-hand side is canonically isomorphic to $\For^{\Iw}_{\Iwu}(\Wak_\lambda) \star_\Iw \nabla^\Iw_{\st(\mu)}$. Since the functor $(-) \star_\Iw \nabla^\Iw_{\st(\mu)}$ is an equivalence of categories (with quasi-inverse $(-) \star_\Iw \Delta^\Iw_{\st(-\mu)}$), this defines an isomorphism
 \[
  \pi_\dag(\Wak^\wedge_\lambda) \cong \For^{\Iw}_{\Iwu}(\Wak_\lambda).
 \]
 One can easily check that this isomorphism does not depend on the choice of $\mu$, hence is indeed canonical. Finally, since $\Wak_\lambda$ is perverse this isomorphism implies that $\Wak_\lambda^\wedge$ is perverse by~\cite[Lemma~5.3(1)]{bezr}.
 
\eqref{it:properties-Wak-3} If $\nu \in X_*^+(T)$ is such that $\mu+\nu \in X_*^+(T)$ and $\lambda+\mu+\nu \in X_*^+(T)$, then using~\eqref{it:properties-Wak-0} we obtain canonical isomorphisms
\[
\Wak_\lambda^\wedge \hatstar \Wak_\mu^\wedge \hatstar \nabla^{\wedge,\can}_{\st(\nu)} \cong \Wak_\lambda^\wedge \hatstar \nabla^{\wedge,\can}_{\st(\mu+\nu)} \cong \nabla^{\wedge,\can}_{\st(\lambda+\mu+\nu)} \cong \Wak_{\lambda+\mu}^\wedge \hatstar \nabla^{\wedge,\can}_{\st(\nu)}.
\]
Since the functor $(-) \hatstar \nabla^{\wedge,\can}_{\st(\nu)}$ is an equivalence of categories (see Remark~\ref{rmk:convolution-equiv}), we deduce an isomorphism
\[
\Wak_\lambda^\wedge \hatstar \Wak_\mu^\wedge \cong \Wak^\wedge_{\lambda+\mu}.
\]
It can be checked that this isomorphism does not depend on $\nu$, hence is indeed canonical.

\eqref{it:properties-Wak-4}
Let $\mu \in X_*^+(T)$ be such that $\lambda+\mu \in X_*^+(T)$. Then as above, using~\eqref{it:properties-Wak-0} we obtain that the functor $(-) \hatstar \nabla^{\wedge,\can}_{\st(\mu)}$ induces an isomorphism
\[
\Hom_{\mathsf{D}^\wedge_{\Iwu,\Iwu}}(\Wak^\wedge_\lambda,\Wak^\wedge_\lambda[n]) \simto \Hom_{\mathsf{D}^\wedge_{\Iwu,\Iwu}}(\nabla^{\wedge,\can}_{\st(\lambda+\mu)},\nabla^{\wedge,\can}_{\st(\lambda+\mu)}[n])
\]
for any $n \in \Z$. When $n=1$, it follows from~\cite[Corollary~4.6]{bezr} and adjunction that the right-hand side vanishes, so that the left-hand side vanishes as well. For $n=0$, it follows from~\eqref{eqn:monodromy-convolution-1}--\eqref{eqn:monodromy-convolution-2} and Lemma~\ref{lem:monodromy-DN} that the composition of this isomorphism with $\mu_{\Wak^\wedge_\lambda}$ is $\mu_{\nabla^{\wedge,\can}_{\st(\lambda+\mu)}}$; the desired claim therefore follows from Lemma~\ref{lem:monodromy-DN}.

\eqref{it:properties-Wak-5}
Let us fix $\mu,\nu \in X_*^+(T)$ such that $\lambda=\mu-\nu$, and an isomorphism $\Wak_\lambda^\wedge \cong \Delta^{\wedge,\can}_{\st(-\nu)} \hatstar \nabla^{\wedge,\can}_{\st(\mu)}$. Identifying these objects via this isomorphism and using~\eqref{eqn:monodromy-convolution-1}--\eqref{eqn:monodromy-convolution-2} and Lemma~\ref{lem:monodromy-DN} we obtain that for $f \in \scO(T^\vee_\bk)$ we have
\begin{multline*}
\mu_{\Wak_\lambda^\wedge}(f \otimes 1) = \mu_{\Delta^{\wedge,\can}_{\st(-\nu)}}(f \otimes 1) \hatstar \id = \mu_{\Delta^{\wedge,\can}_{\st(-\nu)}}(1 \otimes f) \hatstar \id \\
= \id \hatstar \mu_{\nabla^{\wedge,\can}_{\st(\mu)}}(f \otimes 1)
= \id \hatstar \mu_{\nabla^{\wedge,\can}_{\st(\mu)}}(1 \otimes f) = \mu_{\Wak_\lambda^\wedge}(1 \otimes f),
\end{multline*}
which proves the first claim. To prove the second claim, we observe that 
by~\eqref{eqn:monodromy-convolution-1}--\eqref{eqn:monodromy-convolution-2}--\eqref{eqn:monodromy-convolution-3}
and Lemma~\ref{lem:monodromy-DN} we have
\begin{multline*}
\mu^{\mathrm{rot}}_{\Wak_\lambda^\wedge}(x) = \mu^{\mathrm{rot}}_{\Delta^{\wedge,\can}_{\st(-\nu)}}(x) \hatstar \mu_{\nabla^{\wedge,\can}_{\st(\mu)}}^{\mathrm{rot}}(x) = \mu_{\Delta^{\wedge,\can}_{\st(-\nu)}}(1 \otimes e^{\nu}) \hatstar \mu_{\nabla^{\wedge,\can}_{\st(\mu)}}(1 \otimes e^{-\mu}) \\
= \id \hatstar \mu_{\nabla^{\wedge,\can}_{\st(\mu)}}(1 \otimes e^{-\lambda}) = \mu_{\Wak_\lambda^\wedge}(1 \otimes e^{-\lambda}),
\end{multline*}
which finishes the proof.
\end{proof}

Below we will also need the following property, which has no counterpart in the nonmonodromic setting. 

\begin{lem}
\label{lem:Hom-Wak}
If $\lambda,\mu \in X_*(T)$ are distinct, then we have
\[
\Hom_{\mathsf{P}^\wedge_{\Iwu,\Iwu}}(\Wak^\wedge_\lambda,\Wak^\wedge_\mu)=0.
\]
\end{lem}

\begin{proof}
Let $\nu \in -X_*^+(T)$ be such that $\lambda+\nu$ and $\mu + \nu$ belong to $-X_*^+(T)$. Then we have
\[
\Wak_\lambda^\wedge \hatstar \Delta^{\wedge,\can}_{\st(\nu)} \cong \Wak_\lambda^\wedge \hatstar \Wak_\nu^\wedge \cong \Wak^\wedge_{\lambda+\nu} \cong \Delta^{\wedge,\can}_{\st(\lambda+\nu)},
\]
and similarly for $\mu$. Since the functor $(-) \hatstar \Delta^{\wedge,\can}_{\st(\nu)}$ is an equivalence of categories (see Remark~\ref{rmk:convolution-equiv}), we deduce an isomorphism
\[
\Hom_{\mathsf{P}^\wedge_{\Iwu,\Iwu}}(\Wak^\wedge_\lambda,\Wak^\wedge_\mu) \cong \Hom_{\mathsf{P}^\wedge_{\Iwu,\Iwu}}(\Delta^{\wedge,\can}_{\st(\lambda+\nu)},\Delta^{\wedge,\can}_{\st(\mu+\nu)}).
\]
The claim then follows from~\eqref{eqn:Hom-Delta}.
\end{proof}

\subsection{Wakimoto filtrations of free-monodromic perverse sheaves}
\label{ss:wakimoto}

The notion of objects admitting a Wakimoto filtration (see~\S\ref{ss:central-sheaves-prop}) has an obvious analogue in the category $\mathsf{P}^\wedge_{\Iwu,\Iwu}$: we will say that an object $\scF$ \emph{admits a Wakimoto filtration} if there exists a finite filtration on $\scF$ such that each subquotient is a free-monodromic Wakimoto sheaf $\Wak^\wedge_\lambda$ with $\lambda \in X_*(T)$. As in the case of $\mathsf{P}_{\Iw,\Iw}$ and $\mathsf{P}_{\Iwu,\Iw}$, the properties of free-monodromic Wakimoto sheaves stated in Lemma~\ref{lem:properties-Wak}\eqref{it:properties-Wak-1}--\eqref{it:properties-Wak-4} imply that if $\scF$ admits a Wakimoto filtration, then there exists a unique filtration $(\scF_{\leq \lambda} : \lambda \in X_*(T))$ on $\scF$ such that $\scF_{\leq \lambda}=\{0\}$ for some $\lambda$, $\scF_{\leq \mu}=\scF$ for some $\mu$,  and $\scF_{\leq \lambda} / \scF_{< \lambda}$ is a direct sum of copies of $\Wak^\wedge_\lambda$ for each $\lambda \in X_*(T)$. (Here, $\leq$ is the order we fixed in~\S\ref{ss:central-sheaves-prop}.)
Moreover, this filtration is functorial in the same sense as for its ``traditional'' counterpart in~\S\ref{ss:central-sheaves-prop}, which allows to define the functor $\gr_\lambda^\wedge$ sending an object $\scF$ which admits a free-monodromic Wakimoto filtration to
\[
\gr^\wedge_\lambda(\scF):=\scF_{\leq \lambda}/\scF_{< \lambda}.
\]
For such $\scF$, we will also set
\[
\gr^\wedge_\bullet(\scF) = \bigoplus_{\lambda \in X_*(T)} \gr^\wedge_\lambda(\scF).
\]

\begin{prop}
If $\scF,\scG$ in $\mathsf{P}^\wedge_{\Iwu,\Iwu}$ admit Wakimoto filtrations, then the morphism
\[
\Hom_{\mathsf{P}^\wedge_{\Iwu,\Iwu}}(\scF,\scG) \to \Hom_{\mathsf{P}^\wedge_{\Iwu,\Iwu}}(\gr^\wedge_\bullet(\scF),\gr^\wedge_\bullet(\scG))
\]
induced by the functor $\gr^\wedge_\bullet$ is injective.
\end{prop}

\begin{proof}
The claim is an easy consequence of Lemma~\ref{lem:Hom-Wak}, following the same arguments as in~\cite[Corollary~6.3]{bezr}.
\end{proof}

From this proposition we deduce in particular the following claim.

\begin{cor}
\label{cor:Wak-filtr-monodromy}
For any $\scF$ in $\mathsf{P}^\wedge_{\Iwu,\Iwu}$ which admits a Wakimoto filtration, the morphism
\[
\mu_{\scF} : \scO(T^\vee_\bk \times T^\vee_\bk) \to \End(\scF)
\]
factors through the morphism $\scO(T^\vee_\bk \times T^\vee_\bk) \to \scO(T^\vee_\bk)$ induced by the diagonal embedding $T^\vee_\bk \hookrightarrow T^\vee_\bk \times T^\vee_\bk$; in other words, for any $f$ in $\scO(T^\vee_\bk)$ we have
\[
\mu_\scF(f \otimes 1)=\mu_\scF(1 \otimes f).
\]
\end{cor}

\begin{proof}
The functoriality of monodromy implies that the composition
\[
\scO(T^\vee_\bk \times T^\vee_\bk) \xrightarrow{\mu_\scF} \End(\scF) \to \End(\gr^\wedge_\bullet(\scF))
\]
(where the second morphism is induced by the functor $\gr^\wedge_\bullet$) coincides with $\mu_{\gr^\wedge_\bullet(\scF)}$. This reduces the proof to the case $\scF$ is a free-monodromic Wakimoto sheaf, in which case the claim was proved in Lemma~\ref{lem:properties-Wak}\eqref{it:properties-Wak-5}.
\end{proof}

We finish this subsection with a criterion for the existence of Wakimoto filtrations.

\begin{lem}
\label{lem:criterion-Wak}
Let $\scF$ in $\mathsf{D}^\wedge_{\Iwu,\Iwu}$. Then $\scF$ belongs to $\mathsf{P}^\wedge_{\Iwu,\Iwu}$ and admits a Wakimoto filtration iff $\pi_\dag(\scF)$ belongs to $\mathsf{P}_{\Iwu,\Iw}$ and admits a Wakimoto filtration. Moreover, in this case the multiplicity of $\Wak^\wedge_\lambda$ in $\gr^\wedge_\lambda(\scF)$ equals the multiplicity of $\Wak_\lambda$ in $\gr_\lambda(\pi_\dag(\scF))$.
\end{lem}

\begin{proof}
If $\scF$ belongs to $\mathsf{P}^\wedge_{\Iwu,\Iwu}$ and admits a Wakimoto filtration, then Lemma~\ref{lem:properties-Wak}\eqref{it:properties-Wak-2} implies that $\pi_\dag(\scF)$ belongs to $\mathsf{P}^\wedge_{\Iwu,\Iw}$ and admits a Wakimoto filtration. 

Now, assume that $\pi_\dag(\scF)$ belongs to $\mathsf{P}_{\Iwu,\Iw}$ and admits a Wakimoto filtration. Choose $\mu \in X_*^+(T)$ such that $\lambda+\mu$ is dominant for any $\lambda \in X_*(T)$ such that $\gr_\lambda(\pi_\dag(\scF)) \neq 0$. We have
\[
\pi_\dag(\scF \hatstar \nabla^{\wedge,\can}_{\st(\mu)}) \overset{\eqref{eqn:conv-formula-1}}{\cong} \scF \hatstar \pi_\dag(\nabla^{\wedge,\can}_{\st(\mu)}) \overset{\eqref{eqn:pidag-DN-can}}{\cong} \scF \hatstar \For^{\Iw}_{\Iwu}(\nabla^\Iw_{\st(\mu)}) \overset{\eqref{eqn:conv-formula-2}}{\cong} \pi_\dag(\scF) \star_\Iw \nabla^\Iw_{\st(\mu)}.
\]
Our assumption,~\eqref{eqn:convolution-Wak} and the choice of $\mu$ imply that $\pi_\dag(\scF) \star_\Iw \nabla^\Iw_{\st(\mu)}$ is perverse and admits a filtration with subquotients of the form $\nabla^\Iw_{\st(\nu)}$ with $\nu \in X_*^+(T)$. The arguments in the proof of~\cite[Lemma~5.9(1)]{bezr} (see also~\cite[Proposition~9]{be}) show that this condition implies that $\scF \hatstar \nabla^{\wedge,\can}_{\st(\mu)}$ is perverse and admits a filtration with subquotients of the form $\nabla^{\wedge,\can}_{\st(\nu)}$ with $\nu \in X_*(T)$. Applying the functor $(-) \hatstar \Delta^{\wedge,\can}_{\st(-\mu)}$ we deduce that $\scF$ is perverse and admits a Wakimoto filtration, as desired.

The proof of the claim about multiplicities is immediate for these considerations.
\end{proof}

\subsection{Free-monodromic central sheaves}
\label{ss:fm-central-sheaves}

We now explain how to ``upgrade'' Gaitsgory's functor $\sfZ$ (see~\S\ref{ss:central-sheaves-prop}) to a functor with values in $\mathsf{D}^\wedge_{\Iwu,\Iwu}$, following the case of $\overline{\Q}_\ell$-coefficients treated in~\cite[\S 3.5]{be}. (The present construction is based on the slightly different point of view explained in~\S\ref{ss:central-sheaves-const}).

We start with the group scheme $\cH$ defined as the N\'eron blowup of $G \times \mathbb{A}^1_\F$ in $U$ along the divisor $\{0\} \subset \mathbb{A}^1_\F$, in the sense of~\cite{mrr}. Then $\mathcal{H}$ is a smooth group scheme over $\mathbb{A}^1_\F$, and the ind-scheme $\tFl_G$ represents the functor sending an $\F$-algebra $R$ to the set of isomorphism classes of pairs consisting of a principal $\cH$-bundle over $\mathrm{Spec}(R[ \hspace{-1pt} [z] \hspace{-1pt} ])$ together with a trivialization over $\mathrm{Spec}(R( \hspace{-1pt} (z) \hspace{-1pt} ))$. One can then define the ind-scheme $\widetilde{\Gr}^{\mathrm{BD}}_G$ by simply replacing $\cG$ by $\cH$ in the definition of $\Gr^{\mathrm{BD}}_G$. In this way we have canonical identifications
\[
\{0\} \times_{\mathbb{A}^1_\F} \widetilde{\Gr}_G^{\mathrm{BD}} = \tFl_G,
\]
and
\[
(\mathbb{A}^1_\F \smallsetminus \{0\}) \times_{\mathbb{A}^1_\F} \widetilde{\Gr}_G^{\mathrm{BD}} = \Gr_G \times \tFl_G \times (\mathbb{A}^1_\F \smallsetminus \{0\}).
\]
Using nearby cycles we therefore obtain a bifunctor
\[
\widetilde{\mathsf{Y}} : \mathsf{D}_{\Loop^+G, \Loop^+G} \times \mathsf{D}_{\Iwu,\Iwu} \to \Db_c(\tFl_G,\bk).
\]
We have a smooth morphism $\widetilde{\Gr}^{\mathrm{BD}}_G \to \Gr^{\mathrm{BD}}_G$ which restricts to the morphism $\pi : \tFl_G \to \Fl_G$ over $0$, and to the induced morphism $\Gr_G \times \tFl_G \times (\mathbb{A}^1_\F \smallsetminus \{0\}) \to \Gr_G \times \Fl_G \times (\mathbb{A}^1_\F \smallsetminus \{0\})$ over $\mathbb{A}^1_\F \smallsetminus \{0\}$ under the identifications above; by compatibility of nearby cycles with smooth pullback we deduce a canonical isomorphism
\begin{equation}
\label{eqn:C-pullback}
\widetilde{\mathsf{Y}}(\scA,\pi^* \For^{\Iw}_{\Iwu}(\scF)) \cong \pi^* \For^{\Iw}_{\Iwu}(\mathsf{Y}(\scA,\scF))
\end{equation}
for any $\scA$ in $\mathsf{D}_{\Loop^+G, \Loop^+G}$ and $\scF$ in $\mathsf{D}_{\Iw,\Iw}$. In particular, this shows that the functor $\widetilde{\mathsf{Y}}$ factors through a bifunctor
\[
\mathsf{D}_{\Loop^+G, \Loop^+G} \times \mathsf{D}_{\Iwu,\Iwu} \to \mathsf{D}_{\Iwu,\Iwu},
\]
which will also be denoted $\widetilde{\mathsf{Y}}$.

In the following statement, we use the variant of the bifunctor $\mathsf{Y}$ considered in Remark~\ref{rmk:C-noneq}.

\begin{lem}
\label{lem:nearby-cycles-pushforward}
For any $\scA$ in $\mathsf{D}_{\Loop^+G, \Loop^+G}$ and $\scF$ in $\mathsf{D}_{\Iwu,\Iwu}$, we have a canonical isomorphism
\[
\pi_! \widetilde{\mathsf{Y}}(\scA,\scF) \cong \mathsf{Y}(\scA,\pi_! \scF).
\]
\end{lem}

\begin{proof}
By general properties of nearby cycles (see~\cite[Exp.~XIII, \S 2.1.7]{deligne}), there exists a canonical morphism
\[
\pi_! \widetilde{\mathsf{Y}}(\scA,\scF) \to \mathsf{Y}(\scA,\pi_! \scF).
\]
Since $\mathsf{D}_{\Iwu,\Iwu}$ is generated, as a triangulated category, by the objects of the form $\pi^* \For^{\Iw}_{\Iwu}(\scG)$ with $\scG$ in $\mathsf{D}_{\Iw,\Iw}$, to prove that our morphism is an isomorphism it suffices to do so for such objects. Now by the projection formula we have
\[
\pi_! \pi^* \scG \cong \scG \otimes_\bk \pi_! \underline{\bk}_{\tFl_G},
\]
so that
\[
\mathsf{Y}(\scA,\pi_! \pi^*\scG) \cong \mathsf{Y}(\scA,\scG) \otimes_\bk \pi_! \underline{\bk}_{\tFl_G}.
\]
On the other hand, using~\eqref{eqn:C-pullback} and again the projection formula we see that
\[
\pi_! \widetilde{\mathsf{Y}}(\scA,\pi^*\scG) \cong \pi_! \pi^* \mathsf{Y}(\scA,\scG) \cong \mathsf{Y}(\scA,\scG) \otimes_\bk \pi_! \underline{\bk}_{\tFl_G}.
\]
One can check that under these identifications our morphism identifies with the identity morphism of $\mathsf{Y}(\scA,\scG) \otimes_\bk \pi_! \underline{\bk}_{\tFl_G}$; in particular it is indeed an isomorphism, which finishes the proof.
\end{proof}

For $\scA$ in $\mathsf{D}_{\Loop^+G, \Loop^+G}$, the functor $\widetilde{\mathsf{Y}}(\scA,-)$ extends to a functor from $\mathsf{D}^\wedge_{\Iwu,\Iwu}$ to pro-objects in $\mathsf{D}_{\Iwu,\Iwu}$. Using Lemma~\ref{lem:nearby-cycles-pushforward} one sees that this functor in fact takes values in $\mathsf{D}^\wedge_{\Iwu,\Iwu}$; we have therefore obtained a bifunctor
\[
\widetilde{\mathsf{Y}} : \mathsf{D}_{\Loop^+G, \Loop^+G} \times \mathsf{D}^\wedge_{\Iwu,\Iwu} \to \mathsf{D}^\wedge_{\Iwu,\Iwu}
\]
which satisfies
\begin{equation}
\label{eqn:nearby-cycles-pushforward}
\pi_\dag \widetilde{\mathsf{Y}}(\scA,\scF) \cong \mathsf{Y}(\scA, \pi_\dag \scF)
\end{equation}
for any $\scA$ in $\mathsf{D}_{\Loop^+G, \Loop^+G}$ and $\scF$ in $\mathsf{D}^\wedge_{\Iwu,\Iwu}$. 

The same considerations as in~\S\ref{ss:central-sheaves-const} show that for $\scA,\scB$ in $\mathsf{D}_{\Loop^+G, \Loop^+G}$ and $\scF,\scG$ in $\mathsf{D}_{\Iwu,\Iwu}$ we have a canonical isomorphism
\begin{equation*}
\widetilde{\mathsf{Y}}(\scA,\scF) \star_{\Iwu} \widetilde{\mathsf{Y}}(\scB,\scG) \cong \widetilde{\mathsf{Y}}(\scA \star_{\Loop^+G} \scB, \scF \star_{\Iwu} \scG).
\end{equation*}
(The proof of this isomorphism involves the compatibility of nearby cycles with respect to pushforward along a nonproper map; this compatibility does not follow from a general result, but can be obtained using considerations similar to those encountered in the proof of Lemma~\ref{lem:nearby-cycles-pushforward}.) 
Once this is proved, one obtains more generally that for $\scA,\scB$ in $\mathsf{D}_{\Loop^+G,\Loop^+G}$ and $\scF,\scG$ in $\mathsf{D}^\wedge_{\Iwu,\Iwu}$ we have a canonical isomorphism
\begin{equation}
\label{eqn:tC-convolution-comp}
\widetilde{\mathsf{Y}}(\scA,\scF) \hatstar \widetilde{\mathsf{Y}}(\scB,\scG) \cong \widetilde{\mathsf{Y}}(\scA \star_{\Loop^+G} \scB, \scF \hatstar \scG).
\end{equation}

We now define the functor
\[
\tsfZ : \mathsf{D}_{\Loop^+G, \Loop^+G} \to \mathsf{D}^\wedge_{\Iwu,\Iwu}
\]
by setting
\[
\tsfZ(\scA) := \widetilde{\mathsf{Y}}(\scA,\delta^\wedge).
\]
Since this morphism is defined by nearby cycles, it comes with a canonical monodromy automorphism $\widehat{\sm}$. These data possess properties similar to those of the ``traditional'' functor $\sfZ$ of~\S\ref{ss:central-sheaves-prop}, as explained in the following statement.

\begin{thm}
\phantomsection
\label{thm:gaitsgory-mon}
\begin{enumerate}
\item
\label{it:gaitsgory-mon-0}
There exists a canonical isomorphism of functors
\[
\pi_\dag \circ \tsfZ \cong \For^{\Iw}_{\Iwu} \circ \sfZ
\]
which identifies $\pi_\dag \widehat{\sm}$ with $\For^{\Iw}_{\Iwu}(\sm)$.
\item
\label{it:gaitsgory-mon-1}
The functor $\tsfZ$ restricts to an exact functor from $\sfP_{\Loop^+G,\Loop^+G}$ to $\mathsf{P}^\wedge_{\Iwu,\Iwu}$.
\item
\label{it:gaitsgory-mon-2}
There exist canonical isomorphisms
\[
\widehat{\phi}_{\scA,\scB} : \tsfZ(\scA \star_{\Loop^+G} \scB) \simto \tsfZ(\scA) \hatstar \tsfZ(\scB)
\]
for $\scA,\scB$ in $\mathsf{D}_{\Loop^+G, \Loop^+G}$ and $\tsfZ(\delta_{\Gr}) \cong \delta^\wedge$ which endow $\tsfZ$ with a monoidal structure. Moreover, via the identification
\begin{multline*}
 \pi_\dag(\tsfZ(\scA) \hatstar \tsfZ(\scB)) \overset{\eqref{eqn:conv-formula-1}}{\cong} \tsfZ(\scA) \hatstar \pi_\dag(\tsfZ(\scB)) \overset{\eqref{it:gaitsgory-mon-0}}{\cong} \tsfZ(\scA) \hatstar \For^{\Iw}_{\Iwu}(\sfZ(\scB))\\
 \overset{\eqref{eqn:conv-formula-2}}{\cong} \pi_\dag(\tsfZ(\scA)) \star_\Iw \sfZ(\scB) \overset{\eqref{it:gaitsgory-mon-0}}{\cong} \sfZ(\scA) \star_\Iw \sfZ(\scB)
\end{multline*}
we have $\pi_\dag(\widehat{\phi}_{\scA,\scB})=\phi_{\scA,\scB}$, and $\widehat{\phi}_{\scA,\scB}$ intertwines the automorphisms $\widehat{\sm}_{\scA \star_{\Loop^+G} \scB}$ of $\tsfZ(\scA \star_{\Loop^+G} \scB)$ and $\widehat{\sm}_{\scA} \hatstar \widehat{\sm}_{\scB}$ of $\tsfZ(\scA) \hatstar \tsfZ(\scB)$.
\item
\label{it:gaitsgory-mon-4}
For any $\scA$ in $\mathsf{D}_{\Loop^+G, \Loop^+G}$ and $\scF$ in $\D^\wedge_{\Iwu,\Iwu}$, there exists a canonical isomorphism
\[
\widehat{\sigma}_{\scA,\scF} : \tsfZ(\scA) \hatstar \scF \simto \scF \hatstar \tsfZ(\scA)
\]
which identifies $\widehat{\sm}_{\scA} \hatstar \id_{\scF}$ and $\id_{\scF} \hatstar \widehat{\sm}_{\scA}$. Moreover
the functor $\tsfZ$, together with the isomorphisms $\widehat{\sigma}_{\scA,\scF}$ and $\widehat{\phi}_{\scA,\scB}$, define a central functor from $\sfP_{\Loop^+G,\Loop^+G}$ to $\mathsf{D}^\wedge_{\Iwu,\Iwu}$ in the sense of~\cite{bez}; in other words these data define a braided monoidal functor from $\sfP_{\Loop^+G,\Loop^+G}$ to the Drinfeld center of $\mathsf{D}^\wedge_{\Iwu,\Iwu}$.
\item
\label{it:gaitsgory-mon-5}
For any $\scA$ in $\sfP_{\Loop^+G,\Loop^+G}$, the functor
\[
 \tsfZ(\scA) \hatstar (-) : \D^\wedge_{\Iwu,\Iwu} \to \D^\wedge_{\Iwu,\Iwu}
\]
is t-exact.
\item
\label{it:gaitsgory-mon-6}
For any $\scA$ in $\sfP_{\Loop^+G,\Loop^+G}$ we have
\[
 \hat{\sm}_{\scA}=\mu^{\mathrm{rot}}_{\tsfZ(\scA)}(x^{-1}).
\]
\item
\label{it:fm-central-Wak}
 For any $\scA$ in $\sfP_{\Loop^+G,\Loop^+G}$, the perverse sheaf $\tsfZ(\scA)$ admits a Wakimoto filtration. Moreover, for any $\lambda \in X_*(T)$ the multiplicity of $\Wak^\wedge_\lambda$ in $\gr^\wedge_\lambda(\tsfZ(\scA))$ equals the dimension of the $\lambda$-weight space of $\Sat(\scA)$.
\end{enumerate}
\end{thm}

\begin{proof}
\eqref{it:gaitsgory-mon-0}
Using~\eqref{eqn:nearby-cycles-pushforward}, for $\scA$ in $\mathsf{D}_{\Loop^+G, \Loop^+G}$ we obtain a canonical isomorphism
\[
\pi_\dag \tsfZ(\scA) = \pi_\dag \widetilde{\mathsf{Y}}(\scA,\delta^\wedge) \cong \mathsf{Y}(\scA,\pi_\dag \delta^\wedge) \cong \mathsf{Y}(\scA, \delta_{\Fl}).
\]
Now by~\eqref{eqn:isom-C-Z} the right-hand side identifies with $\sfZ(\scA)$, which provides the desired isomorphism. The compatibility with monodromy automorphisms follows from general properties of nearby cycles functors.

\eqref{it:gaitsgory-mon-1}
The isomorphism in~\eqref{it:gaitsgory-mon-0} and the exactness of $\sfZ$ show that $\pi_\dag \tsfZ(\scA)$ is perverse for any $\scA$ in $\sfP_{\Loop^+G,\Loop^+G}$. By~\cite[Lemma~5.3(1)]{bezr}, this implies that $\tsfZ(\scA)$ is perverse. Hence $\tsfZ$ sends $\sfP_{\Loop^+G,\Loop^+G}$ to $\mathsf{P}^\wedge_{\Iwu,\Iwu}$. Exactness of the restriction is automatic since this functor is obtained from a triangulated functor.

\eqref{it:gaitsgory-mon-2}
The isomorphism $\widehat{\phi}_{\scA,\scB}$ is obtained by applying~\eqref{eqn:tC-convolution-comp} in the case $\scF=\scG=\delta^\wedge$ and using the canonical isomorphism $\delta^\wedge \hatstar \delta^\wedge \cong \delta^\wedge$. The compatibility with $\phi_{\scA,\scB}$ follows from the comments in~\S\ref{ss:central-sheaves-const}. The proof of the compatibility with monodromy is similar to that in the case of $\sfZ$.

To justify that $\sfZ(\delta_{\Gr})=\delta^\wedge$, we remark that more generally we have
\begin{equation}
\label{eqn:tC-id}
\widetilde{\mathsf{Y}}(\delta_{\Gr},\scF) \cong \scF
\end{equation}
for any $\scF$ in $\mathsf{D}^\wedge_{\Iwu,\Iwu}$. In fact this follows from considerations analogous to those for the isomorphism $\mathsf{Y}(\delta_{\Gr},-) \cong \id$ in~\S\ref{ss:central-sheaves-const}, using the natural closed embedding $\tFl_G \times \mathbb{A}^1_\F \hookrightarrow \widetilde{\Gr}_G^{\mathrm{BD}}$.


The fact that these data define a monoidal structure on $\tsfZ$ is easy, and left to the reader.

\eqref{it:gaitsgory-mon-4}
As in the case of the functor $\sfZ$ in~\S\ref{ss:central-sheaves-const}, the isomorphism $\widehat{\sigma}_{\scA,\scF}$ is constructed from~\eqref{eqn:tC-convolution-comp} and the isomorphism~\eqref{eqn:tC-id}. Namely, these isomorphisms imply that we have
\[
\tsfZ(\scA) \hatstar \scF \cong \widetilde{\mathsf{Y}}(\scA,\delta^\wedge) \hatstar \widetilde{\mathsf{Y}}(\delta_{\Gr},\scF) \cong \widetilde{\mathsf{Y}}(\scA \star_{\Loop^+G} \delta_{\Gr},\delta^\wedge \hatstar \scF) \cong \widetilde{\mathsf{Y}}(\scA,\scF)
\]
on the one hand, and that
\[
\scF \hatstar \tsfZ(\scA) \cong \widetilde{\mathsf{Y}}(\delta_{\Gr},\scF) \hatstar \widetilde{\mathsf{Y}}(\scA,\delta^\wedge) \cong \widetilde{\mathsf{Y}}(\delta_{\Gr} \star_{\Loop^+G} \scA,\scF \hatstar \delta^\wedge) \cong \widetilde{\mathsf{Y}}(\scA,\scF)
\]
on the other hand. The proof that these data define a central functor can be copied from the case of $\sfZ$ (see~\cite{gaitsgory-app} or~\cite{ar-book}).

\eqref{it:gaitsgory-mon-5}
First we prove that our functor is right exact. For that, since the nonpositive part of the perverse t-structure is generated under extensions by the objects $\Delta^\wedge_w$ ($w \in W$), it suffices to prove that for any such $w$ the object $\tsfZ(\scA) \hatstar \Delta^\wedge_w$ has no perverse cohomology in positive degrees. In fact we will prove that this object is perverse. Indeed, using~\eqref{eqn:conv-formula-1}--\eqref{eqn:conv-formula-2} and~\eqref{it:gaitsgory-mon-0} we obtain that
\[
 \pi_\dag(\tsfZ(\scA) \hatstar \Delta^\wedge_w) \cong \tsfZ(\scA) \hatstar \For^{\Iw}_{\Iwu}(\Delta^{\Iw}_w) \cong \For^{\Iw}_{\Iwu} \bigl( \sfZ(\scA) \star_{\Iw} \Delta^{\Iw}_w \bigr).
\]
Now $\sfZ(\scA) \star_{\Iw} \Delta^{\Iw}_w$ is perverse by ``convolution exactness'' of usual central sheaves (see~\cite[Theorem~4.2(2)]{brr}). Using~\cite[Lemma~5.3(1)]{bezr} we deduce that $\tsfZ(\scA) \hatstar \Delta^\wedge_w$ is perverse, as desired.

To prove left exactness of our functor, we consider the rigid dual $\scA^\vee$ of $\scA$ in the rigid tensor category $\Perv_{\Loop^+G}(\Gr_G,\bk)$. By monoidality of $\tsfZ$ (see~\eqref{it:gaitsgory-mon-2}), the functor $\tsfZ(\scA^{\vee}) \hatstar (-)$ is left adjoint to $\tsfZ(\scA) \hatstar (-)$. Since the former functor is right exact, we deduce that the latter is left exact, which finishes the proof.


\eqref{it:gaitsgory-mon-6}
The proof is similar to that of the corresponding claim for $\sfZ$, see~\eqref{eqn:mon-Z-rot}.

\eqref{it:fm-central-Wak}
The claim follows from Lemma~\ref{lem:criterion-Wak}, in view of~\eqref{it:gaitsgory-mon-0} and the property that $\sfZ(\scA)$ admits a Wakimoto filtration, see~\S\ref{ss:central-sheaves-prop}.
\end{proof}


For simplicity of notation, below we will set
\[
\tsZ := \tsfZ \circ \Sat^{-1} : \Rep(G^\vee_\bk) \to \mathsf{D}^\wedge_{\Iwu,\Iwu},
\]
and write
\[
 \hat{\sm}_V:=\hat{\sm}_{\Sat^{-1}(V)} \in \End(\tsZ(V)).
\]
Theorem~\ref{thm:gaitsgory-mon}\eqref{it:fm-central-Wak} and Corollary~\ref{cor:Wak-filtr-monodromy} imply that for any $V$ in $\Rep(G^\vee_\bk)$ the morphism $\mu_{\tsZ(V)}$ factors through a canonical morphism
\begin{equation}
\label{eqn:muV}
 \mu_V : \scO(T^\vee_\bk) \to \End(\tsZ(V)).
\end{equation}

Let us note for later use that
from Theorem~\ref{thm:gaitsgory-mon}\eqref{it:gaitsgory-mon-4}--\eqref{it:gaitsgory-mon-5} we also deduce, for any $V \in \Rep(G^\vee_\bk)$ and $\scF$ in $\sfD^0_{\Iwu,\Iwu}$, a canonical isomorphism
\begin{equation}
\label{eqn:commutation-Z-0}
 \tsZ(V) \hatstar^0 \scF \simto \scF \hatstar^0 \tsZ(V),
\end{equation}
and that these objects belong to the heart of the perverse t-structure if $\scF$ does.

\subsection{Some tilting perverse sheaves}

In this subsection we assume that the conditions in~\S\ref{ss:regular-quotient-coh} are satisfied.
Under this assumption,
the free-monodromic central sheaves, together with the object $\Xi^\wedge_!$ introduced in~\S\ref{ss:G/B-tilting}, allow to describe a family of tilting objects in $\sfP^\wedge_{\Iwu,\Iwu}$, as follows.

\begin{prop}
\label{prop:tilting-Z}
For any $V \in \Rep(G^\vee_\bk)$ which is tilting, the perverse sheaf
\[
\Xi_!^\wedge \hatstar \tsZ(V)
\]
is tilting.
\end{prop}

\begin{proof}
In view of~\cite[Lemma~5.9]{bezr}, to prove the claim it suffices to prove that the object
$\pi_\dag ( \Xi_!^\wedge \hatstar \tsZ(V) )$ of $\D_{\Iwu,\Iw}$
is a tilting perverse sheaf. Now, using Theorem~\ref{thm:gaitsgory-mon}\eqref{it:gaitsgory-mon-0} we have
\begin{multline*}
\pi_\dag \bigl( \Xi_!^\wedge \hatstar \tsZ(V) \bigr) \overset{\eqref{eqn:conv-formula-1}}{\cong} \Xi_!^\wedge \hatstar \pi_\dag \bigl( \tsZ(V) \bigr) \cong \Xi_!^\wedge \hatstar \For^{\Iw}_{\Iwu}(\sZ(V)) \\
\overset{\eqref{eqn:conv-formula-2}}{\cong} \pi_\dag(\Xi^\wedge_!) \star_\Iw \sZ(V) \cong \Xi_! \star_{\Iw} \sZ(V).
\end{multline*}

Let $\Iwu^+$ be the inverse image of $U^+$ under the evaluation morphism $\Loop^+G \to G$, and consider the composition
\[
\chi_{\Iw} : \Iwu^+ \to U^+ \xrightarrow{\chi} \Ga,
\]
where the first morphism is the obvious projection and $\chi$ is as in~\S\ref{ss:G/B-tilting}. Let us denote by $\sfD_{\mathcal{IW},\Iw}$ the $(\Iwu^+,\chi_\Iw^*(\LAS))$-equivariant derived category of $\bk$-sheaves on $\Fl_G$. 
This category has a natural perverse t-structure, whose heart $\sfP_{\mathcal{IW},\Iw}$ has a canonical structure of highest weight category.

As in the case of $G/U$ in~\S\ref{ss:G/B-tilting},
standard constructions provide t-exact ``averaging'' functors
\[
\Av_{\mathcal{IW}} : \sfD_{\Iwu,\Iw} \to \sfD_{\mathcal{IW},\Iw}, \quad \Av_{\Iwu,!} : \sfD_{\mathcal{IW},\Iw} \to \sfD_{\Iwu,\Iw}, \quad \quad \Av_{\Iwu,*} : \sfD_{\mathcal{IW},\Iw} \to \sfD_{\Iwu,\Iw}
\]
such that $\Av_{\Iwu,!}$ is left adjoint to $\Av_{\mathcal{IW}}$ and $\Av_{\Iwu,*}$ is right adjoint to $\Av_{\mathcal{IW}}$. Standard considerations (see e.g.~\cite[Lemma~3.6]{projGr}) show
that $\Av_{\Iwu,!}$, resp.~$\Av_{\Iwu,*}$, sends objects admitting a standard, resp.~costandard, filtration to objects admitting a standard, resp.~costandard, filtration.

It follows from the definition of $\Xi_!$ that we have a canonical isomorphism of functors
\[
\Xi_! \star_{\Iw} (-) \cong \Av_{\Iwu,!} \circ \Av_{\mathcal{IW}} \circ \For^{\Iw}_{\Iwu};
\]
in particular we deduce that
\[
\pi_\dag \bigl( \Xi_!^\wedge \hatstar \tsZ(V) \bigr) \cong \Av_{\Iwu,!} \circ \Av_{\mathcal{IW}} \circ \For^{\Iw}_{\Iwu}(\sZ(V)).
\]
Now by~\cite[Theorem~8.1]{brr} the perverse sheaf $\Av_{\mathcal{IW}} \circ \For^{\Iw}_{\Iwu}(\sZ(V))$ is tilting, from which we deduce that $\Av_{\Iwu,!} \circ \Av_{\mathcal{IW}} \circ \For^{\Iw}_{\Iwu}(\sZ(V))$ is a perverse sheaf admitting a standard filtration.

Using the isomorphism $\Xi_! \cong \Xi_*$ we similarly obtain that
\[
\pi_\dag \bigl( \Xi_!^\wedge \hatstar \tsZ(V) \bigr) \cong \Av_{\Iwu,*} \circ \Av_{\mathcal{IW}} \circ \For^{\Iw}_{\Iwu}(\sZ(V)),
\]
which implies that $\pi_\dag \bigl( \Xi_!^\wedge \hatstar \tsZ(V) \bigr)$ admits a costandard filtration and finishes the proof.
\end{proof}

\subsection{Quantum trace of monodromy}

Recall (see e.g.~\cite[\S 2.10]{egno}) that if $(\mathsf{A},\odot)$ is a monoidal category with unit object $\mathbf{1}$, and if $X$ is an object of $\mathsf{A}$, a \emph{left dual} 
of $X$ is the data of an object $X^\vee$ together with morphisms
\[
 \ev_X : X^\vee \odot X \to \mathbf{1}, \quad \coev_X : \mathbf{1} \to X \odot X^\vee
\]
such that the compositions
\begin{gather*}
 X \xrightarrow{\coev_X \odot \id} X \odot X^\vee \odot X \xrightarrow{\id \odot \ev_X} X, \\
 X^\vee \xrightarrow{\id \odot \coev_X} X^\vee \odot X \odot X^\vee \xrightarrow{\ev_X \odot \id} X^\vee
\end{gather*}
are the identity morphisms of $X$ and $X^\vee$, respectively. (Here, we omit the unit and associativity isomorphisms.) Similarly, a \emph{right dual} of $X$ is the data of an object ${}^\vee \hspace{-1pt} X$ together with morphisms
\[
 \ev'_X : X \odot {}^\vee \hspace{-1pt} X \to \mathbf{1}, \quad \coev'_X : \mathbf{1} \to {}^\vee \hspace{-1pt} X \odot X
\]
such that the compositions
\begin{gather*}
 {}^\vee \hspace{-1pt} X \xrightarrow{\coev'_X \odot \id} {}^\vee \hspace{-1pt} X \odot X \odot {}^\vee \hspace{-1pt} X \xrightarrow{\id \odot \ev'_X} {}^\vee \hspace{-1pt} X, \\
 X \xrightarrow{\id \odot \coev'_X} X \odot {}^\vee \hspace{-1pt} X \odot X \xrightarrow{\ev'_X \odot \id} X
\end{gather*}
are the identity morphisms of ${}^\vee \hspace{-1pt} X$ and $X$, respectively. The object $X$ is called \emph{left dualizable}, resp.~\emph{right dualizable}, if a left dual, resp.~right dual, exists; in this case such a dual is unique up to unique isomorphism, see~\cite[Proposition~2.10.5]{egno}. This notion is functorial is the following sense: if $X,Y$ are left, resp.~right, dualizable, then there exists a canonical isomorphism
\[
 \Hom_{\mathsf{A}}(X,Y) \simto \Hom_{\mathsf{A}}(Y^\vee,X^\vee), \quad \text{resp.} \quad \Hom_{\mathsf{A}}(X,Y) \simto \Hom_{\mathsf{A}}({}^\vee \hspace{-1pt} Y,{}^\vee \hspace{-1pt}X),
\]
denoted $f \mapsto f^\vee$, resp.~$f \mapsto {}^\vee \hspace{-1pt} f$. Here, given $f : X \to Y$, the morphism $f^\vee$ is the composition
\[
 Y^\vee \xrightarrow{\id \odot \coev_X} Y^\vee \odot X \odot X^\vee \xrightarrow{\id \odot f \odot \id} Y^\vee \odot Y \odot X^\vee \xrightarrow{\ev_Y \odot \id} X^\vee,
\]
and the morphism ${}^\vee \hspace{-1pt} f$ is the composition
\[
 {}^\vee \hspace{-1pt} Y \xrightarrow{\coev'_X \odot \id} {}^\vee \hspace{-1pt} X \odot X \odot {}^\vee \hspace{-1pt} Y \xrightarrow{\id \odot f \odot \id} {}^\vee \hspace{-1pt} X \odot Y \odot {}^\vee \hspace{-1pt} Y \xrightarrow{\id \odot \ev'_Y} {}^\vee \hspace{-1pt}X.
\]
Below we will also use the fact that if $X$ is left dualizable, resp.~right dualizable, then the functor $X^\vee \odot (-)$ is left adjoint to $X \odot (-)$ and the functor $(-) \odot X$ is left adjoint to $(-) \odot X^\vee$, resp.~the functor $X \odot (-)$ is left adjoint to ${}^\vee \hspace{-1pt} X \odot (-)$ and the functor $(-) \odot {}^\vee \hspace{-1pt} X$ is left adjoint to $(-) \odot X$; see~\cite[Proposition~2.10.8]{egno}.

The application of this notion that will be relevant for us is to the definition of \emph{quantum traces}, see~\cite[\S 4.7]{egno}. Namely, consider an object $X$ which is left dualizable, and assume that $X^\vee$ is itself left dualizable (with left dual denoted $X^{\vee\vee}$). Then for any $a \in \Hom_{\mathsf{A}}(X,X^{\vee\vee})$ the \emph{left quantum trace} $\trL(a)$ of $a$ is defined as the endomorphism of $\mathbf{1}$ obtained as the composition
\[
 \mathbf{1} \xrightarrow{\coev_X} X \odot X^\vee \xrightarrow{a \odot \id} X^{\vee\vee} \odot X^\vee \xrightarrow{\ev_{X^\vee}} \mathbf{1}.
\]
A similar definition leads to the notion of the \emph{right quantum trace} of a morphism $a : X \to {}^{\vee\vee} \hspace{-1pt} X$, in case $X$ and ${}^\vee \hspace{-1pt} X$ are right dualizable.

In our present setting, since a monoidal functor sends dualizable objects to dualizable objects, and their duals to the corresponding duals (see~\cite[Exercise~2.10.6]{egno}), and since every object $V$ in $\Rep(G^\vee_\bk)$ is left and right dualizable with left and right duals $V^*$ (together with the obvious evaluation and coevaluation maps), for any $V$ the object $\sZ(V)$, resp.~$\tsZ(V)$, is left and right dualizable in $\mathsf{D}_{\Iw,\Iw}$, resp.~$\mathsf{D}^\wedge_{\Iwu,\Iwu}$, with left and right dual $\sZ(V^*)$, resp.~$\tsZ(V^*)$. Hence the (left) quantum trace $\trL(a)$ is defined for any $a \in \End_{\mathsf{D}_{\Iw,\Iw}}(\sZ(V))$, resp.~$a \in \End_{\mathsf{D}^\wedge_{\Iwu,\Iwu}}(\tsZ(V))$. The case of $\mathsf{D}_{\Iw,\Iw}$ is not very rich, since the endomorphisms of $\delta_{\Fl}$ are $\bk$. But in $\mathsf{D}^\wedge_{\Iwu,\Iwu}$ we have $\End(\delta^\wedge)=\scO(\FN_{T^\vee_\bk}(\{e\}))$ (see Lemma~\ref{lem:monodromy-DN}); the left quantum trace of a morphism is therefore an element in $\scO(\FN_{T^\vee_\bk}(\{e\}))$.

The following lemma will play a technical role in the construction of a functor in Section~\ref{sec:construction}. Its proof will occupy the rest of the section. (No detail of this proof will be used in later sections, so that these subsections can be safely skipped.)

\begin{lem}
\label{lem:trace}
 For any $V$ in $\Rep(G^\vee_\bk)$ we have
 \[
  \trL(\hat{\sm}_V) = \sum_{\mu \in X_*(T)} \dim(V_\mu) \cdot e^\mu
 \]
 where $V_\mu$ is the $\mu$-weight space of $V$.
 In other words, $\trL(\hat{\sm}_V)$ is the image of the character of $V$ (seen as a function on $T^\vee_\bk)$ in $\scO(\FN_{T^\vee_\bk}(\{e\}))$.
\end{lem}

\subsection{Description of duals}

The proof of Lemma~\ref{lem:trace} will use the free-monodromic Wakimoto filtration on $\tsZ(V)$ (see Theorem~\ref{thm:gaitsgory-mon}\eqref{it:fm-central-Wak}). For this we will need to show that each subquotient in this filtration is dualizable, and describe its dual.

For the next lemma, we will have to assume that the order $\leq$ on $X_*(T)$ chosen in~\S\ref{ss:central-sheaves-prop} satisfies the following property:
\[
 \text{for $\lambda,\mu \in X_*(T)$, $\lambda \leq \mu$ if and only if $-\mu \leq -\lambda$.}
\]
(Of course, there exists an order with this property.)

\begin{lem}
\label{lem:dualizability}
 Consider some $V$ in $\Rep(G^\vee_\bk)$, and let $\lambda,\mu \in X_*(T)$ be such that $\lambda \leq \mu$. 
 \begin{enumerate}
 \item The object $\sZ(V)_{\leq \mu} / \sZ(V)_{\leq \lambda}$ is left and right dualizable, with left and right dual $\sZ(V^*)_{< -\lambda} / \sZ(V^*)_{< -\mu}$.
  \item The object $\tsZ(V)_{\leq \mu} / \tsZ(V)_{\leq \lambda}$ is left and right dualizable, with left and right dual $\tsZ(V^*)_{< -\lambda} / \tsZ(V^*)_{< -\mu}$.
 \end{enumerate}
\end{lem}

\begin{rmk}
\label{rmk:dualizability}
 It can be checked that in a monoidal category $(\mathsf{A},\odot)$ where $\mathsf{A}$ is abelian and $\odot$ is exact, the kernel (resp.~cokernel) of a morphism between dualizable objects is dualizable, with dual the cokernel (resp.~kernel) of the dual morphism. This statement does not apply here, since we do not have an obvious abelian subcategory of $\mathsf{D}^\wedge_{\Iwu,\Iwu}$ or $\mathsf{D}_{\Iw,\Iw}$ containing the central sheaves and stable under the convolution product; however the proof below repeats arguments close to those required to prove this property.
\end{rmk}

\begin{proof}[Proof of Lemma~\ref{lem:dualizability}]
 We will treat the two cases in parallel, and work with left duals; the proof for right duals is similar. First, let us assume that $\lambda$ satisfies $\lambda < \nu$ for any $\nu \in X_*(T)$ such that $V_\nu \neq 0$, so that $\sZ(V)_{\leq \lambda}=0$ and $\sZ(V^*)_{<-\lambda}=\sZ(V^*)$ (and similarly for $\tsZ(V)$ and $\tsZ(V^*)$). In this case we will proceed by downward induction on $\mu$, and prove (in addition to the fact that $\sZ(V)_{\leq \mu}$ and $\tsZ(V)_{\leq \mu}$ are dualizable with the duals given in the statement) that the dual of the embedding $\sZ(V)_{\leq \mu} \hookrightarrow \sZ(V)$, resp.~$\tsZ(V)_{\leq \mu} \hookrightarrow \tsZ(V)$, is the projection $\sZ(V^*) \twoheadrightarrow \sZ(V^*)/ \sZ(V^*)_{< -\mu}$, resp.~$\tsZ(V^*) \twoheadrightarrow \tsZ(V^*)/ \tsZ(V^*)_{< -\mu}$.
 
 If $\mu \geq \nu$ for any $\nu \in X_*(T)$ such that $V_\nu \neq 0$ we have $\sZ(V)_{\leq \mu} = \sZ(V)$ and $\sZ(V^*)_{< -\mu}=0$ (and similarly for $\tsZ(V)$ and $\tsZ(V^*)$); in this case the claim has already been justified above Lemma~\ref{lem:trace}. Now we fix $\mu \in X_*(T)$, and assume the claim is known for the successor $\mu'$ of $\mu$, i.e.~that $\sZ(V)_{\leq \mu'}$ and $\tsZ(V)_{\leq \mu'}$ are left dualizable, with left duals
 \[
  \sZ(V^*) / \sZ(V^*)_{< -\mu'} \quad \text{and} \quad \tsZ(V^*) / \tsZ(V^*)_{< -\mu'}
 \]
respectively, and that the dual of the embedding in $\sZ(V)$, resp.~$\tsZ(V)$, is the projection from $\sZ(V^*)$, resp.~$\tsZ(V^*)$. We now consider the exact sequence
 \begin{equation}
 \label{eqn:Wak-exact-sequence}
  \sZ(V)_{\leq \mu} \hookrightarrow \sZ(V)_{\leq \mu'} \twoheadrightarrow \gr_{\mu'}(\sZ(V)).
  \end{equation}
Here the right-hand side is isomorphic to $\Wak_{\mu'}^{\oplus r}$ for some $r \geq 0$; we fix an isomorphism $\gr_{\mu'}(\sZ(V)) \cong \Wak_{\mu'}^{\oplus r}$ and therefore identify the second morphism in~\eqref{eqn:Wak-exact-sequence} with a surjection $f : \sZ(V)_{\leq \mu'} \twoheadrightarrow \Wak_{\mu'}^{\oplus r}$. By assumption $\sZ(V)_{\leq \mu'}$ is left dualizable, and since $\Wak_{\mu'}$ is invertible it is also left dualizable (with left dual $\Wak_{-\mu'}$); hence so is $\Wak_{\mu'}^{\oplus r}$. We can therefore consider the dual morphism
\[
 f^\vee : (\Wak_{\mu'}^{\oplus r})^\vee \to (\sZ(V)_{\leq \mu'})^\vee,
\]
which we interpret as a morphism from $\Wak_{-\mu'}^{\oplus r}$ to $\sZ(V^*) / \sZ(V^*)_{< -\mu'}$. By functoriality of Wakimoto filtrations this morphism factors through a morphism
\begin{equation}
\label{eqn:dual-gr}
 \tilde{f^\vee} : \Wak_{-\mu'}^{\oplus r} \to \gr_{-\mu'}(\sZ(V^*)).
\end{equation}

We claim that $\tilde{f^\vee}$ is an isomorphism. In fact we have $\gr_{-\mu'}(\sZ(V^*)) \cong \Wak_{-\mu'}^{\oplus r}$ since $\dim((V^*)_{-\mu'})=\dim(V_{\mu'})$. Since $\End(\Wak_{-\mu'}) \cong \bk$ (see~\cite[\S 4.5]{brr}), the morphism $\tilde{f^\vee}$ can be represented by an $r \times r$-matrix, and saying that it is invertible is equivalent to this matrix being invertible. If this were not the case, then there would exist an embedding $\Wak_{-\mu'} \hookrightarrow \Wak_{-\mu'}^{\oplus r}$ as a direct summand such that the composition
\[
 \Wak_{-\mu'} \to \Wak_{-\mu'}^{\oplus r} \xrightarrow{f^\vee} \sZ(V^*) / \sZ(V^*)_{< -\mu'}
\]
vanishes. However, by adjunction we have
\begin{align*}
 \Hom(\Wak_{-\mu'},\Wak_{-\mu'}^{\oplus r}) &\cong \Hom(\Wak_{-\mu'} \star_{\Iw} \Wak_{\mu'}^{\oplus r}, \delta_{\Fl}), \\
 \Hom(\Wak_{-\mu'}, \sZ(V^*) / \sZ(V^*)_{< -\mu'}) &\cong \Hom(\Wak_{-\mu'} \star_{\Iw} \sZ(V)_{\leq \mu'}, \delta_{\Fl}),
\end{align*}
and through this identification the morphism $f^\vee \circ (-)$ corresponds to the morphism
\[
 (-) \circ (\id \star f) : \Hom(\Wak_{-\mu'} \star_{\Iw} \Wak_{\mu'}^{\oplus r}, \delta_{\Fl}) \to \Hom(\Wak_{-\mu'} \star_{\Iw} \sZ(V)_{\leq \mu'}, \delta_{\Fl}),
\]
which is injective since $\id \star f$ is surjective. This provides a contradiction, proving therefore that $\tilde{f^\vee}$ indeed is an isomorphism.

We have now obtained an isomorphism
\[
 \bigl( \gr_{\mu'}(\sZ(V)) \bigr)^\vee \simto (\Wak_{\mu'}^{\oplus r})^\vee \simto \Wak_{-\mu'}^{\oplus r} \simto \gr_{-\mu'}(\sZ(V^*)),
\]
which is easily seen not to depend on our initial choice of isomorphism $\gr_{\mu'}(\sZ(V)) \cong \Wak_{\mu'}^{\oplus r}$; it is therefore canonical. Moreover, through this identification the dual of the projection $\sZ(V)_{\leq \mu'} \twoheadrightarrow \gr_{\mu'}(\sZ(V))$ is the embedding $\gr_{-\mu'}(\sZ(V^*)) \hookrightarrow \sZ(V^*) / \sZ(V^*)_{< -\mu'}$.

Let us consider the composition
\[
 \delta_{\Fl} \xrightarrow{\coev} \sZ(V)_{\leq \mu'} \star_\Iw (\sZ(V^*) / \sZ(V^*)_{< -\mu'}) \twoheadrightarrow \sZ(V)_{\leq \mu'} \star_\Iw (\sZ(V^*) / \sZ(V^*)_{< -\mu}).
\]
The preceding considerations show that its composition with the surjection
\[
 \sZ(V)_{\leq \mu'} \star_\Iw (\sZ(V^*) / \sZ(V^*)_{< -\mu}) \twoheadrightarrow \gr_{\mu'}(\sZ(V)) \star_\Iw (\sZ(V^*) / \sZ(V^*)_{< -\mu})
\]
vanishes; this morphism therefore factors through a morphism
\begin{equation}
 \label{eqn:coev-induction}
  \delta_{\Fl} \xrightarrow{\coev} \sZ(V)_{\leq \mu} \star_\Iw (\sZ(V^*) / \sZ(V^*)_{< -\mu}).
\end{equation}
Similarly, the composition
\[
  (\sZ(V^*) / \sZ(V^*)_{< -\mu'}) \star_\Iw \sZ(V)_{\leq \mu} \hookrightarrow  (\sZ(V^*) / \sZ(V^*)_{< -\mu'}) \star_\Iw \sZ(V)_{\leq \mu'} \xrightarrow{\ev} \delta_{\Fl}
\]
factors through a morphism
\begin{equation}
 \label{eqn:ev-induction}
  (\sZ(V^*) / \sZ(V^*)_{< -\mu}) \star_\Iw \sZ(V)_{\leq \mu} \to \delta_{\Fl}.
\end{equation}
It is then not difficult to check that~\eqref{eqn:coev-induction} and~\eqref{eqn:ev-induction} exhibit $\sZ(V^*) / \sZ(V^*)_{< -\mu}$ as the left dual of $\sZ(V)_{\leq \mu}$. Moreover from the construction of the evaluation and coevaluation morphisms one sees that the dual of the embedding $\sZ(V)_{\leq \mu} \hookrightarrow \sZ(V)_{\leq \mu'}$ is the projection $\sZ(V^*) / \sZ(V^*)_{< -\mu'} \twoheadrightarrow \sZ(V^*) / \sZ(V^*)_{< -\mu}$; by compability of duality with composition and the induction hypothesis, it follows that the dual of the embedding $\sZ(V)_{\leq \mu} \hookrightarrow \sZ(V)$ is the natural projection $\sZ(V^*) \twoheadrightarrow \sZ(V^*) / \sZ(V^*)_{< -\mu}$.

Next, we consider the free-monodromic setting, and more specifically the exact sequence
\[
 \tsZ(V)_{\leq \mu} \hookrightarrow \tsZ(V)_{\leq \mu'} \twoheadrightarrow \gr^\wedge_{\mu'}(\tsZ(V)).
\]
Here again, by induction the middle term is left dualizable, and the right-hand side is dualizable because it is isomorphic to a direct sum of invertible objects. If we fix an isomorphism $\gr^\wedge_{\mu'}(\tsZ(V)) \cong (\Wak_{\mu'}^\wedge)^{\oplus r}$, the dual of the surjection $f : \tsZ(V)_{\leq \mu'} \twoheadrightarrow (\Wak_{\mu'}^\wedge)^{\oplus r}$ is a morphism
\[
 f^\vee : (\Wak_{-\mu'}^\wedge)^{\oplus r} \to \tsZ(V^*) / \tsZ(V^*)_{< -\mu'},
\]
which has to factor through a morphism
\[
 (\Wak_{-\mu'}^\wedge)^{\oplus r} \to \gr^\wedge_{-\mu'}(\tsZ(V^*)).
\]
It is clear that the image of this morphism under $\pi_\dag$ is the isomorphism
\[
 \Wak_{-\mu'}^{\oplus r} \to \gr_{-\mu'}(\sZ(V^*))
\]
considered in~\eqref{eqn:dual-gr}; since the functor $\pi_\dag$ is conservative (see~\S\ref{ss:completed-category}) this implies that our morphism is also invertible, and as in the $\Iw$-equivariant setting we deduce a canonical isomorphism
\[
 \bigl( \gr^\wedge_{\mu'}(\tsZ(V)) \bigr)^\vee \cong \gr_{-\mu'}(\tsZ(V^*)).
\]
Once this is established, the same arguments as above allow to prove that $\tsZ(V)_{\leq \mu}$ is left dualizable, with left dual $\tsZ(V^*) / \tsZ(V^*)_{< -\mu}$, and that the dual of the embedding $\tsZ(V)_{\leq \mu} \hookrightarrow \tsZ(V)$ is the surjection $\tsZ(V^*) \twoheadrightarrow \tsZ(V^*) / \tsZ(V^*)_{< -\mu}$.

Now we fix $\mu$, and prove by upward induction on $\lambda$ that the object 
\[
\sZ(V)_{\leq \mu} / \sZ(V)_{\leq \lambda}, \quad \text{resp.} \quad \tsZ(V)_{\leq \mu} / \tsZ(V)_{\leq \lambda},
\]
is left dualizable, with left dual
\[
\sZ(V^*)_{< -\lambda} / \sZ(V^*)_{< -\mu}, \quad \text{resp.} \quad \tsZ(V^*)_{< -\lambda} / \tsZ(V^*)_{< -\mu}.
\]
The two cases are similar, so we only treat the second one.
We consider the exact sequence
\[
\tsZ(V)_{\leq \lambda} \hookrightarrow \tsZ(V)_{\leq \mu} \twoheadrightarrow \tsZ(V)_{\leq \mu} / \tsZ(V)_{\leq \lambda}.
\]
We now know that the first two terms here are left dualizable; moreover the dual of the composition of the first map with the embedding $\tsZ(V)_{\leq \mu} \hookrightarrow \tsZ(V)$ is the surjection $\tsZ(V^*) \twoheadrightarrow \sZ(V^*) / \sZ(V^*)_{< -\lambda}$, with the dual of the latter map being the surjection $\tsZ(V^*) \twoheadrightarrow \sZ(V^*) / \sZ(V^*)_{< -\mu}$; by compatibility of duality with composition this implies that the dual of this first map is the surjection $\sZ(V^*) / \sZ(V^*)_{< -\mu} \twoheadrightarrow \sZ(V^*) / \sZ(V^*)_{< -\lambda}$. From this claim we deduce that the composition
\begin{multline*}
\delta^\wedge \xrightarrow{\coev} \tsZ(V)_{\leq \mu} \hatstar \left( \tsZ(V^*) / \tsZ(V^*)_{< -\mu} \right) \\
\to \left( \tsZ(V)_{\leq \mu} / \tsZ(V)_{\leq \lambda} \right) \hatstar \left( \tsZ(V^*) / \tsZ(V^*)_{< -\mu} \right) \\
\twoheadrightarrow \left( \tsZ(V)_{\leq \mu} / \tsZ(V)_{\leq \lambda} \right) \hatstar \left( \tsZ(V^*) / \tsZ(V^*)_{< -\lambda} \right)
\end{multline*}
vanishes; it follows that the composition of the first two maps factors through a morphism
\[
\delta^\wedge \to \left( \tsZ(V)_{\leq \mu} / \tsZ(V)_{\leq \lambda} \right) \hatstar \left( \tsZ(V^*)_{< -\lambda} / \tsZ(V^*)_{< -\mu} \right).
\]
Similarly, from the evaluation morphism
\[
 \left( \tsZ(V^*) / \tsZ(V^*)_{< -\mu} \right) \hatstar \tsZ(V)_{\leq \mu} \to \delta^\wedge
\]
we obtain a morphism
\[
\left( \tsZ(V^*)_{< -\lambda} / \tsZ(V^*)_{< -\mu} \right) \hatstar \left( \tsZ(V)_{\leq \mu} / \tsZ(V)_{\leq \lambda} \right) \to \delta^\wedge.
\]
It is easily seen that, taken together, these maps exhibit $\tsZ(V^*)_{< -\lambda} / \tsZ(V^*)_{< -\mu}$ as the left dual of $\tsZ(V)_{\leq \mu} / \tsZ(V)_{\leq \lambda}$, which finishes the proof.
%
\end{proof}

\begin{rmk}
From the proof of Lemma~\ref{lem:dualizability} we see that if $\lambda \leq \lambda' \leq \mu' \leq \mu$, the dual of the embedding
\[
\tsZ(V)_{\leq \mu'} / \tsZ(V)_{\leq \lambda} \hookrightarrow \tsZ(V)_{\leq \mu} / \tsZ(V)_{\leq \lambda}
\]
is the projection
\[
\tsZ(V^*)_{< -\lambda} / \tsZ(V^*)_{< -\mu} \twoheadrightarrow \tsZ(V^*)_{< -\lambda} / \tsZ(V^*)_{< -\mu'},
\]
and the dual of the projection
\[
\tsZ(V)_{\leq \mu} / \tsZ(V)_{\leq \lambda} \twoheadrightarrow \tsZ(V)_{\leq \mu} / \tsZ(V)_{\leq \lambda'}
\]
is the embedding
\[
\tsZ(V^*)_{< -\lambda'} / \tsZ(V^*)_{< -\mu} \hookrightarrow \tsZ(V^*)_{< -\lambda} / \tsZ(V^*)_{< -\mu}.
\]
\end{rmk}

\subsection{Proof of Lemma~\ref{lem:trace}}

In order to give the proof of Lemma~\ref{lem:trace} we need another lemma.

\begin{lem}
\label{lem:trace-Wak}
For any $\lambda \in X_*(T)$, we have
\[
\trL(\mu^{\mathrm{rot}}_{\Wak^\wedge_\lambda}(x^{-1}))=e^{\lambda}.
\]
\end{lem}

\begin{proof}
By definition, $\trL(\mu^{\mathrm{rot}}_{\Wak^\wedge_\lambda}(x^{-1}))$ is the composition
\[
\delta^\wedge \xrightarrow[\sim]{\coev} \Wak^\wedge_\lambda \hatstar \Wak^\wedge_{-\lambda} \xrightarrow{\mu^{\mathrm{rot}}_{\Wak^\wedge_\lambda}(x^{-1}) \hatstar \id} \Wak^\wedge_\lambda \hatstar \Wak^\wedge_{-\lambda} \xrightarrow[\sim]{\ev} \delta^\wedge.
\]
Here by Lemma~\ref{lem:properties-Wak}\eqref{it:properties-Wak-5} we have $\mu^{\mathrm{rot}}_{\Wak^\wedge_\lambda}(x^{-1})=\mu_{\Wak^\wedge_\lambda}(e^{\lambda} \otimes 1)$, so that the middle map is $\mu_{\Wak^\wedge_\lambda \hatstar \Wak^\wedge_{-\lambda}}(e^{\lambda} \otimes 1)$ by~\eqref{eqn:monodromy-convolution-1}. By functoriality of monodromy this implies that the composition above is $\mu_{\delta^\wedge}(e^{\lambda} \otimes 1)$, which implies the desired claim.
\end{proof}

A general result about monoidal categories states that in an abelian monoidal category with exact monoidal product, the quantum trace is additive on short exact sequences (for morphisms compatible with the exact sequence), see~\cite[Proposition~4.7.5]{egno}. As in Remark~\ref{rmk:dualizability} this statement does not apply directly in our setting, but our proof of Lemma~\ref{lem:trace} will consist of repeating its proof\footnote{This proof was kindly explained to us by P.~Etingof.} and using Lemma~\ref{lem:trace-Wak} to compute the appropriate trace by induction.

\begin{proof}[Proof of Lemma~\ref{lem:trace}]
By Theorem~\ref{thm:gaitsgory-mon}\eqref{it:gaitsgory-mon-6} we have
\[
\hat{\sm}_V = \mu^{\mathrm{rot}}_{\tsZ(V)}(x^{-1}).
\]
We will prove by induction on $\lambda$ that
\[
\trL \left( \mu^{\mathrm{rot}}_{\tsZ(V)_{\leq \lambda}}(x^{-1}) \right) = \sum_{\mu \leq \lambda} \dim(V_\mu) \cdot e^\mu;
\]
this will imply the desired equality by taking $\lambda$ such that $\nu \leq \lambda$ for any $\nu$ such that $V_\nu \neq 0$.

If $\lambda$ satisfies $\lambda < \nu$ for any $\nu$ such that $V_\nu \neq 0$, then this equality holds since both sides vanish. Now let $\lambda \in X_*(T)$, and assume the equality is known for the predecessor $\lambda'$ of $\lambda$. We consider the exact sequence
\[
\tsZ(V)_{\leq \lambda'} \hookrightarrow \tsZ(V)_{\leq \lambda} \twoheadrightarrow \gr_\lambda(\tsZ(V)).
\]
By functoriality of monodromy, the automorphism $\mu^{\mathrm{rot}}_{\tsZ(V)_{\leq \lambda}}(x^{-1})$ of $\tsZ(V)_{\leq \lambda}$ preserves $\tsZ(V)_{\leq \lambda'}$, and restricts to $\mu^{\mathrm{rot}}_{\tsZ(V)_{\leq \lambda'}}(x^{-1})$ on this subobject. Moreover, the induced automorphism of $\gr_\lambda(\tsZ(V))$ is $\mu^{\mathrm{rot}}_{\gr_\lambda(\tsZ(V))}(x^{-1})$.

The object
\[
\tsZ(V)_{\leq \lambda} \hatstar (\tsZ(V^*)/\tsZ(V^*)_{< -\lambda}) = \tsZ(V)_{\leq \lambda} \hatstar (\tsZ(V)_{\leq \lambda})^\vee
\]
admits a canonical 3-step filtration
\[
\mathscr{M}_1 \subset \mathscr{M}_2 \subset \tsZ(V)_{\leq \lambda} \hatstar (\tsZ(V^*)/\tsZ(V^*)_{< -\lambda})
\]
with successive associated subquotients given by
\begin{multline*}
\tsZ(V)_{\leq \lambda'} \hatstar \gr_{-\lambda}(\tsZ(V^*)), \\
\tsZ(V)_{\leq \lambda'} \hatstar (\tsZ(V^*)/\tsZ(V^*)_{< -\lambda'}) \oplus \gr_\lambda(\tsZ(V)) \hatstar \gr_{-\lambda}(\tsZ(V^*)) \\
\text{and } \gr_\lambda(\tsZ(V)) \hatstar (\tsZ(V^*)/\tsZ(V^*)_{< -\lambda'}).
\end{multline*}
We have
\begin{multline*}
\Hom(\delta^\wedge, \gr_\lambda(\tsZ(V)) \hatstar (\tsZ(V^*)/\tsZ(V^*)_{< -\lambda'})) \\
\cong \Hom(\gr_{-\lambda}(\tsZ(V^*)), \tsZ(V^*)/\tsZ(V^*)_{< -\lambda'})=0,
\end{multline*}
since $(\tsZ(V^*)/\tsZ(V^*)_{< -\lambda'})_{\leq -\lambda}=0$; it follows that the coevaluation map
\[
\delta^\wedge \to \tsZ(V)_{\leq \lambda} \hatstar (\tsZ(V^*)/\tsZ(V^*)_{< -\lambda})
\]
factors through a map $\delta^\wedge \to \mathscr{M}_2$. For similar reasons, the evaluation map
\[
\tsZ(V)_{\leq \lambda} \hatstar (\tsZ(V^*)/\tsZ(V^*)_{< -\lambda}) \to \delta^\wedge
\]
vanishes on $\mathscr{M}_1$, hence factors through a morphism
\[
\bigl( \tsZ(V)_{\leq \lambda} \hatstar (\tsZ(V^*)/\tsZ(V^*)_{< -\lambda}) \bigr) / \mathscr{M}_1 \to \delta^\wedge.
\]
It follows that our trace is the composition
\begin{multline*}
\delta^\wedge \to \tsZ(V)_{\leq \lambda'} \hatstar (\tsZ(V^*)/\tsZ(V^*)_{< -\lambda'}) \oplus \gr_\lambda(\tsZ(V)) \hatstar \gr_{-\lambda}(\tsZ(V^*)) \\
\to \tsZ(V)_{\leq \lambda'} \hatstar (\tsZ(V^*)/\tsZ(V^*)_{< -\lambda'}) \oplus \gr_\lambda(\tsZ(V)) \hatstar \gr_{-\lambda}(\tsZ(V^*)) \to \delta^\wedge
\end{multline*}
where the first, resp.~third, map is the sum of the coevaluation, resp.~evaluation, morphisms for $\tsZ(V)_{\leq \lambda'}$ and $\gr_\lambda(\tsZ(V))$, and the middle arrow is the map induced by $\mu^{\mathrm{rot}}_{\tsZ(V)_{\leq \lambda}}(x^{-1}) \hatstar \id$, i.e.~the direct sum 
\[
\mu^{\mathrm{rot}}_{\tsZ(V)_{\leq \lambda'}}(x^{-1}) \hatstar \id \oplus \mu^{\mathrm{rot}}_{\gr_\lambda(\tsZ(V))}(x^{-1}) \hatstar \id.
\]
We deduce that
\[
\trL(\mu^{\mathrm{rot}}_{\tsZ(V)_{\leq \lambda}}(x^{-1})) = \trL(\mu^{\mathrm{rot}}_{\tsZ(V)_{\leq \lambda'}}(x^{-1})) + \trL(\mu^{\mathrm{rot}}_{\gr_\lambda(\tsZ(V))}(x^{-1})),
\]
which implies the desired formula by the induction hypothesis and Lemma~\ref{lem:trace-Wak}.
\end{proof}

\section{Perverse sheaves on \texorpdfstring{$G/U$}{G/U}}
\label{sec:Perv-G/U}

We continue with the setting of Sections~\ref{ss:const-Fl}--\ref{sec:fmZ}, and consider also the constructions of Section~\ref{sec:Coh-St} in the case $\bG=G^\vee_\bk$ (with the Borel subgroup $B^\vee_\bk$ and the maximal torus $T^\vee_\bk$). In particular, we fix a Steinberg section $\Sigma \subset G^\vee_\bk$ as in~\S\ref{ss:Steinberg-section}. It is clear that in this case the Coxeter system $(\bWf,\bSf)$ of Section~\ref{sec:Coh-St} identifies with the Coxeter system $(\Wf,\Sf)$ of Section~\ref{ss:const-Fl}. We will assume in this section that $G^\vee_\bk$ has simply connected derived subgroup, or in other words that the quotient of $X^*(T)$ by the root lattice is free.

Before constructing the main equivalence of the paper, we explain a similar construction for perverse sheaves on the ``finite'' flag variety $G/B$ (or, in fact, on the basic affine space $G/U$). This construction is essentially a reinterpretation of the main result of~\cite{bezr}; it will serve as a ``toy example'' to illustrate our methods, but will also play a role in the proof of the theorem.

\begin{rmk}
As explained above, the proofs in this section rely on the results of~\cite{bezr}.
 In this reference it is assumed that the group $G$ is semisimple of adjoint type, but all the proofs apply more generally under the present assumption that the dual group has simply connected derived subgroup. (In fact, the main ingredient that requires some assumption is Theorem~\ref{thm:Pittie-Steinberg}, which holds under our present assumption by Remark~\ref{thm:Pittie-Steinberg}.) 
\end{rmk}

\subsection{Categories of sheaves on \texorpdfstring{$G/U$}{G/U}}
\label{ss:cat-sheaves-G/U}

Recall the categories $\sfD_{U,U}$ and $\sfD^\wedge_{U,U}$ considered in~\S\ref{ss:G/B-tilting}. These categories admit perverse t-structures, whose hearts are denoted $\sfP_{U,U}$ and $\sfP^\wedge_{U,U}$ respectively. In fact, pushforward along the closed embedding $G/U \hookrightarrow \tFl_G$ identifies $\sfP_{U,U}$, resp.~$\D_{U,U}$ with the Serre subcategory of $\sfP_{\Iwu,\Iwu}$, resp.~the full triangulated subcategory of $\D_{\Iwu,\Iwu}$, generated by the simple objects $\pi^\dag \For^{\Iw}_{\Iwu}( \IC_w)$ with $w \in \Wf$; it also provides a t-exact fully faithful functor $\sfD^\wedge_{U,U} \to \sfD^\wedge_{\Iwu,\Iwu}$. If we denote by $\sfP_{U,U}^+$, resp.~$\sfD_{U,U}^+$, the Serre subcategory of $\sfP_{U,U}$, resp.~the full triangulated subcategory of $\sfD_{U,U},$ generated by the objects $\pi^\dag \For^{\Iw}_{\Iwu}( \IC_w)$ with $w \in \Wf \smallsetminus \{e\}$, then we can consider the quotient categories
\[
\sfP_{U,U}^0 := \sfP_{U,U} / \sfP_{U,U}^+, \quad \sfD_{U,U}^0 := \sfD_{U,U} / \sfD_{U,U}^+.
\]
By Lemma~\ref{lem:quotient-t-str}, there exists a unique t-structure on $\sfD_{U,U}^0$ such that the quotient functor
\[
 \Pi^0_{U,U} : \sfD_{U,U} \to \sfD_{U,U}^0
\]
is t-exact. This t-structure is bounded, and its heart identifies with $\sfP^0_{U,U}$; it will be called the perverse t-structure, and the associated cohomology functors will once again be denoted $\pH^n(-)$. We have a canonical t-exact functor
\begin{equation}
\label{eqn:D-U-Iu}
\sfD_{U,U}^0 \to \sfD^0_{\Iwu,\Iwu},
\end{equation}
whose restriction to the heart of the perverse t-structure is fully faithful.


As for $\sfD_{\Iwu,\Iwu}$, the category $\sfD_{U,U}$ admits a natural convolution product $\star_U$ which equips it with the structure of a monoidal category (without unit object) such that the embedding $\sfD_{U,U} \to \sfD_{\Iwu,\Iwu}$ is monoidal, and which induces (in the appropriate sense) the product $\hatstar_U$.
We also have a canonical bifunctor
\[
 \hatstar_U : \sfD^\wedge_{U,U} \times \sfD_{U,U} \to \sfD_{U,U}
\]
which defines an action of $(\sfD^\wedge_{U,U}, \hatstar_U)$ on $\sfD_{U,U}$.

 As in~\S\ref{ss:mon-reg-quotient}, the bifunctor $(\scF,\scG) \mapsto \Pi^0_{U,U}(\scF \star_U \scG)$ factors through a triangulated bifunctor
 \[
  \star^0_U : \sfD_{U,U}^0 \times \sfD_{U,U}^0 \to \sfD_{U,U}^0
 \]
which defines a monoidal structure (without unit object) on $\D_{U,U}^0$ so that~\eqref{eqn:D-U-Iu} is monoidal. Moreover, $\star^0_U$ is ``right t-exact'' in the sense that if $\scF,\scG$ belong to the nonpositive part of the perverse t-structure on $\sfD^0_{U,U}$ then so does $\scF \star^0_U \scG$. We therefore obtain a monoidal structure (without unit object) on $\sfP_{U,U}^0$ by setting
\[
 \scF \pstar^0_U \scG := \pH^0(\scF \star^0_U \scG)
\]
for $\scF,\scG$ in $\sfP_{U,U}^0$;
then we have a fully faithful exact monoidal functor
\begin{equation}
\label{eqn:embedding-PUU-PII}
(\sfP_{U,U}^0, \pstar^0_U) \to (\sfP^0_{\Iwu,\Iwu},\pstar^0_{\Iwu}).
\end{equation}

As in~\S\ref{ss:actions-Dwedge-D} we have a canonical bifunctor
\[
 \hatstar^0_U : \sfD^\wedge_{U,U} \times \sfD^0_{U,U} \to \sfD^0_{U,U}
\]
compatible with $\hatstar^0$ in the obvious way, and which defines an action of $(\sfD^\wedge_{U,U},\hatstar_U)$ on the category $\sfD^0_{U,U}$. For $\scF$ in $\sfP^\wedge_{U,U}$ and $\scG$ in $\sfP^0_{U,U}$ we then set
\[
 \scF \phatstar^0_U \scG := \pH^0( \scF \hatstar^0_U \scG).
\]
As in Lemma~\ref{lem:t-exactness-conv-tFl}, this bifunctor is right exact on each side.

Recall also the functor $\mathsf{C}_m$ considered in~\S\ref{ss:truncation-ps}. It is clear that this functor restricts to a functor from $\sfP^\wedge_{U,U}$ to $\sfP_{U,U}$, which will again be denoted $\mathsf{C}_m$.

%

\subsection{Morphisms from \texorpdfstring{$\Pi^0_{U,U}(\Xi^\wedge_!)$}{PiXi}}
\label{ss:morphisms-PiXi}

Recall the object $\Xi^\wedge_!$ defined in~\S\ref{ss:G/B-tilting}. Considering this object as a pro-object in $\sfD_{U,U}$, and applying the extension to pro-objects of the functor $\Pi^0_{U,U}$ we obtain a pro-object $\Pi^0_{U,U}(\Xi^\wedge_!)$ in $\sfD^0_{U,U}$. We can then consider the functor from $\sfP^0_{U,U}$ to the category of $\bk$-vector spaces given by
\[
\scF \mapsto \Hom_{\sfD^0_{U,U}}(\Pi^0_{U,U}(\Xi^\wedge_!), \scF),
\]
where in the right-hand side we mean morphisms in the category of pro-objects in $\sfD^0_{U,U}$. Concretely, if we write $\Xi^\wedge_! = ``\varprojlim_n" \scA_n$ for some objects $\scA_n$ in $\sfD_{U,U}$, then we have
\[
\Hom_{\sfD^0_{U,U}}(\Pi^0_{U,U}(\Xi^\wedge_!), \scF) = \varinjlim_n \Hom_{\sfD^0_{U,U}}(\Pi^0_{U,U}(\scA_n), \scF).
\]
Similarly, given $\scG$ in $\sfP_{U,U}$ we can consider the vector space
\[
\Hom_{\sfP^\wedge_{U,U}}(\Xi^\wedge_!, \scG) = \varinjlim_n \Hom_{\sfD_{U,U}}(\scA_n, \scG).
\]

\begin{lem}
\label{lem:morph-Xiwedge-Pi}
For any $\scG$ in $\sfP_{U,U}$, the canonical morphism
\[
\Hom_{\sfP^\wedge_{U,U}}(\Xi^\wedge_!, \scG) \to \Hom_{\sfD^0_{U,U}}(\Pi^0_{U,U}(\Xi^\wedge_!),\Pi^0_{U,U}(\scG))
\]
is an isomorphism.
\end{lem}

\begin{proof}
Fix $\scG$ in $\sfP_{U,U}$.
As explained above, writing $\Xi^\wedge_! = ``\varprojlim_n" \scA_n$ for some $\scA_n$ in $\sfD_{U,U}$, our morphism can be written more concretely as the morphism
\begin{equation}
\label{eqn:morph-Xiwedge-Pi}
\varinjlim_n \Hom_{\sfD_{U,U}}(\scA_n, \scG) \to \varinjlim_n \Hom_{\sfD^0_{U,U}}(\Pi^0_{U,U}(\scA_n), \Pi^0_{U,U}(\scG))
\end{equation}
induced by the functor $\Pi^0_{U,U}$. 

Let us first show that~\eqref{eqn:morph-Xiwedge-Pi} is surjective. A morphism in the right-hand side is represented by a morphism $\Pi^0_{U,U}(\scA_n) \to \Pi^0_{U,U}(\scG)$ in $\sfD^0_{U,U}$ for some $n$, i.e.~by a diagram
\[
\scA_n \xleftarrow{s} \mathscr{X} \xrightarrow{f} \scG
\]
where $\mathscr{X}$ is an object in $\sfD_{U,U}$ and $s$, $f$ are morphisms in this category such that the cone $\scC$ of $s$ belongs to $\sfD^+_{U,U}$. Now, by Lemma~\ref{lem:Xi-projective} we have
\[
\Hom_{\sfD^\wedge_{U,U}}(\Xi^\wedge_!, \pi^\dag \For^{\Iw}_{\Iwu}(\IC_w)[i])=0
\]
for any $w \in \Wf \smallsetminus \{e\}$ and $i \in \Z$; it follows that
\[
\Hom_{\sfD^\wedge_{U,U}}(\Xi^\wedge_!, \scC)=0,
\]
or in other words that
\[
\varinjlim_m \Hom_{\sfD_{U,U}}(\scA_m, \scC)=0.
\]
We deduce that for $m \gg n$ the composition $\scA_m \to \scA_n \to \scC$ vanishes. Fix such an $m$, and denote by $h$ the structure morphism $\scA_m \to \scA_n$. If we complete the morphisms $s$ and $h$ to a commutative diagram
\[
\xymatrix{
\scA_m \ar[d]_-{h} & \mathscr{Y} \ar[l]_-{t} \ar[d]^-{g} \\
\scA_n & \mathscr{X} \ar[l]^-{s}
}
\]
in $\sfD_{U,U}$ such that the cone of $t$ belongs to $\sfD^+_{U,U}$ (which is always possible since morphisms whose cone belongs to $\sfD^+_{U,U}$ form a multiplicative system, see~\cite[\href{https://stacks.math.columbia.edu/tag/05RG}{Tag 05RG}]{stacks-project}), then the image of our morphism in $\Hom_{\sfD^0_{U,U}}(\Pi^0_{U,U}(\scA_m), \Pi^0_{U,U}(\scG))$ is represented by the diagram
\[
\scA_m \xleftarrow{t} \mathscr{Y} \xrightarrow{f \circ g} \scG.
\]
Now since the composition $\scA_m \to \scA_n \to \scC$ vanishes, there exists $k : \scA_m \to \mathscr{X}$ in $\sfD_{U,U}$ such that $s \circ k = h$. Then in $\sfD^0_{U,U}$ we have
\[
f \circ g \circ t^{-1} = f \circ s^{-1} \circ h = f \circ k.
\]
This shows that the image of our morphism in $\Hom_{\sfD^0_{U,U}}(\Pi^0_{U,U}(\scA_m), \Pi^0_{U,U}(\scG))$ is the image of a morphism in $\Hom_{\sfD_{U,U}}(\scA_m, \scG)$,
 which finishes the proof of surjectivity.

The proof of injectivity is similar. If a morphism $f : \scA_n \to \scG$ has trivial image in $\varinjlim_m \Hom_{\sfD^0_{U,U}}(\Pi^0_{U,U}(\scA_m), \Pi^0_{U,U}(\scG))$, then composing with the morphism $\scA_m \to \scA_n$ for some $m \gg n$ we can assume that $\Pi^0_{U,U}(f)=0$. This means that there exists a morphism $g : \scG \to \scH$ whose cone $\scC$ belongs to $\sfD^+_{U,U}$ such that $g \circ f = 0$. Then $f$ factors through a morphism $\scA_n \to \scC[-1]$. Replacing again $n$ by a larger integer we can assume that this morphism vanishes, so that $f=0$ in the inductive limit, which finishes the proof.
\end{proof}

\subsection{Towards a monoidal structure}
\label{ss:towards-monoidal-U}

For any $\scF$ in $\sfP^0_{U,U}$,
monodromy for the action of $T$ on the left and on the right on $G/U$ equips $\Hom_{\sfD^0_{U,U}}(\Pi^0_{U,U}(\Xi^\wedge_!), \scF)$ with the structure of an $\scO(T^\vee_\bk)$-bimodule. In fact, it follows from Lemma~\ref{lem:monodromy-fiber-prod} that these actions factor through a structure of $\scO(T^\vee_\bk \times_{T^\vee_\bk / \Wf} T^\vee_\bk)$-module.
Our goal in this subsection is to explain how to define, for $\scF,\scG$ in $\sfP^0_{U,U}$, a canonical morphism
\begin{multline}
\label{eqn:morph-monoidality-DUU}
\Hom_{\sfD^0_{U,U}}(\Pi^0_{U,U}(\Xi^\wedge_!), \scF) \otimes_{\scO(T^\vee_\bk)} \Hom_{\sfD^0_{U,U}}(\Pi^0_{U,U}(\Xi^\wedge_!), \scG) \\
\to \Hom_{\sfD^0_{U,U}}(\Pi^0_{U,U}(\Xi^\wedge_!), \scF \pstar^0_U \scG),
\end{multline}
which will eventually be shown to define a monoidal structure on the functor $\Hom_{\sfD^0_{U,U}}(\Pi^0_{U,U}(\Xi^\wedge_!), -)$.
(Here $\scO(T^\vee_\bk)$ acts on $\Hom_{\sfD^0_{U,U}}(\Pi^0_{U,U}(\Xi^\wedge_!), \scF)$ via the projection $T^\vee_\bk \times_{T^\vee_\bk / \Wf} T^\vee_\bk \to T^\vee_\bk$ on the second factor, and on $\Hom_{\sfD^0_{U,U}}(\Pi^0_{U,U}(\Xi^\wedge_!), \scG)$ via the projection $T^\vee_\bk \times_{T^\vee_\bk / \Wf} T^\vee_\bk \to T^\vee_\bk$ on the first factor.) To explain this construction we first need to recall a similar construction from~\cite{bezr}.

First, consider the scheme $\FN_{T^\vee_\bk \times_{T^\vee_\bk / \Wf} T^\vee_\bk}(\{(e,e)\})$.
By Lemma~\ref{lem:Dw-fiber-prod}\eqref{it:Dw-fiber-prod-2} this scheme is the spectrum of the algebra
\[
 \scO(\FN_{T^\vee_\bk}(\{e\})) \otimes_{\scO(\FN_{T^\vee_\bk}(\{e\}))^{\Wf}} \scO(\FN_{T^\vee_\bk}(\{e\}))
\]
that appears e.g.~in~\cite[Theorem~9.1]{bezr}. (Section~\ref{sec:Hecke-cat} is written under the running assumption that $Z(\bG)$ is smooth, but this condition is not required for this specific lemma.)

Recall the category $\sfT^\wedge_{U,U}$ of tilting objects in $\sfP^\wedge_{U,U}$. As explained in~\cite[Remark~7.9]{bezr}, this subcategory is closed under the convolution product $\hatstar_U$.
For $\scF$ in $\sfP^\wedge_{U,U}$, in~\cite{bezr} we explain that monodromy defines on $\Hom_{\sfP_{U,U}^{\wedge}}(\Xi_!^\wedge, \scF)$ the structure of a finitely generated $\scO(\FN_{T^\vee_\bk \times_{T^\vee_\bk / \Wf} T^\vee_\bk}(\{(e,e)\}))$-module. Moreover,
in~\cite[\S 11.3]{bezr} we construct a monoidal structure on the functor 
\[
\Hom_{\sfP_{U,U}^{\wedge}}(\Xi_!^\wedge, -) : \mathsf{T}^\wedge_{U,U} \to \Mod^{\mathrm{fg}}(\scO(\FN_{T^\vee_\bk \times_{T^\vee_\bk / \Wf} T^\vee_\bk}(\{(e,e)\}))),
\]
for the monoidal product on the category $\mathsf{T}^\wedge_{U,U}$ given by $\hatstar_U$, and that on the category of $\scO(\FN_{T^\vee_\bk \times_{T^\vee_\bk / \Wf} T^\vee_\bk}(\{(e,e)\}))$-modules given by tensor product over the algebra $\scO(\FN_{T^\vee_\bk}(\{e\}))$. Using this structure we obtain an isomorphism
\begin{equation}
\label{eqn:Hom-Xi-monoidal}
\Hom_{\sfP_{U,U}^{\wedge}}(\Xi_!^\wedge, \Xi_!^\wedge \hatstar_U \Xi_!^\wedge) \cong \Hom_{\sfP_{U,U}^{\wedge}}(\Xi_!^\wedge, \Xi_!^\wedge) \otimes_{\scO(\FN_{T^\vee_\bk}(\{e\}))} \Hom_{\sfP_{U,U}^{\wedge}}(\Xi_!^\wedge, \Xi_!^\wedge).
\end{equation}
Here the right-hand side has a canonical element, given by $\id_{\Xi_!^\wedge} \otimes \id_{\Xi_!^\wedge}$, which then defines a canonical morphism $\xi : \Xi_!^\wedge \to \Xi_!^\wedge \hatstar_U \Xi_!^\wedge$. Concretely, writing $\Xi^\wedge_! = ``\varprojlim_n" \scA_n$ for some $\scA_n$ in $\sfD_{U,U}$ as in~\S\ref{ss:morphisms-PiXi}, we have
\[
\Xi_!^\wedge \hatstar_U \Xi_!^\wedge = ``\varprojlim_{n,m \geq 0}" \scA_n \star_{U} \scA_m,
\]
so that
\[
\Hom_{\sfP_{U,U}^{\wedge}}(\Xi_!^\wedge, \Xi_!^\wedge \hatstar_U \Xi_!^\wedge) = \varprojlim_{n,m} \varinjlim_q \Hom_{\sfD_{U,U}}(\scA_q, \scA_n \star_{U} \scA_m);
\]
our morphism is therefore defined by a collection of morphisms $\xi_{n,m} : \scA_{q(n,m)} \to \scA_n \star_{U} \scA_m$ for some function $q : (\Z_{\geq 0})^2 \to \Z_{\geq 0}$, which we fix from now on.

Finally we can explain the construction of~\eqref{eqn:morph-monoidality-DUU}.
Consider some elements $f$ in $\Hom_{\sfD^0_{U,U}}(\Pi^0_{U,U}(\Xi^\wedge_!), \scF)$ and $g$ in $\Hom_{\sfD^0_{U,U}}(\Pi^0_{U,U}(\Xi^\wedge_!), \scG)$, represented by morphisms $f : \Pi^0_{U,U}(\scA_n) \to \scF$ and $g : \Pi^0_{U,U}(\scA_m) \to \scG$. Then the image of $f \otimes g$ under~\eqref{eqn:morph-monoidality-DUU} is the composition
\begin{multline*}
\Pi^0_{U,U}(\scA_{q(n,m)}) \xrightarrow{\Pi^0_{U,U}(\xi_{n,m})} \Pi^0_{U,U}(\scA_n \star_{U} \scA_m) = \Pi^0_{U,U}(\scA_n) \star^0_U \Pi^0_{U,U}(\scA_m) \\
\xrightarrow{f \star^0_U g} \scF \star^0_U \scG \to \scF \pstar^0_U \scG
\end{multline*}
where the last morphism is the natural truncation morphism. (Recall that $\scF \star^0_U \scG$ belongs to the nonpositive part of the perverse t-structure, and $\scF \pstar^0_U \scG$ is its degree-$0$ cohomology.)

\subsection{Statement}
\label{ss:statement-G/U}

Consider the category $\Coh_{\{(e,e)\}}(T^\vee_\bk \times_{T^\vee_\bk / \Wf} T^\vee_\bk)$ of coherent shea\-ves on $T^\vee_\bk \times_{T^\vee_\bk / \Wf} T^\vee_\bk$ which are supported set-theoretically on the closed subscheme $\{(e,e)\}$, which we identify with the category of finitely generated $\scO(T^\vee_\bk \times_{T^\vee_\bk / \Wf} T^\vee_\bk)$-modules on which a power of the ideal $\mathcal{I}$ (see~\S\ref{ss:completions}), or equivalently a power of the ideal $\cJ$, acts trivially. This category is monoidal (without unit object) for the product $\circledast$ defined by
\[
M \circledast N := M \otimes_{\scO(T^\vee_\bk)} N.
\]
(Here in the tensor product the action on $M$ is induced by the projection $T^\vee_\bk \times_{T^\vee_\bk / \Wf} T^\vee_\bk \to T^\vee_\bk$ on the second factor, and the action on $N$ is induced by projection on the first factor; the action of $\scO(T^\vee_\bk \times_{T^\vee_\bk / \Wf} T^\vee_\bk)$ on the tensor product is the obvious one, defined in terms of the remaining actions.)

The following proposition is the promised ``finite variant'' of our main result. Its proof will be explained in the next subsection.

\begin{thm}
\label{thm:reg-quotient-finite}
The functor $\Hom_{\sfD^0_{U,U}}(\Pi^0_{U,U}(\Xi^\wedge_!), -)$ induces an equivalence of mo\-noidal abelian categories
\[
\Phi_{U,U} : \left( \sfP_{U,U}^0, \star^0_U \right) \simto \left( \Coh_{\{(e,e)\}}(T^\vee_\bk \times_{T^\vee_\bk / \Wf} T^\vee_\bk), \circledast \right).
\]
\end{thm}

\subsection{``Truncated'' version}
\label{ss:truncation-PUU}

For $m \geq 1$ we consider the affine scheme
\[
\bigl( T^\vee_\bk \times_{T^\vee_\bk / \Wf} T^\vee_\bk \bigr)^{(m)} :=
 \Spec \bigl( \scO(T^\vee_\bk \times_{T^\vee_\bk / \Wf} T^\vee_\bk) / \mathcal{J}^m \cdot \scO(T^\vee_\bk \times_{T^\vee_\bk / \Wf} T^\vee_\bk) \bigr).
\]
Pushforward along the closed embedding $\bigl( T^\vee_\bk \times_{T^\vee_\bk / \Wf} T^\vee_\bk \bigr)^{(m)} \to T^\vee_\bk \times_{T^\vee_\bk / \Wf} T^\vee_\bk$ provides
a fully faithful functor
\[
\Coh \bigl( ( T^\vee_\bk \times_{T^\vee_\bk / \Wf} T^\vee_\bk)^{(m)} \bigr) \to \Coh_{\{(e,e)\}}(T^\vee_\bk \times_{T^\vee_\bk / \Wf} T^\vee_\bk),
\]
and it is clear from definitions that the product $\circledast$ restricts to a monoidal product on $\Coh(( T^\vee_\bk \times_{T^\vee_\bk / \Wf} T^\vee_\bk)^{(m)})$. In fact, 
this collection of functors realizes the category $\Coh_{\{(e,e)\}}(T^\vee_\bk \times_{T^\vee_\bk / \Wf} T^\vee_\bk)$ as the direct limit of its subcategories $\Coh(( T^\vee_\bk \times_{T^\vee_\bk / \Wf} T^\vee_\bk)^{(m)})$, in a way compatible with the monoidal product.

On the constructible side, again for $m \geq 1$ we will
denote by $\sfP_{U,U}^{(m)}$ the full abelian subcategory of $\sfP_{U,U}$ whose objects are the perverse sheaves such that the monodromy action of $\scO(T^\vee_\bk / \Wf)$ (see~\S\ref{ss:tilting-perv}) vanishes on $\mathcal{J}^m$. 
This subcategory contains all the simple objects of $\sfP_{U,U}$, and is stable under subquotients (but not under extensions). If we denote by $\sfP_{U,U}^{(m),0}$ the Serre quotient of $\sfP_{U,U}^{(m)}$ by the Serre subcategory generated by the simple objects $\pi^\dag \For^{\Iw}_{\Iwu}(\IC_w)$ with $w \in \Wf \smallsetminus \{e\}$, then we have a natural fully faithful functor
\[
\sfP_{U,U}^{(m),0} \to \sfP_{U,U}^{0}.
\]
The essential image of this functor can be described as follows. It is clear that monodromy induces, for any $\scF$ in $\sfP_{U,U}^{0}$, a canonical morphism
\[
\mu^0_\scF : \scO(T^\vee_\bk \times_{T^\vee_\bk / \Wf} T^\vee_\bk) \to \End_{\sfP_{U,U}^0}(\scF).
\]
Then the essential image of $\sfP_{U,U}^{(m),0}$ in $\sfP_{U,U}^{0}$ identifies with the subcategory consisting of objects such that $\mu^0_\scF$ vanishes on $\cJ^m$. Indeed, any object in this essential image clearly satisfies this property. On the other hand, if $\mu^0_\scF$ vanishes on $\cJ^m$, writing $\scF=\Pi^0_{U,U}(\scG)$ for some $\scG$ in $\sfP_{U,U}$, we see that the surjection $\scG \twoheadrightarrow \scG / \cJ^m \cdot \scG$ becomes an isomorphism after application of $\Pi^0_{U,U}$, and obviously $\scG / \cJ^m \cdot \scG$ belongs to $\sfP^{(m)}_{U,U}$.

Using~\eqref{eqn:monodromy-convolution-1}--\eqref{eqn:monodromy-convolution-2} one sees that the convolution product $\star^0_U$ restricts to a monoidal product on $\sfP_{U,U}^{(m),0}$. In this way we realize the category $\sfP_{U,U}^{0}$ as the direct limit of its subcategories $\sfP_{U,U}^{(m),0}$, in a way compatible with the monoidal product. It is clear that the restriction of the quotient functor $\Pi^0_{U,U}$ to $\sfP_{U,U}^{(m)}$ takes values in $\sfP_{U,U}^{(m),0}$, and identifies with the quotient functor $\sfP_{U,U}^{(m)} \to \sfP_{U,U}^{(m),0}$.

From these considerations we see that Theorem~\ref{thm:reg-quotient-finite} is a corollary of the following statement.

\begin{prop}
\label{prop:reg-quotient-finite-m}
For any $m \geq 1$,
the functor $\Hom_{\sfD^0_{U,U}}(\Pi^0_{U,U}(\Xi^\wedge_!), -)$ induces an equivalence of abelian categories
\[
\sfP_{U,U}^{(m),0} \simto \Coh \bigl( (T^\vee_\bk \times_{T^\vee_\bk / \Wf} T^\vee_\bk)^{(m)} \bigr).
\]
Moreover these equivalences admit structures of monoidal functors compatible in the obvious way with the natural embeddings
\[
\sfP_{U,U}^{(m),0} \to \sfP_{U,U}^{(m'),0}, \qquad \Coh \bigl( (T^\vee_\bk \times_{T^\vee_\bk / \Wf} T^\vee_\bk)^{(m)} \bigr) \to \Coh \bigl( (T^\vee_\bk \times_{T^\vee_\bk / \Wf} T^\vee_\bk)^{(m')} \bigr)
\]
when $m \leq m'$.
\end{prop}

In the proof of this proposition we will consider, for $m \geq 1$, the object
\[
\Xi_!^{(m)} := \mathsf{C}_m(\Xi^\wedge_!) \quad \in \sfP_{U,U}.
\]
It is clear that this object belongs to the subcategory $\sfP_{U,U}^{(m)}$.

\begin{lem}
\label{eqn:isom-convolution-Xim}
 For any $m \geq 1$, there exists a canonical isomorphism
 \[
  \Pi^0_{U,U} \bigl( \mathsf{C}_m(\Xi^\wedge_! \hatstar_U \Xi^\wedge_!) \bigr) \cong \Pi^0_{U,U}(\Xi_!^{(m)}) \pstar^0_U \Pi^0_{U,U}(\Xi_!^{(m)}).
 \]
\end{lem}

\begin{proof}
 Since the functor $\Xi^\wedge_! \hatstar_U (-)$ is t-exact (see Lemma~\ref{lem:exactness-convolution-Xi}) the morphism
\[
\Xi_!^\wedge \hatstar_U \Xi_!^\wedge \to \Xi_!^\wedge \hatstar_U \Xi_!^{(m)}
\]
is surjective, and identifies the right-hand side with
$\mathsf{C}_m(\Xi_!^\wedge \hatstar_U \Xi_!^\wedge)$.
On the other hand, by definition we have
\[
 \Pi^0_{U,U}(\Xi_!^{(m)}) \pstar^0_U \Pi^0_{U,U}(\Xi_!^{(m)}) \cong \pH^0(\Pi^0_{U,U}(\Xi_!^{(m)} \star_U \Xi_!^{(m)})).
\]
What we have to construct is therefore a canonical isomorphism
\[
 \Pi^0_{U,U} ( \Xi_!^\wedge \hatstar_U \Xi_!^{(m)}) \cong \pH^0(\Pi^0_{U,U}(\Xi_!^{(m)} \star_U \Xi_!^{(m)})).
\]

A choice of a family of $r$ generators of the ideal $\mathcal{J}^m$ defines an exact sequence
\[
(\Xi_!^\wedge)^{\oplus r} \to \Xi_!^\wedge \to \Xi_!^{(m)} \to 0.
\]
Now the functor
\[
 \pH^0(\Pi^0_{U,U}( (-) \hatstar_U \Xi_!^{(m)})) = (-) \phatstar^0_U \Pi^0_{U,U}(\Xi_!^{(m)}) : \sfP^\wedge_{U,U} \to \sfP^0_{U,U}
\]
is right exact (see~\S\ref{ss:cat-sheaves-G/U}); we therefore deduce an exact squence
\[
\Pi^0_{U,U}((\Xi_!^\wedge)^{\oplus r} \hatstar_U \Xi_!^{(m)}) \to \Pi^0_{U,U}(\Xi_!^\wedge \hatstar_U \Xi_!^{(m)}) \to \pH^0(\Pi^0_{U,U}(\Xi_!^{(m)} \star_U \Xi_!^{(m)})) \to 0.
\]
Using~\eqref{eqn:monodromy-convolution-1}--\eqref{eqn:monodromy-convolution-2} one sees that the first morphism in this sequence vanishes, which shows that the second morphism is an isomorphism and finishes the proof.
\end{proof}

%

\begin{proof}[Proof of Proposition~\ref{prop:reg-quotient-finite-m}]
%
%
%

We claim that for $m \geq 1$ the morphism $\mu_{\Xi_!^{(m)}}$ factors through an isomorphism
\begin{equation}
\label{eqn:End-Xi-m}
\scO((T^\vee_\bk \times_{T^\vee_\bk / \Wf} T^\vee_\bk)^{(m)}) \simto \End_{\sfP_{U,U}}(\Xi_!^{(m)}).
\end{equation}
In fact, since by definition the action of $\scO((T^\vee_\bk \times_{T^\vee_\bk / \Wf} T^\vee_\bk)^{(m)})$ on $\Xi_!^{(m)}$ vanishes on $\mathcal{J}^m$ we have
\[
\End_{\sfP_{U,U}}(\Xi_!^{(m)}) \cong \Hom_{\sfP_{U,U}^\wedge}(\Xi^\wedge_!,\Xi_!^{(m)}).
\]
Now by projectivity of $\Xi^\wedge_!$ (see Lemma~\ref{lem:Xi-projective}) we have
\[
\Hom_{\sfP_{U,U}^\wedge}(\Xi^\wedge_!,\Xi_!^{(m)}) \cong \End_{\sfP_{U,U}^\wedge}(\Xi^\wedge_!) \otimes_{\scO(T^\vee_\bk \times_{T^\vee_\bk / \Wf} T^\vee_\bk)} \scO((T^\vee_\bk \times_{T^\vee_\bk / \Wf} T^\vee_\bk)^{(m)}).
\]
Finally, by~\cite[Theorem~9.1]{bezr} and the comments in~\S\ref{ss:towards-monoidal-U}, $\mu_{\Xi_!^\wedge}$ induces an algebra isomorphism
\begin{equation}
\label{eqn:isom-End-Xi}
\scO(\FN_{T^\vee_\bk \times_{T^\vee_\bk / \Wf} T^\vee_\bk}(\{(e,e)\})) \simto \End(\Xi^\wedge_!).
\end{equation}
We deduce that~\eqref{eqn:End-Xi-m} is an isomorphism, as desired.

For any $\scF$ in $\sfP_{U,U}^{(m)}$ we have
\begin{equation*}
\Hom_{\sfP^\wedge_{U,U}}(\Xi^\wedge_!,\scF) = \Hom_{\sfP^{(m)}_{U,U}}(\Xi^{(m)}_!,\scF).
\end{equation*}
From this and Lemma~\ref{lem:Xi-projective} we deduce that
$\Xi_!^{(m)}$ is the projective cover of the simple object $\pi^\dag \For^{\Iw}_{\Iwu}(\IC_e)$ in $\sfP^{(m)}_{U,U}$, which implies that the image $\Pi^0_{U,U}(\Xi_!^{(m)})$ of $\Xi_!^{(m)}$ in $\sfP_{U,U}^{(m),0}$ is the projective cover of the unique simple object in this category; by standard arguments this implies that the functor
\[
 \scF \mapsto \Hom_{\sfP^{(m),0}_{U,U}}(\Pi^0_{U,U}(\Xi^{(m)}_!),\scF)
\]
induces an equivalence of categories
\begin{equation}
\label{eqn:Hom-Xi-m}
\sfP_{U,U}^{(m),0} \simto \Modr^{\mathrm{fg}} \bigl( \End_{\sfP_{U,U}^{(m),0}}(\Pi^0_{U,U}(\Xi_!^{(m)})) \bigr).
\end{equation}
Now since $\Xi^{(m)}_!$ is projective in $\sfP^{(m)}_{U,U}$ with unique simple quotient $\pi^\dag \For^{\Iw}_{\Iwu}(\IC_e)$, for any $\scF$ in $\sfP^{(m)}_{U,U}$ the morphism
\[
 \Hom_{\sfP^{(m)}_{U,U}}(\Xi^{(m)}_!,\scF) \to \Hom_{\sfP^{(m),0}_{U,U}} \bigl( \Pi^0_{U,U}(\Xi^{(m)}_!),\Pi^0_{U,U}(\scF) \bigr)
\]
induced by $\Pi^0_{U,U}$
is an isomorphism. Using the isomorphism~\eqref{eqn:End-Xi-m} this allows to identify the algebra $\End_{\sfP_{U,U}^{(m),0}}(\Pi^0_{U,U}(\Xi_!^{(m)}))$ with $\scO((T^\vee_\bk \times_{T^\vee_\bk / \Wf} T^\vee_\bk)^{(m)})$, hence 
the category $\Modr^{\mathrm{fg}}(\End_{\sfP_{U,U}^{(m),0}}(\Xi_!^{(m)}))$ with $\Coh((T^\vee_\bk \times_{T^\vee_\bk / \Wf} T^\vee_\bk)^{(m)})$.

We claim that the functor
\[
 \Hom_{\sfP^{(m),0}_{U,U}}(\Pi^0_{U,U}(\Xi^{(m)}_!), -) : \sfP_{U,U}^{(m),0} \to \Mod(\scO(T^\vee_\bk \times_{T^\vee_\bk / \Wf} T^\vee_\bk))
\]
identifies canonically with the restriction of
\[
 \Hom_{\sfD^0_{U,U}}(\Pi^0_{U,U}(\Xi^\wedge_!), -) : \sfP_{U,U}^{0} \to \Mod(\scO(T^\vee_\bk \times_{T^\vee_\bk / \Wf} T^\vee_\bk))
\]
to $\sfP_{U,U}^{(m),0}$. In fact, to prove this claim it suffices to construct, for $\scF$ in $\sfP^{(m)}_{U,U}$, a functorial isomorphism
\[
 \Hom_{\sfP^{(m),0}_{U,U}}(\Pi^0_{U,U}(\Xi^{(m)}_!), \Pi^0_{U,U}(\scF)) \cong \Hom_{\sfD^0_{U,U}}(\Pi^0_{U,U}(\Xi^\wedge_!), \Pi^0_{U,U}(\scF)).
\]
And for this, as explained above the left-hand side identifies canonically with $\Hom_{\sfP^{(m)}_{U,U}}(\Xi^{(m)}_!,\scF)$, and then with $\Hom_{\sfP^{\wedge}_{U,U}}(\Xi^{\wedge}_!,\scF)$. The desired identification is therefore provided by Lemma~\ref{lem:morph-Xiwedge-Pi}. This isomorphism and the considerations above show that the restriction of the functor $\Hom_{\sfD^0_{U,U}}(\Pi^0_{U,U}(\Xi^\wedge_!), -)$ to $\sfP_{U,U}^{(m),0}$ takes values in $\Coh((T^\vee_\bk \times_{T^\vee_\bk / \Wf} T^\vee_\bk)^{(m)})$, and induces an equivalence between these categories.


To conclude the proof, we have to construct compatible monoidal structures on our equivalences, which will be done if we prove that~\eqref{eqn:morph-monoidality-DUU} is an isomorphism for any $\scF,\scG$ in $\sfP^{(m),0}_{U,U}$.

First, consider the case $\scF=\scG=\Pi^0_{U,U}(\Xi_!^{(m)})$. In this case we have seen that
\begin{equation}
\label{eqn:Hom-Xi-Xi-m-1}
\Hom_{\sfD^0_{U,U}}(\Pi^0_{U,U}(\Xi^\wedge_!), \Pi^0_{U,U}(\Xi_!^{(m)})) \cong \scO((T^\vee_\bk \times_{T^\vee_\bk / \Wf} T^\vee_\bk)^{(m)}).
\end{equation}
On the other hand, by Lemma~\ref{eqn:isom-convolution-Xim} we have
\begin{multline*}
 \Hom_{\sfD^0_{U,U}}(\Pi^0_{U,U}(\Xi^\wedge_!), \Pi^0_{U,U}(\Xi_!^{(m)}) \pstar^0_U \Pi^0_{U,U}(\Xi_!^{(m)})) \\
 \cong \Hom_{\sfD^0_{U,U}}(\Pi^0_{U,U}(\Xi^\wedge_!),\Pi^0_{U,U} ( \mathsf{C}_m(\Xi^\wedge_! \hatstar_U \Xi^\wedge_!) ) ),
\end{multline*}
and by Lemma~\ref{lem:morph-Xiwedge-Pi} the right-hand side identifies with
\[
 \Hom_{\sfP_{U,U}^\wedge}(\Xi^\wedge_!, \mathsf{C}_m((\Xi^\wedge_! \hatstar_U \Xi^\wedge_!) )).
\]
By projectivity of $\Xi^\wedge_!$ (see Lemma~\ref{lem:Xi-projective}) this space identifies with
\begin{equation*}
\Hom_{\sfP^\wedge_{U,U}}(\Xi^\wedge_!,\Xi_!^\wedge \hatstar_U \Xi_!^\wedge) \otimes_{\scO(T^\vee_\bk \times_{T^\vee_\bk / \Wf} T^\vee_\bk)} \scO((T^\vee_\bk \times_{T^\vee_\bk / \Wf} T^\vee_\bk)^{(m)}),
\end{equation*}
which itself, in view of~\eqref{eqn:Hom-Xi-monoidal} and~\eqref{eqn:isom-End-Xi}, identifies with
\begin{multline*}
 \bigl( \scO(T^\vee_\bk \times_{T^\vee_\bk / \Wf} T^\vee_\bk) \otimes_{\scO(T^\vee_\bk)} \scO(T^\vee_\bk \times_{T^\vee_\bk / \Wf} T^\vee_\bk) \bigr) \\
 \otimes_{\scO(T^\vee_\bk \times_{T^\vee_\bk / \Wf} T^\vee_\bk)} \scO((T^\vee_\bk \times_{T^\vee_\bk / \Wf} T^\vee_\bk)^{(m)}).
\end{multline*}
It is easily seen that under these identifications the morphism~\eqref{eqn:morph-monoidality-DUU} identifies with the natural isomorphism
\begin{multline*}
 \scO((T^\vee_\bk \times_{T^\vee_\bk / \Wf} T^\vee_\bk)^{(m)}) \otimes_{\scO(T^\vee_\bk)} \scO((T^\vee_\bk \times_{T^\vee_\bk / \Wf} T^\vee_\bk)^{(m)}) \cong \\
 \bigl( \scO(T^\vee_\bk \times_{T^\vee_\bk / \Wf} T^\vee_\bk) \otimes_{\scO(T^\vee_\bk)} \scO(T^\vee_\bk \times_{T^\vee_\bk / \Wf} T^\vee_\bk) \bigr) \\
 \otimes_{\scO(T^\vee_\bk \times_{T^\vee_\bk / \Wf} T^\vee_\bk)} \scO((T^\vee_\bk \times_{T^\vee_\bk / \Wf} T^\vee_\bk)^{(m)}),
 \end{multline*}
and is therefore an isomorphism.


Now we prove that~\eqref{eqn:morph-monoidality-DUU} is an isomorphism in case $\scF=\Pi^0_{U,U}(\Xi_!^{(m)})$ and $\scG$ is arbitrary. For this, since we already know the equivalence
\[
\sfP_{U,U}^{(m),0} \simto \Coh \bigl( (T^\vee_\bk \times_{T^\vee_\bk / \Wf} T^\vee_\bk)^{(m)} \bigr)
\]
we know that there exists a presentation
\[
\bigl( \Pi^0_{U,U}(\Xi_!^{(m)}) \bigr)^{\oplus r} \to \bigl( \Pi^0_{U,U}(\Xi_!^{(m)}) \bigr)^{\oplus s} \to \scG \to 0
\]
for some $r,s \in \Z_{\geq 0}$. Then we have exact sequences
\begin{multline*}
\Hom_{\sfD^0_{U,U}}(\Pi^0_{U,U}(\Xi^\wedge_!),\Pi^0_{U,U}(\Xi_!^{(m)})) \otimes_{\scO(T^\vee_\bk)} \Hom_{\sfD^0_{U,U}}(\Pi^0_{U,U}(\Xi^\wedge_!),\Pi^0_{U,U}(\Xi_!^{(m)})^{\oplus r}) \to \\
\Hom_{\sfD^0_{U,U}}(\Pi^0_{U,U}(\Xi^\wedge_!),\Pi^0_{U,U}(\Xi_!^{(m)})) \otimes_{\scO(T^\vee_\bk)} \Hom_{\sfD^0_{U,U}}(\Pi^0_{U,U}(\Xi^\wedge_!),\Pi^0_{U,U}(\Xi_!^{(m)})^{\oplus s}) \to \\
\Hom_{\sfD^0_{U,U}}(\Pi^0_{U,U}(\Xi^\wedge_!),\Pi^0_{U,U}(\Xi_!^{(m)})) \otimes_{\scO(T^\vee_\bk)} \Hom_{\sfD^0_{U,U}}(\Pi^0_{U,U}(\Xi^\wedge_!), \scG) \to 0
\end{multline*}
and
\begin{multline*}
\Hom_{\sfD^0_{U,U}}(\Pi^0_{U,U}(\Xi^\wedge_!), \Pi^0_{U,U}(\Xi^{(m)}_!) \pstar^0_U \Pi^0_{U,U}(\Xi_!^{(m)})^{\oplus r}) \to \\
\Hom_{\sfD^0_{U,U}}(\Pi^0_{U,U}(\Xi^\wedge_!), \Pi^0_{U,U}(\Xi^{(m)}_!) \pstar^0_U \Pi^0_{U,U}(\Xi_!^{(m)})^{\oplus s}) \to \\
\Hom_{\sfD^0_{U,U}}(\Pi^0_{U,U}(\Xi^\wedge_!), \Pi^0_{U,U}(\Xi^{(m)}_!) \pstar^0_U \scG) \to 0,
\end{multline*}
which are related by the corresponding morphisms~\eqref{eqn:morph-monoidality-DUU}. The morphisms relating the first two terms are isomorphisms by the case treated above, hence the one relating the third terms is also an isomorphism, which finishes the proof of the case under consideration.

Finally, one passes from this case to the case of general $\scF,\scG$ using similar arguments, which finishes the proof of the proposition.
\end{proof}

\subsection{Images of truncated costandard objects}

Recall that in~\S\ref{ss:representations-Iadj} we have defined some representations $(\scM_w : w \in W)$ of $\bbI_\Sigma$. By definition, in case $w \in \Wf$, the representation $\scM_w$ is a coherent sheaf on $T^\vee_\bk \times_{T^\vee_\bk / \Wf} T^\vee_\bk$, endowed with the trivial structure as a representation of $\bbI_\Sigma$. In this subsection these representations will be simply considered as coherent sheaves on $T^\vee_\bk \times_{T^\vee_\bk / \Wf} T^\vee_\bk$, or equivalently as finitely generated $\scO(T^\vee_\bk \times_{T^\vee_\bk / \Wf} T^\vee_\bk)$-modules.

Recall the functors $\sfC^0_m$ of~\S\ref{ss:truncation-ps}. The comments in~\S\ref{ss:cat-sheaves-G/U} show that if we consider an object $\scF \in \sfP_{U,U}^\wedge$, seen as an object in $\sfP^\wedge_{\Iwu,\Iwu}$, for any $m \geq 1$ the object $\sfC^0_m(\scF)$ belongs to the full subcategory $\sfP_{U,U}^0 \subset \sfP_{\Iwu,\Iwu}^0$. In this way we obtain a projective system $(\sfC^0_m(\scF) : m \geq 1)$ of objects in $\sfP_{U,U}^0$.

\begin{lem}
\label{lem:PhiUU-cost}
For any $w \in \Wf$, we have an isomorphism of projective systems
 \[
  (\Phi_{U,U}(\sfC^0_m(\nabla_w^\wedge)) : m \geq 1) \cong
 (\scM_w / \cJ^m \cdot \scM_w : m \geq 1).
\]
\end{lem}

\begin{proof}
 For any $m \geq 1$, by Lemma~\ref{lem:morph-Xiwedge-Pi} there exists a canonical isomorphism
 \[
  \Phi_{U,U}(\sfC^0_m(\nabla_w^\wedge)) \cong \Hom_{\sfP^{\wedge}_{U,U}}(\Xi^\wedge_!, \sfC_m(\nabla_w^\wedge)).
 \]
By projectivity of $\Xi^\wedge_!$ (see Lemma~\ref{lem:Xi-projective}), the right-hand side identifies with
\[
 \Hom_{\sfP^{\wedge}_{U,U}}(\Xi^\wedge_!, \nabla_w^\wedge) \otimes_{\scO(T^\vee_\bk)} \bigl( \scO(T^\vee_\bk) / \cJ^m \cdot \scO(T^\vee_\bk) \bigr).
\]
Finally, it is known that the standard object $\Delta^\wedge_w$ appears with multiplicity $1$ in $\Xi^\wedge_!$, see~\cite[\S 9.1]{bezr}; we deduce an isomorphism of $\scO(\FN_{T^\vee_\bk \times_{T^\vee_\bk / \Wf} T^\vee_\bk}(\{(e,e)\}))$-modules $\Hom_{\sfP^{\wedge}_{U,U}}(\Xi^\wedge_!, \nabla_w^\wedge) \cong \scM^\wedge_w$, which finishes the proof.
\end{proof}

\section{Truncation of perverse sheaves}
\label{sec:truncation}

We continue with the setting and assumptions of Section~\ref{sec:Perv-G/U}.
In this section we prove a number of technical statements regarding the functors $\sfC_m$ and $\sfC_m^0$ introduced in~\S\ref{ss:truncation-ps}. These results will be used in the next section in our study of 
the category $\sfP^0_{\Iwu,\Iwu}$.

\subsection{Flatness of standard, costandard and Wakimoto sheaves}
\label{ss:flatness-tilting}

In~\S\ref{ss:flatness-in-cat} we recall what it means for a module in a category to be flat. Here we will be more specifically interested in the case of $\scO(T^\vee_\bk)$-modules in $\mathsf{P}^\wedge_{\Iwu,\Iwu}$. We will say that an object of $\sfP^\wedge_{\Iwu,\Iwu}$ is $\scO(T^\vee_\bk)$-flat if its image under~\eqref{eqn:monodromy-Mod} is flat; in other words, $\scF \in \sfP^\wedge_{\Iwu,\Iwu}$ is $\scO(T^\vee_\bk)$-flat iff the functor
\[
(-) \otimes_{\scO(T^\vee_\bk)} \scF : \Mof(\scO(T^\vee_\bk)) \to \mathsf{P}^\wedge_{\Iwu,\Iwu}
\]
is exact.

Our goal in this subsection is to show that standard perverse sheaves, costandard perverse sheaves and free-monodromic Wakimoto sheaves are $\scO(T^\vee_\bk)$-flat. We start with standard and costandard sheaves.


\begin{lem}
\label{lem:flatness-D-N}
For any $w \in W$, the objects $\Delta^\wedge_w$ and $\nabla^\wedge_w$ are $\scO(T^\vee_\bk)$-flat. As a consequence, every object in $\mathsf{P}^\wedge_{\Iwu,\Iwu}$ which admits a standard or a costandard filtration (in particular, any tilting object) is $\scO(T^\vee_\bk)$-flat.
\end{lem}

\begin{proof}
We will prove the claim for the objects $\Delta_w^\wedge$; the case of the objects $\nabla^\wedge_w$ can be treated similarly, and the claim about objects with a standard or costandard filtration then follows in view of Lemma~\ref{lem:flatness}\eqref{it:flatness-2}.

By Proposition~\ref{prop:exactness-pushforward}, the functor
\[
(\widetilde{\jmath}_w)_! : \sfD^\wedge_{\Iwu}(\tFl_{G,w},\bk) \to \sfD^\wedge_{\Iwu,\Iwu}
\]
(where $(\widetilde{\jmath}_w)_!$ is as in~\S\ref{ss:completed-perverse})
is t-exact, and by~\cite[Proposition~4.5]{bezr} we have a canonical equivalence of triangulated categories
\[
\phi_w : \Db \Mof(\scO(\FN_{T^\vee_\bk}(\{e\}))) \simto \sfD^\wedge_{\Iwu}(\tFl_{G,w},\bk).
\]
By definition, under this identification the perverse t-structure on $\sfD^\wedge_{\Iwu}(\tFl_{G,w},\bk)$ corresponds to the tautological t-structure on $\Db \Mof(\scO(\FN_{T^\vee_\bk}(\{e\})))$. The monodromy construction also provides an $\scO(T^\vee_\bk)$-module structure on any object in $\sfD^\wedge_{\Iwu}(\tFl_{G,w},\bk)$, which under the equivalence $\phi_w$ corresponds to the obvious $\scO(T^\vee_\bk)$-module structure on a complex of $\scO(\FN_{T^\vee_\bk}(\{e\}))$-modules. From these remarks we obtain that for any $N$ in $\Mof(\scO(T^\vee_\bk))$ we have a canonical isomorphism
\[
N \otimes_{\scO(T^\vee_\bk)} \Delta^\wedge_w \cong (\widetilde{\jmath}_w)_! \phi_w(N \otimes_{\scO(T^\vee_\bk)} \scO(\FN_{T^\vee_\bk}(\{e\}))).
\]
The functor on the right-hand side is t-exact by flatness of $\scO(\FN_{T^\vee_\bk}(\{e\}))$ over $\scO(T^\vee_\bk)$ and t-exactness of $(\widetilde{\jmath}_w)_!$ and $\phi_w$, which proves that $\Delta^\wedge_w$ is flat.
\end{proof}

\begin{lem}
\label{lem:Wak-flat-prelim}
For any finitely generated $\scO(T^\vee_\bk)$-module $M$ and any $w,y \in W$, the convolution
\[
\Delta^\wedge_w \hatstar (M \otimes_{\scO(T^\vee_\bk)} \nabla^\wedge_y)
\]
belongs to the heart of the perverse t-structure on $\sfD^\wedge_{\Iwu,\Iwu}$.
\end{lem}

\begin{proof}
As in the proof of Lemma~\ref{lem:flatness-D-N} we have
\[
M \otimes_{\scO(T^\vee_\bk)} \nabla^\wedge_y \cong (\widetilde{\jmath}_y)_* \phi_y(M \otimes_{\scO(T^\vee_\bk)} \scO(\FN_{T^\vee_\bk}(\{e\}))).
\]
By construction of the equivalence $\phi_y$ in~\cite[Proposition~4.5]{bezr} and Lemma~\ref{lem:pro-obj-RA}, if for $m \geq 1$ we denote by $\scF_m$ the $\bk$-local system on $\tFl_{G,y}$ corresponding the $\scO(T^\vee_\bk)$-module $M/\cK^m M$ (where $\cK$ is as in~\S\ref{ss:completions}), then we deduce that
\[
M \otimes_{\scO(T^\vee_\bk)} \nabla^\wedge_y \cong ``\varprojlim_m" (\widetilde{\jmath}_y)_* \scF_m [\ell(y)+\dim(T)],
\]
hence that
\[
\Delta^\wedge_w \hatstar (M \otimes_{\scO(T^\vee_\bk)} \nabla^\wedge_y) \cong ``\varprojlim_m" \Delta^\wedge_w \hatstar \bigl( (\widetilde{\jmath}_y)_* \scF_m [\ell(y)+\dim(T)] \bigr).
\]
Each $\scF_m$ is an extension of copies of the constant local system, so that $\Delta^\wedge_w \hatstar \bigl( (\widetilde{\jmath}_y)_* \scF_m [\ell(y)+\dim(T)] \bigr)$ is an extension (in the sense of triangulated categories) of copies of
\[
\Delta^\wedge_w \hatstar \pi^\dag \For^{\Iw}_{\Iwu} (\nabla^{\Iw}_y) \cong \pi^\dag( \Delta^{\Iw}_w \star_{\Iw} \nabla^{\Iw}_y),
\]
where the isomorphism follows from~\eqref{eqn:conv-formula-2}--\eqref{eqn:conv-formula-3}. Here the right-hand side is perverse (by~\cite[Lemma~4.1.7]{ar-book} and t-exactness of $\pi^\dag$), hence each $\Delta^\wedge_w \hatstar \bigl( (\widetilde{\jmath}_y)_* \scF_m [\ell(y)+\dim(T)] \bigr)$ is perverse. By Proposition~\ref{prop:perverse-t-str-completed-cat}, this implies the lemma.
\end{proof}

\begin{cor}
\label{cor:Wak-flat}
For any $\lambda \in X_*(T)$, the free-monodromic Wakimoto sheaf $\scW^\wedge_\lambda$ is $\scO(T^\vee_\bk)$-flat. As a consequence, every object of $\sfP^\wedge_{\Iwu,\Iwu}$ which admits a Wakimoto filtration is flat.
\end{cor}

\begin{proof}
Once again, using Lemma~\ref{lem:flatness}\eqref{it:flatness-2} it suffices to prove the first claim. Let $\lambda \in X_*(T)$, and
choose $\mu,\nu \in X_*^+(T)$ such that $\lambda = \mu-\nu$, so that $\scW^\wedge_\lambda \cong \nabla^\wedge_{\st(\mu)} \hatstar \Delta^\wedge_{\st(-\nu)}$. First, we claim that for any finitely generated $\scO(T^\vee_\bk)$-module $M$ we have
\begin{equation}
\label{eqn:Wak-tensor}
M \otimes_{\scO(T^\vee_\bk)} \scW^\wedge_\lambda \cong \nabla^\wedge_{\st(\mu)} \hatstar (M \otimes_{\scO(T^\vee_\bk)} \Delta^\wedge_{\st(-\nu)}).
\end{equation}
Indeed, choose a presentation $\scO(T^\vee_\bk)^{\oplus r} \to \scO(T^\vee_\bk)^{\oplus s} \to M \to 0$. We deduce an exact sequence
\[
(\Delta^\wedge_{\st(-\nu)})^{\oplus r} \to (\Delta^\wedge_{\st(-\nu)})^{\oplus s} \to M \otimes_{\scO(T^\vee_\bk)} \Delta^\wedge_{\st(-\nu)} \to 0.
\]
By right exactness of the functor $\nabla^\wedge_{\st(\mu)} \phatstar (-)$ (see Corollary~\ref{cor:exactness-convolution-D-N}\eqref{it:exactness-convolution-N}), we deduce an exact sequence
\[
(\nabla^\wedge_{\st(\mu)} \phatstar\Delta^\wedge_{\st(-\nu)})^{\oplus r} \to (\nabla^\wedge_{\st(\mu)} \phatstar \Delta^\wedge_{\st(-\nu)})^{\oplus s}
\to \nabla^\wedge_{\st(\mu)} \phatstar (M \otimes_{\scO(T^\vee_\bk)} \Delta^\wedge_{\st(-\nu)}) \to 0.
\]
Here the first two terms identify with $(\scW^\wedge_\lambda)^{\oplus r}$ and $(\scW^\wedge_\lambda)^{\oplus s}$ respectively, and by Lemma~\ref{lem:Wak-flat-prelim} one can replace $\phatstar$ by $\hatstar$ in the third term. We deduce~\eqref{eqn:Wak-tensor}.

Now, consider an exact sequence $M_1 \hookrightarrow M_2 \twoheadrightarrow M_3$ of finitely generated $\scO(T^\vee_\bk)$-module. 
By Lemma~\ref{lem:flatness-D-N}, the sequence
\[
0 \to M_1 \otimes_{\scO(T^\vee_\bk)} \nabla^\wedge_{\st(\mu)} \to M_2 \otimes_{\scO(T^\vee_\bk)} \nabla^\wedge_{\st(\mu)} \to M_3 \otimes_{\scO(T^\vee_\bk)} \nabla^\wedge_{\st(\mu)} \to 0
\]
is exact. Applying the triangulated functor $\nabla^\wedge_{\st(\mu)} \hatstar (-)$ we deduce a distinguished triangle
\begin{multline*}
\nabla^\wedge_{\st(\mu)} \hatstar (M_1 \otimes_{\scO(T^\vee_\bk)} \nabla^\wedge_{\st(\mu)}) \to \nabla^\wedge_{\st(\mu)} \hatstar (M_2 \otimes_{\scO(T^\vee_\bk)} \nabla^\wedge_{\st(\mu)}) \\
\to \nabla^\wedge_{\st(\mu)} \hatstar (M_3 \otimes_{\scO(T^\vee_\bk)} \nabla^\wedge_{\st(\mu)} ) \xrightarrow{[1]}.
\end{multline*}
By~\eqref{eqn:Wak-tensor} all terms here are in the heart of the perverse t-structure, and this triangle corresponds to an exact sequence
\[
0 \to M_1 \otimes_{\scO(T^\vee_\bk)} \scW^\wedge_\lambda \to M_2 \otimes_{\scO(T^\vee_\bk)} \scW^\wedge_\lambda \to M_3 \otimes_{\scO(T^\vee_\bk)} \scW^\wedge_\lambda \to 0
\]
in $\sfP^\wedge_{\Iwu,\Iwu}$, which finishes the proof.
\end{proof}

The main consequence of these results that we will use below is the following.

\begin{prop}
\label{prop:exactness-Cm}
Consider an exact sequence
\[
0 \to \scF_1 \to \scF_2 \to \scF_3 \to 0
\]
in $\sfP^\wedge_{\Iwu,\Iwu}$. If $\scF_3$ admits either a standard filtration, or a costandard filtration, or a Wakimoto filtration, then for any $m \geq 1$ the induced sequence
\[
0 \to \sfC_m(\scF_1) \to \sfC_m(\scF_2) \to \sfC_m(\scF_3) \to 0
\]
is exact.
\end{prop}

\begin{proof}
In view of the definition of $\sfC_m$, the claim follows from Lemma~\ref{lem:flatness}\eqref{it:flatness-1} and either Lemma~\ref{lem:flatness-D-N} (in the first two cases) or Corollary~\ref{cor:Wak-flat} (in the third case).
\end{proof}

\subsection{Truncation functors for perverse sheaves: monoidality}
\label{ss:truncation-functor-mon}

We will now study some monoidality properties of the functors $\sfC_m$.

Given $\scF,\scG$ in $\mathsf{P}^\wedge_{\Iwu,\Iwu}$ we have canonical maps
\[
 \scF \to \mathsf{C}_m(\scF), \qquad \scG \to \mathsf{C}_m(\scG),
\]
which give rise to a natural morphism
\[
 \pH^0(\scF \hatstar \scG) \to \pH^0(\mathsf{C}_m(\scF) \star_{\Iwu} \mathsf{C}_m(\scG)).
\]
It follows from~\eqref{eqn:monodromy-convolution-1} that the action of $\scO(T^\vee_\bk)$ on the right-hand side vanishes on $\mathcal{J}^m$; this morphism therefore factors uniquely though a morphism
\[
 \mathsf{C}_m(\pH^0(\scF \hatstar \scG)) \to \pH^0(\mathsf{C}_m(\scF) \star_{\Iwu} \mathsf{C}_m(\scG)).
\]
Finally, applying the functor $\Pi^0_{\Iwu,\Iwu}$ and using the definition of the bifunctor $\pstar^0_{\Iwu}$ (see~\S\ref{ss:mon-reg-quotient}), we obtain a canonical morphism
\begin{equation}
\label{eqn:monoidality-C0-morphism}
  \mathsf{C}_m^0(\pH^0(\scF \hatstar \scG)) \to \mathsf{C}_m^0(\scF) \pstar^0_{\Iwu} \mathsf{C}_m^0(\scG).
\end{equation}

\begin{lem}
\label{lem:monoidality-C0}
 Assume that one of functors
 \[
  \pH^0(\scF \hatstar (-)), \, \pH^0((-) \hatstar \scG) : \mathsf{P}^\wedge_{\Iwu,\Iwu} \to \mathsf{P}^\wedge_{\Iwu,\Iwu}
 \]
 is right exact. Then~\eqref{eqn:monoidality-C0-morphism} is an isomorphism.
\end{lem}

\begin{proof}
 We will write the proof in case the functor $\pH^0(\scF \hatstar (-))$ is right exact; the other case is similar. Choosing a family of $r$ generators of the ideal $\mathcal{J}^m$, we obtain an exact sequence
 \[
  \scG^{\oplus r} \to \scG \to \mathsf{C}^m(\scG) \to 0.
 \]
Applying the right-exact functor $\pH^0(\scF \hatstar (-))$, we deduce an exact sequence
 \[
  \pH^0(\scF \hatstar \scG)^{\oplus r} \to \pH^0(\scF \hatstar \scG) \to \pH^0(\scF \hatstar \mathsf{C}^m(\scG)) \to 0,
 \]
 which provides a canonical isomorphism
 \[
  \mathsf{C}_m(\pH^0(\scF \hatstar \scG)) \cong \pH^0(\scF \hatstar \mathsf{C}_m(\scG)).
 \]
Applying $\Pi^0_{\Iwu,\Iwu}$, we deduce a canonical isomorphism
\[
  \mathsf{C}_m^0(\pH^0(\scF \hatstar \scG)) \cong \scF \phatstar^0 \mathsf{C}_m^0(\scG).
 \]
 At this point, to conclude it suffices to show that the morphism
 \[
  \scF \phatstar^0 \mathsf{C}_m^0(\scG) \to \mathsf{C}_m^0(\scF) \pstar^0_{\Iwu} \mathsf{C}_m^0(\scG)
 \]
induced by the morphism $\scF \to \mathsf{C}^m(\scF)$ is an isomorphism.
 
Our choice of generators for $\mathcal{J}^m$ also provides an exact sequence
 \[
  \scF^{\oplus r} \to \scF \to \mathsf{C}^m(\scF) \to 0.
 \]
By Lemma~\ref{lem:t-exactness-conv-tFl} the functor
$(-) \phatstar^0 \mathsf{C}_m^0(\scG) : \mathsf{P}^\wedge_{\Iwu,\Iwu} \to \mathsf{P}^0_{\Iwu,\Iwu}$
is right exact; we therefore deduce an exact sequence
\[
  (\scF \phatstar^0 \mathsf{C}_m^0(\scG))^{\oplus r} \to \scF \phatstar^0 \mathsf{C}_m^0(\scG) \to \mathsf{C}^m(\scF) \pstar_{\Iwu}^0 \mathsf{C}_m^0(\scG) \to 0.
 \]
By~\eqref{eqn:monodromy-convolution-1}, $\mathcal{J}^m$ acts trivially on $\scF \phatstar^0 \mathsf{C}_m^0(\scG)$, hence the first map in this sequence vanishes; we deduce that the second arrow is an isomorphism, as desired.
\end{proof}

\begin{prop}
\label{prop:monoidality-C0}
 Assume that we are in one of the following settings:
 \begin{enumerate}
  \item either $\scF$ or $\scG$ admits a costandard filtration;
  \item $\scF=\Delta_w^\wedge$ and $\scG=\Delta_y^\wedge$ for some $w,y \in W$ such that $\ell(wy)=\ell(w)+\ell(y)$.
 \end{enumerate}
Then~\eqref{eqn:monoidality-C0-morphism} is an isomorphism.
\end{prop}

\begin{proof}
To treat the first case,
 in view of Lemma~\ref{lem:monoidality-C0} it suffices to show that if $\scF$ admits a costandard filtration the functors
  \[
  \pH^0(\scF \hatstar (-)), \, \pH^0((-) \hatstar \scF) : \mathsf{P}^\wedge_{\Iwu,\Iwu} \to \mathsf{P}^\wedge_{\Iwu,\Iwu}
 \]
 are right exact. For that, it suffices to remark that the functors
  \[
  \scF \hatstar (-), \, (-) \hatstar \scF : \sfD^\wedge_{\Iwu,\Iwu} \to \sfD^\wedge_{\Iwu,\Iwu}
 \]
 are right t-exact, as follows from Corollary~\ref{cor:exactness-convolution-D-N}\eqref{it:exactness-convolution-N}.
 
 Now, let us assume that $\scF=\Delta_w^\wedge$ and $\scG=\Delta_y^\wedge$ for some $w,y \in W$ such that $\ell(wy)=\ell(w)+\ell(y)$. Then by Lemma~\ref{lem:convolution-stand-costand}\eqref{it:convolution-stand-costand-2} we have $\scF \hatstar \scG \cong \Delta^\wedge_{wy}$. Recall from the proof of Lemma~\ref{lem:flatness-D-N} that for any $x \in W$ we have
 \[
  \mathsf{C}_m(\Delta^\wedge_x) \cong (\widetilde{\jmath}_x)_! \phi_x(\scO(T^\vee_\bk)/\mathcal{J}^m \scO(T^\vee_\bk)).
 \]
Now, using~\cite[Lemma~3.4]{bezr} and considerations similar to those encountered in the proof of Lemma~\ref{lem:DN-weights},
it is not difficult to check that we have an isomorphism
\[
 \Delta^\wedge_w \hatstar \bigl( (\widetilde{\jmath}_y)_! \phi_y(\scO(T^\vee_\bk)/\mathcal{J}^m \scO(T^\vee_\bk)) \bigr) \simto (\widetilde{\jmath}_{wy})_! \phi_{wy}(\scO(T^\vee_\bk)/\mathcal{J}^m \scO(T^\vee_\bk)),
\]
i.e.~an isomorphism
\[
 \Delta^\wedge_w \hatstar \mathsf{C}_m(\Delta^\wedge_y) \cong \mathsf{C}_m(\Delta^\wedge_{wy}).
 \]
 Once this is known, the same arguments as in the final part of the proof of Lemma~\ref{lem:monoidality-C0} show the desired claim.
\end{proof}

Recall from~\S\ref{ss:tilting-perv} that the subcategory $\mathsf{T}^\wedge_{\Iwu,\Iwu} \subset \sfD^\wedge_{\Iwu,\Iwu}$ is closed under the convolution product $\hatstar$.
Proposition~\ref{prop:monoidality-C0} implies in particular that for any $m \geq 1$, the functor $\mathsf{C}_m^0$ induces a monoidal functor
\[
 (\mathsf{T}^\wedge_{\Iwu,\Iwu},\hatstar) \to (\sfP^0_{\Iwu,\Iwu},\pstar^0_{\Iwu}).
\]

\subsection{Truncation functors for perverse sheaves: fully faithfulness}
\label{ss:truncation-functor-ff}


We now prove a statement that will allow us to describe tilting objects from their images under the functors $\sfC_m^0$.

\begin{lem}
\label{lem:Cm-properties}
Let $\scT,\scT'$ in $\mathsf{T}^\wedge_{\Iwu,\Iwu}$.
\begin{enumerate}
\item
\label{it:Cm-tilting-2}
The functors $\mathsf{C}_m$ induce an isomorphism
\[
\Hom_{\mathsf{P}^\wedge_{\Iwu,\Iwu}}(\scT,\scT') \simto \varprojlim_m \Hom_{\sfP_{\Iwu,\Iwu}}(\mathsf{C}_m(\scT), \mathsf{C}_m(\scT')).
\]
\item
\label{it:Cm-tilting-3}
For any $m \geq 1$, the functor $\Pi_{\Iwu,\Iwu}^0$ induces an isomorphism
\[
\Hom_{\sfP_{\Iwu,\Iwu}}(\mathsf{C}_m(\scT), \mathsf{C}_m(\scT')) \simto \Hom_{\sfP_{\Iwu,\Iwu}^0}(\mathsf{C}^0_m(\scT), \mathsf{C}^0_m(\scT')).
\]
\end{enumerate}
\end{lem}

\begin{proof}

\eqref{it:Cm-tilting-2}
By definition of the tensor product (and functoriality of monodromy), for any $m \geq 1$ we have
\[
\Hom_{\sfP_{\Iwu,\Iwu}}(\mathsf{C}_m(\scT), \mathsf{C}_m(\scT')) \cong \Hom_{\sfP^\wedge_{\Iwu,\Iwu}}(\scT, \mathsf{C}_m(\scT')).
\]
The isomorphism we have to prove can therefore by written as
\[
\Hom_{\mathsf{P}^\wedge_{\Iwu,\Iwu}}(\scT,\scT') \simto \varprojlim_m \Hom_{\sfP^\wedge_{\Iwu,\Iwu}}(\scT, \mathsf{C}_m(\scT')).
\]
There is an obvious (bifunctorial) map from the left-hand side to the right-hand side, for any $\scT,\scT'$ in $\sfP^\wedge_{\Iwu,\Iwu}$;
we will prove that this map is an isomorphism when $\scT$ has a standard filtration and $\scT'$ has a costandard filtration, by induction on the sum of the lengths of these filtrations.

First, if $\scT = \Delta^\wedge_w$ and $\scT'=\nabla^\wedge_y$ for some $w,y \in W$, then we have
\[
\Hom_{\mathsf{P}^\wedge_{\Iwu,\Iwu}}(\scT,\scT') \cong
\begin{cases}
\scO(\FN_{T^\vee_\bk}(\{e\})) & \text{if $w=y$;} \\
0 & \text{otherwise.}
\end{cases}
\]
On the other hand, from the description of $\mathsf{C}_m(\nabla^\wedge_y)$ in the proof of Lemma~\ref{lem:flatness-D-N} we deduce that
\[
\Hom_{\mathsf{P}^\wedge_{\Iwu,\Iwu}}(\scT,\mathsf{C}_m(\scT')) \cong
\begin{cases}
\scO(T^\vee_\bk) / \mathcal{J}^m \cdot \scO(T^\vee_\bk) & \text{if $w=y$;} \\
0 & \text{otherwise.}
\end{cases}
\]
The claim is therefore clear in this case.

Now, assume that the object $\scT$ admits a standard filtration, that we have an exact sequence
\[
0 \to \nabla^\wedge_y \to \scT' \to \scT'' \to 0
\]
where $\scT''$ has a costandard filtration, and that the claim is known for the pairs $(\scT,\nabla^\wedge_y)$ and $(\scT,\scT'')$. By Proposition~\ref{prop:exactness-Cm},
for any $m \geq 1$ we then have an exact sequence
\[
0 \to \mathsf{C}_m(\nabla^\wedge_y) \to \mathsf{C}_m(\scT') \to \mathsf{C}_m(\scT'') \to 0.
\]
The description of $\mathsf{C}_m(\nabla^\wedge_y)$ in the proof of Lemma~\ref{lem:flatness-D-N} shows that
\[
\Ext^1_{\mathsf{P}^\wedge_{\Iwu,\Iwu}}(\scT,\mathsf{C}_m(\nabla^\wedge_y))=0;
\]
we therefore obtain an exact sequence
\begin{multline*}
0 \to \Hom_{\mathsf{P}^\wedge_{\Iwu,\Iwu}}(\scT,\mathsf{C}_m(\nabla^\wedge_y)) \to \Hom_{\mathsf{P}^\wedge_{\Iwu,\Iwu}}(\scT,\mathsf{C}_m(\scT')) \\
\to \Hom_{\mathsf{P}^\wedge_{\Iwu,\Iwu}}(\scT,\mathsf{C}_m(\scT'')) \to 0.
\end{multline*}
The inverse system $(\Hom_{\mathsf{P}^\wedge_{\Iwu,\Iwu}}(\scT,\mathsf{C}_m(\nabla^\wedge_y)) : m \geq 1)$ is an inverse system of finite-dimensional $\bk$-vector spaces; it therefore automatically satisfies the Mittag-Leffler condition, which implies that the sequence
\begin{multline*}
0 \to \varprojlim_m \Hom_{\mathsf{P}^\wedge_{\Iwu,\Iwu}}(\scT,\mathsf{C}_m(\nabla^\wedge_y)) \to \varprojlim_m \Hom_{\mathsf{P}^\wedge_{\Iwu,\Iwu}}(\scT,\mathsf{C}_m(\scT')) \\
\to \varprojlim_m \Hom_{\mathsf{P}^\wedge_{\Iwu,\Iwu}}(\scT,\mathsf{C}_m(\scT'')) \to 0
\end{multline*}
is exact. Similarly we have an exact sequence
\[
0 \to \Hom_{\mathsf{P}^\wedge_{\Iwu,\Iwu}}(\scT,\nabla^\wedge_y) \to \Hom_{\mathsf{P}^\wedge_{\Iwu,\Iwu}}(\scT,\scT') \to \Hom_{\mathsf{P}^\wedge_{\Iwu,\Iwu}}(\scT,\scT'') \to 0.
\]
Our maps for the pairs $(\scT,\nabla^\wedge_y)$, $(\scT,\scT')$ and $(\scT,\scT'')$ define a morphism of exact sequences; since the first and third maps are isomorphisms by assumption, the second one is also an isomorphism by the five lemma.

Finally, very similar arguments show that if the object $\scT'$ admits a costandard filtration, if we have an exact sequence
\[
0 \to \scT'' \to \scT \to \Delta^\wedge_w \to 0
\]
such that $\scT''$ has a standard filtration, and if the claim is known for the pairs $(\scT'',\scT')$ and $(\Delta^\wedge_w, \scT')$, then it follows for the pair $(\scT,\scT')$, which finishes the proof.

\eqref{it:Cm-tilting-3}
The claim will follow from the description of morphisms in a Serre quotient category provided we prove that $\mathsf{C}_m(\scT)$ does not admit a nonzero morphism to an object of the form $\pi^\dag \For^{\Iw}_{\Iwu}(\IC_w)$ with $\ell(w)>0$ and $\mathsf{C}_m(\scT')$ does not admit a nonzero morphism from such an object. 
Using Proposition~\ref{prop:exactness-Cm}
we obtain that if $\scF$ admits a standard, resp.~costandard, filtration in $\mathsf{P}^\wedge_{\Iwu,\Iwu}$ then $\mathsf{C}_m(\scF)$ admits a filtration whose subquotients have the form $\mathsf{C}_m(\Delta^\wedge_w)$, resp.~$\mathsf{C}_m(\nabla^\wedge_w)$, with $w \in W$.
The desired claim will therefore follow if we prove that for $y \in W$ the object $\mathsf{C}_m(\nabla^\wedge_y)$ does not admit a nonzero morphism to an object of the form $\pi^\dag \For^{\Iw}_{\Iwu}(\IC_w)$ with $\ell(w)>0$, and $\mathsf{C}_m(\Delta^\wedge_y)$ does not admit a nonzero morphism from such an object. However, 
from the proof of Lemma~\ref{lem:flatness-D-N} we know that
$\mathsf{C}_m(\nabla^\wedge_y)$ is an extension of copies of $\pi^\dag \For^{\Iw}_{\Iwu}(\nabla^\Iw_y)$ and $\mathsf{C}_m(\Delta^\wedge_y)$ is an extension of copies of $\pi^\dag \For^{\Iw}_{\Iwu}(\Delta^\Iw_y)$; the claim therefore follows from the fact that the head of $\nabla^\Iw_y$ and the socle of $\Delta^\Iw_y$ are both of the form $\IC_z$ with $\ell(z)=0$, see~\cite[Lemma~4.5]{brr}.
\end{proof}

\section{Perverse sheaves on \texorpdfstring{$\tFl_G$}{Fl}}
\label{sec:construction}


We continue with the setting of Section~\ref{sec:Perv-G/U}, and make the following assumptions:
\begin{enumerate}
 \item the quotient of $X^*(T)$ by the root lattice of $(G,T)$ is free;
 \item the quotient of $X_*(T)$ by the coroot lattice of $(G,T)$ has no $\ell$-torsion;
 \item for any indecomposable factor in the root system of $(G,T)$, $\ell$ is strictly bigger than the corresponding value in Figure~\ref{fig:bounds}.
\end{enumerate}
(As in~\S\ref{ss:regular-quotient-coh}, we expect that the third assumption can be weakened. What will be used below is that the main result of~\cite{brr} holds.)
Our goal is to prove the first main result of the paper, Theorem~\ref{thm:main-intro-1}.

\subsection{Statement}

We will use the constructions of Section~\ref{sec:Coh-St}--\ref{sec:Hecke-cat}, for the group $\bG=G^\vee_\bk$, its Borel subgroup $B^\vee_\bk$, and its maximal torus $T^\vee_\bk$. In particular, we
fix a Steinberg section $\Sigma \subset G^\vee_\bk$ as in~\S\ref{ss:Steinberg-section}. Then we have the universal centralizer $\mathbb{J}_\Sigma \subset G^\vee_\bk \times \Sigma$, a smooth affine group scheme over $\Sigma$. 
We have a canonical morphism $T^\vee_\bk \times_{T^\vee_\bk / \Wf} T^\vee_\bk \to \Sigma$ obtained by composing the obvious projection $T^\vee_\bk \times_{T^\vee_\bk / \Wf} T^\vee_\bk \to T^\vee_\bk / \Wf$ with the inverse of the isomorphism $\Sigma \simto T^\vee_\bk / \Wf$, and the smooth affine group scheme
\[
\mathbb{I}_\Sigma = (T^\vee_\bk \times_{T^\vee_\bk / \Wf} T^\vee_\bk) \times_{\Sigma} \mathbb{J}_\Sigma
\]
over the affine scheme $T^\vee_\bk \times_{T^\vee_\bk / \Wf} T^\vee_\bk$. Consider as in~\S\ref{ss:representations-Iadj}
the abelian category $\Rep(\mathbb{I}_\Sigma)$ of representations of $\mathbb{I}_\Sigma$ which are finitely generated over $\scO(T^\vee_\bk \times_{T^\vee_\bk / \Wf} T^\vee_\bk)$, and its monoidal product $\circledast$. Recall that this product is right exact on both sides, and has as unit object $\scO_{\Delta T^\vee_\bk}$ where $\Delta T^\vee_\bk \subset T^\vee_\bk \times_{T^\vee_\bk / \Wf} T^\vee_\bk$ is the diagonal copy of $T^\vee_\bk$, with the trivial structure as a representation. We will also denote by
\[
\Rep_0(\mathbb{I}_\Sigma)
\]
the full subcategory of $\Rep(\mathbb{I}_\Sigma)$ whose objects are the representations which are set-theoretically supported on the base point $(e,e) \in T^\vee_\bk \times_{T^\vee_\bk / \Wf} T^\vee_\bk$ (i.e., whose restriction to the open complement vanishes). Using the notation of~\S\ref{ss:completions}, the objects in this subcategory can also be described as those on which the ideal $\cI$ acts nilpotently, or equivalently as those on which $\cJ$ acts nilpotently.

The subcategory $\Rep_0(\mathbb{I}_\Sigma) \subset \Rep(\mathbb{I}_\Sigma)$ is a nonunital monoidal subcategory. If $\Coh_{\{(e,e)\}}(T^\vee_\bk \times_{T^\vee_\bk / \Wf} T^\vee_\bk)$ is as in~\S\ref{ss:statement-G/U}, we have a fully faithful exact monoidal functor
\begin{equation}
\label{eqn:Coh-Rep-JD}
\Coh_{\{(e,e)\}}(T^\vee_\bk \times_{T^\vee_\bk / \Wf} T^\vee_\bk) \to \Rep_0(\mathbb{I}_\Sigma)
\end{equation}
sending a coherent sheaf to itself with the trivial structure as a representation. (This justifies our choice of notation for the monoidal products.)

Let now $\su \in \Sigma \subset G^\vee_\bk$ be the point corresponding to the image of $e \in T^\vee_\bk$ in $T^\vee_\bk/\Wf$. Then $\su$ is a regular unipotent element, so that as explained in~\S\ref{ss:regular-quotient-coh} the constructions of~\cite{brr} provide an equivalence of abelian monoidal categories
\[
\Phi_{\Iw,\Iw} : \left( \mathsf{P}^0_{\Iw,\Iw}, \star_\Iw^0 \right) \simto \left( \Rep(\rmZ_{G^\vee_\bk}(\su)),\otimes \right).
\]
The fiber of $\mathbb{I}_\Sigma$ over $(e,e)$ is by definition the fiber of $\mathbb{J}_\Sigma$ over $\su$, which identifies with $\rmZ_{G^\vee_\bk}(\su)$. The functor of pushforward along the closed embedding $\{(e,e)\} \hookrightarrow T^\vee_\bk \times_{T^\vee_\bk / \Wf} T^\vee_\bk$ therefore defines an exact fully faithful functor
\begin{equation}
\label{eqn:functor-statement-Rep}
\Rep(\rmZ_{G^\vee_\bk}(\su)) \to \Rep_0(\mathbb{I}_\Sigma),
\end{equation}
which is easily seen to admit a canonical monoidal structure. (The essential image of this functor consists of representations on which the ideal $\cI$ acts trivially.) 

On the other hand, recall the exact fully faithful monoidal functor
\[
\pi_0^\dag : \left( \mathsf{P}_{\Iw,\Iw}^0,\star_{\Iw}^0 \right) \to \left( \mathsf{P}_{\Iwu,\Iwu}^0,\pstar_{\Iwu}^0 \right)
\]
considered in~\S\ref{ss:relation-reg-quotient}, and the exact monoidal functor~\eqref{eqn:embedding-PUU-PII}. In this section we will prove the following theorem, which is a more precise version of Theorem~\ref{thm:main-intro-1}.

\begin{thm}
\label{thm:main}
There exists an equivalence of abelian monoidal categories
\[
\Phi_{\Iwu,\Iwu} : \left( \mathsf{P}_{\Iwu,\Iwu}^0,\pstar_{\Iwu}^0 \right) \simto \left( \Rep_0(\mathbb{I}_\Sigma), \circledast \right)
\]
such that the diagrams
\begin{equation}
\label{eqn:diag-compatibility-Phi}
\vcenter{
\xymatrix@C=1.5cm{
\mathsf{P}^0_{\Iw,\Iw} \ar[r]_-{\sim}^-{\Phi_{\Iw,\Iw}} \ar[d]_-{\pi_0^\dag} & \Rep(\rmZ_{G^\vee_\bk}(\su)) \ar[d]^-{\eqref{eqn:functor-statement-Rep}} \\
\mathsf{P}_{\Iwu,\Iwu}^0 \ar[r]_-{\sim}^-{\Phi_{\Iwu,\Iwu}} & \Rep_0(\mathbb{I}_\Sigma).
}
}
\end{equation}
and
\begin{equation}
\label{eqn:diag-compatibility-Phi-2}
\vcenter{
\xymatrix@C=1.5cm{
\mathsf{P}^0_{U,U} \ar[r]_-{\sim}^-{\Phi_{U,U}} \ar[d]_-{\eqref{eqn:embedding-PUU-PII}} & \Coh_{\{(e,e)\}}(T^\vee_\bk \times_{T^\vee_\bk / \Wf} T^\vee_\bk) \ar[d]^-{\eqref{eqn:Coh-Rep-JD}} \\
\mathsf{P}_{\Iwu,\Iwu}^0 \ar[r]_-{\sim}^-{\Phi_{\Iwu,\Iwu}} & \Rep_0(\mathbb{I}_\Sigma).
}
}
\end{equation}
commute up to isomorphism.
\end{thm}

The proof of this theorem will occupy the whole section. Our strategy will be to define an appropriate ``deformation'' of the functor $\Phi_{\Iw,\Iw}$ as described in~\S\ref{ss:regular-quotient-coh}, and check that this functor has the required properties by reducing most of them to the similar properties of $\Phi_{\Iw,\Iw}$ or $\Phi_{U,U}$.

\subsection{Truncation and completion of representations}
\label{ss:truncation-Rep}

As in~\S\ref{ss:completions} we consider the scheme $\FN_{T^\vee_\bk \times_{T^\vee_\bk / \Wf} T^\vee_\bk}(\{(e,e)\})$ 
and the smooth affine group scheme
\begin{multline*}
\bbI_\Sigma^\wedge = \FN_{T^\vee_\bk \times_{T^\vee_\bk / \Wf} T^\vee_\bk}(\{(e,e)\}) \times_{T^\vee_\bk \times_{T^\vee_\bk / \Wf} T^\vee_\bk} \bbI_\Sigma \\
\cong \FN_{T^\vee_\bk \times_{T^\vee_\bk / \Wf} T^\vee_\bk}(\{(e,e)\}) \times_{T^\vee_\bk / \Wf} \mathbb{J}_\Sigma
\end{multline*}
over the affine scheme $\FN_{T^\vee_\bk \times_{T^\vee_\bk / \Wf} T^\vee_\bk}(\{(e,e)\})$. Recall that the category $\Rep(\bbI_\Sigma^\wedge)$ of representations of this group scheme which are of finite type over the ring $\scO(\FN_{T^\vee_\bk \times_{T^\vee_\bk / \Wf} T^\vee_\bk}(\{(e,e)\}))$ admits a natural monoidal product $\oast$. We have a natural fully faithful exact monoidal functor
\begin{equation}
\label{eqn:functor-Rep-wedge}
(\Rep_0(\bbI_\Sigma), \oast) \to (\Rep(\bbI_\Sigma^\wedge), \oast)
\end{equation}
whose essential image consists of modules on which $\cI$ (equivalently, $\cJ$) acts nilpotently.


We also have ``truncation'' operations which produce objects in $\Rep_0(\bbI_\Sigma)$ out of objects in $\Rep(\bbI_\Sigma^\wedge)$. Namely, for $m \geq 1$ we can consider the functor
\[
\mathsf{D}_m : \Rep(\bbI_\Sigma^\wedge) \to \Rep_0(\bbI_\Sigma)
\]
given by restriction to the closed subscheme
\[
\bigl( T^\vee_\bk \times_{T^\vee_\bk / \Wf} T^\vee_\bk \bigr)^{(m)} \subset \FN_{T^\vee_\bk \times_{T^\vee_\bk / \Wf} T^\vee_\bk}(\{(e,e)\})
\]
from~\S\ref{ss:truncation-PUU}.
The following claim is clear from definitions.

\begin{lem}
\label{lem:truncation-Rep-mon}
Let $M,M' \in \Rep(\bbI_\Sigma^\wedge)$.
For any $m \geq 1$ we have
\[
\mathsf{D}_m(M \circledast M') \cong \mathsf{D}_m(M) \circledast \mathsf{D}_m(M').
\]
\end{lem}

\subsection{Extension of \texorpdfstring{$\tsZ$}{Z} to coherent sheaves and definition of \texorpdfstring{$\scR^\wedge$}{R}}
\label{ss:extension-Z-wedge}

Recall the category $\Coh^{G^\vee_\bk}_{\mathrm{fr}}(G^\vee_\bk)$ considered in~\S\ref{ss:extension-Z}. (Here the action of $G^\vee_\bk$ on itself is the adjoint action.) Applying Lemma~\ref{lem:extension-functor} to the monoidal functor
\[
\tsZ : \Rep(G^\vee_\bk) \to \D^\wedge_{\Iwu,\Iwu}
\]
and its automorphism $\widehat{\sm}_{(-)}$ (see~\S\ref{ss:fm-central-sheaves}) we obtain a canonical monoidal functor
\[
\sZ^{\wedge,\Coh} : \Coh^{G^\vee_\bk}_{\mathrm{fr}}(G^\vee_\bk) \to \D^\wedge_{\Iwu,\Iwu}
\]
taking values in the subcategory $\mathsf{P}^\wedge_{\Iwu,\Iwu}$. Recall that for any $V$ in $\Rep(G^\vee_\bk)$ and $\scF$ in $\D^\wedge_{\Iwu,\Iwu}$ the isomorphism $\widehat{\sigma}_{\Sat^{-1}(V),\scF}$ from Theorem~\ref{thm:gaitsgory-mon}\eqref{it:gaitsgory-mon-4} provides a canonical isomorphism
\begin{equation}
\label{eqn:commutation-Z-Satake}
\tsZ(V) \hatstar \scF \simto \scF \hatstar \tsZ(V),
\end{equation}
or in other words a canonical isomorphism
\begin{equation}
\label{eqn:commutation-ZCoh-Satake}
\sZ^{\wedge,\Coh}(V \otimes \scO_{G^\vee_\bk}) \hatstar \scF \simto \scF \hatstar \sZ^{\wedge,\Coh}(V \otimes \scO_{G^\vee_\bk}).
\end{equation}
We note the following property for later use.

\begin{lem}
\label{lem:commutation-ZCoh-Satake}
The isomorphisms~\eqref{eqn:commutation-ZCoh-Satake} define an isomorphism of bifunctors from $\Coh^{G^\vee_\bk}_{\mathrm{fr}}(G^\vee_\bk) \times \D^\wedge_{\Iwu,\Iwu}$ to $\D^\wedge_{\Iwu,\Iwu}$.
\end{lem}

\begin{proof}
Consider the monoidal $\bk$-linear additive category $\mathsf{A}$ of $\bk$-linear endofunctors of $\D^\wedge_{\Iwu,\Iwu}$ (with monoidal structure given by composition). We have two $\bk$-linear monoidal functors from $\Rep(G^\vee_\bk)$ to $\mathsf{A}$, sending respectively $V$ to the endofunctors $\tsZ(V) \hatstar (-)$ and $(-) \hatstar \tsZ(V)$. Each of these functors admits an automorphism, given by $\widehat{\sm}_V \hatstar (-)$ and $(-) \hatstar \widehat{\sm}_V$ respectively. Lemma~\ref{lem:extension-functor} provides extensions of these functors determined by the corresponding automorphism, which are given by $V \otimes \scO_{G^\vee_\bk} \mapsto \sZ^{\wedge,\Coh}(V) \hatstar (-)$ and $V \otimes \scO_{G^\vee_\bk} \mapsto (-) \hatstar \sZ^{\wedge,\Coh}(V)$ respectively. Now the isomorphism~\eqref{eqn:commutation-Z-Satake} intertwines $\widehat{\sm}_V \hatstar \id_{\scF}$ and $\id_\scF \hatstar \widehat{\sm}_V$, see Theorem~\ref{thm:gaitsgory-mon}\eqref{it:gaitsgory-mon-4}. By unicity in Lemma~\ref{lem:extension-functor}, this means that~\eqref{eqn:commutation-ZCoh-Satake} is an isomorphism of bifunctors, as desired.
\end{proof}

Recall (see~~\eqref{eqn:adj-quotient}) that the adjoint quotient $G^\vee_\bk / G^\vee_\bk$ identifies with $T^\vee_\bk / \Wf$.
The quotient morphism $G^\vee_\bk \to G^\vee_\bk / G^\vee_\bk$ provides, for any $\scF$ in $\Coh^{G^\vee_\bk}(G^\vee_\bk)$, a canonical algebra morphism
\[
\scO(G^\vee_\bk / G^\vee_\bk) \to \End_{\Coh^{G^\vee_\bk}(G^\vee_\bk)}(\scF),
\]
and therefore an algebra morphism $\scO(T^\vee_\bk / \Wf) \to \End(\scF)$. With these morphisms, the category $\Coh^{G^\vee_\bk}(G^\vee_\bk)$ becomes an $\scO(T^\vee_\bk / \Wf)$-linear category. 

\begin{lem}
\label{lem:monodromy-tZ-F}
 For any $V$ in $\Rep(G^\vee_\bk)$, the composition
 \[
  \scO(T^\vee_\bk / \Wf) \to \End_{\Coh^{G^\vee_\bk}(G^\vee_\bk)}(V \otimes_\bk \scO_{G^\vee_\bk}) \to \End_{\mathsf{P}^{\wedge}_{\Iwu,\Iwu}}(\tsZ(V))
 \]
 (where the second map is induced by $\sZ^{\wedge,\Coh}$)
coincides with the restriction of the morphism $\mu_V$ (see~\eqref{eqn:muV}) to $\scO(T^\vee_\bk / \Wf)$.
\end{lem}

\begin{proof}
 First we consider the case when $V=\bk$ is the trivial $G^\vee_\bk$-module. In this case, $\tsZ(\bk)$ is the unit object $\delta^\wedge$ in $\D^{\wedge}_{\Iwu,\Iwu}$. If $M$ is a finite-dimensional $G^\vee_\bk$-module, then we denote by $\mathrm{ch}_M \in \scO(G^\vee_\bk / G^\vee_\bk)=\scO(T^\vee_\bk / \Wf)$ the associated character. It is well known that these elements generate $\scO(T^\vee_\bk / \Wf)$ as a vector space, so that to prove the desired claim it suffices to check that our maps coincide on such elements. Now $\mathrm{ch}_M$ can be interpreted as the composition
 \[
  \scO(G^\vee_\bk) \to M \otimes \scO(G^\vee_\bk) \otimes M^* \to M \otimes \scO(G^\vee_\bk) \otimes M^* \to \scO(G^\vee_\bk)
 \]
where the first (resp.~third) morphism is induced by the canonical map $\bk \to M \otimes M^*$ (resp.~$M \otimes M^* \to \bk$), and the middle one is $\sm^{\mathrm{taut}}_M \otimes \id_{M^*}$. (See~\S\ref{ss:extension-Z} for the definition of $\sm^{\mathrm{taut}}_M$.) Therefore, its image in $\End(\delta^\wedge)$ is the composition
\[
 \delta^\wedge \to \tsZ(M) \hatstar \tsZ(M^*) \xrightarrow{\widehat{\sm}_M \hatstar \id_{\tsZ(M^*)}} \tsZ(M) \hatstar \tsZ(M^*) \to \delta^\wedge,
 \]
 where the first and third morphisms are the images of the maps considered above.
 This map has been computed in Lemma~\ref{lem:trace}, and is known to equal $\mu_{\delta^\wedge}(\mathrm{ch}_M \otimes 1)=\mu_\bk(\mathrm{ch}_M)$; this proves the desired claim in this case.
 
Now we deduce the general case. It is clear that the canonical morphism
 \[
  \scO(T^\vee_\bk / \Wf) \to \End_{\Coh^{G^\vee_\bk}(G^\vee_\bk)}(V \otimes_\bk \scO_{G^\vee_\bk})
 \]
is the composition
\[
\scO(T^\vee_\bk / \Wf) \to \End_{\Coh^{G^\vee_\bk}(G^\vee_\bk)}(\scO_{G^\vee_\bk}) \to \End_{\Coh^{G^\vee_\bk}(G^\vee_\bk)}(V \otimes_\bk \scO_{G^\vee_\bk})
\]
where the first map is the canonical morphism associated with the object $\scO_{G^\vee_\bk}$, and the second one is induced by the tensor product (on the left) with $V \otimes_\bk \scO_{G^\vee_\bk}$. Since $\sZ^{\wedge,\Coh}$ is monoidal, using the case already treated, it follows that its composition with the morphism induced by $\tsZ$ is the composition
\[
 \scO(T^\vee_\bk / \Wf) \xrightarrow{\mu_\bk} \End_{\mathsf{P}^{\wedge}_{\Iwu,\Iwu}}(\tsZ(\bk)) \to \End_{\mathsf{P}^{\wedge}_{\Iwu,\Iwu}}(\tsZ(V))
\]
where the second map is induced by convolution on the left with $\tsZ(V)$. Now, interpreting $\mu_V$ in terms of right monodromy, it is clear that $\mu_V=\id_{\tsZ(V)} \hatstar \mu_\bk$, which completes the proof.
\end{proof}

\begin{rmk}
\label{rmk:action-adjquo-Z}
Theorem~\ref{thm:gaitsgory-mon}\eqref{it:gaitsgory-mon-0} and the unicity in Lemma~\ref{lem:extension-functor} imply that we have $\pi_\dag \circ \sZ^{\wedge,\Coh} \cong \sZ^\Coh$.
Hence Lemma~\ref{lem:monodromy-tZ-F} implies that the composition of~\eqref{eqn:ZCoh-morph} with the natural morphism
\[
\scO(T^\vee_\bk/\Wf) \to \End_{\Coh^{G^\vee_\bk}(G^\vee_\bk)}(V \otimes_\bk \scO_{G^\vee_\bk})
\]
factors through the quotient $\scO(T^\vee_\bk/\Wf) \to \scO(T^\vee_\bk/\Wf) / \mathcal{J} = \bk$.
\end{rmk}

Since the complement of the open subset $G^\vee_{\bk,\reg} \subset G^\vee_\bk$ is known to have codimension at least $3$ (see~\cite[\S 4.13]{humphreys}), restriction induces an isomorphism $\scO(G^\vee_\bk) \simto \scO(G^\vee_{\bk,\reg})$, hence a fully faithful monoidal functor
\[
\Coh^{G^\vee_\bk}_{\mathrm{fr}}(G^\vee_\bk) \to \Coh^{G^\vee_\bk}(G^\vee_{\bk,\reg}).
\]
Recall that restriction to $\Sigma$ induces an equivalence of monoidal categories
\[
\Coh^{G^\vee_\bk}(G^\vee_{\bk,\reg}) \simto \Rep(\mathbb{J}_\Sigma),
\]
see Proposition~\ref{prop:rest-Steinberg-section}. We use this equivalence and the functor above to see the category $\Coh^{G^\vee_\bk}_{\mathrm{fr}}(G^\vee_\bk)$ as a full subcategory in $\Rep(\mathbb{J}_\Sigma)$. In these terms, the canonical functor $\Rep(G^\vee_\bk) \to \Rep(\mathbb{J}_\Sigma)$ is given by $V \mapsto V \otimes_\bk \scO_\Sigma$, and the $\scO(T^\vee_\bk/\Wf)$-linear structure on $\Coh^{G^\vee_\bk}_{\mathrm{fr}}(G^\vee_\bk)$ corresponds to the natural $\scO(\Sigma)$-linear structure on $\Rep(\mathbb{J}_\Sigma)$ via the identification $\Sigma \simto T^\vee_\bk/\Wf$.

Let us consider $\scO(G^\vee_\bk)$ with the $G^\vee_\bk$-module structure induced by multiplication on the left.
In~\S\ref{ss:extension-Z} we have considered a morphism
\begin{equation}
\label{eqn:morph-action-OG}
\scO(G^\vee_\bk) \to \End_{{\ind}\Coh^{G^\vee_\bk}(G^\vee_\bk)}(\scO(G^\vee_\bk) \otimes \scO_{G^\vee_\bk})
\end{equation}
constructed using the morphism
\[
\scO(G^\vee_\bk) \to \Hom_{{\ind}\Coh^{G^\vee_\bk}(G^\vee_\bk)}(\scO_{G^\vee_\bk}, \scO(G^\vee_\bk) \otimes \scO_{G^\vee_\bk})
= \scO(G^\vee_\bk \times G^\vee_\bk)^{G^\vee_\bk}
\]
induced by the map $G^\vee_\bk \times G^\vee_\bk \to G^\vee_\bk$ given by $(g,h) \mapsto g^{-1}hg$. In the terms above the right-hand side identifies with the space $\scO(G^\vee_\bk \times \Sigma)^{\mathbb{J}_\Sigma}$ of $\mathbb{J}_\Sigma$-invariant functions on $G^\vee_\bk \times \Sigma$, and our morphism
\[
\scO(G^\vee_\bk) \to \scO(G^\vee_\bk \times \Sigma)^{\mathbb{J}_\Sigma}
\]
is induced by the morphism $G^\vee_\bk \times \Sigma \to G^\vee_\bk$ given again by $(g,h) \mapsto g^{-1}hg$.

We now consider the ind-object $\tsZ(\scO(G^\vee_\bk)) \in {\ind}\D^\wedge_{\Iwu,\Iwu}$. This object is a ring ind-object, and taking the images of the morphisms above we obtain a ring morphism
\[
\scO(G^\vee_\bk) \to \Hom_{{\ind}\D^\wedge_{\Iwu,\Iwu}}(\delta^\wedge, \tsZ(\scO(G^\vee_\bk))) \subset \End_{{\ind}\D^\wedge_{\Iwu,\Iwu}}(\tsZ(\scO(G^\vee_\bk))),
\]
where the embedding is as in~\S\ref{ss:extension-Z}.
We can therefore consider the tensor product
\[
\scR^\wedge := \tsZ(\scO(G^\vee_\bk)) \otimes_{\scO(G^\vee_\bk)} \scO(\Sigma)
\]
in $\sfP^\wedge_{\Iwu,\Iwu}$ (based on the general construction recalled in~\S\ref{ss:mod-cat}),
which is the quotient of $\tsZ(\scO(G^\vee_\bk))$ by a left ideal. The same considerations as for $\sZ(\scO(G^\vee_\bk))$ based on the fact that $\tsZ$ is a central functor (see~\cite[p.~73]{bez} for details) imply that any left ideal in $\tsZ(\scO(G^\vee_\bk))$ is also a right ideal, so that $\scR^\wedge$ also has a canonical structure of ring object in ${\ind}\D^\wedge_{\Iwu,\Iwu}$, such that the surjection
\begin{equation}
\label{eqn:surj-OH}
\tsZ(\scO(G^\vee_\bk)) \to \scR^\wedge
\end{equation}
is a ring morphism.

\begin{rmk}
\label{rmk:action-Sigma-2}
From the definition we see that the restriction of~\eqref{eqn:morph-action-OG} to the subalgebra $\scO(G^\vee_\bk/G^\vee_\bk) \cong \scO(T^\vee_\bk/\Wf)$ coincides with the morphism considered in Lemma~\ref{lem:monodromy-tZ-F}; by this lemma, it therefore coincides with the restriction of monodromy. As a consequence, the action of $\scO(\Sigma)$ on $\scR^\wedge$ induced by the obvious action on $\scO(\Sigma)$ coincides, via the identification $\scO(\Sigma) \cong \scO(T^\vee_\bk/\Wf)$, with the monodromy action of $\scO(T^\vee_\bk/\Wf)$.
\end{rmk}

\subsection{Some properties of \texorpdfstring{$\scR^\wedge$}{Rwedge}}

In this subsection we prove a number of properties of the object $\scR^\wedge$.

\begin{lem}
\label{lem:tsZ-R}
 For any $V \in \Rep(G^\vee_\bk)$ we have a canonical isomorphism
 \[
  \scR^\wedge \hatstar \tsZ(V) \cong \scR^\wedge \otimes_\bk V
 \]
 in ${\ind}\D^\wedge_{\Iwu,\Iwu}$.
\end{lem}

\begin{proof}
By exactness of the functor $(-) \hatstar \tsZ(V)$ (see Theorem~\ref{thm:gaitsgory-mon}\eqref{it:gaitsgory-mon-5}) we have
\begin{equation}
 \label{eqn:isom-tsZ-R-0}
 \scR^\wedge \hatstar \tsZ(V) \cong \bigl( \tsZ(\scO(G^\vee_\bk)) \hatstar \tsZ(V) \bigr) \otimes_{\scO(G^\vee_\bk)} \scO(\Sigma),
\end{equation}
where $\scO(G^\vee_\bk)$ acts on $\tsZ(\scO(G^\vee_\bk)) \hatstar \tsZ(V)$ via its action on $\tsZ(\scO(G^\vee_\bk))$.
Now by monoidality of $\tsZ$ (see Theorem~\ref{thm:gaitsgory-mon}\eqref{it:gaitsgory-mon-2}) we have
 \[
  \tsZ(\scO(G^\vee_\bk)) \hatstar \tsZ(V) \cong \tsZ(\scO(G^\vee_\bk) \otimes_\bk V)
 \]
where $G^\vee_\bk$ acts diagonally on $\scO(G^\vee_\bk) \otimes_\bk V$. By standard arguments this tensor product identifies canonically with the similar tensor product where $G^\vee_\bk$ acts trivially on $V$, which provides an isomorphism
 \begin{equation}
 \label{eqn:isom-tsZ-R}
  \tsZ(\scO(G^\vee_\bk)) \hatstar \tsZ(V) \cong \tsZ(\scO(G^\vee_\bk)) \otimes_\bk V.
 \end{equation}
 This isomorphism is the image under $\sZ^{\wedge,\Coh}$ of an isomorphism
 \[
  \bigl( \scO(G^\vee_\bk) \otimes_\bk \scO_{G^\vee_\bk} \bigr) \otimes_{\scO_{G^\vee_\bk}} \bigl( V \otimes_\bk \scO_{G^\vee_\bk} \bigr) \cong \bigl( \scO(G^\vee_\bk) \otimes_\bk \scO_{G^\vee_\bk} \bigr) \otimes_\bk V
 \]
in $\Coh^{G^\vee_\bk}_{\mathrm{fr}}(G^\vee_\bk)$, which is easily seen to be $\scO(G^\vee_\bk)$-equivariant where on each side $\scO(G^\vee_\bk)$ acts via its action on $\scO(G^\vee_\bk) \otimes_\bk \scO_{G^\vee_\bk}$. Hence~\eqref{eqn:isom-tsZ-R} is $\scO(G^\vee_\bk)$-equivariant where on each side $\scO(G^\vee_\bk)$ acts via its action on $\tsZ(\scO(G^\vee_\bk))$. Combining this with~\eqref{eqn:isom-tsZ-R-0} we deduce the desired isomorphism.
\end{proof}


Since $\scR^\wedge$ is defined as a quotient of $\tsZ(\scO(G^\vee_\bk))$, the following claim follows from the similar property of $\tsZ(\scO(G^\vee_\bk))$ proved at the end of~\S\ref{ss:fm-central-sheaves}.

\begin{lem}
\label{lem:monodromy-Rwedge}
The monodromy morphism $\mu_{\scR^\wedge}$ factors through the multiplication morphism $\scO(T^\vee_\bk \times T^\vee_\bk) = \scO(T^\vee_\bk) \otimes \scO(T^\vee_\bk) \to \scO(T^\vee_\bk)$.
\end{lem}

Recall that
since the category $\mathsf{P}^0_{\Iwu,\Iwu}$ is defined as a quotient of $\mathsf{P}_{\Iwu,\Iwu}$, any object in this category admits a canonical action of $\scO(T^\vee_\bk \times T^\vee_\bk)$ which factors through an action of $\scO(T^\vee_\bk \times_{T^\vee_\bk / \Wf} T^\vee_\bk)$ (see Lemma~\ref{lem:monodromy-fiber-prod}), and these actions commute with any morphism in $\mathsf{P}^0_{\Iwu,\Iwu}$. From Lemma~\ref{lem:monodromy-Rwedge} and~\eqref{eqn:monodromy-convolution-1}--\eqref{eqn:monodromy-convolution-2}
one obtains that for any $\scF$ in $\mathsf{P}^0_{\Iwu,\Iwu}$ we have
\begin{equation}
\label{eqn:monodromy-conv-R}
\mu_{\scR^\wedge \hatstar^0 \scF} = \id_{\scR^\wedge} \hatstar^0 \mu_{\scF}.
\end{equation}

Let consider the object $\pi_\dag(\scR^\wedge)$ in ${\ind}\D_{\Iwu,\Iw}$. The category ${\ind}\D_{\Iwu,\Iw}$ is not triangulated in any obvious way, nor can we consider any form of ``perverse'' t-structure on it. However, for any $n$ the perverse cohomology functor $\pH^n : \D_{\Iwu,\Iw} \to \sfP_{\Iwu,\Iw}$ induces a functor on ind-objects; we can therefore consider the object $\pH^0(\pi_\dag(\scR^\wedge))$ in ${\ind}\sfP_{\Iwu,\Iw}$. Recall also the object $\scR$ in ${\ind}\sfP_{\Iw,\Iw}$ considered in~\eqref{eqn:defn-R}.

\begin{lem}
\label{lem:pidag-Rwedge}
We have a canonical isomorphism
\[
\pH^0(\pi_\dag(\scR^\wedge)) \cong \For^{\Iw}_{\Iwu}(\scR)
\]
in ${\ind}\sfP_{\Iwu,\Iw}$.
\end{lem}

\begin{proof}
Choose a presentation of $\scO(G^\vee_\bk)$-modules
\[
\scO(G^\vee_\bk)^{\oplus r} \to \scO(G^\vee_\bk) \to \scO(\Sigma) \to 0,
\]
and consider the induced exact sequence
\[
\tsZ(\scO(G^\vee_\bk))^{\oplus r} \to \tsZ(\scO(G^\vee_\bk)) \to \scR^\wedge \to 0
\]
in ${\ind}\sfP^\wedge_{\Iwu,\Iwu}$. Since $\pi_\dag$ is right t-exact (see~\S\ref{ss:completed-perverse}), the functor
\[
\pH^0(\pi_\dag(-)) : \sfP^\wedge_{\Iwu,\Iwu} \to \sfP_{\Iwu,\Iw}
\]
is right exact. By~\cite[Corollary~8.6.8]{ks} it follows that the induced functor on ind-objects is also right exact, and using Theorem~\ref{thm:gaitsgory-mon}\eqref{it:gaitsgory-mon-0} we deduce an exact sequence
\[
\For^{\Iw}_{\Iwu} (\sZ(\scO(G^\vee_\bk)))^{\oplus r} \to \For^{\Iw}_{\Iwu} (\sZ(\scO(G^\vee_\bk))) \to \pH^0(\pi_\dag \scR^\wedge) \to 0,
\]
which shows that
\[
\pH^0(\pi_\dag \scR^\wedge) \cong \For^{\Iw}_{\Iwu} (\sZ(\scO(G^\vee_\bk)) \otimes_{\scO(G^\vee_\bk)} \scO(\Sigma))
\]
where the action of $\scO(G^\vee_\bk)$ on $\sZ(\scO(G^\vee_\bk))$ is as in~\S\ref{ss:extension-Z}, or equivalently is obtained from the action on $\tsZ(\scO(G^\vee_\bk))$ by application of $\pi_\dag$. By Remarks~\ref{rmk:action-adjquo-Z} and~\ref{rmk:action-Sigma-2}, the restriction of this action to $\scO(T^\vee_\bk / \Wf)$ factors through the quotient morphism $\scO(T^\vee_\bk / \Wf) \to \scO(T^\vee_\bk / \Wf)/\cJ=\bk$, so that the action of $\scO(G^\vee_\bk)$ on $\sZ(\scO(G^\vee_\bk))$ factors through an action of the subscheme
\[
G^\vee_\bk \times_{T^\vee_\bk / \Wf} \Spec(\bk),
\]
where the morphism $\Spec(\bk) \to T^\vee_\bk / \Wf$ corresponds to the image of $e \in T^\vee_\bk$. We deduce that
\[
\sZ(\scO(G^\vee_\bk)) \otimes_{\scO(G^\vee_\bk)} \scO(\Sigma) = \sZ(\scO(G^\vee_\bk)) \otimes_{\scO(G^\vee_\bk \times_{T^\vee_\bk / \Wf} \Spec(\bk))} \scO(\Sigma \times_{T^\vee_\bk / \Wf} \Spec(\bk)).
\]
Now since the morphism $\Sigma \to T^\vee_\bk/\Wf$ is an isomorphism we have
\[
\Sigma \times_{T^\vee_\bk / \Wf} \Spec(\bk) = \{\su\},
\]
so that finally
\[
\sZ(\scO(G^\vee_\bk)) \otimes_{\scO(G^\vee_\bk)} \scO(\Sigma) = \sZ(\scO(G^\vee_\bk)) \otimes_{\scO(G^\vee_\bk)} \scO(\{\su\}) = \scR,
\]
which finishes the proof.
\end{proof}

\subsection{The coaction morphism}

Consider the comultiplication morphism~\eqref{eqn:comult-OG}. As in~\S\ref{ss:regular-quotient-coh}
this morphism defines a morphism of ind-objects
\begin{equation}
\label{eqn:comult-Zwedge}
\tsZ(\scO(G^\vee_\bk)) \to \tsZ(\scO(G^\vee_\bk)) \otimes_\bk \scO(G^\vee_\bk).
\end{equation}
Here we can interpret the right-hand side as the tensor product
\[
\tsZ(\scO(G^\vee_\bk)) \otimes_{\scO(G^\vee_\bk)} \scO(G^\vee_\bk \times G^\vee_\bk)
\]
where $\scO(G^\vee_\bk)$ acts on $\tsZ(\scO(G^\vee_\bk))$ as in the definition of $\scR^\wedge$ and the morphism $\scO(G^\vee_\bk) \to \scO(G^\vee_\bk \times G^\vee_\bk)$ is induced by the first projection. Hence using the morphism 
$\scO(G^\vee_\bk \times G^\vee_\bk) \to \scO(\bbJ_\Sigma)$ induced by the composition
\[
\bbJ_\Sigma \hookrightarrow \Sigma \times G^\vee_\bk \hookrightarrow G^\vee_\bk \times G^\vee_\bk
\]
we obtain a canonical morphism
\[
\tsZ(\scO(G^\vee_\bk)) \to \tsZ(\scO(G^\vee_\bk)) \otimes_{\scO(G^\vee_\bk)} \scO(\bbJ_\Sigma)
\]
where the morphism $\scO(G^\vee_\bk) \to \scO(\bbJ_\Sigma)$ is induced by the composition
\[
\bbJ_\Sigma \hookrightarrow \Sigma \times G^\vee_\bk \to \Sigma \hookrightarrow G^\vee_\bk
\]
where the second morphism is the obvious projection.
Now, using~\eqref{eqn:tensor-product-bimodule} we obtain a canonical isomorphism
\[
\tsZ(\scO(G^\vee_\bk)) \otimes_{\scO(G^\vee_\bk)} \scO(\bbJ_\Sigma) \cong \scR^\wedge  \otimes_{\scO(\Sigma)} \scO(\mathbb{J}_\Sigma),
\]
where the morphism $\scO(\Sigma) \to \scO(\mathbb{J}_\Sigma)$ is induced by the natural projection $\bbJ_\Sigma \to \Sigma$. We can therefore consider our morphism as a morphism
\begin{equation}
\label{eqn:coprod-Rwedge}
\tsZ(\scO(G^\vee_\bk)) \to \scR^\wedge \otimes_{\scO(\Sigma)} \scO(\mathbb{J}_\Sigma).
\end{equation}

\begin{lem}
\label{lem:coprod-Rwedge}
The morphism~\eqref{eqn:coprod-Rwedge} factors (uniquely) through a morphism
\[
\mathrm{coact}_{\scR^\wedge} : \scR^\wedge \to \scR^\wedge \otimes_{\scO(\Sigma)} \scO(\mathbb{J}_\Sigma).
\]
\end{lem}

\begin{proof}
Consider the action of  $\scO(G^\vee_\bk) \otimes \scO(G^\vee_\bk)$ on
\[
\tsZ(\scO(G^\vee_\bk)) \otimes_\bk \scO(G^\vee_\bk)
\]
where the left copy acts via the action on $\tsZ(\scO(G^\vee_\bk))$ used in the definition of $\scR^\wedge$, and the right copy acts via multiplication on $\scO(G^\vee_\bk)$. Of course, via the identification
\[
\tsZ(\scO(G^\vee_\bk)) \otimes_\bk \scO(G^\vee_\bk) \cong \tsZ(\scO(G^\vee_\bk)) \otimes_{\scO(G^\vee_\bk)} \scO(G^\vee_\bk \times G^\vee_\bk)
\]
this action corresponds to the action on the right-hand side induced by the obvious action of $\scO(G^\vee_\bk \times G^\vee_\bk)$ on itself.

An explicit computation using the Hopf algebra operations in $\scO(G^\vee_\bk)$ shows that
the morphism~\eqref{eqn:comult-Zwedge} is $\scO(G^\vee_\bk)$-linear, where the action on the left-hand side is as in the definition of $\scR^\wedge$ and that on the right-hand side is obtained from the action of $\scO(G^\vee_\bk) \otimes \scO(G^\vee_\bk)$ considered above via the morphism $\scO(G^\vee_\bk) \to \scO(G^\vee_\bk) \otimes \scO(G^\vee_\bk)$ induced by the map
\[
G^\vee_\bk \times G^\vee_\bk \to G^\vee_\bk
\]
given by $(g,h) \mapsto h^{-1}gh$. Hence our morphism
\[
\tsZ(\scO(G^\vee_\bk)) \to \tsZ(\scO(G^\vee_\bk)) \otimes_{\scO(G^\vee_\bk)} \scO(G^\vee_\bk \times G^\vee_\bk)
\]
is $\scO(G^\vee_\bk)$-linear where the action on the right-hand side is induced by the same morphism $\scO(G^\vee_\bk) \to \scO(G^\vee_\bk) \otimes \scO(G^\vee_\bk)$ and the obvious action of $\scO(G^\vee_\bk) \otimes \scO(G^\vee_\bk)$. If follows that the morphism
\[
\tsZ(\scO(G^\vee_\bk)) \to \tsZ(\scO(G^\vee_\bk)) \otimes_{\scO(G^\vee_\bk)} \scO(\bbJ_\Sigma)
\]
used to define~\eqref{eqn:coprod-Rwedge} is $\scO(G^\vee_\bk)$-linear where the action on the right-hand side is obtained from the obvious action of $\scO(\bbJ_\Sigma)$ on itself via the morphism $\scO(G^\vee_\bk) \to \scO(\bbJ_\Sigma)$ induced by the map
\[
\bbJ_\Sigma \to G^\vee_\bk
\]
given by $(s,h) \mapsto h^{-1} s h$. By definition of $\bbJ_\Sigma$ this morphism coincides with the projection $\bbJ_\Sigma \to \Sigma$; in particular, this action of $\scO(G^\vee_\bk)$ on $\tsZ(\scO(G^\vee_\bk)) \otimes_{\scO(G^\vee_\bk)} \scO(\bbJ_\Sigma)$ factors through the restriction morphism $\scO(G^\vee_\bk) \to \scO(\Sigma)$, which implies the desired property for~\eqref{eqn:coprod-Rwedge}.
%
\end{proof}

It is clear by construction that the morphism $\mathrm{coact}_{\scR^\wedge}$ from Lemma~\ref{lem:coprod-Rwedge} is ``counital" in the sense that the composition 
\[
\scR^\wedge \xrightarrow{\mathrm{coact}_{\scR^\wedge}} \scR^\wedge \otimes_{\scO(\Sigma)} \scO(\mathbb{J}_\Sigma) \to \scR^\wedge
\]
(where the second map is induced by restriction to the identity section of $\mathbb{J}_\Sigma$) is $\id_{\scR^\wedge}$, and
``coassociative" in the sense that the composition
\[
\scR^\wedge \xrightarrow{\mathrm{coact}_{\scR^\wedge}} \scR^\wedge \otimes_{\scO(\Sigma)} \scO(\mathbb{J}_\Sigma) \to \scR^\wedge \otimes_{\scO(\Sigma)} \scO(\mathbb{J}_\Sigma) \otimes_{\scO(\Sigma)} \scO(\mathbb{J}_\Sigma)
\]
(where the second map is induced by the comultiplication map for the group scheme $\bbJ_\Sigma$) coincides with the composition
\[
\scR^\wedge \xrightarrow{\mathrm{coact}_{\scR^\wedge}} \scR^\wedge \otimes_{\scO(\Sigma)} \scO(\mathbb{J}_\Sigma) \xrightarrow{\mathrm{coact}_{\scR^\wedge} \otimes \id} \scR^\wedge \otimes_{\scO(\Sigma)} \scO(\mathbb{J}_\Sigma) \otimes_{\scO(\Sigma)} \scO(\mathbb{J}_\Sigma).
\]

\subsection{A monoidality morphism}
\label{ss:monoidality-morph}

In this subsection we explain the construction of a morphism which will be an ingredient in the construction of the monoidal structure on the functor $\Phi_{\Iwu,\Iwu}$.

Recall that the product $\tsZ(V) \hatstar \scF$ is perverse for any $V \in \Rep(G^\vee_\bk)$ and $\scF \in \sfP_{\Iwu,\Iwu}$, see Theorem~\ref{thm:gaitsgory-mon}\eqref{it:gaitsgory-mon-5}. In view of the construction of the product $\hatstar^0$ (see~\S\ref{ss:actions-Dwedge-D}), it follows that for any $V \in \Rep(G^\vee_\bk)$ and $\scF \in \sfP_{\Iwu,\Iwu}^0$ the product $\tsZ(V) \hatstar^0 \scF$ belongs to $\sfP_{\Iwu,\Iwu}^0$.

\begin{lem}
\label{lem:convolution-R-0}
 Let $\scF \in \sfP^0_{\Iwu,\Iwu}$.
  We have a canonical identification
 \[
 \scR^\wedge \phatstar^0 \scF = (\tsZ(\scO(G^\vee_\bk)) \hatstar^0 \scF) \otimes_{\scO(G^\vee_\bk)} \scO(\Sigma),
 \]
where the tensor product in the right-hand side is taken in the abelian category ${\ind}\sfP^0_{\Iwu,\Iwu}$, and the action of $\scO(G^\vee_\bk)$ is induced by that on $\tsZ(\scO(G^\vee_\bk))$.
\end{lem}

\begin{proof}
By definition we have $\scR^\wedge = \tsZ(\scO(G^\vee_\bk)) \otimes_{\scO(G^\vee_\bk)} \scO(\Sigma)$. Choosing a 
presentation as in the proof of Lemma~\ref{lem:pidag-Rwedge} we obtain
an exact sequence
\[
 \tsZ(\scO(G^\vee_\bk))^{\oplus r} \to \tsZ(\scO(G^\vee_\bk)) \to \scR^\wedge \to 0.
\]
By Lemma~\ref{lem:t-exactness-conv-tFl} the functor
$(-) \phatstar^0 \scF : \sfP^\wedge_{\Iwu,\Iwu} \to \sfP^0_{\Iwu,\Iwu}$
is right exact. By~\cite[Corollary~8.6.8]{ks}, it follows that the same is true for the extension of this functor to ind-objects, and we deduce an exact sequence
\[
 (\tsZ(\scO(G^\vee_\bk)) \hatstar^0 \scF)^{\oplus r} \to \tsZ(\scO(G^\vee_\bk)) \hatstar^0 \scF \to \scR^\wedge \phatstar^0 \scF \to 0
\]
in ${\ind}\sfP^0_{\Iwu,\Iwu}$. It follows that
\[
 \scR^\wedge \phatstar^0 \scF = (\tsZ(\scO(G^\vee_\bk)) \hatstar^0 \scF) \otimes_{\scO(G^\vee_\bk)} \scO(\Sigma),
\]
as desired. 
%
\end{proof}

Using this lemma, we will now explain how to construct, for $\scF,\scG$ in $\sfP^0_{\Iwu,\Iwu}$, a canonical morphism
\begin{equation}
\label{eqn:morph-conv-R}
 (\scR^\wedge \phatstar^0 \scF) \pstar^0_{\Iwu} (\scR^\wedge \phatstar^0 \scG) \to \scR^\wedge \phatstar^0 (\scF \pstar^0_{\Iwu} \scG)
\end{equation}
in ${\ind}\sfP^0_{\Iwu,\Iwu}$.
First, from Lemma~\ref{lem:convolution-R-0} and the right t-exactness of $\pstar^0_{\Iwu}$ (see~\S\ref{ss:mon-reg-quotient}) we deduce a canonical isomorphism
\begin{multline*}
 (\scR^\wedge \phatstar^0 \scF) \pstar^0_{\Iwu} (\scR^\wedge \phatstar^0 \scG) = \\
 \bigl( (\tsZ(\scO(G^\vee_\bk)) \hatstar^0 \scF) \pstar^0_{\Iwu} (\tsZ(\scO(G^\vee_\bk)) \hatstar^0 \scG) \bigr) \otimes_{\scO(G^\vee_\bk \times G^\vee_\bk)} \scO(\Sigma \times \Sigma),
\end{multline*}
where $\scO(G^\vee_\bk \times G^\vee_\bk)=\scO(G^\vee_\bk) \otimes \scO(G^\vee_\bk)$ acts via its action on the two factors $\tsZ(\scO(G^\vee_\bk))$. Hence to construct~\eqref{eqn:morph-conv-R} it suffices to construct a morphism
\begin{equation}
\label{eqn:morph-conv-R-2}
 (\tsZ(\scO(G^\vee_\bk)) \hatstar^0 \scF) \pstar^0_{\Iwu} (\tsZ(\scO(G^\vee_\bk)) \hatstar^0 \scG) \to \scR^\wedge \phatstar^0 (\scF \pstar^0_{\Iwu} \scG)
\end{equation}
which is annihilated by the ideal of $\Sigma \times \Sigma \subset G^\vee_\bk \times G^\vee_\bk$. Now from the definition of $\pstar^0_{\Iwu}$,~\eqref{eqn:commutation-Z-0} and~\eqref{eqn:compatibility-hatstar-0} we obtain isomorphisms
\begin{multline*}
 (\tsZ(\scO(G^\vee_\bk)) \hatstar^0 \scF) \pstar^0_{\Iwu} (\tsZ(\scO(G^\vee_\bk)) \hatstar^0 \scG) = \\
 \pH^0 \bigl( (\tsZ(\scO(G^\vee_\bk)) \hatstar^0 \scF) \star^0_{\Iwu} (\tsZ(\scO(G^\vee_\bk)) \hatstar^0 \scG) \bigr) \cong \\
 \pH^0 \bigl( (\scF \hatstar^0 \tsZ(\scO(G^\vee_\bk))) \star^0_{\Iwu} (\tsZ(\scO(G^\vee_\bk)) \hatstar^0 \scG) \bigr) \cong \\
 \pH^0 \bigl( (\scF \hatstar^0 ( \tsZ(\scO(G^\vee_\bk)) \hatstar \tsZ(\scO(G^\vee_\bk)) )) \star^0_{\Iwu} \scG\bigr).
\end{multline*}
By Lemma~\ref{lem:commutation-ZCoh-Satake} these isomorphisms commute with the actions of $\scO(G^\vee_\bk \times G^\vee_\bk)$ induced by the actions on the factors $\tsZ(\scO(G^\vee_\bk))$. Now multiplication in $\scO(G^\vee_\bk)$ induces a morphism
\[
 \tsZ(\scO(G^\vee_\bk)) \hatstar \tsZ(\scO(G^\vee_\bk)) \to \tsZ(\scO(G^\vee_\bk))
\]
which is $\scO(G^\vee_\bk \times G^\vee_\bk)$-equivariant, where the action on the right-hand side is the composition of the product morphism $\scO(G^\vee_\bk) \otimes \scO(G^\vee_\bk) \to \scO(G^\vee_\bk)$ with the given action of $\scO(G^\vee_\bk)$. We deduce a canonical morphism
\[
 (\tsZ(\scO(G^\vee_\bk)) \hatstar^0 \scF) \pstar^0_{\Iwu} (\tsZ(\scO(G^\vee_\bk)) \hatstar^0 \scG) \to \pH^0((\scF \hatstar^0 \tsZ(\scO(G^\vee_\bk))) \star^0_{\Iwu} \scG).
\]
Using~\eqref{eqn:commutation-Z-0} and~\eqref{eqn:compatibility-hatstar-other-0}, the right-hand side identifies with
\[
 \tsZ(\scO(G^\vee_\bk)) \hatstar^0 (\scF \pstar^0_{\Iwu} \scG);
\]
it therefore admits a canonical morphism to $\scR^\wedge \phatstar^0 (\scF \pstar^0_{\Iwu} \scG)$. Combining these morphisms we obtain a morphism~\eqref{eqn:morph-conv-R-2}, and from the construction and the comments above on equivariance one can check that this morphism is indeed annihilated by the ideal of $\Sigma \times \Sigma \subset G^\vee_\bk \times G^\vee_\bk$; it therefore induces the whished-for morphism~\eqref{eqn:morph-conv-R}.

\subsection{Exactness}

This (technical) subsection is devoted to the proof of the following claim, which will be crucial for our considerations below.

\begin{lem}
\label{lem:convolution-0-exact}
The functor
\begin{equation}
\label{eqn:convolution-R}
\scR^\wedge \phatstar^0 (-) : \mathsf{P}^0_{\Iwu,\Iwu} \to {\ind}\mathsf{P}^0_{\Iwu,\Iwu}.
\end{equation}
is exact. Moreover, for $\scG$ in $\sfP_{\Iw,\Iw}^0$ we have a canonical isomorphism
\begin{equation}
\label{eqn:convolution-R-pidag}
\scR^\wedge \phatstar^0 (\pi_0^\dag \scG) \cong \pi_0^\dag(\scR^0 \star^0_\Iw \scG).
\end{equation}
\end{lem}

To prove this lemma we will need some preliminary results. Let us choose a complex of $\scO(G^\vee_\bk)$-modules
\begin{equation}
\label{eqn:resolution-OSigma}
 \cdots \to 0 \to P^{-2} \xrightarrow{a} P^{-1} \xrightarrow{b} P^0 \to 0 \to \cdots
\end{equation}
where each $P^j$ is free of finite rank (and placed in degree $j$), the natural morphism $\mathrm{Im}(a) \to \ker(b)$ is an isomorphism (in other words, our complex is exact in degree $-1$) and $\mathrm{coker}(b) \cong \scO(\Sigma)$. Tensoring with $\sZ^0(\scO(G^\vee_\bk))$ we deduce a complex
\begin{multline*}
  \cdots \to 0 \to \sZ^0(\scO(G^\vee_\bk)) \otimes_{\scO(G^\vee_\bk)} P^{-2} \xrightarrow{\tilde{a}} \sZ^0(\scO(G^\vee_\bk)) \otimes_{\scO(G^\vee_\bk)} P^{-1} \\
  \xrightarrow{\tilde{b}} \sZ^0(\scO(G^\vee_\bk)) \otimes_{\scO(G^\vee_\bk)} P^0 \to 0 \to \cdots
\end{multline*}
of objects in ${\ind}\sfP^0_{\Iw,\Iw}$.

\begin{lem}
\label{lem:exactness-piR}
 The natural morphism $\mathrm{Im}(\tilde{a}) \to \ker(\tilde{b})$ is an isomorphism.
\end{lem}

\begin{proof}
 Recall from~\S\ref{ss:regular-quotient-coh} the equivalence of categories $\Phi_{\Iw,\Iw} :\sfP^0_{\Iw,\Iw} \simto \Rep(\rmZ_{G^\vee_\bk}(\su))$. Passing to ind-objects we deduce an equivalence ${\ind}\sfP^0_{\Iw,\Iw} \simto {\ind}\Rep(\rmZ_{G^\vee_\bk}(\su))$. Now the category ${\ind}\Rep(\rmZ_{G^\vee_\bk}(\su))$ identifies with the category $\Rep^\infty(\rmZ_{G^\vee_\bk}(\su))$ of all algebraic $\rmZ_{G^\vee_\bk}(\su)$-modules (see~\cite[\S 6.3]{ks}), and under the equivalence
 \[
  {\ind}\sfP^0_{\Iw,\Iw} \simto \Rep^\infty(\rmZ_{G^\vee_\bk}(\su))
 \]
the object $\sZ^0(\scO(G^\vee_\bk))$ corresponds to $\scO(G^\vee_\bk)$, with the structure of $\rmZ_{G^\vee_\bk}(\su)$-module induced by multiplication on the left on $G^\vee_\bk$. Through these identifications, the action of $\scO(G^\vee_\bk)$ on $\sZ^0(\scO(G^\vee_\bk))$ corresponds to the action on $\scO(G^\vee_\bk)$ where $\varphi \in \scO(G^\vee_\bk)$ acts by multiplication by the function $g \mapsto \varphi(g^{-1} \su g)$ (see Remark~\ref{rmk:extension-restriction}). To prove our claim, it therefore suffices to prove that the complex of $\rmZ_{G^\vee_\bk}(\su)$-modules
\begin{multline*}
 \cdots \to 0 \to \scO(G^\vee_\bk) \otimes_{\scO(G^\vee_\bk)} P^{-2} \to \scO(G^\vee_\bk) \otimes_{\scO(G^\vee_\bk)} P^{-1} \\
 \to \scO(G^\vee_\bk) \otimes_{\scO(G^\vee_\bk)} P^{0} \to 0 \to \cdots
\end{multline*}
has no cohomology in degree $-1$. Now, the cohomology in degree $-1$ of this complex is
\[
 \mathrm{Tor}_1^{\scO(G^\vee_\bk)}(\scO(G^\vee_\bk), \scO(\Sigma)).
\]
If we let $G^\vee_\bk$ act on $\scO(G^\vee_\bk)$ via the right regular action, then for the action above $\scO(G^\vee_\bk)$ becomes a $G^\vee_\bk$-equivariant $\scO(G^\vee_\bk)$-module (where $G^\vee_\bk$ acts on the algebra $\scO(G^\vee_\bk)$ via conjugation). The desired claim therefore follows from Corollary~\ref{cor:Tor-vanishing}.
\end{proof}

Now we can come to the main step towards Lemma~\ref{lem:convolution-0-exact}. Here, as for Lem\-ma~\ref{lem:pidag-Rwedge} we will use the fact that it makes sense to apply a functor $\pH^n$ to an object in ${\ind}\D^0_{\Iwu,\Iw}$, even though there is no ``perverse t-structure'' on this category.

\begin{lem}
\label{lem:pi-dag-R}
We have
\[
\pH^n(\Pi^0_{\Iwu,\Iw}(\pi_\dag \scR^\wedge)) = \begin{cases}
\For^{\Iw,0}_{\Iwu}(\scR^0) & \text{if $n=0$;} \\
0 & \text{if $n>0$ or $n=-1$.}
\end{cases}
\]
\end{lem}

\begin{proof}
By definition of the perverse t-structure on $\D^0_{\Iwu,\Iw}$ the functor $\Pi^0_{\Iwu,\Iw} : \D_{\Iwu,\Iw} \to \D^0_{\Iwu,\Iw}$ is t-exact with respect to the perverse t-structures; this functor therefore commutes with the functor $\pH^n$, and then the same property holds for the extensions to ind-objects. Using Lemma~\ref{lem:pidag-Rwedge}, we deduce the case $n=0$. Similarly, since $\scR^\wedge$ belongs to ${\ind}\sfP^\wedge_{\Iwu,\Iwu}$, and since $\pi_\dag$ is right t-exact (see~\S\ref{ss:completed-perverse}), we have $\pH^n(\pi_\dag \scR^\wedge)=0$ for any $n>0$, which implies the claim in this case.

It remains to treat the case $n=-1$.
Consider again the complex~\eqref{eqn:resolution-OSigma}, and the complex
\begin{multline}
\label{eqn:resolution-Rwedge}
 \cdots \to 0 \to \tsZ(\scO(G^\vee_\bk)) \otimes_{\scO(G^\vee_\bk)} P^{-2} \xrightarrow{f} \tsZ(\scO(G^\vee_\bk)) \otimes_{\scO(G^\vee_\bk)} P^{-1} \\
 \xrightarrow{g} \tsZ(\scO(G^\vee_\bk)) \otimes_{\scO(G^\vee_\bk)} P^0 \to 0 \to \cdots
\end{multline}
in ${\ind}\sfP^\wedge_{\Iwu,\Iwu}$ obtained by tensoring with $\tsZ(\scO(G^\vee_\bk))$. Let us denote by $\mathsf{A}$ the full subcategory of $\sfP^\wedge_{\Iwu,\Iwu}$ whose objects are the perverse sheaves $\scF$ such that $\Pi^0_{\Iwu,\Iw} (\pi_\dag(\scF))$ belongs to the heart of the perverse t-structure. Then $\mathsf{A}$ is an additive category, and the natural functor ${\ind}\mathsf{A} \to {\ind}\sfP^\wedge_{\Iwu,\Iwu}$ is fully faithful by~\cite[Proposition~6.1.10]{ks}. The object $\tsZ(\scO(G^\vee_\bk))$ belongs to the essential image of this functor, see Theorem~\ref{thm:gaitsgory-mon}\eqref{it:gaitsgory-mon-0}; the same is therefore true for any term of our complex~\eqref{eqn:resolution-Rwedge}.

For an additive category $\sfB$, let us denote by $C^{[-2,0]}(\sfB)$ the category of complexes of objects of $\sfB$ whose components are zero in all degrees except possibly $-2$, $-1$ and $0$.
By~\cite[Lemma~15.4.1]{ks}, the natural functor
\[
{\ind} C^{[-2,0]}(\mathsf{A}) \to C^{[-2,0]}({\ind}\mathsf{A})
\]
is an equivalence of categories. This implies that there exist a filtrant category $I$, inductive systems $(\scM_i^{-2} : i \in I)$, $(\scM_i^{-1} : i \in I)$ and $(\scM_i^{0} : i \in I)$ of objects of $\mathsf{A}$, and morphisms of inductive systems $(f_i : \scM_i^{-2} \to \scM_i^{-1})_{i \in I}$, $(g_i : \scM_i^{-1} \to \scM_i^{0})_{i \in I}$ such that $g_i \circ f_i=0$ for any $i$, and such that~\eqref{eqn:resolution-Rwedge} is isomorphic to
\[
\cdots \to 0 \to ``\varinjlim_{i \in I}{}" \scM_i^{-2} \xrightarrow{``\varinjlim{}" f_i} ``\varinjlim_{i \in I}{}" \scM_i^{-1} \xrightarrow{``\varinjlim{}" g_i} ``\varinjlim_{i \in I}{}" \scM_i^0 \to 0 \to \cdots
\]
(as a complex of objects in ${\ind}\sfP^\wedge_{\Iwu,\Iwu}$).
For any $i \in I$ we set $\scQ_i := \mathrm{coker}(g_i)$; these objects define in a natural way an inductive system of objects in $\sfP^\wedge_{\Iwu,\Iwu}$. The object $\scR^\wedge$ is isomorphic to the cokernel of $g$; in view of the description of cokernels in ind-objects in an abelian category (see~\cite[Lemma~8.6.4(ii)]{ks}), we therefore have
\[
\scR^\wedge \cong ``\varinjlim_{i \in I}{}" \scQ_i,
\]
so that what we have to prove is that
\begin{equation}
\label{eqn:vanishing-Qi}
``\varinjlim_{i \in I}{}" \pH^{-1} \Pi^0_{\Iwu,\Iw} \pi_\dag( \scQ_i ) = 0.
\end{equation}
Note that with these notations, Lemma~\ref{lem:exactness-piR} (combined with the t-exactness of $\For^{\Iw,0}_{\Iwu}$) says that the complex
\begin{multline}
\label{eqn:piR-exact-complex}
\cdots \to 0 \to ``\varinjlim_{i \in I}{}" \Pi^0_{\Iwu,\Iw} \pi_\dag (\scM_i^{-2}) \xrightarrow{``\varinjlim{}" \Pi^0_{\Iwu,\Iw} \pi_\dag(f_i)} ``\varinjlim_{i \in I}{}" \Pi^0_{\Iwu,\Iw} \pi_\dag(\scM_i^{-1}) \\
\xrightarrow{``\varinjlim{}" \Pi^0_{\Iwu,\Iw} \pi_\dag(g_i)} ``\varinjlim_{i \in I}{}" \Pi^0_{\Iwu,\Iw} \pi_\dag(\scM_i^0) \to 0 \to \cdots
\end{multline}
of objects in ${\ind}\sfP^0_{\Iwu,\Iw}$
has no cohomology in degree $-1$.

Recall from~\eqref{eqn:equiv-KbTilt-wedge} the equivalence of triangulated categories
\[
 \Db \sfP^\wedge_{\Iwu,\Iwu} \simto \sfD^\wedge_{\Iwu,\Iwu}
\]
provided by the ``realization functor.'' For any $i \in I$ we consider the complex
\[
\widetilde{\scS}_i = (\cdots \to 0 \to \scM^{-2}_i \xrightarrow{f_i} \scM^{-1}_i \xrightarrow{g_i} \scM_i^0 \to 0 \to \cdots)
\]
of objects in $\sfP^\wedge_{\Iwu,\Iwu}$ (seen as an object in $\Db \sfP^\wedge_{\Iwu,\Iwu}$), and denote by $\scS_i$ its image in $\sfD^\wedge_{\Iwu,\Iwu}$. If we set
\[
\widetilde{\scS}'_i = (\cdots \to 0 \to \scM^{-2}_i \xrightarrow{f_i} \ker(g_i) \to 0 \to \cdots)
\]
(seen as an object in $\Db \sfP^\wedge_{\Iwu,\Iwu}$)
where $\scM^{-2}_i$ is in degree $-2$, and denote by $\scS'_i$ its image in $\sfD^\wedge_{\Iwu,\Iwu}$ then we have a distinguished triangle
\[
 \widetilde{\scS}'_i \to \widetilde{\scS}_i \to \scQ_i \xrightarrow{[1]}
\]
in $\Db \sfP^\wedge_{\Iwu,\Iwu}$, hence a distinguished triangle
\[
 \scS_i' \to \scS_i \to \scQ_i \xrightarrow{[1]}
\]
in $\sfD^\wedge_{\Iwu,\Iwu}$. Applying the triangulated functor $\Pi^0_{\Iwu,\Iw} \pi_\dag$, then taking the long exact sequence in perverse cohomology, and finally formal direct limits, we deduce an exact sequence
\[ 
 ``\varinjlim_{i \in I}{}" \pH^{-1} \Pi^0_{\Iwu,\Iw} \pi_\dag (\scS_{i}) \to ``\varinjlim_{i \in I}{}" \pH^{-1} \Pi^0_{\Iwu,\Iw} \pi_\dag(\scQ_i)
 \to ``\varinjlim_{i \in I}{}" \pH^{0} \Pi^0_{\Iwu,\Iw} \pi_\dag(\scS_i')
 \]
of objects in ${\ind}\sfP_{\Iwu,\Iw}^0$. By right t-exactness of the functor $\pi_\dag$, and since each $\scS_i'$ is concentrated in negative perverse degrees, the third term in this sequence vanishes. As a consequence, to prove~\eqref{eqn:vanishing-Qi} it suffices to prove that
\begin{equation}
\label{eqn:vanishing-Si}
``\varinjlim_{i \in I}{}" \pH^{-1} \Pi^0_{\Iwu,\Iw} \pi_\dag( \scS_i ) = 0.
\end{equation}

Now we set
\[
\widetilde{\scS}''_i = (\cdots \to 0 \to \scM^{-2}_i \xrightarrow{f_i} \scM_i^{-1} \to 0 \to \cdots)
\]
(seen as an object in $\Db \sfP^\wedge_{\Iwu,\Iwu}$)
where $\scM^{-2}_i$ is in degree $-2$, and denote by $\scS''_i$ its image in $\sfD^\wedge_{\Iwu,\Iwu}$. We have distinguished triangles
\[
 \scM_i^0 \to \scS_i \to \scS''_i \xrightarrow{[1]}, \qquad \scM_i^{-1}[1] \to \scS_i'' \to \scM_i^{-2}[2] \xrightarrow{[1]}
\]
where in the second triangle the morphism $\scM_i^{-2}[2] \to \scM_i^{-1}[2]$ is $f_i[2]$, and in the first triangle the morphism $\scS_i'' \to \scM_i^0[1]$ is the unique morphism whose composition with the map $\scM_i^{-1}[1] \to \scS_i''$ appearing in the second triangle is $g_i[1]$. (The existence and unicity of this morphism is guaranteed by the long exact sequence obtained by applying $\Hom(-,\scM_i^0[1])$ to the second triangle.) Applying the triangulated functor $\Pi^0_{\Iwu,\Iw} \pi_\dag$, we obtain distinguished triangles
\begin{gather*}
 \Pi^0_{\Iwu,\Iw} \pi_\dag(\scM_i^0) \to \Pi^0_{\Iwu,\Iw} \pi_\dag(\scS_i) \to \Pi^0_{\Iwu,\Iw} \pi_\dag(\scS''_i) \xrightarrow{[1]}, \\
 \Pi^0_{\Iwu,\Iw} \pi_\dag(\scM_i^{-1})[1] \to \Pi^0_{\Iwu,\Iw} \pi_\dag(\scS_i'') \to \Pi^0_{\Iwu,\Iw} \pi_\dag(\scM_i^{-2})[2] \xrightarrow{[1]}.
\end{gather*}
Since 
$\Pi^0_{\Iwu,\Iw} \pi_\dag(\scM_i^{-2})$ is perverse by definition of $\sfA$, taking the long exact sequence of perverse cohomology associated with the second triangle we obtain an exact sequence
\begin{equation}
\label{eqn:piR-es}
\Pi^0_{\Iwu,\Iw} \pi_\dag(\scM_i^{-2}) \to \Pi^0_{\Iwu,\Iw} \pi_\dag(\scM_i^{-1}) \to \pH^{-1}  \Pi^0_{\Iwu,\Iw} \pi_\dag(\scS_i'') \to 0,
\end{equation}
which identifies $\pH^{-1}  \Pi^0_{\Iwu,\Iw} \pi_\dag(\scS_i'')$ with $\mathrm{coker}(\Pi^0_{\Iwu,\Iw} \pi_\dag(f_i))$. On the other hand, the same procedure applied to the first distinguished triangle produces an exact sequence
\[
0 \to \pH^{-1}  \Pi^0_{\Iwu,\Iw} \pi_\dag(\scS_i) \to \pH^{-1}  \Pi^0_{\Iwu,\Iw} \pi_\dag(\scS_i'') \to \Pi^0_{\Iwu,\Iw} \pi_\dag(\scM_i^0).
\]
Here, by construction the composition of the right morphism with the surjection $ \Pi^0_{\Iwu,\Iw} \pi_\dag(\scM_i^{-1}) \to \pH^{-1}  \Pi^0_{\Iwu,\Iw} \pi_\dag(\scS_i'')$ from~\eqref{eqn:piR-es} is $\Pi^0_{\Iwu,\Iw} \pi_\dag(g_i)$. Hence, through the identification $\pH^{-1}  \Pi^0_{\Iwu,\Iw} \pi_\dag(\scS_i'') \cong \mathrm{coker}(\Pi^0_{\Iwu,\Iw} \pi_\dag(f_i))$, this exact sequence identifies $\pH^{-1}  \Pi^0_{\Iwu,\Iw} \pi_\dag(\scS_i)$ with the kernel of the morphism
\[
\mathrm{coker}(\Pi^0_{\Iwu,\Iw} \pi_\dag(f_i)) \to \Pi^0_{\Iwu,\Iw} \pi_\dag(\scM_i^0)
\]
induced by $\Pi^0_{\Iwu,\Iw} \pi_\dag(g_i)$. Passing to formal direct limits and then using the description of kernels and cokernels in ind-objects in an abelian category (see~\cite[Lemma~8.6.4(ii)]{ks}), we deduce that $``\varinjlim_i{}" \pH^{-1}  \Pi^0_{\Iwu,\Iw} \pi_\dag(\scS_i)$ identifies with the kernel of the morphism
\[
\mathrm{coker}(``\varinjlim_i{}" \Pi^0_{\Iwu,\Iw} \pi_\dag(f_i)) \to ``\varinjlim_i{}" \Pi^0_{\Iwu,\Iw} \pi_\dag( \scM_i^0)
\]
induced by $``\varinjlim_i{}" \Pi^0_{\Iwu,\Iw} \pi_\dag(g_i)$. The exactness of the complex~\eqref{eqn:piR-exact-complex} in degree $-1$ exactly says that this morphism is injective, which shows~\eqref{eqn:vanishing-Si} and finishes the proof.
\end{proof}

\begin{proof}[Proof of Lemma~\ref{lem:convolution-0-exact}]
By construction, the bifunctor
\[
\sfD^\wedge_{\Iwu,\Iwu} \times \sfD^0_{\Iw,\Iw} \to \sfD^0_{\Iwu,\Iwu}
\]
given by $(\scF,\scG) \mapsto \scF \hatstar^0 (\pi_0^\dag \scG)$
is the unique bifunctor through which the bifunctor
\[
\sfD^\wedge_{\Iwu,\Iwu} \times \sfD_{\Iw,\Iw} \to \sfD^0_{\Iwu,\Iwu}
\]
given by $(\scF,\scG) \mapsto \Pi^0_{\Iwu,\Iwu}( \scF \hatstar (\pi^\dag \For^{\Iw}_{\Iwu}( \scG)))$ factors. Now, by~\eqref{eqn:conv-formula-2}--\eqref{eqn:conv-formula-3}, for $\scF$ in $\sfD^\wedge_{\Iwu,\Iwu}$ and $\scG$ in $\sfD_{\Iw,\Iw}$ we have a canonical isomorphism
\[
\scF \hatstar (\pi^\dag \For^{\Iw}_{\Iwu}(\scG)) \cong \pi^\dag((\pi_\dag \scF) \star_\Iw \scG).
\]
We deduce, for $\scF$ in $\sfD^\wedge_{\Iwu,\Iwu}$ and $\scG$ in $\sfD_{\Iw,\Iw}^0$, a bifunctorial isomorphism
\[
\scF \hatstar^0 (\pi_0^\dag \scG) \cong \pi^{\dag,0} \bigl( (\Pi^0_{\Iwu,\Iw} \pi_\dag \scF) \star^0_{\Iw} \scG \bigr)
\]
(where the bifunctor $\star^0_\Iw$ here is that defined in~\S\ref{ss:another-convolution}, and the functor $\pi^{\dag,0}$ is defined in~\S\ref{ss:relation-reg-quotient})
and then, by t-exactness of $\pi^{\dag,0}$ and $\star^0_{\Iw}$, for any $n \in \Z$ we deduce an isomorphism
\[
\pH^n(\scF \hatstar^0 (\pi_0^\dag \scG)) \cong \pi^{\dag,0} \bigl( \pH^n(\Pi^0_{\Iwu,\Iw} \pi_\dag \scF) \star^0_{\Iw} \scG \bigr).
\]
This isomorphism extends to ind-objects, and provides for $\scG$ in $\sfP_{\Iw,\Iw}^0$ and $n \in \Z$ an isomorphism
\begin{equation}
\label{eqn:isom-convolution-0-exact}
\pH^n(\scR^\wedge \hatstar^0 (\pi_0^\dag \scG)) \cong \pi^{\dag,0} \bigl( \pH^n(\Pi^0_{\Iwu,\Iw} \pi_\dag \scR^\wedge) \star^0_{\Iw} \scG \bigr).
\end{equation}

Applying~\eqref{eqn:isom-convolution-0-exact} in case $n=0$ and using Lemma~\ref{lem:pi-dag-R}, we obtain an isomorphism
\[
\scR^\wedge \phatstar^0 (\pi_0^\dag \scG) \cong \pi^{\dag,0}((\For^{\Iw,0}_{\Iwu} \scR^0) \star^0_\Iw \scG).
\]
Using~\eqref{eqn:For-conv-0} and~\eqref{eqn:pidag-0-For}, we deduce the isomorphism~\eqref{eqn:convolution-R-pidag}.

The isomorphism~\eqref{eqn:isom-convolution-0-exact} and Lemma~\ref{lem:pi-dag-R} also imply that for any $\scG$ in $\sfP_{\Iw,\Iw}^0$ we have
\[
\pH^n(\scR^\wedge \hatstar^0 (\pi_0^\dag \scG))=0 \quad \text{if $n>0$ or $n=-1$.}
\]
We claim that in fact, for any $\scF$ in $\sfP_{\Iwu,\Iwu}^0$ we have
\begin{equation}
\label{eqn:vanishing-conv-Rwedge}
\pH^n(\scR^\wedge \hatstar^0 \scF)=0 \quad \text{if $n>0$ or $n=-1$.}
\end{equation}
Indeed, write $\scR^\wedge = ``\varinjlim_{i}" \scR^\wedge_i$ with each $\scR^\wedge_i$ in $\sfP_{\Iwu,\Iwu}^\wedge$. Given an exact sequence
\[
\scF_1 \hookrightarrow \scF_2 \twoheadrightarrow \scF_3
\]
in $\sfP_{\Iwu,\Iwu}^0$, for any $i$ we have a distinguished triangle
\[
\scR^\wedge_i \hatstar^0 \scF_1 \to \scR^\wedge_i \hatstar^0 \scF_2 \to \scR^\wedge_i \hatstar^0 \scF_3 \xrightarrow{[1]}
\]
in $\sfD_{\Iwu,\Iwu}^0$, and then a long exact sequence
\begin{multline*}
\cdots \to \pH^{n-1}(\scR^\wedge_i \hatstar^0 \scF_3) \to \pH^n(\scR^\wedge_i \hatstar^0 \scF_1) \to \pH^n(\scR^\wedge_i \hatstar^0 \scF_2) \\
\to \pH^n(\scR^\wedge_i \hatstar^0 \scF_3) \to \pH^{n+1}(\scR^\wedge_i \hatstar^0 \scF_1) \to \cdots
\end{multline*}
in $\sfP_{\Iwu,\Iwu}^0$. Taking the formal inductive limit, we deduce a long exact sequence
\begin{multline*}
\cdots \to \pH^{n-1}(\scR^\wedge \hatstar^0 \scF_3) \to \pH^n(\scR^\wedge \hatstar^0 \scF_1) \to \pH^n(\scR^\wedge \hatstar^0 \scF_2) \\
\to \pH^n(\scR^\wedge \hatstar^0 \scF_3) \to \pH^{n+1}(\scR^\wedge \hatstar^0 \scF_1) \to \cdots
\end{multline*}
in ${\ind}\sfP_{\Iwu,\Iwu}^0$. This exact sequence shows that if~\eqref{eqn:vanishing-conv-Rwedge} is true for two objects, then it follows for any extension between them. Since this statement is known for any object of the form $\pi_0^\dag \scG$ with $\scG$ in $\sfP_{\Iw,\Iw}^0$, and since any object in $\sfP_{\Iwu,\Iwu}^0$ is a successive extension of objects of this form, this proves~\eqref{eqn:vanishing-conv-Rwedge} for all $\scF$.

Now that~\eqref{eqn:vanishing-conv-Rwedge} is known, the same long exact sequence as above shows that the functor~\eqref{eqn:convolution-R} transforms exact sequences in $\sfP_{\Iwu,\Iwu}^0$ into exact sequences in ${\ind}\sfP_{\Iwu,\Iwu}^0$, i.e.~is exact.
\end{proof}

\subsection{Definition of the functor}
\label{ss:definition-Phi-Iwu}

We start with the following observation: consider a category $\mathsf{A}$, a pro-object
\[
X = ``\varprojlim_{i \in I}" X_i,
\]
in $\mathsf{A}$, and an ind-object
\[
Y = ``\varinjlim_{j \in J}" Y_j
\]
in $\mathsf{A}$. Then $\mathsf{A}$ embeds in the category ${\ind}\mathsf{A}$ of ind-objects in $\sfA$, and also in the category ${\pro}\mathsf{A}$ of pro-objects in $\mathsf{A}$. Using the induced functors on categories of pro-objects and ind-objects respectively, we can see $X$ and $Y$ either as objects in ${\pro}{\ind}\mathsf{A}$, or as objects in ${\ind}{\pro}\mathsf{A}$. The spaces of morphisms from $X$ to $Y$ in these two categories coincide: they both canonically identify with
\[
\varinjlim_{i \in I} \varinjlim_{j \in J} \Hom_{\mathsf{A}}(X_i, Y_j) = \varinjlim_{(i,j) \in I \times J} \Hom_{\mathsf{A}}(X_i, Y_j) = \varinjlim_{j \in J} \varinjlim_{i \in I} \Hom_{\mathsf{A}}(X_i, Y_j),
\]
where the equalities follow from~\cite[Proposition~2.1.7]{ks}.
This space will simply be denoted $\Hom_{\mathsf{A}}(X,Y)$.


In the present setting we have the pro-object $\Xi^\wedge_!$ in $\D_{\Iwu,\Iwu}$; applying (the functor on pro-objects induced by) $\Pi_{\Iwu,\Iwu}^0$ we deduce a pro-object $\Pi_{\Iwu,\Iwu}^0(\Xi^\wedge_!)$ in $\D^0_{\Iwu,\Iwu}$. (In other words, $\Pi_{\Iwu,\Iwu}^0(\Xi^\wedge_!)$ is the image of the pro-object $\Pi_{U,U}^0(\Xi^\wedge_!)$ considered in~\S\ref{ss:morphisms-PiXi} under the functor on pro-objects induced by~\eqref{eqn:D-U-Iu}.) On the other hand, given $\scF$ in $\mathsf{P}_{\Iwu,\Iwu}^0$ we have the ind-object $\scR^\wedge \phatstar^0 \scF$ in $\mathsf{P}_{\Iwu,\Iwu}^0$. Now we have a fully faithful functor ${\ind}\mathsf{P}_{\Iwu,\Iwu}^0 \to {\ind}\mathsf{D}_{\Iwu,\Iwu}^0$, see~\cite[Proposition~6.1.10]{ks}; $\scR^\wedge \phatstar^0 \scF$ can therefore also be seen as an ind-object in $\D_{\Iwu,\Iwu}^0$. Using the notation above we can therefore consider the vector space
\[
\Phi_{\Iwu,\Iwu}(\scF) := \Hom_{\D_{\Iwu,\Iwu}^0} \bigl( \Pi_{\Iwu,\Iwu}^0(\Xi^\wedge_!), \scR^\wedge \phatstar^0 \scF \bigr).
\]

For $\scF$ in $\mathsf{P}_{\Iwu,\Iwu}^0$,
monodromy endows $\Phi_{\Iwu,\Iwu}(\scF)$ with a canonical action of $\scO(T^\vee_\bk \times T^\vee_\bk)$, which by Lemma~\ref{lem:monodromy-fiber-prod} factors through an action of $\scO(T^\vee_\bk \times_{T^\vee_\bk/\Wf} T^\vee_\bk)$. (In view of Remark~\ref{rmk:action-Sigma-2}, the restriction of this action to $\scO(T^\vee_\bk/\Wf)$ coincides with the action of $\scO(\Sigma)$ induced by the action on $\scR^\wedge$.) In this way, $\Phi_{\Iwu,\Iwu}$ can be seen as a functor
\[
 \sfP^0_{\Iwu,\Iwu} \to \Mod(\scO(T^\vee_\bk \times_{T^\vee_\bk/\Wf} T^\vee_\bk)).
\]

Again for $\scF$ in $\mathsf{P}_{\Iwu,\Iwu}^0$,
using the right exactness of the functor $(-) \phatstar^0 \scF$ (see Lemma~\ref{lem:t-exactness-conv-tFl}),
the morphism of Lemma~\ref{lem:coprod-Rwedge} provides a canonical morphism
\[
\scR^\wedge \phatstar^0 \scF \to \bigl( \scR^\wedge \phatstar^0 \scF \bigr) \otimes_{\scO(\Sigma)} \scO(\mathbb{J}_\Sigma)
\]
in ${\ind}\mathsf{P}^0_{\Iwu,\Iwu}$. (Here the action of $\scO(\Sigma)$ on $\scR^\wedge \phatstar^0 \scF$ is induced by that on $\scR^\wedge$, or equivalently by monodromy.)

\begin{lem}
\label{lem:morph-coaction-Psi}
For any $\scF$ in $\mathsf{P}^0_{\Iwu,\Iwu}$, there exists a canonical isomorphism
\[
\Phi_{\Iwu,\Iwu}(\scF) \otimes_{\scO(\Sigma)} \scO(\mathbb{J}_\Sigma) \simto \Hom_{\sfD^0_{\Iwu,\Iwu}}\bigl( \Pi^0_{\Iwu,\Iwu}(\Xi^\wedge_!), (\scR^\wedge \phatstar^0 \scF) \otimes_{\scO(\Sigma)} \scO(\mathbb{J}_\Sigma) \bigr).
\]
\end{lem}

\begin{proof}
The object $\Xi^\wedge_!$ is perverse; by Proposition~\ref{prop:perverse-t-str-completed-cat} it can therefore be written as $``\varprojlim_n" \scA_n$ for some projective system $(\scA_n : n \geq 0)$ of objects in $\pD^{\leq 0}_{\Iwu,\Iwu}$. On the other hand, $\scR^\wedge \phatstar^0 \scF$ is an ind-object in $\sfP^0_{\Iwu,\Iwu}$; it can therefore be written as $``\varinjlim_i" \scG_i$ for some objects $\scG_i$ in $\sfP^0_{\Iwu,\Iwu}$. 
We then have
\[
 \Phi_{\Iwu,\Iwu}(\scF) = \varinjlim_{n,i} \Hom_{\sfD^0_{\Iwu,\Iwu}}(\scA_n, \scG_i) = \varinjlim_{n,i} \Hom_{\sfP^0_{\Iwu,\Iwu}}(\pH^0(\scA_n), \scG_i).
\]
If we write $\scO(\bbJ_\Sigma) = \varinjlim_j M_j$ for some finitely generated $\scO(\Sigma)$-modules $M_j$, then we similarly have
\begin{multline*}
 \Hom_{\sfD^0_{\Iwu,\Iwu}}\bigl( \Pi^0_{\Iwu,\Iwu}(\Xi^\wedge_!), (\scR^\wedge \phatstar^0 \scF) \otimes_{\scO(\Sigma)} \scO(\mathbb{J}_\Sigma) \bigr) \\
 = \varinjlim_{n,i,j} \Hom_{\sfP^0_{\Iwu,\Iwu}}(\pH^0(\scA_n), \scG_i \otimes_{\scO(\Sigma)} M_j),
\end{multline*}
where $\scO(\Sigma)$ acts on $\scG_i$ via monodromy.
By Lemma~\ref{lem:morphism-tensor-prod}, for any $n,i,j$ we have a canonical morphism
\begin{equation}
\label{eqn:morph-coaction-Psi}
 \Hom_{\sfP^0_{\Iwu,\Iwu}}(\pH^0(\scA_n), \scG_i) \otimes_{\scO(\Sigma)} M_j \to \Hom_{\sfP^0_{\Iwu,\Iwu}}(\pH^0(\scA_n), \scG_i \otimes_{\scO(\Sigma)} M_j),
\end{equation}
which defines the desired morphism.


To prove that this morphism is an isomorphism, we observe that
since $\scO(\mathbb{J}_\Sigma)$ is flat over $\scO(\Sigma)$ (see Lemma~\ref{lem:J-smooth}), by Lazard's theorem (see e.g.~\cite[\href{https://stacks.math.columbia.edu/tag/058G}{Tag 058G}]{stacks-project}) the objects $M_j$ can be chosen to be finite free $\scO(\Sigma)$-modules. Then~\eqref{eqn:morph-coaction-Psi} is an isomorphism for any $n,i,j$, which concludes the proof.
\end{proof}

This lemma shows that the morphism of Lemma~\ref{lem:coprod-Rwedge} induces, for any $\scF$ in $\mathsf{P}^0_{\Iwu,\Iwu}$, a canonical morphism
\[
\Phi_{\Iwu,\Iwu}(\scF) \to \Phi_{\Iwu,\Iwu}(\scF) \otimes_{\scO(\Sigma)} \scO(\mathbb{J}_\Sigma),
\]
which is easily seen to define a structure of $\scO(\mathbb{J}_\Sigma)$-comodule on $\Phi_{\Iwu,\Iwu}(\scF)$. Combining these structures, we see that $\Phi_{\Iwu,\Iwu}$ defines a functor from $\mathsf{P}^0_{\Iwu,\Iwu}$ to the category of $\scO(\bbI_\Sigma)$-comodules.

We now explain how to construct, for $\scF,\scG$ in $\sfP^0_{\Iwu,\Iwu}$, a bifunctorial morphism
\begin{equation}
\label{eqn:monoidality-morphism}
 \Phi_{\Iwu,\Iwu}(\scF) \circledast  \Phi_{\Iwu,\Iwu}(\scG) \to \Phi_{\Iwu,\Iwu}(\scF \pstar^0_{\Iwu} \scG).
\end{equation}
First, the same construction as for~\eqref{eqn:morph-monoidality-DUU} provides, for any $\scF,\scG$ in $\sfP^0_{\Iwu,\Iwu}$, a canonical morphism
\begin{multline*}
 \Hom_{\sfD^0_{\Iwu,\Iwu}}(\Pi^0_{\Iwu,\Iwu}(\Xi^\wedge_!), \scF) \otimes_{\scO(\Sigma)} \Hom_{\sfD^0_{\Iwu,\Iwu}}(\Pi^0_{\Iwu,\Iwu}(\Xi^\wedge_!), \scG) \\
 \to
 \Hom_{\sfD^0_{\Iwu,\Iwu}}(\Pi^0_{\Iwu,\Iwu}(\Xi^\wedge_!), \scF \pstar^0_{\Iwu} \scG).
\end{multline*}
We then deduce a similar morphism for ind-objects, which provides a canonical morphism
\[
 \Phi_{\Iwu,\Iwu}(\scF) \circledast  \Phi_{\Iwu,\Iwu}(\scG) \to \Hom_{\sfD^0_{\Iwu,\Iwu}}(\Pi^0_{\Iwu,\Iwu}(\Xi^\wedge_!), (\scR^\wedge \phatstar^0 \scF) \pstar^0_{\Iwu} (\scR^\wedge \phatstar^0 \scG)).
\]
Composing this morphism with the morphism
\begin{multline*}
 \Hom_{\sfD^0_{\Iwu,\Iwu}}(\Pi^0_{\Iwu,\Iwu}(\Xi^\wedge_!), (\scR^\wedge \phatstar^0 \scF) \pstar^0_{\Iwu} (\scR^\wedge \phatstar^0 \scG)) \to \\
 \Hom_{\sfD^0_{\Iwu,\Iwu}}(\Pi^0_{\Iwu,\Iwu}(\Xi^\wedge_!), \scR^\wedge \phatstar^0 (\scF \pstar^0 \scG))
\end{multline*}
induced by~\eqref{eqn:morph-conv-R}, we deduce the whished-for morphism~\eqref{eqn:monoidality-morphism}.

We will later see that~\eqref{eqn:monoidality-morphism} is an isomorphism for any $\scF,\scG$ in $\sfP^0_{\Iwu,\Iwu}$, but this will require proving first some other properties of $\Phi_{\Iwu,\Iwu}$.

\subsection{Image of monodromy}

Recall the monodromy construction with respect to the loop rotation action, see Remark~\ref{rmk:loop-rot}. This morphism provides, for any $\scF$ in $\sfP^0_{\Iwu,\Iwu}$, a functorial automorphism $\mu_\scF^{\mathrm{rot}}(x) : \scF \simto \scF$.

On the other hand, consider the category $\Rep^\infty(\bbI_\Sigma)$ of representations of the group scheme $\bbI_\Sigma$, or in other words of $\scO(\bbI_\Sigma)$-comodules. As explained in Remark~\ref{rmk:section-J}, the group scheme $\bbJ_\Sigma$ admits a canonical section, hence so does $\bbI_\Sigma$. This implies that any $M$ in $\Rep^\infty(\bbI_\Sigma)$ admits a ``tautological'' automorphism, defined as the composition
\[
M \to M \otimes_{\scO(T^\vee_\bk \times_{T^\vee_\bk/\Wf} T^\vee_\bk)} \scO(\bbI_\Sigma) \to M \otimes_{\scO(T^\vee_\bk \times_{T^\vee_\bk/\Wf} T^\vee_\bk)} \scO(T^\vee_\bk \times_{T^\vee_\bk/\Wf} T^\vee_\bk) = M
\]
where the first morphism is the coaction, and the second one is induced by restriction to the canonical section. Here again this automorphism is functorial (in the sense that it defines an automorphism of the identity functor).

\begin{lem}
\label{lem:image-monodromy}
For any $\scF$ in $\sfP^0_{\Iwu,\Iwu}$, $\Phi_{\Iwu,\Iwu}(\mu_\scF^{\mathrm{rot}}(x)^{-1})$ is the tautological automorphism of $\Phi_{\Iwu,\Iwu}(\scF)$.
\end{lem}

\begin{proof}
Since the loop rotation action is trivial on $G/U$, for any $\scG$ in $\sfP_{U,U}^0$, seen as an object in $\sfP^0_{\Iwu,\Iwu}$, we have $\mu_\scG^{\mathrm{rot}}(x)=\id$. On the other hand, by~\eqref{eqn:monodromy-convolution-3} we have
\[
\mu^{\mathrm{rot}}_{\scR^\wedge \phatstar^0 \scF}(x)=\mu^{\mathrm{rot}}_{\scR^\wedge}(x) \phatstar^0 \mu_\scF^{\mathrm{rot}}(x)
\]
where $\mu^{\mathrm{rot}}_{\scR^\wedge}$ is the automorphism induced by $\mu^{\mathrm{rot}}_{\tsZ(\scO(G^\vee_\bk))}$. Since any morphism in $\sfD_{\Iwu,\Iwu}^0$ commutes with monodromy, it follows that $\Phi_{\Iwu,\Iwu}(\mu_\scF^{\mathrm{rot}}(x)^{-1})$ coincides with the automorphism of $\Phi_{\Iwu,\Iwu}(\scF)$ induced by $\mu^{\mathrm{rot}}_{\scR^\wedge}(x)$. By definition, this automorphism is obtained by passage to the quotient from $\mu^{\mathrm{rot}}_{\tsZ(\scO(G^\vee_\bk))}(x)$, which by Theorem~\ref{thm:gaitsgory-mon}\eqref{it:gaitsgory-mon-6} coincides with $(\hat{\sm}_{\scO(G^\vee_\bk)})^{-1}$. By construction of this functor (see~\S\ref{ss:extension-Z-wedge}), the latter automorphism is the image under $\sZ^{\wedge,\Coh}$ of the inverse of the tautological automorphism of $\scO(G^\vee_\bk) \otimes \scO_{G^\vee_\bk}$ as considered in~\S\ref{ss:extension-Z}.

On the other hand, by definition of the coaction on $\Phi_{\Iwu,\Iwu}(\scF)$, the tautological automorphism of this representation is induced by the automorphism of $\scR^\wedge$ given by the composition
\[
\scR^\wedge \xrightarrow{\mathrm{coact}_{\scR^\wedge}} \scR^\wedge \otimes_{\scO(\Sigma)} \scO(\bbJ_\Sigma) \to \scR^\wedge \otimes_{\scO(\Sigma)} \scO(\Sigma) = \scR^\wedge
\]
where the second morphism is induced by restriction to the canonical section. This automorphism is obtained by passage to the quotient from the automorphism of $\tsZ(\scO(G^\vee_\bk))$ given by the composition
\begin{multline*}
\tsZ(\scO(G^\vee_\bk)) \to \tsZ(\scO(G^\vee_\bk)) \otimes_\bk \scO(G^\vee_\bk) = \tsZ(\scO(G^\vee_\bk)) \otimes_{\scO(G^\vee_\bk)} \scO(G^\vee_\bk \times G^\vee_\bk) \\
\to \tsZ(\scO(G^\vee_\bk)) \otimes_{\scO(G^\vee_\bk)} \scO(G^\vee_\bk) = \tsZ(\scO(G^\vee_\bk))
\end{multline*}
where the first morphism and the first identification are as in the construction of $\mathrm{coact}_{\scR^\wedge}$, and the second morphism is induced by restriction to the diagonal. The latter morphism is the image under $\sZ^{\wedge,\Coh}$ of the automorphism of $\scO(G^\vee_\bk) \otimes \scO_{G^\vee_\bk}$ given by the composition
\begin{multline*}
\scO(G^\vee_\bk) \otimes \scO_{G^\vee_\bk} \to \scO(G^\vee_\bk) \otimes \scO(G^\vee_\bk) \otimes \scO_{G^\vee_\bk} = \bigl( \scO(G^\vee_\bk) \otimes \scO_{G^\vee_\bk} \bigr) \otimes_{\scO(G^\vee_\bk)} \scO(G^\vee_\bk \times G^\vee_\bk) \\
\to \bigl( \scO(G^\vee_\bk) \otimes \scO_{G^\vee_\bk} \bigr) \otimes_{\scO(G^\vee_\bk)} \scO(G^\vee_\bk) = \scO(G^\vee_\bk) \otimes \scO_{G^\vee_\bk}
\end{multline*}
where the first morphism is induced by the comultiplication in $\scO(G^\vee_\bk)$ and the second one by restriction to the diagonal.

These remarks show that the claim will follow if we check that the two given automorphisms of $\scO(G^\vee_\bk) \otimes \scO_{G^\vee_\bk}$ coincide. This is an easy exercise of manipulation with the Hopf algebra $\scO(G^\vee_\bk)$.
\end{proof}

\subsection{Exactness and compatibility with \texorpdfstring{$\Phi_{\Iw,\Iw}$}{PhiII} and \texorpdfstring{$\Phi_{U,U}$}{PhiUU}}


Our goal in this subsection is to show that the functor $\Phi_{\Iwu,\Iwu}$ constructed in~\S\ref{ss:definition-Phi-Iwu} factors through an exact functor
\[
\mathsf{P}^0_{\Iwu,\Iwu} \to \Rep_0(\bbI_\Sigma),
\]
which is moreover compatible with the functor $\Phi_{\Iw,\Iw}$ from~\S\ref{ss:regular-quotient-coh} and the functor $\Phi_{U,U}$ from Theorem~\ref{thm:reg-quotient-finite} in the appropriate sense.



\begin{lem}
\label{lem:exactness-Phi}
For any $\scF$ in $\sfP_{\Iw,\Iw}^0$, we have a canonical isomorphism of $\scO(\bbI_\Sigma)$-comodules
\[
\Phi_{\Iwu,\Iwu}(\pi_0^\dag \scF) \cong \Phi_{\Iw,\Iw}(\scF)
\]
(where the coaction on the right-hand side is provided by the functor~\eqref{eqn:functor-statement-Rep}),
and moreover
\[
\Hom_{\mathsf{D}^0_{\Iwu,\Iwu}} \bigl( \Pi_{\Iwu,\Iwu}^0(\Xi^\wedge_!), \scR^\wedge \phatstar^0 (\pi_0^\dag \scF) [n] \bigr) = 0
\]
if $n \neq 0$.
\end{lem}

\begin{proof}
By Lemma~\ref{lem:convolution-0-exact} and~\eqref{eqn:pidag-0-For}, for $\scF$ in $\sfP^0_{\Iw,\Iw}$ we have
\begin{equation}
\label{eqn:isom-exactness-Phi}
\scR^\wedge \phatstar^0 (\pi_0^\dag \scF) \cong \pi_0^\dag (\scR^0 \star^0_{\Iw} \scF) \cong \pi^{\dag,0} \For^{\Iw,0}_{\Iwu} (\scR^0 \star^0_{\Iw} \scF).
\end{equation}
On the other hand, it is easily seen that the functor
\[
\Pi_{\Iwu,\Iw}^0 \circ \pi_\dag : \mathsf{D}_{\Iwu,\Iwu} \to \mathsf{D}_{\Iwu,\Iw}^0
\]
factors through a triangulated functor
\[
\pi_{\dag,0} : \mathsf{D}_{\Iwu,\Iwu}^0 \to \mathsf{D}_{\Iwu,\Iw}^0
\]
which is left adjoint to $\pi^{\dag,0}$. Moreover, since $\pi_\dag(\Xi^\wedge_!) = \Xi_!$ we have
\begin{equation}
\label{eqn:pidag-Xi}
\pi_{\dag,0} \circ \Pi^0_{\Iwu,\Iwu}(\Xi^\wedge_!) = \Pi_{\Iwu,\Iw}^0(\Xi_!).
\end{equation}
(In particular, this pro-object in $\mathsf{D}_{\Iwu,\Iw}^0$ in fact belongs to $\mathsf{D}_{\Iwu,\Iw}^0$.)
We can now compute using these considerations: for $n \in \Z$ we have
\begin{multline*}
\Hom_{\D^0_{\Iwu,\Iwu}} \bigl( \Pi^0_{\Iwu,\Iwu}(\Xi^\wedge_!), \scR^\wedge \phatstar^0 \pi_0^\dag \scF [n] \bigr) \\
\cong \Hom_{\D^0_{\Iwu,\Iwu}} \bigl( \Pi^0_{\Iwu,\Iwu}(\Xi^\wedge_!), \pi^{\dag,0} \For^{\Iw,0}_{\Iwu} (\scR^0 \star^0_{\Iw} \scF)[n] \bigr) \\
\cong \Hom_{\D^0_{\Iwu,\Iw}} \bigl( \pi_{\dag,0} \Pi^0_{\Iwu,\Iwu}(\Xi^\wedge_!), \For^{\Iw,0}_{\Iwu} (\scR^0 \star^0_{\Iw} \scF)[n] \bigr) \\
\cong \Hom_{\mathsf{D}^0_{\Iwu,\Iw}} \bigl( \Pi_{\Iwu,\Iw}^0(\Xi_!), \For^{\Iw,0}_{\Iwu} (\scR^0 \star^0_{\Iw} \scF)[n] \bigr).
\end{multline*}
(Here, the last space is simply a space of morphisms in the category ${\ind}\mathsf{D}^0_{\Iwu,\Iw}$.)

To proceed further,
consider the ``Iwahori--Whittaker'' category $\mathsf{D}_{\IW,\Iw}$ of sheaves on $\Fl_G$ considered in~\cite[\S 7.1]{brr}. As for the ``finite" flag variety in~\S\ref{ss:G/B-tilting}, we have ``averaging'' functors
\[
\Av_{\Iwu,!} : \mathsf{D}_{\IW,\Iw} \to \mathsf{D}_{\Iwu,\Iw}, \qquad
\Av_{\IW} : \mathsf{D}_{\Iwu,\Iw} \to \mathsf{D}_{\IW,\Iw}
\]
such that $\Av_{\Iwu,!}$ is left adjoint to $\Av_{\IW}$ and both functors are t-exact;
moreover, by construction we have
\[
\Xi_! = \Av_{\Iwu,!} \circ \Av_{\IW} \circ \For^{\Iw}_{\Iwu}(\delta_{\Fl}).
\]

If we denote by $\sfP_{\IW,\Iw}$ the heart of the perverse t-structure on $\D_{\IW,\Iw}$, then the simple objects in $\sfP_{\IW,\Iw}$ are parametrized in a natural way by the subset of $W$ consisting of elements $w$ which have minimal length in their coset $\Wf w$. We can therefore consider the Serre and triangulated subcategories of $\mathsf{P}_{\IW,\Iw}$ and $\D_{\IW,\Iw}$ respectively generated by the simple objects labelled by elements of positive length, and then
the corresponding Serre quotient $\mathsf{P}_{\IW,\Iw}^0$ and Verdier quotient $\mathsf{D}_{\IW,\Iw}^0$, and the quotient functor $\Pi^0_{\IW,\Iw} : \mathsf{D}_{\IW,\Iw} \to \mathsf{D}_{\IW,\Iw}^0$. By Lemma~\ref{lem:quotient-t-str} there exists a unique t-structure on $\sfD_{\IW,\Iw}^0$ such that $\Pi^0_{\IW,\Iw}$ is t-exact; this t-structure is bounded, and its heart identifies with $\sfP_{\IW,\Iw}^0$. The methods of~\cite[\S\S 3.2--3.3]{bgs} can be used to show that the realization functor
\[
\Db \mathsf{P}_{\IW,\Iw} \to \sfD_{\IW,\Iw}
\]
is an equivalence of categories; combining this with Proposition~\ref{prop:miyachi}, we obtain an equivalence of triangulated categories
\begin{equation}
\label{eqn:real-D0-IW}
\Db \sfP^0_{\IW,\Iw} \simto \sfD^0_{\IW,\Iw}
\end{equation}
whose restriction to $\sfP^0_{\IW,\Iw}$ is the obvious embedding.

One can easily check that the functor $\Pi^0_{\Iwu,\Iw} \circ \Av_{\Iwu,!}$, resp.~$\Pi^0_{\IW,\Iw} \circ \Av_{\IW}$, factors uniquely through a t-exact triangulated functor
\[
\Av^0_{\Iwu,!} : \mathsf{D}^0_{\IW,\Iw} \to \mathsf{D}^0_{\Iwu,\Iw}, \quad \text{resp.} \quad \Av^0_{\IW} : \mathsf{D}^0_{\Iwu,\Iw} \to \mathsf{D}^0_{\IW,\Iw},
\]
and that $\Av^0_{\Iwu,!}$ is left adjoint to $\Av^0_{\IW}$.
We then obtain that
\begin{multline*}
\Hom_{\mathsf{D}^0_{\Iwu,\Iw}} \bigl( \Pi_{\Iwu,\Iw}^0(\Xi_!), \For^{\Iw,0}_{\Iwu} (\scR^0 \star^0_{\Iw} \scF)[n] \bigr) \\
\cong \Hom_{\mathsf{D}^0_{\Iwu,\Iw}} \bigl( \Av^0_{\Iwu,!} \Av^0_{\IW} \For^{\Iw,0}_{\Iwu} (\delta^0), \For^{\Iw,0}_{\Iwu} (\scR^0 \star^0_{\Iw} \scF)[n] \bigr) \\
\cong \Hom_{\mathsf{D}^0_{\IW,\Iw}} \bigl( \Av^0_{\IW} \For^{\Iw,0}_{\Iwu} (\delta^0), \Av^0_{\IW} \For^{\Iw,0}_{\Iwu} (\scR^0 \star^0_{\Iw} \scF)[n] \bigr) \\
\cong \Hom_{\Db\mathsf{P}^0_{\IW,\Iw}} \bigl( \Av^0_{\IW} \For^{\Iw,0}_{\Iwu} (\delta^0), \Av^0_{\IW} \For^{\Iw,0}_{\Iwu} (\scR^0 \star^0_{\Iw} \scF)[n] \bigr),
\end{multline*}
where the last step uses~\eqref{eqn:real-D0-IW}.
(Following our conventions, the last space means morphisms in the category ${\ind}\Db\mathsf{P}^0_{\IW,\Iw}$.)

It follows from~\cite[Corollary~9.2]{brr} and the ``transitivity" of the Serre quotient that the functor
\[
\Av_{\IW}^0 \circ \For^{\Iw,0}_{\Iwu} : \mathsf{P}^0_{\Iw,\Iw} \to \mathsf{P}^0_{\IW,\Iw}
\]
is an equivalence of categories. We deduce an equivalence of categories
\[
\Db\sfP^0_{\Iw,\Iw} \to \Db\sfP^0_{\IW,\Iw},
\]
and finally an isomorphism
\[
\Hom_{\D^0_{\Iwu,\Iwu}} \bigl( \Pi_{\Iwu,\Iwu}^0(\Xi^\wedge_!), \scR^\wedge \phatstar^0 (\pi_0^\dag \scF) [n] \bigr) \cong \Hom_{\Db\sfP^0_{\Iw,\Iw}}( \delta^0, \scR^0 \star^0_{\Iw} \scF[n] ),
\]
where in the right-hand side we mean morphisms in ${\ind}\Db\sfP^0_{\Iw,\Iw}$.

Now consider the case $n=0$. Since the natural functor ${\ind}\sfP^0_{\Iw,\Iw} \to {\ind}\Db\sfP^0_{\Iw,\Iw}$ is fully faithful (see~\cite[Proposition~6.1.10]{ks}), we can compute the morphism space above in ${\ind}\sfP^0_{\Iw,\Iw}$; by definition we recover $\Phi_{\Iw,\Iw}(\scF)$, which proves the isomorphism of the lemma.

If $n \neq 0$, we use the monoidal equivalence
\[
\Phi_{\Iw,\Iw} : \sfP_{\Iw,\Iw}^0 \simto \Rep(\rmZ_{G^\vee}(\su))
\]
that sends $\scR^0$ to $\scO(\rmZ_{G^\vee_\bk}(\su))$ (see~\S\ref{ss:regular-quotient-coh}) to obtain an isomorphism
\begin{multline*}
\Hom_{{\ind}\Db\sfP^0_{\Iw,\Iw}}( \delta^0, \scR^0 \star^0_{\Iw} \scF[n] ) \cong \\
\Hom_{{\ind}\Db \Rep(\rmZ_{G^\vee_\bk}(\su))}(\bk, \scO(\rmZ_{G^\vee_\bk}(\su)) \otimes_\bk \Phi_{\Iw,\Iw}(\scF)[n]).
\end{multline*}
If we write
\[
\scO(\rmZ_{G^\vee_\bk}(\su)) = ``\varinjlim_{i \in I}" M_i
\]
where $I$ is filtrant and each $M_i$ belongs to $\Rep(\rmZ_{G^\vee_\bk}(\su))$, then we have
\begin{multline*}
\Hom_{{\ind}\Db \Rep(\rmZ_{G^\vee_\bk}(\su))}(\bk, \scO(\rmZ_{G^\vee_\bk}(\su)) \otimes_\bk \Phi_{\Iwu,\Iwu}(\scF)[n]) = \\
\varinjlim_{i \in I} \Hom_{\Db \Rep(\rmZ_{G^\vee_\bk}(\su))}(\bk, M_i \otimes_\bk \Phi_{\Iwu,\Iwu}(\scF)[n]).
\end{multline*}
It is known that the natural functor
$\Db \Rep(\rmZ_{G^\vee_\bk}(\su)) \to \Db \Rep^\infty(\rmZ_{G^\vee_\bk}(\su))$
is fully faithful. (This follows e.g.~from the much more general results in~\cite[Corollary~2.11 and its proof]{ab}.) We deduce that for any $i \in I$ we have
\[
\Hom_{\Db \Rep(\rmZ_{G^\vee_\bk}(\su))}(\bk, M_i \otimes_\bk \Phi_{\Iw,\Iw}(\scF)[n]) \cong \mathsf{H}^n(\rmZ_{G^\vee_\bk}(\su), M_i \otimes_\bk \Phi_{\Iw,\Iw}(\scF)),
\]
and then using the fact that cohomology commutes with filtrant direct limits (see~\cite[Lemma~I.4.17]{jantzen}) that
\begin{multline*}
\Hom_{{\ind}\Db \Rep(\rmZ_{G^\vee_\bk}(\su))}(\bk, \scO(\rmZ_{G^\vee_\bk}(\su)) \otimes_\bk \Phi_{\Iw,\Iw}(\scF)[n]) =\\
 \mathsf{H}^n(\rmZ_{G^\vee_\bk}(\su), \scO(\rmZ_{G^\vee_\bk}(\su)) \otimes_\bk \Phi_{\Iw,\Iw}(\scF)).
\end{multline*}
Finally, we use the fact that $\scO(\rmZ_{G^\vee_\bk}(\su)) \otimes_\bk \Phi_{\Iw,\Iw}(\scF)$ is injective in $\Rep^\infty(\rmZ_{G^\vee_\bk}(\su))$ (see~\cite[\S\S I.3.9--I.3.10]{jantzen}) to conclude that this space vanishes.
\end{proof}


As a consequence of Lemma~\ref{lem:exactness-Phi} we obtain the following properties.

\begin{prop}
\label{prop:Phi-exact}
The functor $\Phi_{\Iwu,\Iwu}$ is exact, and takes values in $\Rep_0(\bbI_\Sigma)$. Moreover, the diagram~\eqref{eqn:diag-compatibility-Phi} commutes.
\end{prop}

\begin{proof}
Recall that $\sfP^0_{\Iwu,\Iwu}$ is a finite-length category, and that its simple objects are the objects $\pi^\dag_0(\scF)$ with $\scF$ simple in $\sfP^0_{\Iw,\Iw}$. Given a short exact sequence
\[
\scF_1 \hookrightarrow \scF_2 \twoheadrightarrow \scF_3
\]
in $\sfP^0_{\Iwu,\Iwu}$, by Lemma~\ref{lem:convolution-0-exact} we have a short exact sequence
\begin{equation}
\label{eqn:ses-conv-R}
\scR^\wedge \phatstar^0 \scF_1 \hookrightarrow \scR^\wedge \phatstar^0 \scF_2 \twoheadrightarrow \scR^\wedge \phatstar^0 \scF_3
\end{equation}
in ${\ind}\sfP^0_{\Iwu,\Iwu}$. By~\cite[Proposition~8.6.6(1)]{ks}, there exists a filtrant category $I$ and an inductive system of short exact sequences
\[
\scM_i^1 \hookrightarrow \scM_i^2 \twoheadrightarrow \scM_i^3
\]
in $\sfP^0_{\Iwu,\Iwu}$
from which~\eqref{eqn:ses-conv-R} is obtained by taking formal direct limits. Write also 
\[
\Pi_{\Iwu,\Iwu}^0(\Xi^\wedge_!) = ``\varprojlim_n{}" \scG_n
\]
for some objects $\scG_n$ in $\sfD^0_{\Iwu,\Iwu}$.
Then for any $n$ and $i$ we have an exact sequence
\begin{multline*}
\Hom_{\D^0_{\Iwu,\Iwu}}(\scG_n, \scM_i^1[-1]) \to \Hom_{\D^0_{\Iwu,\Iwu}}(\scG_n, \scM_i^2[-1]) \to \Hom_{\D^0_{\Iwu,\Iwu}}(\scG_n, \scM_i^3[-1]) \\
\to \Hom_{\D^0_{\Iwu,\Iwu}}(\scG_n, \scM_i^1) \to \Hom_{\D^0_{\Iwu,\Iwu}}(\scG_n, \scM_i^2)
\to \Hom_{\D^0_{\Iwu,\Iwu}}(\scG_n, \scM_i^3) \\
\to \Hom_{\D^0_{\Iwu,\Iwu}}(\scG_n, \scM_i^1[1]) \to \Hom_{\D^0_{\Iwu,\Iwu}}(\scG_n, \scM_i^2[1]) \to \Hom_{\D^0_{\Iwu,\Iwu}}(\scG_n, \scM_i^3[1]).
\end{multline*}
By exactness of filtrant direct limits we deduce an exact sequence
\begin{multline*}
\Hom_{\mathsf{D}^0_{\Iwu,\Iwu}} \bigl( \Pi_{\Iwu,\Iwu}^0(\Xi^\wedge_!), \scR^\wedge \phatstar^0 \scF_1[-1]) \to \Hom_{\mathsf{D}^0_{\Iwu,\Iwu}} \bigl( \Pi_{\Iwu,\Iwu}^0(\Xi^\wedge_!), \scR^\wedge \phatstar^0 \scF_2[-1]) \\
\to \Hom_{\mathsf{D}^0_{\Iwu,\Iwu}} \bigl( \Pi_{\Iwu,\Iwu}^0(\Xi^\wedge_!), \scR^\wedge \phatstar^0 \scF_3[-1]) \to \Phi_{\Iwu,\Iwu} (\scF_1) \to \Phi_{\Iwu,\Iwu} (\scF_2) \to \Phi_{\Iwu,\Iwu} (\scF_3) \\
\to \Hom_{\mathsf{D}^0_{\Iwu,\Iwu}} \bigl( \Pi_{\Iwu,\Iwu}^0(\Xi^\wedge_!), \scR^\wedge \phatstar^0 \scF_1[1])
\to \Hom_{\mathsf{D}^0_{\Iwu,\Iwu}} \bigl( \Pi_{\Iwu,\Iwu}^0(\Xi^\wedge_!), \scR^\wedge \phatstar^0 \scF_2[1]) \\
\to \Hom_{\mathsf{D}^0_{\Iwu,\Iwu}} \bigl( \Pi_{\Iwu,\Iwu}^0(\Xi^\wedge_!), \scR^\wedge \phatstar^0 \scF_3[1]).
\end{multline*}
Using these exact sequences and Lemma~\ref{lem:exactness-Phi} one proves by induction on the length that for any $\scF$ in $\sfP_{\Iwu,\Iwu}^0$ the module $\Phi_{\Iwu,\Iwu} (\scF)$ is finite-dimensional and annihilated by a power of $\cJ$, and that
\[
\Hom_{\mathsf{D}^0_{\Iwu,\Iwu}} \bigl( \Pi_{\Iwu,\Iwu}^0(\Xi^\wedge_!), \scR^\wedge \phatstar^0 \scF[-1])=0=\Hom_{\mathsf{D}^0_{\Iwu,\Iwu}} \bigl( \Pi_{\Iwu,\Iwu}^0(\Xi^\wedge_!), \scR^\wedge \phatstar^0 \scF[1]). 
\]
The first property shows that $\Phi_{\Iwu,\Iwu}$ takes values in $\Rep_0(\bbI_\Sigma)$, and the second one implies (using again the exact sequence above) the exactness of $\Phi_{\Iwu,\Iwu}$.

The commutativity of the diagram~\eqref{eqn:diag-compatibility-Phi} has been established in Lem\-ma~\ref{lem:exactness-Phi}.
\end{proof}

We can now prove the compatibility of $\Phi_{\Iwu,\Iwu}$ with the equivalence $\Phi_{U,U}$ of Theorem~\ref{thm:reg-quotient-finite}.

\begin{prop}
\label{prop:isom-Phi-U-Iu}
The diagram~\eqref{eqn:diag-compatibility-Phi-2} commutes; in other words,
 for any $\scF$ in $\sfP_{U,U}^0$ there exists a canonical isomorphism of $\bbI_\Sigma$-modules
 \[
  \Phi_{U,U}(\scF) \simto \Phi_{\Iwu,\Iwu}(\scF)
 \]
 where the left-hand side is endowed with the trivial structure as a representation.
\end{prop}

\begin{proof}
 From the definition of the functors we see that there exists a functorial morphism
  \[
  \Phi_{U,U}(\scF) \to \Phi_{\Iwu,\Iwu}(\scF)
 \]
 induced by the unit morphism $\delta^{\wedge} \to \scR^\wedge$ and the functor~\eqref{eqn:D-U-Iu}. 
 Using Lemma~\ref{lem:exactness-Phi} we see that this morphism is an isomorphism when $\scF$ is the unique simple object in $\sfP_{U,U}^0$, namely $\pi^\dag(\For^{\Iw}_{\Iwu} \IC_e)$. Since both functors are exact, the five-lemma implies that this morphism is invertible for any $\scF$ in $\sfP_{U,U}^0$, which finishes the proof.
\end{proof}

\subsection{Images of some truncated Wakimoto sheaves}

Recall the representations of $\bbI_\Sigma^\wedge$ introduced in~\S\ref{ss:completions}, and the ``truncation" functors introduced in~\S\ref{ss:truncation-ps} and~\S\ref{ss:truncation-Rep}. By construction, given $\scF \in \sfP^\wedge_{\Iwu,\Iwu}$, resp.~$M \in \Rep(\bbI_\Sigma^\wedge)$, these functors provide a projective system $(\mathsf{C}_m^0(\Wak^\wedge_{w_\circ(\lambda)}) : m \geq 1)$ of objects in $\sfP^0_{\Iwu,\Iwu}$, resp.~a projective system $(\mathsf{D}_m(M) : m \geq 1)$ of objects in $\Rep_0(\bbI_\Sigma)$.
Our goal in this subsection is to prove the following claims.

\begin{prop}
\phantomsection
\label{prop:Phi-Wak-antidom}
\begin{enumerate}
\item
\label{it:Phi-Wak-antidom}
For any $\lambda \in X_*^+(T)$, there exists an isomorphism of projective systems
 \[
  \bigl( \Phi_{\Iwu,\Iwu}(\mathsf{C}_m^0(\Wak^\wedge_{w_\circ(\lambda)})) : m \geq 1 \bigr) \cong \bigl( \mathsf{D}_m(\mathscr{M}^\wedge_{\st(w_\circ(\lambda))}) : m \geq 1 \bigr).
 \]
\item
\label{it:Phi-Delta-omega}
For any $\omega \in \Omega$, there exists an isomorphism of projective systems
 \[
  \bigl( \Phi_{\Iwu,\Iwu}(\mathsf{C}_m^0(\Delta^\wedge_{\omega})) : m \geq 1 \bigr) \cong \bigl( \mathsf{D}_m(\mathscr{M}^\wedge_{\omega}) : m \geq 1 \bigr).
 \]
\end{enumerate}
 \end{prop}

We will need some preliminaries.

\begin{lem}
\label{lem:Phi-ZV}
For any $V \in \Rep(G^\vee_\bk)$ and $\scF$ in $\sfP^0_{\Iwu,\Iwu}$ we have a canonical isomorphism of $\bbI_\Sigma$-modules
\[
\Phi_{\Iwu,\Iwu} \left( \tsZ(V) \phatstar^0 \scF \right) \cong V \otimes_\bk \Phi_{\Iwu,\Iwu}(\scF).
\]
\end{lem}

\begin{proof}
If we write $\scF=\Pi^0_{\Iwu,\Iwu}(\scG)$ with $\scG \in \mathsf{P}_{\Iwu,\Iwu}$ then by exactness of the functor $\tsZ(V) \hatstar (-)$ (see Theorem~\ref{thm:gaitsgory-mon}\eqref{it:gaitsgory-mon-5}) we have
\[
 \tsZ(V) \phatstar^0 \scF = \Pi^0_{\Iwu,\Iwu}(\tsZ(V) \hatstar \scG).
\]
Using Lemma~\ref{lem:tsZ-R} we deduce an isomorphism
\[
 \scR^\wedge \phatstar^0 (\tsZ(V) \phatstar^0 \scF) \cong (\scR^\wedge \phatstar^0 \scF) \otimes_\bk V,
\]
and then an isomorphism
\[
 \Phi_{\Iwu,\Iwu}(\tsZ(V) \phatstar^0 \scF) \cong \Phi_{\Iwu,\Iwu}(\scF) \otimes_\bk V,
\]
as desired.
\end{proof}

By exactness of the functor $\tsZ(V) \hatstar (-)$ (see Theorem~\ref{thm:gaitsgory-mon}\eqref{it:gaitsgory-mon-5}), for any $m \geq 1$ we have
\[
 \mathsf{C}_m^0(\tsZ(V)) \cong \tsZ(V) \hatstar^0 \mathsf{C}_m^0(\delta^\wedge).
\]
In view of Lemma~\ref{lem:Phi-ZV} and Proposition~\ref{prop:isom-Phi-U-Iu}, we deduce a canonical isomorphism
\[
 \Phi_{\Iwu,\Iwu}(\mathsf{C}_m^0(\tsZ(V))) \cong V \otimes_\bk \Phi_{U,U}(\mathsf{C}_m^0(\delta^\wedge)).
\]
Now if we denote by $\Delta (T^\vee_\bk)^{(m)}$ the spectrum of $\scO(T^\vee_\bk)/(\mathcal{J}^m \cdot \scO(T^\vee_\bk))$, embedded diagonally as a closed subscheme of $T^\vee_\bk \times_{T^\vee_\bk / \Wf} T^\vee_\bk$, it is clear that $\Phi_{U,U}(\mathsf{C}_m^0(\delta^\wedge)) \cong \scO_{\Delta (T^\vee_\bk)^{(m)}}$, so that we finally obtain an isomorphism
\begin{equation}
 \label{eqn:image-Phi-Z}
 \Phi_{\Iwu,\Iwu}(\mathsf{C}_m^0(\tsZ(V))) \cong V \otimes_\bk \scO_{\Delta (T^\vee_\bk)^{(m)}}.
\end{equation}
By Theorem~\ref{thm:gaitsgory-mon}\eqref{it:gaitsgory-mon-6} and Lemma~\ref{lem:image-monodromy}, under this identification the automorphism of the left-hand side induced by $\hat{\sm}_V$ corresponds to the tautological automorphism of the $\bbI_\Sigma$-module $V \otimes_\bk \scO_{\Delta (T^\vee_\bk)^{(m)}}$. By Remark~\ref{rmk:taut-autom}, this automorphism can also be obtained from the tautological automorphism of $V \otimes \scO_{G^\vee_\bk} \in \Coh^{G^\vee_\bk}(G^\vee_\bk)$ by pullback under the composition
\[
\Delta (T^\vee_\bk)^{(m)} \hookrightarrow T^\vee_\bk \times_{T^\vee_\bk / \Wf} T^\vee_\bk \hookrightarrow \mathrm{St}_{\mathrm{m}} \to G^\vee_\bk.
\]
(Here, $\mathrm{St}_{\mathrm{m}}$ is the multiplicative Steinberg variety of $G^\vee_\bk$, and the second embedding is provided by~\eqref{eqn:section-Stm}.)

%
%

\begin{proof}[Proof of Proposition~\ref{prop:Phi-Wak-antidom}]
\eqref{it:Phi-Wak-antidom}
We fix $\lambda \in X_*^+(T)$, and consider $V \in \Rep(G^\vee_\bk)$ which has highest weight $\lambda$ (in the sense of~\S\ref{ss:coh-Groth}). Recall from Theorem~\ref{thm:gaitsgory-mon}\eqref{it:fm-central-Wak} that the perverse sheaf $\tsZ(V)$ admits a Wakimoto filtration whose subquotients have as labels the weights of $V$. In particular there exists an embedding $\Wak^\wedge_{w_\circ(\lambda)} \hookrightarrow \tsZ(V)$ whose cokernel admits a Wakimoto filtration. By Proposition~\ref{prop:exactness-Cm}, the induced morphism
\begin{equation}
\label{eqn:embeddings-Wak-Z}
 \mathsf{C}^0_m(\Wak^\wedge_{w_\circ(\lambda)}) \to \mathsf{C}^0_m(\tsZ(V))
\end{equation}
is injective for any $m \geq 1$.
By Lemma~\ref{lem:properties-Wak}\eqref{it:properties-Wak-5} and Theorem~\ref{thm:gaitsgory-mon}\eqref{it:gaitsgory-mon-6}, the endomorphism $f:=\hat{\sm}_V-\mu_{\tsZ(V)}(e^{w_\circ(\lambda)} \otimes 1)$ of $\tsZ(V)$ vanishes on the image of $\Wak^\wedge_{w_\circ(\lambda)}$. We deduce that, for any $m \geq 1$, $\mathsf{C}^0_m(f)$ vanishes on the image of $\mathsf{C}^0_m(\Wak^\wedge_{w_\circ(\lambda)})$.

For any $m \geq 1$ we can consider the finitely generated $\scO((T^\vee_\bk \times_{T^\vee_\bk / \Wf} T^\vee_\bk)^{(m)})$-module 
\[
\Phi_{\Iwu,\Iwu}(\mathsf{C}_m^0(\Wak^\wedge_{w_\circ(\lambda)})),
\]
and, by exactness of $\Phi_{\Iwu,\Iwu}$, for any $m' \geq m$ we have
\[
 \scO((T^\vee_\bk \times_{T^\vee_\bk / \Wf} T^\vee_\bk)^{(m)}) \otimes_{\scO(T^\vee_\bk \times_{T^\vee_\bk / \Wf} T^\vee_\bk)} \Phi_{\Iwu,\Iwu}(\mathsf{C}_{m'}^0(\Wak^\wedge_{w_\circ(\lambda)})) \cong \Phi_{\Iwu,\Iwu}(\mathsf{C}_m^0(\Wak^\wedge_{w_\circ(\lambda)})).
\]
By~\cite[Proposition~7.2.9]{ega1}, it follows that
\[
 K_\lambda := \varprojlim_m \Phi_{\Iwu,\Iwu}(\mathsf{C}_m^0(\Wak^\wedge_{w_\circ(\lambda)}))
\]
is a finitely generated $\scO(\FN_{T^\vee_\bk \times_{T^\vee_\bk / \Wf} T^\vee_\bk}(\{(e,e)\}))$-module such that
\begin{multline}
\label{eqn:reduction-Klambda}
 \scO((T^\vee_\bk \times_{T^\vee_\bk / \Wf} T^\vee_\bk)^{(m)}) \otimes_{\scO(\FN_{T^\vee_\bk \times_{T^\vee_\bk / \Wf} T^\vee_\bk}(\{(e,e)\}))} K_\lambda \\
  \cong \Phi_{\Iwu,\Iwu}(\mathsf{C}_m^0(\Wak^\wedge_{w_\circ(\lambda)}))
\end{multline}
for any $m \geq 1$. From the embeddings~\eqref{eqn:embeddings-Wak-Z}, and by exactness of $\Phi_{\Iwu,\Iwu}$, we obtain an embedding
\[
 K_\lambda \hookrightarrow \varprojlim_m \Phi_{\Iwu,\Iwu}(\mathsf{C}^0_m(\tsZ(V))),
\]
i.e., using~\eqref{eqn:image-Phi-Z}, an embedding
\[
 K_\lambda \hookrightarrow V \otimes \scO(\Delta \FN_{T^\vee_\bk}(\{e\}))
\]
where $\Delta \FN_{T^\vee_\bk}(\{e\})$ is the image of the diagonal embedding of $\FN_{T^\vee_\bk}(\{e\})$ in the scheme $\FN_{T^\vee_\bk \times_{T^\vee_\bk / \Wf} T^\vee_\bk}(\{(e,e)\})$.

Now, consider the multiplicative Grothendieck resolution $\widetilde{G^\vee_\bk}$ of $G^\vee_\bk$, and
recall from Lemma~\ref{lem:kernel-end-free-coh}
that $V \otimes \scO_{\Delta \widetilde{G^\vee_\bk}}$ has a canonical endomorphism whose kernel identifies with $\scO_{\Delta \widetilde{G^\vee_\bk}}(w_\circ(\lambda))$. Restricting to the regular part, using the equivalence of Proposition~\ref{prop:rest-Steinberg-section} and Lemma~\ref{lem:line-bundles-J}, and then completing, we deduce that $V \otimes \scO(\Delta \FN_{T^\vee_\bk}(\{e\}))$ has a canonical endomorphism whose kernel is isomorphic to $\mathscr{M}^\wedge_{w_\circ(\lambda)}$. By construction and the comments just before the proof, this endomorphism vanishes on the image of $K_\lambda$, which provides an embedding
\[
 K_\lambda \hookrightarrow \mathscr{M}_{w_\circ(\lambda)}^\wedge.
\]
We will prove that this embedding is also surjective, hence an isomorphism, which will conclude the proof in view of~\eqref{eqn:reduction-Klambda}.

By~\cite[Corollaire~7.1.14]{ega1}, to prove this surjectivity it suffices to prove that the induced morphism
\[
 K_\lambda / (\mathcal{J} \cdot K_\lambda) \to \mathscr{M}^\wedge_{w_\circ(\lambda)} / (\mathcal{J} \cdot \mathscr{M}^\wedge_{w_\circ(\lambda)})
\]
is surjective. Now, by construction this morphism is injective, so that to conclude it suffices to prove that
\[
 \dim(K_\lambda / (\mathcal{J} \cdot K_\lambda))=\dim(\mathscr{M}^\wedge_{w_\circ(\lambda)} / (\mathcal{J} \cdot \mathscr{M}^\wedge_{w_\circ(\lambda)})).
\]
It is clear that the right-hand side equals $\dim(\scO(T^\vee_\bk)/(\mathcal{J} \cdot \scO(T^\vee_\bk)))$. For the left-hand side, we remark that
\[
 K_\lambda / (\mathcal{J} \cdot K_\lambda) = \Phi_{\Iwu,\Iwu}(\mathsf{C}_1^0(\Wak^\wedge_{w_\circ(\lambda)})).
\]
From the proof of Lemma~\ref{lem:flatness-D-N} one sees that the object $\mathsf{C}_1^0(\Wak^\wedge_{w_\circ(\lambda)})$ is an extension of $\dim(\scO(T^\vee_\bk)/(\mathcal{J} \cdot \scO(T^\vee_\bk)))$ many copies of the simple object $\Pi^0_{\Iw,\Iw}(\Delta^\Iw_{\st(w_\circ(\lambda))})$. (The fact that this object is simple follows from~\cite[Lemma~4.5]{brr}.) We deduce that
\[
 \dim(K_\lambda / (\mathcal{J} \cdot K_\lambda))=\dim(\scO(T^\vee_\bk)/(\mathcal{J} \cdot \scO(T^\vee_\bk))),
 \]
 as desired.
 
 \eqref{it:Phi-Delta-omega}
 Since $\tFl_{G,\omega}$ is closed the natural morphism $\Delta^\wedge_\omega \to \nabla^\wedge_\omega$ is an isomorphism; in the statement one can therefore replace $\Delta^\wedge_\omega$ by $\nabla^\wedge_\omega$. Let us write $\omega=w \st(\lambda)$ with $w \in \Wf$ and $\lambda \in X_*(T)$; then we have $\omega \st(-\lambda)=w$ with $\ell(\omega \st(\lambda)) = \ell(\omega) + \ell(\st(-\lambda))$, hence by Lemma~\ref{lem:convolution-stand-costand}\eqref{it:convolution-stand-costand-1}--\eqref{it:convolution-stand-costand-2} we have
 \[
  \nabla^\wedge_\omega \cong \nabla^\wedge_w \hatstar \Delta^\wedge_{\st(\lambda)}.
 \]
After fixing such an isomorphism, using Proposition~\ref{prop:monoidality-C0} we obtain for any $m \geq 1$ an isomorphism
\[
 \sfC_m^0(\nabla^\wedge_\omega) \cong \sfC_m^0(\nabla^\wedge_w) \pstar^0_{\Iwu} \sfC_m^0(\Delta^\wedge_{\st(\lambda)}).
\]
Applying~\eqref{eqn:monoidality-morphism}, one deduces a canonical morphism
\[
 \Phi_{\Iwu}(\sfC_m^0(\nabla^\wedge_w)) \oast \Phi_{\Iwu,\Iwu}(\sfC_m^0(\Delta^\wedge_{\st(\lambda)})) \to \Phi_{\Iwu,\Iwu}(\nabla^\wedge_\omega).
\]
  From~\eqref{eqn:formula-length} one sees that $\lambda$ is antidominant; using~\eqref{it:Phi-Wak-antidom}, we deduce isomorphisms
  \[
  \Phi_{\Iwu,\Iwu}(\mathsf{C}^0_m(\Delta^\wedge_{\st(\lambda)})) \cong \Phi_{\Iwu,\Iwu}(\mathsf{C}^0_m(\scW^\wedge_{\lambda})) \cong \mathsf{D}_m(\mathscr{M}^\wedge_{\st(\lambda)}).
  \]
  On the other hand, using Lemma~\ref{lem:PhiUU-cost} and Proposition~\ref{prop:isom-Phi-U-Iu} one obtains isomorphisms
  \[
   \Phi_{\Iwu,\Iwu}(\mathsf{C}^0_m(\Delta^\wedge_{w})) \cong \mathsf{D}_m(\mathscr{M}^\wedge_w)
  \]
  for any $m \geq 1$.
Using Lemma~\ref{lem:truncation-Rep-mon} and~\eqref{eqn:convolution-D}, we deduce that our construction provides morphisms
\[
 \mathsf{D}_m(\mathscr{M}^\wedge_\omega) \to \Phi_{\Iwu,\Iwu}(\mathsf{C}^0_m(\Delta^\wedge_\omega))
\]
defining a morphism of projective systems,
and to conclude it suffices to prove that these morphisms are isomorphisms. Computing as in~\eqref{it:Phi-Wak-antidom} one sees that the two modules involved have the same dimension, so that it suffices to prove that these morphisms are surjective. Now $\scO((T^\vee_\bk \times_{T^\vee_\bk / \Wf} T^\vee_\bk)^{(m)})$ is a local ring, with maximal ideal the image of the ideal $\cI$ considered in~\S\ref{ss:completions}. By Nakayama's lemma, to prove surjectivity it therefore suffices to prove that the induced morphism
\[
 \mathsf{D}_m(\mathscr{M}^\wedge_\omega) / \cI \cdot \mathsf{D}_m(\mathscr{M}^\wedge_\omega) \to \Phi_{\Iwu,\Iwu}(\mathsf{C}^0_m(\Delta^\wedge_\omega)) / \cI \cdot (\Phi_{\Iwu,\Iwu}(\mathsf{C}^0_m(\Delta^\wedge_\omega)))
\]
is an isomorphism. Here by exactness of $\Phi_{\Iwu,\Iwu}$ the right-hand side identifies with $\Phi_{\Iw,\Iw}(\Pi^0_{\Iw,\Iw}(\pi^\dag \Delta_\omega^{\Iw}))$, and the morphism is an isomorphism by monoidality of $\Phi_{\Iw,\Iw}$.
\end{proof}


\subsection{Fully faithfulness}
\label{ss:fully-faithfulness}

We will now prove that the functor $\Phi_{\Iwu,\Iwu}$ is fully faithful. The proof will rely on the following easy lemma.

\begin{lem}
\label{lem:ff-criterion}
 Let $\mathsf{A}$, $\mathsf{A}'$ be abelian categories, and let $F : \mathsf{A} \to \mathsf{A}'$ be an exact functor. Assume that every object in $\mathsf{A}$ has finite length, and that for any simple objects $M,M'$ in $\mathsf{A}$:
 \begin{itemize}
  \item the morphism
  \[
  \Hom_{\mathsf{A}}(M,M') \to \Hom_{\mathsf{A}'}(F(M),F(M'))
  \]
  induced by $F$ is an isomorphism;
  \item the morphism
  \[
  \Ext^1_{\mathsf{A}}(M,M') \to \Ext^1_{\mathsf{A}'}(F(M),F(M'))
  \]
  induced by $F$ is injective.
 \end{itemize}
Then $F$ is fully faithful.
\end{lem}

\begin{proof}
 One proves, by induction on the sum of the lengths of the objects involved and using the four- and five-lemmas, that for any objects $M,M' \in \mathsf{A}$ the morphism
\[ 
 \Hom_{\mathsf{A}}(M,M') \to \Hom_{\mathsf{A}'}(F(M),F(M')),
 \]
 resp.
 \[
 \Ext^1_{\mathsf{A}}(M,M') \to \Ext^1_{\mathsf{A}'}(F(M),F(M')),
 \]
 is an isomorphism, resp.~is injective, which implies the claim.
\end{proof}

In order to use Lemma~\ref{lem:ff-criterion} we will need to describe some groups of extensions in $\sfP^0_{\Iwu,\Iwu}$, which we will relate to some groups of extensions in $\sfP^0_{\Iwu,\Iw}$.
First, the forgetful functor
\[
\For^{\Iw}_{\Iwu} : \D_{\Iw,\Iw} \to \D_{\Iwu,\Iw}
\]
induces a fully faithful functor $\sfP_{\Iw,\Iw}^0 \to \sfP_{\Iwu,\Iw}^0$; we deduce a canonical injective morphism
\begin{equation}
\label{eqn:Ext1-1}
\Ext^1_{\sfP_{\Iw,\Iw}^0}(\delta^0,\delta^0) \to \Ext^1_{\sfP_{\Iwu,\Iw}^0}(\For^{\Iw,0}_{\Iwu}(\delta^0),\For^{\Iw,0}_{\Iwu}(\delta^0)).
\end{equation}
On the other hand, we also have the similar quotient category $\sfP^0_{U,B}$ constructed out of the category $\sfP_{U,B}$ considered in the proof of Lemma~\ref{lem:Xi-projective}. The closed embedding $G/B \to \Fl_G$ induces a fully faithful functor $\sfP^0_{U,B} \to \sfP^0_{\Iwu,\Iw}$, and $\For^{\Iw,0}_{\Iwu}(\delta^0)$ is the image under this functor of a canonical object of $\sfP^0_{U,B}$ denoted in the same way. We deduce another canonical injective morphism
\begin{equation}
\label{eqn:Ext1-2}
\Ext^1_{\sfP_{U,B}^0}(\For^{\Iw,0}_{\Iwu}(\delta^0),\For^{\Iw,0}_{\Iwu}(\delta^0)) \to \Ext^1_{\sfP_{\Iwu,\Iw}^0}(\For^{\Iw,0}_{\Iwu}(\delta^0),\For^{\Iw,0}_{\Iwu}(\delta^0)).
\end{equation}

\begin{lem}
\label{lem:Ext1-DIuI}
The morphisms~\eqref{eqn:Ext1-1} and~\eqref{eqn:Ext1-2} induce an isomorphism of $\bk$-vector spaces
\[
\Ext^1_{\sfP_{\Iw,\Iw}^0}(\delta^0,\delta^0) \oplus \Ext^1_{\sfP_{U,B}^0}(\For^{\Iw,0}_{\Iwu}(\delta^0),\For^{\Iw,0}_{\Iwu}(\delta^0)) \simto \Ext^1_{\sfP_{\Iwu,\Iw}^0}(\For^{\Iw,0}_{\Iwu}(\delta^0),\For^{\Iw,0}_{\Iwu}(\delta^0)).
\]
\end{lem}

\begin{proof}
Recall from the proof of Lemma~\ref{lem:exactness-Phi} the exact functor
\[
\Av_\IW^0 : \sfP_{\Iwu,\Iw}^0 \to \sfP_{\IW,\Iw}^0
\]
such that $\Av_{\IW}^0 \circ \For^{\Iw,0}_{\Iwu}$ induces an equivalence $\sfP_{\Iw,\Iw}^0 \simto \sfP_{\IW,\Iw}^0$. The functor $\Av_\IW^0$ induces a morphism
\begin{multline}
\label{eqn:Ext1-3}
\Ext^1_{\sfP_{\Iwu,\Iw}^0}(\For^{\Iw,0}_{\Iwu}(\delta^0),\For^{\Iw,0}_{\Iwu}(\delta^0)) \to \\
 \Ext^1_{\sfP_{\IW,\Iw}^0}(\Av_\IW^0(\For^{\Iw,0}_{\Iwu}(\delta^0)),\Av_\IW^0(\For^{\Iw,0}_{\Iwu}(\delta^0)))
\end{multline}
whose composition with~\eqref{eqn:Ext1-1} is an isomorphism; to prove the lemma it therefore suffices to prove that its kernel is the image of~\eqref{eqn:Ext1-2}.

Now, recall also (from the same proof) the functor
\[
\Av^0_{\Iwu,!} : \D_{\IW,\Iw}^0 \to \D_{\Iwu,\Iw}^0
\]
which is left adjoint to $\Av_\IW^0 : \D_{\Iwu,\Iw}^0 \to \D_{\IW,\Iw}^0$, and which satisfies
\[
\Av^0_{\Iwu,!} \circ \Av_\IW^0 \circ \For^{\Iw,0}_{\Iwu} (\delta^0) = \Pi^0_{\Iwu,\Iw}(\Xi_!).
\]
By adjunction and standard properties of t-structures we have
\begin{multline*}
\Ext^1_{\sfP_{\IW,\Iw}^0}(\Av_\IW^0(\For^{\Iw,0}_{\Iwu}(\delta^0)),\Av_\IW^0(\For^{\Iw,0}_{\Iwu}(\delta^0))) = \\
\Hom_{\D_{\IW,\Iw}^0}(\Av_\IW^0(\For^{\Iw,0}_{\Iwu}(\delta^0)),\Av_\IW^0(\For^{\Iw,0}_{\Iwu}(\delta^0))[1]) \cong \\
\Hom_{\D_{\Iwu,\Iw}^0}(\Pi^0_{\Iwu,\Iw}(\Xi_!),\For^{\Iw,0}_{\Iwu}(\delta^0)[1]),
\end{multline*}
and the morphism~\eqref{eqn:Ext1-3} identifies with the morphism
\[
\Hom_{\sfD_{\Iwu,\Iw}^0}(\For^{\Iw,0}_{\Iwu}(\delta^0),\For^{\Iw,0}_{\Iwu}(\delta^0)[1]) \to \Hom_{\D_{\Iwu,\Iw}^0}(\Pi^0_{\Iwu,\Iw}(\Xi_!),\For^{\Iw,0}_{\Iwu}(\delta^0)[1])
\]
induced by the natural surjection $\Pi^0_{\Iwu,\Iw}(\Xi_!) \twoheadrightarrow \For^{\Iw,0}_{\Iwu}(\delta^0)$. If we denote by $\mathscr{K}$ the kernel of this morphism, we have a long exact sequence
\begin{multline}
\label{eqn:les-Ext1}
0 \to \Hom_{\sfP_{\Iwu,\Iw}^0}(\For^{\Iw,0}_{\Iwu}(\delta^0),\For^{\Iw,0}_{\Iwu}(\delta^0)) \to \Hom_{\sfP_{\Iwu,\Iw}^0}(\Pi^0_{\Iwu,\Iw}(\Xi_!),\For^{\Iw,0}_{\Iwu}(\delta^0)) \\
\xrightarrow{f} \Hom_{\sfP_{\Iwu,\Iw}^0}(\mathscr{K},\For^{\Iw,0}_{\Iwu}(\delta^0))
\to \Ext^1_{\sfP_{\Iwu,\Iw}^0}(\For^{\Iw,0}_{\Iwu}(\delta^0),\For^{\Iw,0}_{\Iwu}(\delta^0)) \\
 \to \Ext^1_{\sfP_{\Iwu,\Iw}^0}(\Pi^0_{\Iwu,\Iw}(\Xi_!),\delta^0) \to \cdots,
\end{multline}
which identifies the kernel of~\eqref{eqn:Ext1-3} with $\mathrm{coker}(f)$.
But since all the objects involved belong to the full subcategory $\sfP_{U,B}^0$, we also have a similar long exact sequence
\begin{multline*}
0 \to \Hom_{\sfP_{U,B}^0}(\For^{\Iw,0}_{\Iwu}(\delta^0),\For^{\Iw,0}_{\Iwu}(\delta^0)) \to \Hom_{\sfP_{U,B}^0}(\Pi^0_{\Iwu,\Iw}(\Xi_!),\For^{\Iw,0}_{\Iwu}(\delta^0)) \\
 \xrightarrow{g} \Hom_{\sfP_{U,B}^0}(\mathscr{K},\For^{\Iw,0}_{\Iwu}(\delta^0))
\to \Ext^1_{\sfP_{U,B}^0}(\For^{\Iw,0}_{\Iwu}(\delta^0),\For^{\Iw,0}_{\Iwu}(\delta^0)) \\
 \to \Ext^1_{\sfP_{U,B}^0}(\Pi^0_{\Iwu,\Iw}(\Xi_!),\For^{\Iw,0}_{\Iwu}(\delta^0)) \to \cdots.
\end{multline*}
Since $\Xi_!$ is the projective cover of $\For^{\Iw}_{\Iwu}(\delta_{\Fl})$ in $\sfP_{U,B}$ (see the proof of Lemma~\ref{lem:Xi-projective}), $\Pi^0_{\Iwu,\Iw}(\Xi_!)$ is the projective cover of $\For^{\Iw,0}_{\Iwu}(\delta^0)$ in $\sfP^0_{U,B}$, which implies that
\[
\Ext^1_{\sfP_{U,B}^0}(\Pi^0_{\Iwu,\Iw}(\Xi_!),\For^{\Iw,0}_{\Iwu}(\delta^0))=0,
\]
and therefore allows to identify $\Ext^1_{\sfP_{U,B}^0}(\For^{\Iw,0}_{\Iwu}(\delta^0),\For^{\Iw,0}_{\Iwu}(\delta^0))$ with $\mathrm{coker}(g)$. By fully faithfulness of the functor $\sfP_{U,B}^0 \to \sfP_{\Iwu,\Iw}^0$ the domains and codomains of $f$ and $g$ identify, in a way compatible with these morphisms, so that their cokernels are canonically isomorphic.

Gathering these identifications we obtain an identification of the kernel of~\eqref{eqn:Ext1-3} with $\Ext^1_{\sfP_{U,B}^0}(\For^{\Iw,0}_{\Iwu}(\delta^0),\For^{\Iw,0}_{\Iwu}(\delta^0))$; it is clear by construction that this identification is induced by the morphism~\eqref{eqn:Ext1-2}, which finishes the proof.
\end{proof}

The other ingredient we will need is a way to ``pass from $\delta^0$ to $\delta^0_\omega$,'' which will be provided by the following lemma. (See~\S\ref{ss:regular-quotient} for the definition of the object $\delta^0_\omega$.) Here, given a simple object $\scF$ in $\sfP^0_{\Iwu,\Iwu}$ we will denote by $\langle \scF \rangle_{\mathrm{Serre}}$ the Serre subcategory generated by $\scF$, i.e.~the full subcategory whose objects are those all of whose composition factors are isomorphic to $\scF$.

\begin{lem}
\label{lem:convolution-delta-omega}
 For any $\omega \in \Omega$, the equivalence
 \[
  \Delta^\wedge_{\omega} \hatstar^0 (-) : \sfD^0_{\Iwu,\Iwu} \simto \sfD^0_{\Iwu,\Iwu}
 \]
restricts to an equivalence $\langle \pi^\dag_0 \delta^0 \rangle_{\mathrm{Serre}} \simto \langle \pi^\dag_0 \delta_\omega^0 \rangle_{\mathrm{Serre}}$. Moreover the following diagram commutes:
\[
 \xymatrix@C=2cm{
 \langle \pi^\dag_0 \delta^0 \rangle_{\mathrm{Serre}} \ar[d]_-{\Phi_{\Iwu,\Iwu}} \ar[r]^-{\Delta^\wedge_{\omega} \hatstar^0 (-)} & \langle \pi^\dag_0 \delta_\omega^0 \rangle_{\mathrm{Serre}} \ar[d]^-{\Phi_{\Iwu,\Iwu}} \\
 \Rep_{0}(\bbI_\Sigma) \ar[r]^{\mathscr{M}^\wedge_\omega \circledast (-)} & \Rep_{0}(\bbI_\Sigma).
 }
\]
\end{lem}

\begin{proof}
 Using~\eqref{eqn:conv-formula-2}--\eqref{eqn:conv-formula-3} we see that
 \[
  \Delta^\wedge_{\omega} \hatstar^0 \pi^\dag_0 \delta^0 \cong \pi^\dag_0 \delta^0_\omega,
 \]
which implies the first claim. In order to prove the second claim, we consider some object $\scF$ in $\langle \pi^\dag_0 \delta^0 \rangle_{\mathrm{Serre}}$, and choose $m \geq 1$ such that $\cJ^m$ acts trivially on $\scF$.
Then using Lemma~\ref{lem:t-exactness-conv-tFl} we obtain a canonical isomorphism
\[
 \Delta^\wedge_{\omega} \hatstar^0 \scF \cong \mathsf{C}_m^0 (\Delta^\wedge_\omega) \pstar^0_{\Iwu} \scF.
\]
Now~\eqref{eqn:monoidality-morphism} provides a functorial morphism
\[
 \Phi_{\Iwu,\Iwu}(\mathsf{C}_m^0(\Delta_\omega^\wedge)) \circledast \Phi_{\Iwu,\Iwu}(\scF) \to \Phi_{\Iwu,\Iwu}(\mathsf{C}_m^0 (\Delta^\wedge_\omega) \pstar^0_{\Iwu} \scF).
\]
Using the identification above and Proposition~\ref{prop:Phi-Wak-antidom}\eqref{it:Phi-Delta-omega}, this morphism can be interpreted as a functorial morphism
\[
 \mathscr{M}_\omega^\wedge \circledast \Phi_{\Iwu,\Iwu}(\scF) \to \Phi_{\Iwu,\Iwu}(\Delta^\wedge_\omega \hatstar^0 \scF),
\]
which does not depend on the choice of $m$.
When $\scF=\pi^\dag_0 \delta^0$, this morphism is an isomorphism by Proposition~\ref{prop:Phi-exact} and monoidality of $\Phi_{\Iw,\Iw}$. By the five lemma it is then an isomorphism for any $\scF$, which concludes the proof.
\end{proof}

We can finally prove the desired property.

\begin{prop}
\label{prop:Phi-ff}
The functor $\Phi_{\Iwu,\Iwu}$ is fully faithful.
\end{prop}

\begin{proof}
Recall that the simple objects in $\sfP_{\Iwu,\Iwu}^0$ are the images under $\pi_0^\dag$ of the simple objects in $\sfP_{\Iw,\Iw}^0$. In view of
Lemma~\ref{lem:ff-criterion}, to prove the proposition it therefore suffices to check that for any simple objects $\scF,\scG$ in $\sfP_{\Iw,\Iw}^0$ the functor $\Phi_{\Iwu,\Iwu}$ induces an isomorphism
\[
\Hom_{\sfP_{\Iwu,\Iwu}^0}(\pi^\dag_0 \scF,\pi^\dag_0 \scG) \simto \Hom_{\Rep_0(\bbI_\Sigma)}(\Phi_{\Iwu,\Iwu}(\pi^\dag_0 \scF),\Phi_{\Iwu,\Iwu}(\pi^\dag_0 \scG))
\]
and an injection
\[
\Ext^1_{\sfP_{\Iwu,\Iwu}^0}(\pi^\dag_0 \scF,\pi^\dag_0 \scG) \hookrightarrow \Ext^1_{\Rep_0(\bbI_\Sigma)}(\Phi_{\Iwu,\Iwu}(\pi^\dag_0 \scF),\Phi_{\Iwu,\Iwu}(\pi^\dag_0 \scG)).
\]
For the $\Hom$-spaces, this simply follows from the commutativity of~\eqref{eqn:diag-compatibility-Phi} (proved in Proposition~\ref{prop:Phi-exact}) since $\Phi_{\Iw,\Iw}$ is known to be an equivalence, hence in particular fully faithful.

To prove the claim about the $\Ext^1$ spaces, recall that nonisomorphic simple objects in $\sfP_{\Iw,\Iw}^0$ are supported on distinct connected components of $\Fl_G$; hence if $\scF \not\cong \scG$ then $\Ext^1_{\sfP_{\Iwu,\Iwu}^0}(\pi^\dag_0 \scF,\pi^\dag_0 \scG)=0$, and there is nothing to prove. Otherwise we can assume that $\scF=\scG=\delta^0_\omega$ for some $\omega \in \Omega$. In fact, Lemma~\ref{lem:convolution-delta-omega} reduces this case to the case when 
$\scF=\scG=\delta^0$. (Indeed the horizontal arrows in the diagram of Lemma~\ref{lem:convolution-delta-omega} are equivalences, hence induce isomorphisms on $\Ext^1$ spaces.) In this case we have $\Phi_{\Iwu,\Iwu}(\pi^\dag_0 \delta^0)=\bk$, seen as the skyscraper sheaf at $\{(e,e)\}$, endowed with the trivial structure as a representation.

The same considerations as for~\eqref{eqn:Ext1-1} and~\eqref{eqn:Ext1-2} provide canonical embeddings
\begin{align*}
\Ext^1_{\sfP_{\Iw,\Iw}^0}(\delta^0,\delta^0) &\to \Ext^1_{\sfP_{\Iwu,\Iwu}^0}(\pi^\dag_0 \delta^0, \pi^\dag_0 \delta^0), \\
\Ext^1_{\sfP_{U,U}^0}(\pi^\dag_0 \delta^0, \pi^\dag_0\delta^0) &\to \Ext^1_{\sfP_{\Iwu,\Iwu}^0}(\pi^\dag_0 \delta^0, \pi^\dag_0 \delta^0).
\end{align*}
We claim that these morphisms induce an isomorphism
\[
\Ext^1_{\sfP_{\Iw,\Iw}^0}(\delta^0,\delta^0) \oplus \Ext^1_{\sfP_{U,U}^0}(\pi^\dag_0 \delta^0, \pi^\dag_0\delta^0) \simto \Ext^1_{\sfP_{\Iwu,\Iwu}^0}(\pi^\dag_0 \delta^0, \pi^\dag_0 \delta^0).
\]
Indeed, using the adjunction $(\pi_{\dag,0},\pi^{\dag,0})$ (see the proof of Lemma~\ref{lem:exactness-Phi}) we obtain an isomorphism
\[
\Ext^1_{\sfP_{\Iwu,\Iwu}^0}(\pi^\dag_0 \delta^0, \pi^\dag_0 \delta^0) \cong \Hom_{\D_{\Iwu,\Iw}^0}(\pi_{\dag,0} \pi^{\dag,0} \delta^0,  \delta^0[1]) \cong \Ext^1_{\sfP_{\Iwu,\Iw}^0}(\delta^0, \delta^0) \oplus \mathsf{H}^{2r-1}_c(T;\bk)^*
\]
where $r=\dim(T)$,
since $\pi_{\dag,0} \pi^{\dag,0} \delta^0 \cong \bigoplus_n \delta^0 \otimes_\bk \mathsf{H}^n_c(T;\bk)[2r-n]$. Similarly we have
\[
\Ext^1_{\sfP_{U,U}^0}(\pi^\dag_0 \delta^0, \pi^\dag_0\delta^0) = \Ext^1_{\sfP_{U,B}^0}(\delta^0, \delta^0) \oplus \mathsf{H}^{2r-1}_c(T;\bk)^*,
\]
so that our claim follows from Lemma~\ref{lem:Ext1-DIuI}.

In view of this isomorphism, to conclude the proof we need to check that the morphism
\begin{equation}
\label{eqn:morph-Ext1-ff}
\Ext^1_{\sfP_{\Iw,\Iw}^0}(\delta^0,\delta^0) \oplus \Ext^1_{\sfP_{U,U}^0}(\pi^\dag_0 \delta^0, \pi^\dag_0\delta^0) \to \Ext^1_{\Rep_0(\bbI_\Sigma)}(\bk,\bk) 
\end{equation}
induced by $\Phi_{\Iwu,\Iwu}$ is injective. Note that on the first summand this morphism identifies with that induced by $\Phi_{\Iw,\Iw}$ (via the morphism induced by the fully faithful functor~\eqref{eqn:functor-statement-Rep}), and on the second summand it identifies with the morphism induced by $\Phi_{U,U}$ (via the morphism induced by the fully faithful functor~\eqref{eqn:Coh-Rep-JD}). Consider elements $c_1 \in \Ext^1_{\sfP_{\Iw,\Iw}^0}(\delta^0,\delta^0)$ and $c_2 \in \Ext^1_{\sfP_{U,U}^0}(\pi^\dag_0 \delta^0, \pi^\dag_0\delta^0)$ such that the image of $c_1 + c_2$ vanishes. We have a forgetful functor
\[
\Rep_0(\bbI_\Sigma) \to \Coh_{\{(e,e)\}}(T^\vee_\bk \times_{T^\vee_\bk / \Wf} T^\vee_\bk);
\]
the composition of the induced morphism
\[
\Ext^1_{\Rep_0(\bbI_\Sigma)}(\bk,\bk) \to \Ext^1_{\Coh_{\{(e,e)\}}(T^\vee_\bk \times_{T^\vee_\bk / \Wf} T^\vee_\bk)}(\bk,\bk) 
\]
with~\eqref{eqn:morph-Ext1-ff} vanishes on the first summand, and identifies with the isomorphism induced by $\Phi_{U,U}$ on the second summand. We deduce that $c_2=0$, and then that $c_1=0$ since $\Phi_{\Iw,\Iw}$ is an equivalence, which completes the proof.
\end{proof}

\subsection{Essential surjectivity}
\label{ss:essential-surjectivity}

We will now prove that the functor $\Phi_{\Iwu,\Iwu}$ is essentially surjective. For this we need a preliminary lemma. Given $V \in \Rep(G^\vee_\bk)$ and $M \in \Coh(T^\vee_\bk \times_{T^\vee_\bk / \Wf} T^\vee_\bk)$, the coherent sheaf $V \otimes_\bk M$ has a canonical structure of module for the group scheme $G^\vee_\bk \times (T^\vee_\bk \times_{T^\vee_\bk / \Wf} T^\vee_\bk)$ over $T^\vee_\bk \times_{T^\vee_\bk / \Wf} T^\vee_\bk$. Now by construction $\bbI_\Sigma$ is a subgroup scheme of $G^\vee_\bk \times (T^\vee_\bk \times_{T^\vee_\bk / \Wf} T^\vee_\bk)$, hence by restriction $V \otimes_\bk M$ has a natural structure of $\bbI_\Sigma$-module. If $M$ belongs to $\Coh_{\{(e,e)\}}(T^\vee_\bk \times_{T^\vee_\bk / \Wf} T^\vee_\bk)$, then this module belongs to $\Rep_0(\bbI_\Sigma)$. Recall also the closed subschemes $(T^\vee_\bk \times_{T^\vee_\bk / \Wf} T^\vee_\bk)^{(m)} \subset T^\vee_\bk \times_{T^\vee_\bk / \Wf} T^\vee_\bk$ considered in~\S\ref{ss:truncation-PUU}.

\begin{lem}
\label{lem:J-modules-quot}
 Any object in $\Rep_0(\bbI_\Sigma)$ is a quotient of a module of the form $V \otimes_\bk \scO_{(T^\vee_\bk \times_{T^\vee_\bk / \Wf} T^\vee_\bk)^{(m)}}$ with $V \in \Rep(G^\vee_\bk)$ and $m \geq 1$.
\end{lem}

\begin{proof}
We will in fact prove that any $M$ in $\Rep(\bbI_\Sigma)$ is a quotient of a module $V \otimes_\bk \scO_{T^\vee_\bk \times_{T^\vee_\bk / \Wf} T^\vee_\bk}$ with $V \in \Rep(G^\vee_\bk)$; if $M$ is in $\Rep_0(\bbI_\Sigma)$ the corresponding surjection $V \otimes_\bk \scO_{T^\vee_\bk \times_{T^\vee_\bk / \Wf} T^\vee_\bk} \twoheadrightarrow M$ will necessarily factor through a surjection $V \otimes_\bk \scO_{(T^\vee_\bk \times_{T^\vee_\bk / \Wf} T^\vee_\bk)^{(m)}} \twoheadrightarrow M$ for some $m \geq 1$. Recall the equivalence of categories
\[
\Coh^{G^\vee_\bk}(T^\vee_\bk \times_{T^\vee_\bk/\Wf} (G^\vee_\bk)_\reg \times_{T^\vee_\bk/\Wf} T^\vee_\bk) \simto \Rep(\bbI_\Sigma)
\]
induced by restriction to the preimage of $\Sigma$,
see Proposition~\ref{prop:rest-Steinberg-section}. If $\scF$ is the equivariant coherent sheaf corresponding to $M$, then there exists a $G^\vee_\bk$-equivariant coherent sheaf $\scF'$ on $T^\vee_\bk \times_{T^\vee_\bk/\Wf} G^\vee_\bk \times_{T^\vee_\bk/\Wf} T^\vee_\bk$ whose restriction to the open subscheme $T^\vee_\bk \times_{T^\vee_\bk/\Wf} (G^\vee_\bk)_\reg \times_{T^\vee_\bk/\Wf} T^\vee_\bk$ is $\scF$; see e.g.~\cite[Lemma~2.12]{arinkin-bezrukavnikov}. Now since $T^\vee_\bk \times_{T^\vee_\bk/\Wf} G^\vee_\bk \times_{T^\vee_\bk/\Wf} T^\vee_\bk$ is affine, there exists $V \in \Rep(G^\vee_\bk)$ and a surjection of $G^\vee_\bk$-equivariant coherent sheaves
\[
V \otimes_\bk \scO_{T^\vee_\bk \times_{T^\vee_\bk/\Wf} G^\vee_\bk \times_{T^\vee_\bk/\Wf} T^\vee_\bk} \twoheadrightarrow \scF'.
\]
Restricting to $T^\vee_\bk \times_{T^\vee_\bk/\Wf} (G^\vee_\bk)_\reg \times_{T^\vee_\bk/\Wf} T^\vee_\bk$ and then to the preimage of $\Sigma$, we obtain the desired surjection $V \otimes_\bk \scO_{T^\vee_\bk \times_{T^\vee_\bk / \Wf} T^\vee_\bk} \twoheadrightarrow M$.
\end{proof}

\begin{prop}
\label{prop:Phi-es}
 The functor
 \[
  \Phi_{\Iwu,\Iwu} : \mathsf{P}_{\Iwu,\Iwu}^0 \to \Rep_0(\mathbb{I}_\Sigma)
 \]
 is essentially surjective.
\end{prop}

\begin{proof}
We now know that $\Phi_{\Iwu,\Iwu}$ is exact (see Proposition~\ref{prop:Phi-exact}) and fully faithful (see Proposition~\ref{prop:Phi-ff}).
By Lemma~\ref{lem:J-modules-quot}, any object in $\Rep_0(\mathbb{I}_\Sigma)$ is isomorphic to the cokernel of a morphism between objects of the form 
$V \otimes_\bk M$ with $V \in \Rep(G^\vee_\bk)$ and $M \in \Coh_{\{(e,e)\}}(T^\vee_\bk \times_{T^\vee_\bk / \Wf} T^\vee_\bk)$; to prove the proposition it therefore suffices to prove that any such object is isomorphic to the image of an object of $\mathsf{P}_{\Iwu,\Iwu}^0$. Let us fix $V \in \Rep(G^\vee_\bk)$ and $M \in \Coh_{\{(e,e)\}}(T^\vee_\bk \times_{T^\vee_\bk / \Wf} T^\vee_\bk)$. Since $\Phi_{U,U}$ is essentially surjective (see
Theorem~\ref{thm:reg-quotient-finite}),
Proposition~\ref{prop:isom-Phi-U-Iu} implies that there exists $\scF$ in $\mathsf{P}_{\Iwu,\Iwu}^0$ such that $\Phi_{\Iwu,\Iwu}(\scF) \cong M$. Then Lemma~\ref{lem:Phi-ZV} implies that
\[
 \Phi_{\Iwu,\Iwu}(\tsZ(V) \phatstar^0 \scF) \cong V \otimes_\bk M,
\]
which finishes the proof.
\end{proof}

\subsection{Monoidality}
\label{ss:monoidality-PhiIuIu}

Combining Propositions~\ref{prop:Phi-exact},~\ref{prop:isom-Phi-U-Iu},~\ref{prop:Phi-ff} and~\ref{prop:Phi-es}, we have now pro\-ved that $\Phi_{\Iwu,\Iwu}$ is an equivalence of categories, and that the diagrams~\eqref{eqn:diag-compatibility-Phi}--\eqref{eqn:diag-compatibility-Phi-2} commute. To conclude the proof of Theorem~\ref{thm:main}, it therefore only remains to prove that $\Phi_{\Iwu,\Iwu}$ is monoidal. Recall that in~\eqref{eqn:monoidality-morphism} we have defined a canonical ``monoidality'' morphism; what we will now prove is that this morphism is an isomorphism for any $\scF,\scG$ in $\sfP^0_{\Iwu,\Iwu}$.

We start with a special case.

\begin{lem}
\label{lem:monoidality}
The morphism~\eqref{eqn:monoidality-morphism} is an isomorphism in case
\[
\scF=\tsZ(V) \hatstar^0 \scF' \quad \text{and} \quad \scG=\tsZ(V') \hatstar^0 \scG'
\]
for some $V,V' \in \Rep(G^\vee_\bk)$ and $\scF',\scG' \in \sfP^0_{U,U}$.
\end{lem}

\begin{proof}
Assume that $\scF$ and $\scG$ are as in the lemma. By Lemma~\ref{lem:Phi-ZV} and Proposition~\ref{prop:isom-Phi-U-Iu} we have
\[
\Phi_{\Iwu,\Iwu}(\scF) \cong V \otimes_\bk \Phi_{U,U}(\scF'), \quad \Phi_{\Iwu,\Iwu}(\scG) \cong V' \otimes_\bk \Phi_{U,U}( \scG').
\]
On the other hand, using~\eqref{eqn:commutation-Z-0} we obtain a canonical isomorphism
\[
\scF \pstar^0_{\Iwu} \scG \cong \tsZ(V \otimes_\bk V') \hatstar^0 (\scF' \star^0_U \scG'),
\]
so that similarly we have
\[
\Phi_{\Iwu,\Iwu}(\scF \star^0_{\Iwu} \scG) \cong (V \otimes_\bk V') \otimes_\bk \Phi_{U,U}(\scF' \star^0_U \scG').
\]
Under these identifications the morphism~\eqref{eqn:monoidality-morphism} is induced by the isomorphism
\[
\Phi_{U,U} (\scF' \star^0_U \scG') \cong \Phi_{U,U}(\scF') \circledast \Phi_{U,U}(\scG')
\]
defining the monoidal structure of $\Phi_{U,U}$ (see Theorem~\ref{thm:reg-quotient-finite}); it is therefore an isomorphism.
\end{proof}

\begin{prop}
The morphism~\eqref{eqn:monoidality-morphism} is an isomorphism for all $\scF,\scG$ in $\sfP^0_{\Iwu,\Iwu}$; in other words, the functor $\Phi_{\Iwu,\Iwu}$ admits a canonical monoidal structure.
\end{prop}

\begin{proof}
By right exactness of the bifunctors $\star^0_{\Iwu}$ and $\circledast$ and the five lemma, if given $\scF,\scG,\scG'$ in $\sfP^0_{\Iwu,\Iwu}$ the claim is known for the pairs of objects $(\scF,\scG)$ and $(\scF,\scG')$, then it will hold for the pair $(\scF,\scG'')$ for any cokernel $\scG''$ of a morphism $\scG \to \scG'$. The same property holds when the order of the factors is switched. In view of Lemma~\ref{lem:monoidality}, to conclude the proof it therefore suffices to prove that any object in $\sfP^0_{\Iwu,\Iwu}$ is isomorphic to the cokernel of a morphism between objects of the form $\tsZ(V) \hatstar^0 \scF'$ with $V \in \Rep(G^\vee_\bk)$ and $\scF' \in \sfP^0_{U,U}$. In view of the computation in the proof of Lemma~\ref{lem:monoidality}, this follows from Lemma~\ref{lem:J-modules-quot} and the fact that $\Phi_{U,U}$ and $\Phi_{\Iwu,\Iwu}$ are equivalences of categories.
%
\end{proof}

\section{A Soergel-type description of tilting perverse sheaves}
\label{sec:Soergel-type}

We continue with the assumptions of Section~\ref{sec:construction}. In this section we explain how to deduce Theorem~\ref{thm:main-intro-2} from Theorem~\ref{thm:main-intro-1}.

\subsection{Tilting completed perverse sheaves}
\label{ss:Twedge}

Recall the category $\mathsf{P}^\wedge_{\Iwu,\Iwu}$ and its subcategory $\mathsf{T}^\wedge_{\Iwu,\Iwu}$ from~\S\ref{ss:tilting-perv}. Recall also that $\mathsf{T}^\wedge_{\Iwu,\Iwu}$ is stable under the monoidal product $\hatstar$. In~\S\ref{ss:G/B-tilting} we have defined, for any $s \in \Sf$, an object $\Xi^\wedge_{s,!}$ in $\mathsf{T}^\wedge_{\Iwu,\Iwu}$. On the other hand, in~\S\ref{ss:Waff} we have chosen, for any $s \in S \smallsetminus \Sf$, elements $w \in W$ and $s' \in \Sf$ such that $\ell(ws')=\ell(w)+1$ and $s=ws'w^{-1}$. We then have
\[
 \Delta^\wedge_w \hatstar \Delta^\wedge_{s'} \cong \Delta^\wedge_s \hatstar \Delta^\wedge_w \quad \text{and} \quad \nabla^\wedge_{w^{-1}} \hatstar \nabla^\wedge_s \cong \nabla^{\wedge}_{s'} \hatstar \nabla^\wedge_{w^{-1}}
\]
by Lemma~\ref{lem:convolution-stand-costand}\eqref{it:convolution-stand-costand-2}, hence
\[
 \Delta^\wedge_{s} \cong \Delta^\wedge_w \hatstar \Delta^\wedge_{s'} \hatstar \nabla^\wedge_{w^{-1}} \quad \text{and} \quad \nabla^{\wedge}_{s} \cong \Delta^\wedge_w \hatstar \nabla^\wedge_{s'} \hatstar \nabla^\wedge_{w^{-1}}
\]
by Lemma~\ref{lem:convolution-stand-costand}\eqref{it:convolution-stand-costand-1}. From these isomorphisms we deduce that the object
\[
 \Xi^\wedge_{s,!} := \Delta^\wedge_w \hatstar \Xi^\wedge_{s',!} \hatstar \nabla^\wedge_{w^{-1}}
\]
is a representative for the indecomposable tilting object $\mathscr{T}^\wedge_s$.
On the other hand, for any $\omega \in \Omega$ the natural morphism $\Delta^\wedge_\omega \to \nabla^\wedge_\omega$ is an isomorphism since $\tFl_{G,\omega}$ is closed, which shows that these perverse sheaves are tilting. 


Recall that the category $\mathsf{D}^\wedge_{\Iwu,\Iwu}$ is Krull--Schmidt, see~\S\ref{ss:completed-category}.
Standard arguments (see e.g.~\cite[Remark~7.9]{bezr}) show that any object in $\mathsf{T}^\wedge_{\Iwu,\Iwu}$ is a direct sum of direct summands of objects of the form
\[
\Delta^\wedge_\omega \hatstar \Xi^\wedge_{s_1,!} \hatstar \cdots \hatstar \Xi^\wedge_{s_i,!}
\]
with $\omega \in \Omega$ and $s_1, \cdots, s_i \in S$.

Let us recall also that
for any $V$ in $\Rep(G^\vee_\bk)$ which is tilting the object
\[
\Xi^\wedge_! \hatstar \tsZ(V) \cong \tsZ(V) \hatstar \Xi^\wedge_!
\]
belongs to $\mathsf{T}^\wedge_{\Iwu,\Iwu}$, see Proposition~\ref{prop:tilting-Z}.


\subsection{Soergel representations}
\label{ss:SRep}


Recall the category
\[
\SRep(\bbI^\wedge_\Sigma)
\]
defined in~\S\ref{ss:completions}. Note that any object in this category is finite and flat, i.e.~finite and projective, as an $\scO(\FN_{T^\vee_\bk}(\{e\}))$-module for the action on the right.

Let us note the following technical property for later use.

\begin{lem}
\label{lem:Iwedge-integral}
The scheme $\FN_{T^\vee_\bk}(\{e\}) \times_{T^\vee_\bk / \Wf} \mathbb{J}_\Sigma$ is integral.
\end{lem}

\begin{proof}
First, we note that since $T^\vee_\bk$ is isomorphic to a product of copies of the multiplicative group, the scheme $\FN_{T^\vee_\bk}(\{e\})$ is integral. Recall the open subscheme $(T^\vee_\bk)_\circ \subset T^\vee_\bk$ introduced in~\S\ref{ss:rs}; in concrete terms, $\scO((T^\vee_\bk)_\circ)$ is the spectrum of the localization of $\scO(T^\vee_\bk)$ with respect to the elements $\alpha-1$ where $\alpha$ runs over the coroots of $(G,T)$. We define similarly $\FN_{T^\vee_\bk}(\{e\})_\circ$ as the spectrum of the localization of $\scO(\FN_{T^\vee_\bk}(\{e\}))$ with respect to the elements $\alpha-1$ where $\alpha$ runs over the coroots of $(G,T)$. Then $\FN_{T^\vee_\bk}(\{e\})_\circ$ is integral, and $\scO(\FN_{T^\vee_\bk}(\{e\})_\circ)$ is also the similar localization of $\scO(\FN_{T^\vee_\bk}(\{e\}))$ viewed as an $\scO(T^\vee_\bk)$-module; in other words we have a canonical identification
\begin{equation}
\label{eqn:Trs}
(T^\vee_\bk)_\circ \times_{T^\vee_\bk} \FN_{T^\vee_\bk}(\{e\}) \simto \FN_{T^\vee_\bk}(\{e\})_\circ.
\end{equation}

Since $\mathbb{J}_\Sigma$ is flat over $T^\vee_\bk / \Wf$ (see Lemma~\ref{lem:J-smooth}), the natural morphism
\[
\scO(\FN_{T^\vee_\bk}(\{e\}) \times_{T^\vee_\bk / \Wf} \mathbb{J}_\Sigma) \to \scO(\FN_{T^\vee_\bk}(\{e\})_\circ \times_{T^\vee_\bk / \Wf} \mathbb{J}_\Sigma)
\]
is injective. To prove our claim it therefore suffices to prove that the right-hand side is a domain. However, in view of~\eqref{eqn:Trs} we have
\[
\FN_{T^\vee_\bk}(\{e\})_\circ \times_{T^\vee_\bk / \Wf} \mathbb{J}_\Sigma = \FN_{T^\vee_\bk}(\{e\})\times_{T^\vee_\bk} \left( (T^\vee_\bk)_\circ \times_{T^\vee_\bk / \Wf} \mathbb{J}_\Sigma \right).
\]
By Lemma~\ref{lem:morph-gp-schemes-Jreg} the right-hand side identifies with
\[
\FN_{T^\vee_\bk}(\{e\}) \times_{T^\vee_\bk} ((T^\vee_\bk)_\circ \times T^\vee_\bk) = \FN_{T^\vee_\bk}(\{e\})_\circ \times T^\vee_\bk,
\]
which is integral since it is a split torus over the integral scheme $\FN_{T^\vee_\bk}(\{e\})_\circ$.
\end{proof}

\subsection{Statement}
\label{ss:Soergel-type-statement}

Our main application of Theorem~\ref{thm:main} is the following statement, which establishes Theorem~\ref{thm:main-intro-2}. 

\begin{thm}
\label{thm:main-application}
There exists a canonical equivalence of additive monoidal categories
\[
\Phi_{\mathsf{T}} : (\mathsf{T}^\wedge_{\Iwu,\Iwu}, \hatstar) \simto (\SRep(\bbI^\wedge_\Sigma), \circledast)
\]
which satisfies the following properties:
\begin{enumerate}
\item
\label{it:main-application-1}
$\Phi_{\mathsf{T}}(\Xi^\wedge_{s,!}) \cong \mathscr{B}^\wedge_s$ for any $s \in S$;
\item
\label{it:main-application-2}
$\Phi_{\mathsf{T}}(\Delta^\wedge_\omega) \cong \mathscr{M}^\wedge_\omega$ for any $\omega \in \Omega$;
\item
\label{it:main-application-3}
$\Phi_{\mathsf{T}} ( \tsZ(V) \hatstar \Xi^\wedge_! ) \cong V \otimes_\bk \scO(\FN_{T^\vee_\bk \times_{T^\vee_\bk / \Wf} T^\vee_\bk}(\{(e,e)\}))$ for any $V \in \Rep(G^\vee_\bk)$ tilting.
\end{enumerate}
\end{thm}

Note that property~\eqref{it:main-application-3} shows in particular that $V \otimes_\bk \scO(\FN_{T^\vee_\bk \times_{T^\vee_\bk / \Wf} T^\vee_\bk}(\{(e,e)\}))$ belongs to $\SRep(\bbI^\wedge_\Sigma)$ for any $V \in \Rep(G^\vee_\bk)$ tilting, which is not clear from definitions.
The proof of Theorem~\ref{thm:main-application} occupies the rest of the section. The idea of the proof is to give ``sequential" descriptions of the categories $\mathsf{T}^\wedge_{\Iwu,\Iwu}$ and $\SRep(\bbI^\wedge_\Sigma)$, in terms of the categories $\mathsf{P}^0_{\Iwu,\Iwu}$ and $\Rep_0(\bbI_\Sigma)$ respectively, and then use Theorem~\ref{thm:main} to relate these two descriptions.

More specifically, using the notation introduced in~\S\ref{ss:truncation-ps}, on the perverse side, we will
denote by $\mathsf{P}^{\wedge,\mathrm{seq}}_{\Iwu,\Iwu}$ the monoidal additive $\bk$-linear category with:
\begin{itemize}
\item
objects the projective systems $(\scF_m : m \geq 1)$ of objects in $\mathsf{P}^0_{\Iwu,\Iwu}$ such that $\mathcal{J}^m$ acts trivially on $\scF_m$ for any $m \geq 1$, and for $m' \geq m$ the morphism
\[
\scO((T^\vee_\bk)^{(m)}) \otimes_{\scO(T^\vee_\bk)} \scF_{m'} \to \scF_m
\]
induced by the transition morphism $\scF_{m'} \to \scF_m$ is an isomorphism;
\item
morphisms from $(\scF_m : m \geq 1)$ to $(\scG_m : m \geq 1)$ the inverse limit
\[
\varprojlim_m \Hom_{\mathsf{P}^0_{\Iwu,\Iwu}}(\scF_m,\scG_m)
\]
where for $m' \geq m$ the morphism $\Hom(\scF_{m'},\scG_{m'}) \to \Hom(\scF_m,\scG_m)$ is induced by the functor $\scO((T^\vee_\bk)^{(m)}) \otimes_{\scO(T^\vee_\bk)} (-)$;
\item
monoidal product sending a pair of objects $((\scF_m : m \geq 1),(\scG_m : m \geq 1))$ to the object $(\scF_m \pstar^0_{\Iwu} \scG_m : m \geq 1)$. (This projective system is indeed an object of our category by right exactness of $\pstar^0_{\Iwu}$.)
\end{itemize}
With this definition, Proposition~\ref{prop:monoidality-C0} and Lemma~\ref{lem:Cm-properties} imply that the assignment
\[
\scT \mapsto (\mathsf{C}_m^0(\scT) : m \geq 1)
\]
defines a fully faithful monoidal functor
\begin{equation}
\label{eqn:seq-functor-Perv}
\mathsf{T}^\wedge_{\Iwu,\Iwu} \to \mathsf{P}^{\wedge,\mathrm{seq}}_{\Iwu,\Iwu}.
\end{equation}

\subsection{Truncation functors for representations: fully faithfulness}
\label{ss:truncation-Rep-ff}

We now consider the category $\Rep(\bbI_\Sigma^\wedge)$, and the functors $\sfD_m$ introduced in~\S\ref{ss:truncation-Rep}. 

\begin{lem}
\label{lem:truncation-Rep}
Let $M,M' \in \Rep(\bbI_\Sigma^\wedge)$.
If $M'$ is projective for the right action of $\scO(\FN_{T^\vee_\bk}(\{e\}))$, then
the functors $\mathsf{D}_m$ induce an isomorphism
\[
\Hom_{\Rep(\bbI_\Sigma^\wedge)}(M,M') \simto \varprojlim_m \Hom_{\Rep_0(\bbI_\Sigma)}(\mathsf{D}_m(M),\mathsf{D}_m(M')).
\]
\end{lem}

\begin{proof}
For any $M,M'$ in $\Rep(\bbI_\Sigma^\wedge)$ and any $m \geq 1$ the forgetful functor $\Rep_0(\bbI_\Sigma) \to \Coh_{\{(e,e)\}}(T^\vee_\bk \times_{T^\vee_\bk / \Wf} T^\vee_\bk)$ induces an embedding
\[
\Hom_{\Rep_0(\bbI_\Sigma)}(\mathsf{D}_m(M),\mathsf{D}_m(M')) \hookrightarrow \Hom_{\Coh_{\{(e,e)\}}(T^\vee_\bk \times_{T^\vee_\bk / \Wf} T^\vee_\bk)}(\mathsf{D}_m(M),\mathsf{D}_m(M')).
\]
These embeddings give rise to a morphism fitting as the right vertical arrow in the following commutative diagram, where the other arrows are the obvious maps, and where in the lower line we consider morphisms in $\Coh(\FN_{T^\vee_\bk \times_{T^\vee_\bk / \Wf} T^\vee_\bk}(\{(e,e)\}))$ and in $\Coh_{\{(e,e)\}}(T^\vee_\bk \times_{T^\vee_\bk / \Wf} T^\vee_\bk)$ respectively:
\[
\xymatrix{
\Hom_{\Rep(\bbI_\Sigma^\wedge)}(M,M') \ar[r] \ar@{^{(}->}[d] & \varprojlim_m \Hom_{\Rep_0(\bbI_\Sigma)}(\mathsf{D}_m(M),\mathsf{D}_m(M')) \ar[d] \\
\Hom(M,M') \ar[r] & \varprojlim_m \Hom(\mathsf{D}_m(M),\mathsf{D}_m(M')).
}
\]
The lower horizontal arrow in this diagram is an isomorphism by~\cite[Chap.~0, Corollaire~7.2.10]{ega1}. It follows that the upper horizontal arrow, which is the morphism we need to study, is injective.

To prove surjectivity, we consider a projective system
\[
(f_m)_{m \geq 1} \in \varprojlim_m \Hom_{\Rep_0(\bbI_\Sigma)}(\mathsf{D}_m(M),\mathsf{D}_m(M')).
\]
Each $f_m$ is a morphism of $\scO(T^\vee_\bk \times_{T^\vee_\bk / \Wf} T^\vee_\bk)$-modules; this collection therefore defines a morphism $f : M \to M'$ of $\scO(\FN_{T^\vee_\bk \times_{T^\vee_\bk / \Wf} T^\vee_\bk}(\{(e,e)\}))$-modules, and what we have to show is that $f$ is also a morphism of $\scO(\bbI_\Sigma^\wedge)$-comodules. We consider the second projection $\FN_{T^\vee_\bk \times_{T^\vee_\bk / \Wf} T^\vee_\bk}(\{(e,e)\}) \to \FN_{T^\vee_\bk}(\{e\})$. Then we have
\[
\bbI_\Sigma^\wedge = \FN_{T^\vee_\bk \times_{T^\vee_\bk / \Wf} T^\vee_\bk}(\{(e,e)\}) \times_{\FN_{T^\vee_\bk}(\{e\})} \left( \FN_{T^\vee_\bk}(\{e\}) \times_{T^\vee_\bk / \Wf} \mathbb{J}_\Sigma \right),
\]
so that $M$ and $M'$ can be considered as representations of $\FN_{T^\vee_\bk}(\{e\}) \times_{T^\vee_\bk / \Wf} \mathbb{J}_\Sigma$. From this point of view, to prove the desired claim it suffices to show that $f$ is a morphism of $\scO(\FN_{T^\vee_\bk}(\{e\}) \times_{T^\vee_\bk / \Wf} \mathbb{J}_\Sigma)$-comodules. Now $\mathbb{J}_\Sigma$ is smooth over the smooth $\bk$-scheme $T^\vee_\bk / \Wf$, and $\FN_{T^\vee_\bk}(\{e\})$ is flat over $T^\vee_\bk / \Wf$ as the composition of flat morphisms
\[
\FN_{T^\vee_\bk}(\{e\}) \to \FN_{T^\vee_\bk/\Wf}(\{e\}) \to T^\vee_\bk / \Wf
\]
(see Lemma~\ref{lem:Dw-fiber-prod} for the first map). By Lemma~\ref{lem:infflat-base-change} and Corollary~\ref{cor:smooth-infflat}, this implies that $\FN_{T^\vee_\bk}(\{e\}) \times_{T^\vee_\bk / \Wf} \mathbb{J}_\Sigma$ is infinitesimally flat. This group scheme is also integral by Lemma~\ref{lem:Iwedge-integral}, and noetherian since it is of finite type over the noetherian scheme $\FN_{T^\vee_\bk}(\{e\})$.
We can therefore apply Lemma~\ref{lem:morph-Dist} which reduces the proof to checking that $f$ is a morphism of $\mathrm{Dist}(\FN_{T^\vee_\bk}(\{e\}) \times_{T^\vee_\bk / \Wf} \mathbb{J}_\Sigma)$-modules. However each $f_m$ is a morphism of $\mathrm{Dist}(\FN_{T^\vee_\bk}(\{e\}) \times_{T^\vee_\bk / \Wf} \mathbb{J}_\Sigma)$-modules, so that this claim is clear.
\end{proof}

Let us denote by $\Rep^{\mathrm{seq}}(\bbI_\Sigma^\wedge)$ the monoidal additive $\bk$-linear category with:
\begin{itemize}
\item
objects the projective systems $(M_m : m \geq 1)$ of objects of $\Rep_0(\bbI_\Sigma)$ such that $\mathcal{J}^m$ acts trivially on $M_m$ for any $m \geq 1$, and for $m' \geq m$ the morphism
\[
\mathsf{D}_m(M_{m'}) \to M_m
\]
induced by the transition morphism $M_{m'} \to M_m$ is an isomorphism;
\item
morphisms from $(M_m : m \geq 1)$ to $(N_m : m \geq 1)$ the inverse limit
\[
\varprojlim_m \Hom_{\Rep_0(\bbI_\Sigma)}(M_m, N_m)
\]
where for $m' \geq m$ the transition morphism
\[
\Hom_{\Rep_0(\bbI_\Sigma)}(M_{m'}, N_{m'}) \to \Hom_{\Rep_0(\bbI_\Sigma)}(M_m, N_m)
\]
is induced by the functor $\mathsf{D}_m$;
\item
monoidal product sending a pair of objects $((M_m : m \geq 1),(N_m : m \geq 1))$ to the object $(M_m \circledast N_m : m \geq 1)$.
\end{itemize}
With this definition, Lemma~\ref{lem:truncation-Rep-mon} and Lemma~\ref{lem:truncation-Rep} imply that the assignment
\[
M \mapsto (\mathsf{D}_m(M) : m \geq 1)
\]
defines a monoidal functor
\begin{equation}
\label{eqn:seq-functor-Rep}
\Rep(\bbI^\wedge_\Sigma) \to \Rep^{\mathrm{seq}}(\bbI_\Sigma^\wedge)
\end{equation}
which is fully faithful on the full subcategory of objects which are projective for the right action of $\scO(\FN_{T^\vee_\bk}(\{e\}))$.
In particular, objects of $\SRep(\bbI_\Sigma^\wedge)$ satisfy this condition (see~\S\ref{ss:SRep}), hence~\eqref{eqn:seq-functor-Rep} is fully faithful on $\SRep(\bbI_\Sigma^\wedge)$.

\subsection{Construction of \texorpdfstring{$\Phi_{\mathsf{T}}$}{PhiT}}

Comparing the definitions of the categories $\mathsf{P}^{\wedge,\mathrm{seq}}_{\Iwu,\Iwu}$ (see~\S\ref{ss:Soergel-type-statement}) and $ \Rep^{\mathrm{seq}}(\bbI_\Sigma^\wedge)$ (see~\S\ref{ss:truncation-Rep-ff}),
we see that Theorem~\ref{thm:main} provides a canonical equivalence of monoidal additive $\bk$-linear categories
\[
\Phi_{\Iwu,\Iwu}^{\mathrm{seq}} : \mathsf{P}^{\wedge,\mathrm{seq}}_{\Iwu,\Iwu} \simto \Rep^{\mathrm{seq}}(\bbI_\Sigma^\wedge).
\]
In the remaining subsections we will prove the following claim. 

\begin{lem}
\phantomsection
\label{lem:isom-PhiT}
\begin{enumerate}
\item
\label{it:isom-main-app-1}
For any $s \in S$ there exists an isomorphism of projective systems
\[
\left( \Phi_{\Iwu,\Iwu} \bigl( \mathsf{C}^0_m(\Xi^\wedge_{s,!}) \bigr) : m \geq 1 \right) \cong ( \mathsf{D}_m(\mathscr{B}^\wedge_s) : m \geq 1 ).
\]
\item
\label{it:isom-main-app-3}
There exists an isomorphism of projective systems
\[
\left( \Phi_{\Iwu,\Iwu} \bigl( \mathsf{C}_m^0(\Xi^\wedge_{!}) \bigr) : m \geq 1 \right) \cong ( \mathsf{D}_m(\scO(\FN_{T^\vee_\bk \times_{T^\vee_\bk / \Wf} T^\vee_\bk}(\{(e,e)\}))) : m \geq 1).
\]
\end{enumerate}
\end{lem}

For now, let us explain why Lemma~\ref{lem:isom-PhiT} allows to complete the proof of Theorem~\ref{thm:main-application}.

\begin{proof}[Proof of Theorem~\ref{thm:main-application}]
Lemma~\ref{lem:isom-PhiT}\eqref{it:isom-main-app-1} implies that the equivalence $\Phi_{\Iwu,\Iwu}^{\mathrm{seq}}$ matches the image of each $\Xi^\wedge_{s,!}$ ($s \in S$) under~\eqref{eqn:seq-functor-Perv} with the image of $\mathscr{B}^\wedge_s$ under~\eqref{eqn:seq-functor-Rep}. Similarly, Proposition~\ref{prop:Phi-Wak-antidom}\eqref{it:Phi-Delta-omega}
implies that $\Phi_{\Iwu,\Iwu}^{\mathrm{seq}}$ matches the image of each $\Delta^\wedge_\omega$ ($\omega \in \Omega$) under~\eqref{eqn:seq-functor-Perv} with the image of $\mathscr{M}^\wedge_\omega$ under~\eqref{eqn:seq-functor-Rep}.
In view of the discussion in~\S\ref{ss:Twedge} and the definition of $\SRep(\bbI_\Sigma^\wedge)$,
these properties and monoidality imply that $\Phi_{\Iwu,\Iwu}^{\mathrm{seq}}$ identifies the essential images of the functor~\eqref{eqn:seq-functor-Perv} with the essential image of $\SRep(\bbI_\Sigma^\wedge)$ under~\eqref{eqn:seq-functor-Rep}. We deduce the equivalence $\Phi_{\mathsf{T}}$, and also that this equivalence satisfies properties~\eqref{it:main-application-1}--\eqref{it:main-application-2}.

Finally, consider some $V \in \Rep(G^\vee_\bk)$ which is tilting. As explained in~\S\ref{ss:Twedge}, we have the object $\tsZ(V) \hatstar \Xi_!^\wedge$ in $\mathsf{T}^\wedge_{\Iwu,\Iwu}$.
Lemma~\ref{lem:isom-PhiT}\eqref{it:isom-main-app-3} and Lemma~\ref{lem:Phi-ZV} provide an isomorphism of projective systems
\begin{multline*}
\left( \Phi_{\Iwu,\Iwu} \bigl( \mathsf{C}^0_m(\tsZ(V) \hatstar \Xi^\wedge_{!}) \bigr) : m \geq 1 \right) \cong \\
( \mathsf{D}_m(V \otimes_\bk \scO(\FN_{T^\vee_\bk \times_{T^\vee_\bk / \Wf} T^\vee_\bk}(\{(e,e)\}))) : m \geq 1),
\end{multline*}
which allows to identify the images of the objects $\Phi_{\mathsf{T}}(\tsZ(V) \hatstar \Xi^\wedge_{!})$ and $V \otimes_\bk \scO(\FN_{T^\vee_\bk \times_{T^\vee_\bk / \Wf} T^\vee_\bk}(\{(e,e)\}))$ under~\eqref{eqn:seq-functor-Rep}. Now $\Phi_{\mathsf{T}}( \tsZ(V) \hatstar \Xi^\wedge_{!} )$ satisfies the assumption in Lemma~\ref{lem:truncation-Rep} because it belongs to $\SRep(\bbI_\Sigma^\wedge)$, and the module $V \otimes_\bk \scO(\FN_{T^\vee_\bk \times_{T^\vee_\bk / \Wf} T^\vee_\bk}(\{(e,e)\}))$ also does by Lemma~\ref{lem:Dw-fiber-prod}\eqref{it:Dw-fiber-prod-1}. We deduce that these objects are isomorphic, proving that $\Phi_{\mathsf{T}}$ satisfies property~\eqref{it:main-application-3}.
\end{proof}

\subsection{The case of \texorpdfstring{$G/U$}{G/U}}

Recall the category $\sfD^\wedge_{U,U}$ from~\S\ref{ss:G/B-tilting}, the heart $\sfP^\wedge_{U,U}$ of the perverse t-structure on this category, and the full subcategory $\sfT^\wedge_{U,U}$ of tilting perverse sheaves. Recall also that we denote by $\sfP_{U,U}$ the heart of the perverse t-structure on $\sfD_{U,U}$. Pushforward along the closed embedding $G/U \hookrightarrow \tFl_G$ defines a fully faithful monoidal functor $(\sfD^\wedge_{U,U}, \hatstar_U) \to (\sfD^\wedge_{\Iwu,\Iwu}, \hatstar)$, which restricts to a fully faithful monoidal functor $(\sfT^\wedge_{U,U}, \hatstar_U) \to (\sfT^\wedge_{\Iwu,\Iwu}, \hatstar)$. As noted in~\S\ref{ss:cat-sheaves-G/U}, for any $m \geq 1$ the restriction of the functor $\mathsf{C}_m$ to $\sfP^\wedge_{U,U}$ factors through a functor
$\sfP^\wedge_{U,U} \to \sfP_{U,U}$, which will be denoted similarly.

Consider also the category $\Coh(\FN_{T^\vee_\bk \times_{T^\vee_\bk / \Wf} T^\vee_\bk}(\{(e,e)\}))$ of coherent sheaves on the noetherian scheme $\FN_{T^\vee_\bk \times_{T^\vee_\bk / \Wf} T^\vee_\bk}(\{(e,e)\})$. We have a natural fully faithful functor
\begin{equation}
\label{eqn:Coh-Rep-Dwedge}
\Coh(\FN_{T^\vee_\bk \times_{T^\vee_\bk / \Wf} T^\vee_\bk}(\{(e,e)\})) \to \Rep(\bbI^\wedge_\Sigma)
\end{equation}
sending a coherent sheaf to itself with the trivial structure as a representation of $\bbI^\wedge_\Sigma$, whose essential image contains the objects $\mathscr{B}^\wedge_s$ for $s \in \Sf$. It is clear from definition that the monoidal product $\circledast$ on $\Rep(\bbI^\wedge_\Sigma)$ restricts to a monoidal product on $\Coh(\FN_{T^\vee_\bk \times_{T^\vee_\bk / \Wf} T^\vee_\bk}(\{(e,e)\}))$ (which will be denoted similarly). It is clear also that for any $m \geq 1$ the restriction of the functor $\mathsf{D}_m$ to $\Coh(\FN_{T^\vee_\bk \times_{T^\vee_\bk / \Wf} T^\vee_\bk}(\{(e,e)\}))$ factors through a functor
\[
\Coh(\FN_{T^\vee_\bk \times_{T^\vee_\bk / \Wf} T^\vee_\bk}(\{(e,e)\})) \to \Coh_{\{(e,e)\}}(T^\vee_\bk \times_{T^\vee_\bk / \Wf} T^\vee_\bk),
\]
which will be denoted similarly.
We will denote by
\[
\SCoh(\FN_{T^\vee_\bk \times_{T^\vee_\bk / \Wf} T^\vee_\bk}(\{(e,e)\}))
\]
the full monoidal additive subcategory of $\Coh(\FN_{T^\vee_\bk \times_{T^\vee_\bk / \Wf} T^\vee_\bk}(\{(e,e)\}))$ generated (under the monoidal product, direct sums and direct summands) by the unit object $\mathscr{M}^\wedge_e$ and the objects $\mathscr{B}^\wedge_s$ for $s \in \Sf$. With this definition,~\eqref{eqn:Coh-Rep-Dwedge} identifies $\SCoh(\FN_{T^\vee_\bk \times_{T^\vee_\bk / \Wf} T^\vee_\bk}(\{(e,e)\}))$ with a full subcategory of $\SRep(\bbI^\wedge_\Sigma)$.

In~\cite[Theorem~11.8]{bezr} we have shown that the functor $\Hom_{\sfP^\wedge_{U,U}}(\Xi_!^\wedge,-)$ defines an equivalence of monoidal categories
\[
\Phi_{\sfT,U} : (\sfT^\wedge_{U,U}, \hatstar_U) \simto (\SCoh(\FN_{T^\vee_\bk \times_{T^\vee_\bk / \Wf} T^\vee_\bk}(\{(e,e)\})), \circledast)
\]
which satisfies
\begin{equation}
\label{eqn:PhiTU-Xis}
\Phi_{\sfT,U}(\Xi^\wedge_{s,!}) \cong \mathscr{B}^\wedge_s \quad \text{for $s \in \Sf$}
\end{equation}
and
\begin{equation}
\label{eqn:PhiTU-Xi}
\Phi_{\sfT,U}(\Xi^\wedge_{!}) \cong \scO(\FN_{T^\vee_\bk \times_{T^\vee_\bk / \Wf} T^\vee_\bk}(\{(e,e)\})).
\end{equation}

\begin{lem}
\label{lem:PhiT-comparison-G/B}
For any $m \geq 1$ and any $\scT$ in $\sfT^\wedge_{U,U}$ we have a canonical isomorphism
\[
\Phi_{\Iwu,\Iwu}(\mathsf{C}^0_{m}(\scT)) \cong \mathsf{D}_{m}(\Phi_{\sfT,U}(\scT)),
\]
where the right-hand side is seen as an object of $\Rep_0(\bbI_\Sigma)$ via~\eqref{eqn:Coh-Rep-JD}.
\end{lem}

\begin{proof}
The object $\mathsf{C}^0_m(\scT)$ belongs to the subcategory $\sfP^0_{U,U} \subset \sfP^0_{\Iwu,\Iwu}$. By Proposition~\ref{prop:isom-Phi-U-Iu}, this implies that
\[
\Phi_{\Iwu,\Iwu}(\mathsf{C}^0_m(\scT)) \cong \Phi_{U,U}(\mathsf{C}^0_m(\scT)) \cong \Hom_{\sfP^\wedge_{U,U}}(\Xi_!^\wedge,\mathsf{C}_m(\scT)),
\]
where the second isomorphism uses Lemma~\ref{lem:morph-Xiwedge-Pi}.
Now by projectivity of $\Xi^\wedge_!$ (see Lemma~\ref{lem:Xi-projective}) we have
\[
\Hom_{\sfP^\wedge_{U,U}}(\Xi_!^\wedge,\mathsf{C}_m(\scT)) \cong \Hom_{\sfP^\wedge_{U,U}}(\Xi_!^\wedge,\scT) \otimes_{\scO(T^\vee_\bk)} \scO((T^\vee_\bk)^{(m)}),
\]
which finishes the proof in view of the definitions of $\Phi_{\sfT,U}$ and $\mathsf{D}_{m}$.
\end{proof}


\subsection{Images of truncated standard and costandard objects}

We now need to describe the images under $\Phi_{\Iwu,\Iwu}$ of (images in $\sfP^0_{\Iwu,\Iwu}$ of) truncated standard and costandard free-monodromic perverse sheaves. Recall that the images of truncated costandard perverse sheaves associated with elements in $\Wf$ have been described in Lemma~\ref{lem:PhiUU-cost}. We next prove the analogous statement for standard objects.

\begin{lem}
\label{lem:PhiUU-st}
For any $w \in \Wf$, there exists an isomorphism of projective systems of $\scO(T^\vee_\bk \times_{T^\vee_\bk / \Wf} T^\vee_\bk)$-modules
 \[
  (\Phi_{U,U}(\sfC^0_m(\Delta_w^\wedge)) : m \geq 1) \cong
 (\scM_w / \cJ^m \cdot \scM_w : m \geq 1).
\]
\end{lem}

\begin{proof}
By Lemma~\ref{lem:convolution-stand-costand}\eqref{it:convolution-stand-costand-1} we have an isomorphism
\[
 \Delta^\wedge_w \hatstar \nabla^\wedge_{w^{-1}} \cong \delta^\wedge.
\]
After fixing such an isomorphism, using Proposition~\ref{prop:monoidality-C0} we deduce for any $m \geq 1$ an isomorphism
\[
 \sfC_m^0(\Delta^\wedge_w) \pstar^0_{\Iwu} \sfC^0_m(\nabla^\wedge_{w^{-1}}) \cong \sfC^0_m(\delta^\wedge),
\]
which by monoidality of $\Phi_{U,U}$ provides an isomorphism
\begin{equation}
\label{eqn:computation-Phi-Delta}
 \Phi_{U,U}(\sfC_m^0(\Delta^\wedge_w)) \oast \Phi_{U,U}(\sfC^0_m(\nabla^\wedge_{w^{-1}})) \cong \Phi_{U,U}(\sfC^0_m(\delta^\wedge)).
\end{equation}
Now, consider
\[
 M_w := \varprojlim_{m \geq 1} \Phi_{U,U}(\sfC_m^0(\Delta^\wedge_w)).
\]
For any $m \geq 1$, the $\scO(T^\vee_\bk \times_{T^\vee_\bk / \Wf} T^\vee_\bk)$-module $\Phi_{U,U}(\sfC_m^0(\Delta^\wedge_w))$ is finitely generated, and, by exactness of $\Phi_{U,U}$, for any $m' \geq m$ the natural morphism
\begin{multline*}
 \bigl( \scO(T^\vee_\bk \times_{T^\vee_\bk / \Wf} T^\vee_\bk) / \cJ^m \cdot \scO(T^\vee_\bk \times_{T^\vee_\bk / \Wf} T^\vee_\bk) \bigr) \otimes_{\scO(T^\vee_\bk \times_{T^\vee_\bk / \Wf} T^\vee_\bk)} \Phi_{U,U}(\sfC_{m'}^0(\Delta^\wedge_w)) \\
 \to \Phi_{U,U}(\sfC_m^0(\Delta^\wedge_w))
\end{multline*}
is an isomorphism. By~\cite[Chap.~0, Proposition~7.2.9]{ega1}, it follows that $M_w$ is a finitely generated $\scO(\FN_{T^\vee_\bk \times_{T^\vee_\bk / \Wf} T^\vee_\bk}(\{(e,e)\}))$-module, and that for any $m \geq 1$ the natural morphism
\[
 M_w / \cJ^m \cdot M_w \to \Phi_{U,U}(\sfC_m^0(\Delta^\wedge_w))
\]
is an isomorphism. To conclude the proof,
 it therefore suffices to construct an isomorphism of $\scO(\FN_{T^\vee_\bk \times_{T^\vee_\bk / \Wf} T^\vee_\bk}(\{(e,e)\}))$-modules $M_w \cong \scM_w^\wedge$. Now the isomorphisms~\eqref{eqn:computation-Phi-Delta} and the description of images of costandard objects in Lemma~\ref{lem:PhiUU-cost} provide an isomorphism
\[
 M_w \oast \scM_{w^{-1}}^\wedge \cong \scM_e^\wedge.
\]
Since $\scM_{w^{-1}}^\wedge$ is invertible for the product $\oast$, with inverse $\scM_w^\wedge$, we deduce the desired isomorphism $M_w \cong \scM_w^\wedge$.
\end{proof}

We can finally prove the desired general statement.

\begin{lem}
\label{lem:Phi-C-Delta}
 For any $w \in W$, there exist isomorphisms of projective systems
 \[
  \bigl( \Phi_{\Iwu,\Iwu}(\mathsf{C}_m^0(\Delta^\wedge_w)) : m \geq 1 \bigr) \cong \bigl( \mathsf{D}_m(\mathscr{M}_w^\wedge) : m \geq 1 \bigr)
 \]
 and
 \[
  \bigl( \Phi_{\Iwu,\Iwu}(\mathsf{C}_m^0(\nabla^\wedge_w)) : m \geq 1 \bigr) \cong \bigl( \mathsf{D}_m(\mathscr{M}_w^\wedge) : m \geq 1 \bigr).
 \]
\end{lem}

\begin{proof}
First we prove the isomorphisms when $w=s \in S$. In case $s \in \Sf$, they follow from Lemma~\ref{lem:PhiUU-cost} and Lemma~\ref{lem:PhiUU-st}.
Now assume that $s \in S \smallsetminus \Sf$. Then there exist $w \in \Wf$ and $\lambda \in X_*(T)$ antidominant such that $s=w \st(\lambda)$ with $\ell(\st(\lambda)) = \ell(w)+1$. By Lemma~\ref{lem:convolution-stand-costand}\eqref{it:convolution-stand-costand-1}--\eqref{it:convolution-stand-costand-2} we then have
\[
 \Delta^\wedge_s \cong \nabla^\wedge_w \hatstar \Delta^\wedge_{\st(\lambda)}.
\]
Using Proposition~\ref{prop:monoidality-C0} and the monoidality of $\Phi_{\Iwu,\Iwu}$ we deduce for any $m \geq 1$ an isomorphism
\[
\Phi_{\Iwu,\Iwu}(\sfC^0_m(\Delta^\wedge_s)) \cong \Phi_{\Iwu,\Iwu}(\sfC^0_m(\nabla^\wedge_w)) \oast \Phi_{\Iwu,\Iwu}(\sfC^0_m(\Delta^\wedge_{\st(\lambda)})).
\]
Now $\Phi_{\Iwu,\Iwu}(\sfC^0_m(\nabla^\wedge_w))$ is described in Lemma~\ref{lem:PhiUU-cost}, and $\Phi_{\Iwu,\Iwu}(\sfC^0_m(\Delta^\wedge_{\st(\lambda)}))$ is described in Proposition~\ref{prop:Phi-Wak-antidom}\eqref{it:Phi-Wak-antidom}. Using these descriptions, Lemma~\ref{lem:truncation-Rep-mon} and~\eqref{eqn:convolution-D}, we deduce the desired isomorphism of projective systems
 \[
  \bigl( \Phi_{\Iwu,\Iwu}(\mathsf{C}_m^0(\Delta^\wedge_s)) : m \geq 1 \bigr) \cong \bigl( \mathsf{D}_m(\mathscr{M}_s^\wedge) : m \geq 1 \bigr).
 \]
 The isomorphism
  \[
  \bigl( \Phi_{\Iwu,\Iwu}(\mathsf{C}_m^0(\nabla^\wedge_s)) : m \geq 1 \bigr) \cong \bigl( \mathsf{D}_m(\mathscr{M}_s^\wedge) : m \geq 1 \bigr)
 \]
 follows, using the same arguments as in the proof of Lemma~\ref{lem:PhiUU-st}.
 
 Note that in case $w \in \Omega$, we have $\Delta^\wedge_w \cong \nabla^\wedge_w$, so that both isomorphisms follow from Proposition~\ref{prop:Phi-Wak-antidom}\eqref{it:Phi-Delta-omega}.
 
 Finally we treat the general case. Given a reduced expression $w = \omega s_1 \cdots s_r$ (with $\omega \in \Omega$, $s_1, \cdots, s_r \in S$ and $r=\ell(w)$), by Lemma~\ref{lem:convolution-stand-costand}\eqref{it:convolution-stand-costand-2} we have
 \[
 \Delta^\wedge_w \cong \Delta^\wedge_\omega \hatstar  \Delta^\wedge_{s_1} \hatstar \cdots \hatstar \Delta^\wedge_{s_r}, \quad
  \nabla^\wedge_w \cong \nabla^\wedge_\omega \hatstar \nabla^\wedge_{s_1} \hatstar \cdots \hatstar \nabla^\wedge_{s_r}.
 \]
Using Proposition~\ref{prop:monoidality-C0} and monoidality of $\Phi_{\Iwu,\Iwu}$ we deduce isomorphisms
 \begin{align*}
 \Phi_{\Iwu,\Iwu}(\sfC^0_m(\Delta^\wedge_w)) &\cong \Phi_{\Iwu,\Iwu}(\sfC^0_m(\Delta^\wedge_\omega)) \oast  \Phi_{\Iwu,\Iwu}(\sfC^0_m(\Delta^\wedge_{s_1})) \oast \cdots \oast \Phi_{\Iwu,\Iwu}(\sfC^0_m(\Delta^\wedge_{s_r})), \\
  \Phi_{\Iwu,\Iwu}(\sfC^0_m(\nabla^\wedge_w)) &\cong \Phi_{\Iwu,\Iwu}(\sfC^0_m(\nabla^\wedge_\omega)) \oast \Phi_{\Iwu,\Iwu}(\sfC^0_m(\nabla^\wedge_{s_1})) \oast \cdots \oast \Phi_{\Iwu,\Iwu}(\sfC^0_m(\nabla^\wedge_{s_r})).
 \end{align*}
 We deduce the desired isomorphisms using the case of the elements $\omega$ and $s_i$, together with Lemma~\ref{lem:truncation-Rep-mon} and~\eqref{eqn:convolution-D}.
 \end{proof}
 
 \begin{rmk}
 Since $\Phi_{\Iwu,\Iwu}$ is an equivalence, Lemma~\ref{lem:Phi-C-Delta} implies in particular that the projective systems $( \mathsf{C}_m^0(\Delta^\wedge_w) : m \geq 1 )$ and $( \mathsf{C}_m^0(\nabla^\wedge_w) : m \geq 1 )$ are isomorphic. In other words, in the quotient $\sfP_{\Iwu,\Iwu}^0$ one cannot see the difference between standard and costandard objects. (This is a standard phenomenon in Soergel theory.)
 \end{rmk}

\subsection{Completion of the proof}

We can finally complete the proof of Lem\-ma~\ref{lem:isom-PhiT}.

\begin{proof}[Proof of Lemma~\ref{lem:isom-PhiT}]
From Lemma~\ref{lem:PhiT-comparison-G/B} and~\eqref{eqn:PhiTU-Xis} we deduce~\eqref{it:isom-main-app-1} in case $s \in \Sf$. Similarly,~\eqref{it:isom-main-app-3} follows from Lemma~\ref{lem:PhiT-comparison-G/B} and~\eqref{eqn:PhiTU-Xi}.

Finally, consider some $s \in S \smallsetminus \Sf$, and recall the elements $w,s'$ chosen in~\S\ref{ss:Waff}. Then by definition we have
\[
 \Xi^\wedge_{s,!} = \Delta^\wedge_w \hatstar \Xi^\wedge_{s',!} \hatstar \nabla^\wedge_{w^{-1}}.
\]
Using Proposition~\ref{prop:monoidality-C0} and
monoidality of $\Phi_{\Iwu,\Iwu}$ we deduce for any $m \geq 1$ an isomorphism
\[
 \Phi_{\Iwu,\Iwu}(\mathsf{C}_m^0(\Xi^\wedge_{s,!})) \cong \Phi_{\Iwu,\Iwu}(\mathsf{C}_m^0(\Delta^\wedge_w)) \circledast \Phi_{\Iwu,\Iwu}(\mathsf{C}_m^0(\Xi^\wedge_{s',!})) \circledast \Phi_{\Iwu,\Iwu}(\mathsf{C}_m^0(\nabla^\wedge_{w^{-1}})).
\]
Here we know that $\Phi_{\Iwu,\Iwu}(\mathsf{C}_m^0(\Xi^\wedge_{s',!})) \cong \mathsf{D}_m(\mathscr{B}^\wedge_{s'})$. On the other hand, by Lem\-ma~\ref{lem:Phi-C-Delta}
 we have isomorphisms
\[
\Phi_{\Iwu,\Iwu}(\mathsf{C}_m^0(\Delta^\wedge_w)) \cong \mathsf{D}_m(\mathscr{M}^\wedge_{w}), \quad
\Phi_{\Iwu,\Iwu}(\mathsf{C}_m^0(\nabla^\wedge_{w^{-1}})) \cong \mathsf{D}_m(\mathscr{M}^\wedge_{w^{-1}}).
\]
We deduce isomorphisms
\[
 \Phi_{\Iwu,\Iwu}(\mathsf{C}_m^0(\Xi^\wedge_{s,!})) \cong \mathsf{D}_m(\mathscr{M}^\wedge_{w}) \circledast \mathsf{D}_m(\mathscr{B}^\wedge_{s'}) \circledast \mathsf{D}_m(\mathscr{M}^\wedge_{w^{-1}}).
 \]
 Using Lemma~\ref{lem:truncation-Rep-mon} and the definition (see~\eqref{eqn:def-Bs}), one sees that the right-hand side identifies with $\mathsf{D}_m(\mathscr{B}^\wedge_{s})$, which concludes the proof.
\end{proof}

\begin{rmk}
\label{rmk:canonical-Xis}
The objects $\Xi^\wedge_{s,!}$ considered above have been defined in a canonical way for $s \in \Sf$, and in this case the isomorphism in Theorem~\ref{thm:main-application}\eqref{it:main-application-1} is canonical. But we do not know any canonical construction in case $s \in S \smallsetminus \Sf$, and the objects $\Delta^\wedge_\omega$ ($\omega \in \Omega$) are defined only up to isomorphism. However, one can always fix some
objects $\Xi^{\wedge,\can}_{s}$ ($s \in S$) and $\Delta^{\wedge,\can}_\omega$ ($\omega \in \Omega$) in $\sfT^\wedge_{\Iwu,\Iwu}$ and identifications
\[
\Phi_{\sfT}(\Xi^{\wedge,\can}_{s}) = \scB^\wedge_s, \qquad \Phi_{\sfT}(\Delta^{\wedge,\can}_{\omega}) = \scM^\wedge_\omega
\]
for $s \in S \smallsetminus \Sf$ and $\omega \in \Omega$. 
Using Lemma~\ref{lem:conjugation-B-RepJ} and monoidality of $\Phi_{\sfT}$, we deduce canonical isomorphisms
\begin{equation}
\label{eqn:conjugation-Xi-can}
\Delta^{\wedge,\can}_{\omega} \hatstar \Xi^{\wedge,\can}_{s} \hatstar \Delta^{\wedge,\can}_{\omega^{-1}} \cong \Xi^{\wedge,\can}_{\omega s \omega^{-1}}
\end{equation}
for any $s \in S$ and $\omega \in \Omega$.
One can then define the category $\mathsf{T}^{\wedge,\mathrm{BS}}_{\Iwu,\Iwu}$ with
\begin{itemize}
\item
objects the collections $(\omega, s_1, \cdots, s_i)$ with $\omega \in \Omega$ and $s_1, \cdots, s_i \in S$;
\item
morphisms from $(\omega, s_1, \cdots, s_i)$ to $(\omega', s'_1, \cdots, s'_j)$ given by
\[
\Hom_{\mathsf{T}^\wedge_{\Iwu,\Iwu}}(\Delta^{\wedge,\can}_{\omega} \hatstar \Xi^{\wedge,\can}_{s_1} \hatstar \cdots \hatstar \Xi^{\wedge,\can}_{s_i}, \Delta^{\wedge,\can}_{\omega'} \hatstar \Xi^{\wedge,\can}_{s'_1} \hatstar \cdots \hatstar \Xi^{\wedge,\can}_{s'_j} \bigr).
\]
\end{itemize}
Using the isomorphisms~\eqref{eqn:conjugation-Xi-can} one can define on $\mathsf{T}^{\wedge,\mathrm{BS}}_{\Iwu,\Iwu}$ a canonical monoidal structure, such that we have a canonical equivalence of monoidal categories
\begin{equation}
\label{eqn:equiv-TBS-BSRep}
\sfT^{\wedge,\mathrm{BS}}_{\Iwu,\Iwu} \simto \BSRep(\bbI_\Sigma^\wedge)
\end{equation}
which is the identity on objects,
where $\BSRep(\bbI_\Sigma^\wedge)$ is as in~\S\ref{ss:completions} (for $\bG=G^\vee_\bk$).
We also have a canonical fully faithful monoidal functor
\begin{equation*}
\mathsf{T}^{\wedge,\mathrm{BS}}_{\Iwu,\Iwu} \to \mathsf{T}^\wedge_{\Iwu,\Iwu}
\end{equation*}
sending $(\omega, s_1, \cdots, s_i)$ to $\Delta^{\wedge,\can}_{\omega} \hatstar \Xi^{\wedge,\can}_{s_1} \hatstar \cdots \hatstar \Xi^{\wedge,\can}_{s_i}$, and $\mathsf{T}^{\wedge}_{\Iwu,\Iwu}$ identifies with the Karoubian closure of the additive hull of $\mathsf{T}^{\wedge,\mathrm{BS}}_{\Iwu,\Iwu}$.
\end{rmk}

\section{The case of \texorpdfstring{$\Fl_G$}{Fl}}
\label{sec:Fl}

We continue with the assumptions of Sections~\ref{sec:construction}--\ref{sec:Soergel-type}.
In this section we briefly indicate how the constructions of Sections~\ref{sec:Perv-G/U},~\ref{sec:construction} and~\ref{sec:Soergel-type} can be adapted to give a description of the more familiar category of tilting $\Iwu$-equivariant perverse sheaves on $\Fl_G$. (These results are not used in~\cite{br-pt3}.)

\subsection{The case of \texorpdfstring{$G/B$}{G/B}}

We start with the analogue of Theorem~\ref{thm:reg-quotient-finite}. Recall the categories $\sfP_{U,B}$ and $\sfP^0_{U,B}$ considered in~\S\ref{ss:fully-faithfulness}, and denote by $\Pi^0_{U,B} : \sfP_{U,B} \to \sfP_{U,B}^0$ the quotient functor. We can also consider the Verdier quotient of the category $\sfD_{U,B}$ from the proof of Lemma~\ref{lem:Xi-projective} by the full triangulated subcategory generated by the simple perverse sheaves $\For^{\Iw}_{\Iwu}(\IC_w)$ with $w \in \Wf \smallsetminus \{e\}$. The functor $\pi^{\dag,0}$ from~\S\ref{ss:relation-reg-quotient} is t-exact, and it restricts to a functor $\sfD^0_{U,B} \to \sfD_{U,U}^0$; it therefore restricts to an exact functor
\begin{equation}
\label{eqn:pullback-flag}
 \sfP_{U,B}^0 \to \sfP_{U,U}^0,
\end{equation}
which is moreover fully faithful (because so is the restriction of $\pi^\dag$ to perverse sheaves). The same construction as for the convolution product $\pstar^0_U$ (see~\S\ref{ss:cat-sheaves-G/U}) provides a left action of the category $\sfP_{U,U}^0$ on $\sfP_{U,B}^0$, via a bifunctor also denoted $\pstar^0_U$, and which is right exact on both sides.

On the other hand, 
consider the finite $\bk$-scheme
\[
(T^\vee_\bk \times_{T^\vee_\bk / \Wf} T^\vee_\bk) \times_{T^\vee_\bk} \{e\} = T^\vee_\bk \times_{T^\vee_\bk / \Wf} \{e\}
\]
where the morphism $T^\vee_\bk \times_{T^\vee_\bk / \Wf} T^\vee_\bk \to T^\vee_\bk$ is projection on the second factor and $e \in T^\vee_\bk$ is the unit element. The closed embedding
\[
T^\vee_\bk \times_{T^\vee_\bk / \Wf} \{e\} \hookrightarrow T^\vee_\bk \times_{T^\vee_\bk / \Wf} T^\vee_\bk
\]
induces an exact fully faithful functor
\begin{equation}
\label{eqn:functor-bE-bD}
 \Coh(T^\vee_\bk \times_{T^\vee_\bk / \Wf} \{e\}) \to \Coh_{\{(e,e)\}}(T^\vee_\bk \times_{T^\vee_\bk / \Wf} T^\vee_\bk).
\end{equation}
We also have a natural left action of $\Coh_{\{(e,e)\}}(T^\vee_\bk \times_{T^\vee_\bk / \Wf} T^\vee_\bk)$ on $\Coh(T^\vee_\bk \times_{T^\vee_\bk / \Wf} \{e\})$, via a bifunctor denoted $\circledast$.

\begin{thm}
\label{thm:reg-quotient-finite-B}
 There exists a canonical equivalence of abelian categories
 \[
  \Phi_{U,B} : \sfP_{U,B}^0 \simto \Coh(T^\vee_\bk \times_{T^\vee_\bk / \Wf} \{e\})
 \]
such that the following diagram commutes up to isomorphism:
\[
 \xymatrix@C=2cm{
 \sfP_{U,B}^0 \ar[r]^-{\eqref{eqn:pullback-flag}} \ar[d]_-{\Phi_{U,B}}^-{\wr} & \sfP_{U,U}^0 \ar[d]_-{\wr}^-{\Phi_{U,U}} \\
 \Coh(T^\vee_\bk \times_{T^\vee_\bk / \Wf} \{e\}) \ar[r]^-{\eqref{eqn:functor-bE-bD}} & \Coh_{\{(e,e)\}}(T^\vee_\bk \times_{T^\vee_\bk / \Wf} T^\vee_\bk).
 }
\]
Moreover, this equivalence intertwines the actions of the category $\sfP_{U,U}^0$ on $\sfP_{U,B}^0$ and of the category $\Coh_{\{(e,e)\}}(T^\vee_\bk \times_{T^\vee_\bk / \Wf} T^\vee_\bk)$ on $\Coh(T^\vee_\bk \times_{T^\vee_\bk / \Wf} \{e\})$ via the equivalence $\Phi_{U,U}$.
\end{thm}

\begin{proof}
Recall the object $\Xi_!$ from~\S\ref{ss:G/B-tilting}, which is a projective cover of the simple object $\delta_{\Fl}$ in $\sfP_{U,B}$. The object $\Pi_{U,B}^0(\Xi_!)$ is then the projective cover of the unique simple object in $\sfP^0_{U,B}$, so that the functor $\Hom_{\sfP^0_{U,B}}(\Pi_{U,B}^0(\Xi_!), -)$ induces an equivalence of abelian categories
\[
 \sfP^0_{U,B} \simto \Mofr(\End_{\sfP^0_{U,B}}(\Pi_{U,B}^0(\Xi_!))).
\]
Now, a much simplified version of the proof of Lemma~\ref{lem:morph-Xiwedge-Pi} shows that for any $\scF$ in $\sfP_{U,B}$ the morphism
\[
 \Hom_{\sfP_{U,B}}(\Xi_!, \scF) \to \Hom_{\sfP^0_{U,B}}(\Pi_{U,B}^0(\Xi_!), \Pi_{U,B}^0(\scF))
\]
induced by $\Pi_{U,B}^0$ is an isomorphism. In particular, we have an algebra isomorphism
\[
 \End_{\sfP_{U,B}}(\Xi_!) \simto \End_{\sfP_{U,B}^0}(\Pi^0_{U,B}(\Xi_!)).
\]
By~\cite[Corollary~9.2]{bezr}, the left-hand side identifies (via monodromy) with the algebra $\scO(T^\vee_\bk \times_{T^\vee_\bk / \Wf} \{e\})$, which provides the desired equivalence $\Phi_{U,B}$.

The compatibility with $\Phi_{U,U}$ can be checked using adjunction and~\eqref{eqn:pidag-Xi}. The compatibility with the action of $\sfP^0_{U,U}$ follows from (a simplified version of) the same arguments as for the monoidality of $\Phi_{U,U}$.
\end{proof}

\subsection{Description of the regular quotient}

We now set
\[
\bbI'_{\Sigma} := (T^\vee_\bk \times_{T^\vee_\bk / \Wf} \{e\}) \times_{T^\vee_\bk \times_{T^\vee_\bk / \Wf} T^\vee_\bk} \bbI_\Sigma.
\]
Since the fiber of $\bbJ_\Sigma$ over the image of $e$ in $T^\vee_\bk / \Wf$ is $\rmZ_{G^\vee_\bk}(\su)$, we in fact have
\[
\bbI'_{\Sigma} \cong (T^\vee_\bk \times_{T^\vee_\bk / \Wf} \{e\}) \times \rmZ_{G^\vee_\bk}(\su)
\]
as group schemes over $(T^\vee_\bk \times_{T^\vee_\bk / \Wf} \{e\})$. We will consider the abelian category $\Rep(\bbI'_\Sigma)$ of representations of $\bbI'_\Sigma$ on coherent $\scO_{T^\vee_\bk \times_{T^\vee_\bk / \Wf} \{e\}}$-modules. Pushforward along the embeddings
\[
 \{(e,e)\} \hookrightarrow T^\vee_\bk \times_{T^\vee_\bk / \Wf} \{e\} \hookrightarrow T^\vee_\bk \times_{T^\vee_\bk / \Wf} T^\vee_\bk
\]
provides exact and fully faithful functors
\begin{equation}
\label{eqn:functors-Rep-J}
 \Rep(\rmZ_{G^\vee_\bk}(\su)) \to \Rep(\bbI'_\Sigma) \to \Rep_0(\bbI_\Sigma)
\end{equation}
whose composition is~\eqref{eqn:functor-statement-Rep}.
We also have a canonical left action of $\Rep_0(\bbI_\Sigma)$ on $\Rep(\bbI'_\Sigma)$, and a canonical right action of $\Rep(\rmZ_{G^\vee_\bk}(\su))$ on $\Rep(\bbI'_\Sigma)$; the corresponding bifunctors will be denoted $\circledast$ and $\otimes$ respectively. (The latter bifunctor is exact on both sides, but the former is only right exact on both sides.) Finally, we have a natural exact fully faithful functor
\begin{equation}
\label{eqn:embedding-CohE}
 \Coh(T^\vee_\bk \times_{T^\vee_\bk / \Wf} \{e\}) \to \Rep(\bbI'_\Sigma)
\end{equation}
sending a coherent sheaf to itself with the trivial structure as a representation.

On the other hand, recall from~\S\ref{ss:another-convolution} the category $\sfD_{\Iwu,\Iw}^0$, and the heart $\sfP_{\Iwu,\Iw}^0$ of its perverse t-structure. We have canonical exact and fully faithful functors
\[
 \sfP_{\Iw,\Iw}^0 \xrightarrow{\For^{\Iw,0}_{\Iwu}} \sfP_{\Iwu,\Iw}^0 \xrightarrow{\pi^{\dag,0}} \sfP_{\Iwu,\Iwu}^0.
\]
As explained in~\S\ref{ss:another-convolution}, we also have a natural right action of $\sfP_{\Iw,\Iw}^0$ on $\sfP_{\Iwu,\Iw}^0$ (via a bifunctor denoted $\star^0_{\Iw}$, which is exact on both sides), and as for the construction of the convolution product $\pstar_{\Iwu}^0$ (see~\S\ref{ss:mon-reg-quotient}) we have a natural left action of $\sfP_{\Iwu,\Iwu}^0$ on $\sfP_{\Iwu,\Iw}^0$, via a bifunctor also denoted $\pstar_{\Iwu}^0$ (and which is also right exact on both sides). We have a natural embedding $G/B=\overline{\Fl_{G,w_\circ}} \subset \Fl_G$, which provides via pushforward an exact and fully faithful functor
\begin{equation}
\label{eqn:embedding-PUB}
 \sfP^0_{U,B} \to \sfP^0_{\Iwu,\Iw}.
\end{equation}

\begin{thm}
\label{thm:main-Fl}
 There exists a canonical equivalence of abelian categories
 \[
  \Phi_{\Iwu,\Iw} : \sfP_{\Iwu,\Iw}^0 \simto \Rep(\bbI'_\Sigma)
 \]
such that the diagrams
\[
 \xymatrix{
 \sfP_{\Iw,\Iw}^0 \ar[r]^-{\For^{\Iw,0}_{\Iwu}} \ar[d]_-{\Phi_{\Iw,\Iw}}^-{\wr} & \sfP_{\Iwu,\Iw}^0 \ar[r]^-{\pi^{\dag,0}} \ar[d]^-{\Phi_{\Iwu,\Iw}}_-{\wr} & \sfP_{\Iwu,\Iwu}^0 \ar[d]^-{\Phi_{\Iwu,\Iwu}}_-{\wr} \\
 \Rep(\rmZ_{G^\vee_\bk}(\su)) \ar[r] & \Rep(\bbI'_\Sigma) \ar[r] & \Rep_0(\bbI_\Sigma)
 }
 \]
(where on the left-hand side the lower horizontal row is as in~\eqref{eqn:functors-Rep-J}) and
 \[
 \xymatrix@C=1.5cm{
 \sfP^0_{U,B} \ar[r]^-{\eqref{eqn:embedding-PUB}} \ar[d]_-{\Phi_{U,B}}^-{\wr} & \sfP^0_{\Iwu,\Iw} \ar[d]_-{\wr}^-{\Phi_{\Iwu,\Iw}} \\
 \Coh(T^\vee_\bk \times_{T^\vee_\bk / \Wf} \{e\}) \ar[r]^-{\eqref{eqn:embedding-CohE}} & \Rep(\bbI'_\Sigma)
 }
\]
 commute (up to isomorphisms).
Moreover, this equivalence intertwines the actions of $\sfP_{\Iwu,\Iwu}^0$ and $\sfP_{\Iw,\Iw}^0$ on $\sfP_{\Iwu,\Iw}^0$ and of $\Rep(\rmZ_{G^\vee_\bk}(\su))$ and $\Rep_0(\bbI_\Sigma)$ on $\Rep(\bbI'_\Sigma)$ via the equivalences $\Phi_{\Iw,\Iw}$ and $\Phi_{\Iwu,\Iwu}$.
\end{thm}

\subsection{(Sketch of) proof of Theorem~\ref{thm:main-Fl}}


The same considerations as in~\S\ref{ss:actions-Dwedge-D} provide a left action of the category $\sfD^\wedge_{\Iwu,\Iwu}$ on $\sfD^0_{\Iwu,\Iw}$, via a bifunctor denoted $\hatstar^0$, and which satisfies
\begin{equation}
\label{eqn:convolution-pidag}
 \pi^{\dag,0}(\scF \hatstar^0 \scG) \cong \scF \hatstar^0 (\pi^{\dag,0} \scG)
\end{equation}
for any $\scF$ in $\sfD^{\wedge}_{\Iwu,\Iwu}$ and $\scG$ in $\sfD^0_{\Iwu,\Iw}$. Taking $0$-th perverse cohomology we then define $\phatstar^0$ as in the setting of $\sfP^0_{\Iwu,\Iwu}$.

One can deduce from Lemma~\ref{lem:monodromy-fiber-prod} that the monodromy action of $\scO(T^\vee_\bk)$ on any object in $\sfP_{\Iwu,\Iw}$ (induced by the $T$-action on $\Fl_G$ by multiplication on the left) factors through an action of $\scO(T^\vee_\bk \times_{T^\vee_\bk / \Wf} \{e\})$.
We can then define the functor $\Phi_{\Iwu,\Iw}$, initially with values in $\scO(T^\vee_\bk \times_{T^\vee_\bk / \Wf} \{e\})$-modules, by setting
\[
 \Phi_{\Iwu,\Iw}(\scF) = \Hom_{\sfD^0_{\Iwu,\Iw}}(\Pi^0_{\Iwu,\Iw}(\Xi_!), \scR^\wedge \phatstar^0 \scF).
\]
(Here $\Pi^0_{\Iwu,\Iw}(\Xi_!)$ is a honest object in $\sfD^0_{\Iwu,\Iw}$, and $\scR^\wedge \phatstar^0 \scF$ is an ind-object in $\sfD^0_{\Iwu,\Iw}$; morphisms are taken in the category ${\ind}\sfD^0_{\Iwu,\Iw}$.) Using~\eqref{eqn:convolution-pidag}, exactness of $\pi^{\dag,0}$, adjunction and~\eqref{eqn:pidag-Xi} one sees that for $\scF$ in $\sfP^0_{\Iwu,\Iw}$ we have a canonical isomorphism
\[
 \Phi_{\Iwu,\Iwu}(\pi^{\dag,0}(\scF)) \cong \Phi_{\Iwu,\Iw}(\scF).
\]
This implies that $\Phi_{\Iwu,\Iw}$ is exact and takes values in finite-dimensional $\scO(T^\vee_\bk \times_{T^\vee_\bk / \Wf} \{e\})$-modules. The same considerations as in~\S\ref{ss:definition-Phi-Iwu} can be used to endow $\Phi_{\Iwu,\Iw}(\scF)$ with the structure of a representation of $\bbI'_\Sigma$, so that we have in fact constructed an exact functor
\[
 \Phi_{\Iwu,\Iw} : \sfP^0_{\Iwu,\Iw} \to \Rep(\bbI'_\Sigma).
\]
The considerations above show that the left-hand diagram in Theorem~\ref{thm:main-Fl} commutes. One can next check commutativity of the right-hand diagram as in Lem\-ma~\ref{prop:isom-Phi-U-Iu}.

Using the comparison with $\Phi_{\Iwu,\Iwu}$ we see also that $\Phi_{\Iwu,\Iw}$ is fully faithful. Essential surjectivity can be checked as in~\S\ref{ss:essential-surjectivity}: namely, using pushforward to $T^\vee_\bk \times_{T^\vee_\bk / \Wf} T^\vee_\bk$, Lemma~\ref{lem:J-modules-quot}, and then pullback to $T^\vee_\bk \times_{T^\vee_\bk / \Wf} \{e\}$, one sees that any object in $\Rep(\bbI'_\Sigma)$ is a quotient of an object of the form $V \otimes \scO(T^\vee_\bk \times_{T^\vee_\bk / \Wf} \{e\})$ with $V$ in $\Rep(G^\vee_\bk)$. This allows to conclude since
\[
 \Phi_{\Iwu,\Iw}(\tsZ(V) \hatstar^0 \Pi^0_{\Iwu,\Iwu}(\Xi_!)) \cong V \otimes \scO(T^\vee_\bk \times_{T^\vee_\bk / \Wf} \{e\}).
\]

Finally, compatibility with the actions can be checked by considerations similar to those used to prove monoidality of $\Phi_{\Iwu,\Iwu}$ (see~\S\ref{ss:monoidality-morph} and~\S\ref{ss:monoidality-PhiIuIu}).

\subsection{Description of tilting perverse sheaves}
\label{ss:description-tiltings-Iu-I}

We finally explain how to adapt Theorem~\ref{thm:main-application} to the present setting. We will denote by $\sfT_{\Iwu,\Iw}$ the full subcategory of tilting objects in the highest weight category $\sfP_{\Iwu,\Iw}$. We have a natural functor
\[
 \pi_\dag : \sfT^\wedge_{\Iwu,\Iwu} \to \sfT_{\Iwu,\Iw},
\]
and an action of the category $\sfT^\wedge_{\Iwu,\Iwu}$ on $\sfT_{\Iwu,\Iw}$ (via $\hatstar$), from which the functor $\pi_\dag$ can be recoved as the action on the object $\For^{\Iw}_{\Iwu}(\delta_{\Fl})$. Standard considerations show that $\sfT_{\Iwu,\Iw}$ is the smallest additive full subcategory of $\sfP_{\Iwu,\Iw}$ that contains $\For^{\Iw}_{\Iwu}(\delta_{\Fl})$ and is stable under direct summands and the action of the objects $\Xi^\wedge_{s,!}$ ($s \in S$) and $\Delta^\wedge_\omega$ ($\omega \in \Omega$).


The following claim follows from~\cite[Proposition~5.9(2)]{bezr}.

\begin{lem}
\label{lem:Hom-tilting}
For any $\scF,\scG$ in $\sfT^\wedge_{\Iwu,\Iwu}$, the functor $\pi_\dag$ induces an isomorphism
\[
\Hom_{\sfT^\wedge_{\Iwu,\Iwu}}(\scF,\scG) \otimes_{\scO(T^\vee_\bk)} \bk \simto \Hom_{\sfT_{\Iwu,\Iw}}(\pi_\dag(\scF), \pi_\dag(\scG)),
\]
where the action of $\scO(T^\vee_\bk)$ is via monodromy associated with the right action of $T$ on $\tFl_G$, and $\bk$ is viewed as an $\scO(T^\vee_\bk)$-module via evaluation at $e \in T^\vee_\bk$.
\end{lem}

On the other hand, we have a natural functor
\[
 \Rep(\bbI_{\Sigma}^\wedge) \to \Rep(\bbI'_\Sigma)
\]
given by restriction along the closed immersion
\[
T^\vee_\bk \times_{T^\vee_\bk / \Wf} \{e\} \hookrightarrow \FN_{T^\vee_\bk \times_{T^\vee_\bk / \Wf} T^\vee_\bk}(\{(e,e)\}).
\]
We also have a left action of $\Rep(\bbI_{\Sigma}^\wedge)$ on $\Rep(\bbI'_{\Sigma})$ by convolution (the corresponding bifunctor will once again be denoted $\circledast$), and the functor above is given by the action 
on the skyscraper sheaf at the base point. We will denote by $\SRep(\bbI'_\Sigma)$ the full additive subcategory of $\Rep(\bbI'_\Sigma)$ generated under direct sums and direct summands by the image of $\SRep(\bbI_{\Sigma}^\wedge)$ under this functor. In other words, $\SRep(\bbI'_\Sigma)$ is the smallest Karoubian additive subcategory of $\Rep(\bbI'_\Sigma)$ containing the skyscraper sheaf at the base point and stable under action by the objects $\mathscr{B}^\wedge_s$ ($s \in S$) and $\mathscr{M}^\wedge_\omega$ ($\omega \in \Omega$).

\begin{thm}
\label{thm:main-application-Fl}
 There exists a canonical equivalence of additive categories
 \[
  \Psi_{\sfT} : \sfT_{\Iwu,\Iw} \simto \SRep(\bbI'_\Sigma)
 \]
which sends $\For^{\Iw}_{\Iwu}(\delta_{\Fl})$ to the skyscraper sheaf at the base point and intertwines the actions of $\sfT^\wedge_{\Iwu,\Iwu}$ on $\sfT_{\Iwu,\Iw}$ and $\SRep(\bbI_{\Sigma}^\wedge)$ on $\SRep(\bbI'_{\Sigma})$ via the equivalence $\Phi_{\sfT}$.
\end{thm}

\begin{proof}
 It is a standard fact that the quotient functor $\Pi^0_{\Iwu,\Iw}$ is fully faithful on the subcategory $\sfT_{\Iwu,\Iw}$ (see~\cite[\S 2.1]{bbm} for the similar claim on $G/B$). Since the equivalence $\Phi_{\Iwu,\Iw}^0$ of Theorem~\ref{thm:main-Fl} sends $\For^{\Iw}_{\Iwu}(\delta_{\Fl})$ to the skyscraper sheaf at the base point, to prove the theorem it therefore suffices to check that this functor satisfies
 \[
  \Phi_{\Iwu,\Iw}(\Xi^\wedge_{s,!} \phatstar^0 \scF) \cong \mathscr{B}^\wedge_s \circledast \Phi_{\Iwu,\Iw}(\scF), \quad \Phi_{\Iwu,\Iw}(\Delta^\wedge_\omega \phatstar^0 \scF) \cong \mathscr{M}^\wedge_\omega \circledast \Phi_{\Iwu,\Iw}(\scF)
 \]
for $s \in S$, $\omega \in \Omega$ and $\scF$ in $\sfP^0_{\Iwu,\Iw}$. However, $\cJ$ acts trivially on $\scF$. As in the proof of Lemma~\ref{lem:monoidality-C0}, one can check that we have a canonical isomorphism
\[
 \Xi^\wedge_{s,!} \phatstar^0 \scF \cong \mathsf{C}_1^0(\Xi^\wedge_{s,!}) \pstar^0_{\Iwu} \scF.
\]
It follows that
\begin{multline*}
  \Phi_{\Iwu,\Iw}(\Xi^\wedge_{s,!} \phatstar^0 \scF) \cong \Phi_{\Iwu,\Iwu}(\mathsf{C}_1^0(\Xi^\wedge_{s,!})) \circledast \Phi_{\Iwu,\Iw}(\scF) \\
  \cong \mathsf{D}_1(\mathscr{B}^\wedge_s) \circledast \Phi_{\Iwu,\Iw}(\scF) \cong \mathscr{B}^\wedge_s \circledast \Phi_{\Iwu,\Iw}(\scF),
  \end{multline*}
  which proves the first isomorphism. The second one can be checked similarly.
\end{proof}

\begin{rmk}
 Comparing Theorem~\ref{thm:main-application} and Theorem~\ref{thm:main-application-Fl}, and using Lemma~\ref{lem:Hom-tilting}, we see that for any $M,N$ in $\SRep(\bbI_\Sigma^\wedge)$, there exists a canonical isomorphism
 \[
  \Hom_{\SRep(\bbI_\Sigma^\wedge)}(M,N) \otimes_{\scO(T^\vee_\bk)} \bk \simto \Hom_{\SRep(\bbI'_\Sigma)}(M_{|T^\vee_\bk \times_{T^\vee_\bk / \Wf} \{e\}},N_{|T^\vee_\bk \times_{T^\vee_\bk / \Wf} \{e\}}).
 \]
Such a property is standard in the usual theory of Soergel bimodules, see e.g.~\cite[Proposition~1.13]{riche-bourbaki}.
\end{rmk}

\appendix

\section{Verdier quotients, Serre quotients, and t-structures}

\subsection{Quotient categories}
\label{ss:quotient-cat}

If $\mathsf{A}$ is an abelian category, recall that a \emph{Serre subcategory} of $\mathsf{A}$ is a nonempty strictly full subcategory which is stable under subquotients and extensions. Given a Serre subcategory $\mathsf{B} \subset \mathsf{A}$, one can form the quotient category $\mathsf{A}/\mathsf{B}$ and the exact functor $Q : \mathsf{A} \to \mathsf{A}/\mathsf{B}$ which satisfy the universal property that given an abelian category $\mathsf{A}'$ and an exact functor $F : \mathsf{A} \to \mathsf{A}'$ such that $F(M)=0$ for any $M$ in $\mathsf{B}$, there exists a unique functor $G : \mathsf{A}/\mathsf{B} \to \mathsf{A}'$ such that $F=G \circ Q$. (Moreover, in this situation $G$ is exact.)

\begin{rmk}
 There exist at least two different constructions of $\mathsf{A}/\mathsf{B}$: the initial construction of Gabriel given in~\cite[\S III.1]{gabriel}, and the construction as a localization explained in~\cite[\href{https://stacks.math.columbia.edu/tag/02MN}{Tag 02MN}]{stacks-project}. Since these constructions provide categories satisfying the same universal property, the corresponding categories must be canonically equivalent.
\end{rmk}

Now, let $\mathsf{D}$ be a triangulated category, and $\mathsf{E} \subset \mathsf{D}$ a full triangulated subcategory. Then one can form the quotient category $\mathsf{D}/\mathsf{E}$ and the quotient functor $\Pi : \mathsf{D} \to \mathsf{D}/\mathsf{E}$ using a localization procedure as in~\cite[\href{https://stacks.math.columbia.edu/tag/05RA}{Tag 05RA}]{stacks-project}. This category satisfies the following universal properties (see~\cite[\href{https://stacks.math.columbia.edu/tag/05RJ}{Tag 05RJ}]{stacks-project}):
\begin{enumerate}
 \item for any triangulated category $\mathsf{D}'$ and any triangulated functor $F : \mathsf{D} \to \mathsf{D}'$ such that $F(M)=0$ for any $M$ in $\mathsf{E}$, there exists a unique functor $G : \mathsf{D}/\mathsf{E} \to \mathsf{D}'$ such that $F=G \circ \Pi$; moreover $G$ is triangulated;
 \item for any abelian category $\mathsf{A}$ and any cohomological functor $H : \mathsf{D} \to \mathsf{A}$ such that $H(M)=0$ for any $M$ in $\mathsf{E}$, there exists a unique functor $H' : \mathsf{D}/\mathsf{E} \to \mathsf{A}$ such that $H=H' \circ \Pi$; moreover $H'$ is a cohomological functor.
\end{enumerate}
Another property which is checked similarly is that given a triangulated bifunctor $F : \sfD \times \sfD \to \sfD'$ such that $F(X,Y)=0$ if either $X$ or $Y$ is in $\sfE$, there exists a unique bifunctor $G : (\sfD/\sfE) \times (\sfD/\sfE) \to \sfD'$ such that $F=G(\Pi(-),\Pi(-))$, and moreover $G$ is triangulated.

\subsection{Verdier quotient and t-structures}
\label{ss:Verdier-t-str}

Let $\mathsf{D}$ be a triangulated category equipped with a bounded t-structure $(\mathsf{D}^{\leq 0}, \mathsf{D}^{\geq 0})$. We will denote by $\mathsf{A}$ the heart of this t-structure, and by $(H^n : n \in \Z)$ the associated cohomology functors. Recall that a bounded t-structure is automatically non-degenerate; in particular we have
\begin{align*}
 \mathsf{D}^{\leq 0} &= \{X \in \mathsf{D} \mid \forall n \in \Z_{>0}, \, H^n(X)=0\}; \\
 \mathsf{D}^{\geq 0} &= \{X \in \mathsf{D} \mid \forall n \in \Z_{<0}, \, H^n(X)=0\},
\end{align*}
see~\cite[Proposition~1.3.7]{bbd}.

Let also $\mathsf{B} \subset \mathsf{A}$ be a Serre subcategory, and denote by $\mathsf{D}_{\mathsf{B}}$ the full triangulated subcategory of $\mathsf{D}$ generated by $\mathsf{B}$; it is easily seen that
\[
 \mathsf{D}_{\mathsf{B}} = \{X \in \mathsf{D} \mid \forall n \in \Z, \, H^n(X) \in \mathsf{B}\}.
\]
We set
\[
 \mathsf{E} := \mathsf{D} / \mathsf{D}_{\mathsf{B}},
\]
and denote the quotient functor by $\Pi : \sfD \to \sfE$.

\begin{lem}
\phantomsection
\label{lem:quotient-t-str}
\begin{enumerate}
 \item 
 \label{it:quot-t-str-unicity}
 There exists a unique t-structure on $\sfE$ such that $\Pi$ is t-exact.
 \item 
 \label{it:quot-t-str-description}
 This t-structure is bounded, and given explicitly by
 \begin{align*}
 \sfE^{\leq 0} &= \{X \in \sfE \mid \forall m \in \Z_{>0}, \, H^m(X) \in \sfB\}; \\
 \sfE^{\geq 0} &= \{X \in \sfZ \mid \forall m \in \Z_{<0}, \, H^m(X) \in \sfB\},
\end{align*}
where we identify the objects of $\sfD$ and $\sfE$.
\item If $\sfA'$ is the heart of this t-structure, then the restriction of $\Pi_\sfD$ to $\sfA$, seen as a functor $\sfA \to \sfA'$, factors through an equivalence of categories $\sfA/\sfB \simto \sfA'$.
\end{enumerate}
\end{lem}

\begin{proof}
 We define the full subcategories $\sfE^{\leq 0}$ and $\sfE^{\geq 0}$ as in~\eqref{it:quot-t-str-description}, and then set $\sfE^{\leq n} := \sfE^{\leq 0}[-n]$ and $\sfE^{\geq n} := \sfE^{\geq 0}[-n]$ for $n \in \Z$; we also have
  \begin{align*}
 \sfE^{\leq n} &= \{X \in \mathsf{D}' \mid \forall m \in \Z_{>n}, \, H^m(X) \in \sfB\}; \\
 \sfE^{\geq n} &= \{X \in \mathsf{D}' \mid \forall m \in \Z_{<n}, \, H^m(X) \in \sfB\}.
\end{align*}
 First, let us show that $(\sfE^{\leq 0},\sfE^{\geq 0})$ is a t-structure on $\sfE$. 
 
 The first axiom we have to check is that if $X \in \sfE^{\leq 0}$ and $Y \in \sfE^{\geq 1}$ we have $\Hom_{\sfE}(X,Y)=0$. Recall that a morphism $f : X \to Y$ in $\sfE$ is the equivalence class of a diagram
 \[
  X \xleftarrow{g} Z \xrightarrow{h} Y
 \]
where $Z \in \sfD$, $g,h$ are morphisms in $\sfD$, and the cone $C$ of $g$ belongs to $\sfD_{\sfB}$. Since $X \in \sfE^{\leq 0}$, the long exact sequence in cohomology associated with the distinguished triangle
\[
 Z \xrightarrow{g} X \to C \xrightarrow{[1]}
\]
shows that $H^n(Z) \in \sfB$ for any $n > 0$. It follows that the canonical morphism $\varphi : \tau_{\leq 0} Z \to Z$ is such that $\Pi(\varphi)$ is an isomorphism. (Here, $\tau_{\leq 0}$ is the truncation functor for our given t-structure on $\sfD$.) Similarly the canonical morphism $\psi : Y \to \tau_{\geq 1} Y$ is such that $\Pi(\psi)$ is an isomorphism. We deduce that $f$ can be written as the composition
\[
 X \xrightarrow{\Pi(g)^{-1}} Z \xrightarrow{\Pi(\varphi)^{-1}} \tau_{\leq 0} Z \xrightarrow{\Pi(\varphi)} Z \xrightarrow{\Pi(h)} Y \xrightarrow{\Pi(\psi)} \tau_{\geq 1} Y \xrightarrow{\Pi(\psi)^{-1}} Y.
\]
Now $\psi \circ h \circ \varphi=0$ since $(\mathsf{D}_{\leq 0}, \mathsf{D}_{\geq 0})$ is a t-structure on $\sfD$, hence $f=0$.

It is clear that we have $\sfE^{\leq 0} \subset \sfE^{\leq 1}$ and $\sfE^{\geq 0} \supset \sfE^{\geq 1}$, so that our data satisfy the second axiom of a t-structure. Finally, for any $X \in \sfE$, considered as an object in $\sfD$, the canonical triangle $\tau_{\leq 0} X \to X \to \tau_{\geq 1} X \xrightarrow{[1]}$ in $\sfD$ induces a distinguished triangle
\[
 \Pi(\tau_{\leq 0} X) \to \Pi(X) \to \Pi(\tau_{\geq 1} X) \xrightarrow{[1]}
\]
in $\sfE$, proving that the third axiom is also satisfied.

To prove the unicity claim in~\eqref{it:quot-t-str-unicity}, we consider another t-structure on $\sfE$ such that $\Pi$ is exact (which we will call ``new'' to distinguish from the one constructed above). Then if $X \in \sfE^{\leq 0}$, we consider $X$ as an object in $\sfD$ and the canonical map $\tau_{\leq 0} X \to X$. (Here, as above, $\tau_{\leq 0}$ is the truncation functor for our given t-structure on $\sfD$.) The image of this map under $\Pi$ is an isomorphism, which implies that $X$ belongs to the nonpositive part of the new t-structure. By similar arguments, $\sfE^{\geq 0}$ is contained in the nonnegative part of the new t-structure. We conclude using
the standard fact that two t-structures on a given triangulated category such that the nonpositive part of the first one is contained in the nonpositive part of the second one and the nonnegative part of the first one is contained in the nonnegative part of the second one must coincide.

Let us now denote by $\sfA'$ the heart of our t-structure on $\sfE$. Since $\Pi$ is t-exact it restricts to an exact functor $\sfA \to \sfA'$. It is clear that this functor sends all objects of $\sfB$ to $0$; by the universal property of the Serre quotient it therefore factors through an exact functor
\begin{equation}
\label{eqn:quotient-t-str-1}
\sfA/\sfB \to \sfA'.
\end{equation}
On the other hand, consider the quotient functor $Q : \sfA \to \sfA/\sfB$, and the cohomological functor $Q \circ H^0 : \sfD \to \sfA/\sfB$. It is clear that this functor sends all objects of $\sfD_\sfB$ to $0$; it therefore factors through a cohomological functor $\sfE \to \sfA/\sfB$, which by restriction to $\sfA'$ provides an exact functor
\begin{equation}
\label{eqn:quotient-t-str-2}
\sfA' \to \sfA/\sfB.
\end{equation}
It is clear that~\eqref{eqn:quotient-t-str-1} and~\eqref{eqn:quotient-t-str-2} are quasi-inverse to each other, so that~\eqref{eqn:quotient-t-str-1} is an equivalence of categories.
\end{proof}

\subsection{Quotients of derived categories}

Let now $\sfA$ be an abelian category, and $\sfB \subset \sfA$ be a Serre subcategory. We consider as in~\S\ref{ss:Verdier-t-str} the full triangulated subcategory $\sfD_\sfB$ of $\Db(\sfA)$ generated by $\sfB$. Consider the Serre quotient $\sfA/\sfB$, and the quotient functor $Q : \sfA \to \sfA/\sfB$. The following statement is~\cite[Theorem~3.2]{miyachi}.

\begin{prop}
\label{prop:miyachi}
 The functor
 \[
  \Db(Q) : \Db(\sfA) \to \Db(\sfA/\sfB)
 \]
factors through an equivalence of triangulated categories
\[
 \Db(\sfA) / \sfD_\sfB \simto \Db(\sfA/\sfB).
\]
\end{prop}

By uniqueness, in this particular setting the t-structure obtained using Lem\-ma~\ref{lem:quotient-t-str} from the standard t-structure on $\Db(\sfA)$ is just the standard t-structure on $\Db(\sfA/\sfB)$.

\section{Flat modules in categories}

\subsection{Modules in categories}
\label{ss:mod-cat}

Let $k$ be a commutative ring, $R$ be a $k$-algebra, and $\mathsf{A}$ be a $k$-linear abelian category. Recall (see~\cite[\S 8.5]{ks}) that an $R$-module in $\mathsf{A}$ is a pair $(X,\xi_X)$ where $X$ is an object in $\mathsf{A}$ and $\xi_X : R \to \End_{\mathsf{A}}(X)$ is a $k$-algebra morphism. The $R$-modules in $\mathsf{A}$ are the objects in a $k$-linear abelian category $\Mod(R,\mathsf{A})$, where morphisms from $(X,\xi_X)$ to $(Y,\xi_Y)$ are defined as morphisms $f : X \to Y$ in $\mathsf{A}$ which satisfy $f \circ \xi_X(r) = \xi_Y(r) \circ f$ for any $r \in R$. 
Usually the morphism $\xi_X$ will be omitted from notation. Note that if $X \in \Mod(R,\mathsf{A})$ and $Y \in \mathsf{A}$ the $k$-module $\Hom_{\mathsf{A}}(X,Y)$ admits a natural structure of right $R$-module, where $r \in R$ acts on a morphism $f : X \to Y$ by sending it to $f \circ \xi_X(r)$.

Given a right $R$-module $M$ and an object $X \in \Mod(R,\mathsf{A})$, if the functor
\[
Y \mapsto \Hom_{R^{\mathrm{op}}}(M, \Hom_{\mathsf{A}}(X,Y))
\]
is representable we denote by $M \otimes_R X$ the object that represents it. This condition is automatic in the following cases:
\begin{enumerate}
\item
\label{it:tensor-product-case1}
$\mathsf{A}$ admits small inductive limits, see~\cite[Proposition~8.5.5(a)]{ks}. (In particular, this condition is satisfied in case $\mathsf{A}={\ind}\mathsf{C}$ with $\mathsf{C}$ a $k$-linear abelian category, see~\cite[Theorem~8.6.5(iii)]{ks}.)
\item
\label{it:tensor-product-case2}
$M$ is of finite presentation, see~\cite[Remark~8.5.7]{ks}.
\end{enumerate}
If we denote by $\Modr(R)$ the category of right $R$-modules, then
in case~\eqref{it:tensor-product-case1} we therefore obtain a canonical bifunctor
\begin{equation*}
(-) \otimes_R (-) : \Modr(R) \times \Mod(R,\mathsf{A}) \to \mathsf{A},
\end{equation*}
which is additive and right exact on both sides. It is clear that if $S$ is another $k$-algebra, $M$ is an $(S,R)$-bimodule and $N$ is a right $S$-module, then if $X \in \Mod(R,\mathsf{A})$ the object $M \otimes_R X$ has a canonical structure of $S$-module in $\mathsf{A}$, and moreover we have a canonical isomorphism
\begin{equation}
\label{eqn:tensor-product-bimodule}
N \otimes_S \bigl( M \otimes_R X \bigr) \cong (N \otimes_S M) \otimes_R X.
\end{equation}

If we assume that $R$ is right noetherian and denote by $\Mofr(R)$ the category of finitely generated right $R$-modules, then in view of~\eqref{it:tensor-product-case2} above we similarly have a canonical bifunctor
\begin{equation}
\label{eqn:otimes}
(-) \otimes_R (-) : \Mofr(R) \times \Mod(R,\mathsf{A}) \to \mathsf{A}
\end{equation}
which again is additive and right exact on both sides and satisfies~\eqref{eqn:tensor-product-bimodule} when $M$ is finitely generated as a right $R$-module and $N$ is finitely generated as a right $S$-module.
Concretely, given a free presentation
\[
R^{\oplus n} \xrightarrow{f} R^{\oplus m} \to M \to 0, 
\]
the morphism $f$, seen as a matrix with coefficients in $R$, defines via $\xi_X$ a morphism $X^{\oplus n} \to X^{\oplus m}$, and $M \otimes_R X$ is canonically isomorphic to the cokernel of this map. 

\begin{lem}
\label{lem:morphism-tensor-prod}
 Let $X \in \mathsf{A}$, $Y \in \Mod(R,\mathsf{A})$ and $M \in \Modr(R)$, and assume that the tensor product $M \otimes_R Y$ is defined. Then $\Hom_{\mathsf{A}}(X,Y)$ has a canonical structure of left $R$-module, and there exists a canonical morphism of $k$-modules
 \[
  M \otimes_R \Hom_{\mathsf{A}}(X,Y) \to \Hom_{\mathsf{A}}(X,M \otimes_R Y).
 \]
\end{lem}

\begin{proof}
The $R$-module structure on $\Hom_{\mathsf{A}}(X,Y)$ is defined so that an element $r \in R$ acts on a morphism $f$ by sending it to $\xi_Y(r) \circ f$.
 By the standard adjunction for tensor products of $R$-modules, to construct a morphism as in the lemma we need to define a morphism of right $R$-modules
 \[
  M \to \Hom_k(\Hom_{\mathsf{A}}(X,Y), \Hom_{\mathsf{A}}(X,M \otimes_R Y)).
 \]
In other words, given $m \in M$ and a morphism $f : X \to Y$, we need to construct a morphism $X \to M \otimes_R Y$. By the Yoneda lemma, to construct such a morphism one needs to define, for any $Z \in \mathsf{A}$, a morphism
\[
 \Hom_{\mathsf{A}}(M \otimes_R Y,Z) \to \Hom_{\mathsf{A}}(X,Z)
\]
functorial in $Z$. Now by definition we have
\[
 \Hom_{\mathsf{A}}(M \otimes_R Y,Z) = \Hom_{R^\op}(M,\Hom_{\mathsf{A}}(Y,Z)).
\]
To a morphism $\varphi : M \to \Hom_{\mathsf{A}}(Y,Z)$ one can associate the morphism
\[
 \varphi(m) \circ f : X \to Z,
\]
which provides the desired construction.
\end{proof}

\begin{rmk}
\phantomsection
\label{rmk:flatness-def}
\begin{enumerate}
\item
One can define in a similar way the notion of right $R$-module in $\mathsf{A}$, and the tensor product $X \otimes_R M$ if $X$ is a right $R$-module in $\mathsf{A}$ and $M$ is a left $R$-module. In practice, we will only consider this construction in case $R$ is commutative, so that left and right $R$-modules are the same. We will choose the most convenient notation among $X \otimes_R M$ and $M \otimes_R X$, depending on the context.
\item
\label{it:flatness-def}
In case $\sfA$ is the category of $k$-modules, then an $R$-module in $\sfA$ is nothing but a left $R$-module in the usual sense. Moreover, the above definition of the tensor product coincides with the usual definition.
\end{enumerate}
\end{rmk}

\subsection{Flatness}
\label{ss:flatness-in-cat}

Let again $k$ be a commutative ring, $\sfA$ a $k$-linear abelian category, and $R$ a $k$-algebra. We will assume in addition
that $R$ is right noetherian. We will say that an object $X \in \Mod(R,\mathsf{A})$ is $R$-flat (or simply flat if $R$ is clear from the context) if the functor
\[
(-) \otimes_R X : \Mofr(R) \to \mathsf{A}
\]
is exact, i.e.~if for any injection $M_1 \hookrightarrow M_2$ of finitely generated $R$-modules the induced morphism $M_1 \otimes_R X \to M_2 \otimes_R X$ is injective. 

\begin{rmk}
In view of the standard characterization of flatness in terms of morphisms between finitely generated modules (see e.g.~\cite[\href{https://stacks.math.columbia.edu/tag/00HD}{Tag 00HD}]{stacks-project} in the commutative case), 
this definition is equivalent to the usual definition of flatness for $R$-modules in case 
$\sfA$ is the category of $k$-modules (see Remark~\ref{rmk:flatness-def}\eqref{it:flatness-def}).
\end{rmk}

The next lemma states that this notion satisfies the usual properties of flat modules.

\begin{lem}
\phantomsection
\label{lem:flatness}
\begin{enumerate}
\item
\label{it:flatness-1}
Consider a short exact sequence
\[
0 \to X_1 \to X_2 \to X_3 \to 0
\]
in $\Mod(R,\mathsf{A})$. If $X_3$ is $R$-flat, then for any $M$ in $\Mofr(R)$ the induced sequence
\[
0 \to M \otimes_R X_1 \to M \otimes_R X_2 \to M \otimes_R X_3 \to 0
\]
is a short exact sequence.
\item
\label{it:flatness-2}
Consider a short exact sequence
\[
0 \to X_1 \to X_2 \to X_3 \to 0
\]
in $\Mod(R,\mathsf{A})$. If $X_1$ and $X_3$ are $R$-flat, then so is $X_2$.
\end{enumerate}
\end{lem}

\begin{proof}
\eqref{it:flatness-1}
Assume that $X_3$ is flat, fix $M$ in $\Mofr(R)$, and consider an exact sequence
\[
0 \to M_1 \to M_2 \to M \to 0
\]
in $\Mofr(R)$, where $M_2$ is free. Then we obtain a commutative diagram
\[
\xymatrix@R=0.5cm{
&&&0 \ar[d] & \\
& M_1 \otimes_R X_1 \ar[r] \ar[d] & M_1 \otimes_R X_2 \ar[r] \ar[d] & M_1 \otimes_R X_3 \ar[r] \ar[d] & 0 \\
0 \ar[r] & M_2 \otimes_R X_1 \ar[r] \ar[d] & M_2 \otimes_R X_2 \ar[r] \ar[d] & M_2 \otimes_R X_3 \ar[r] \ar[d] & 0 \\
& M \otimes_R X_1 \ar[r] \ar[d] & M \otimes_R X_2 \ar[r] \ar[d] & M \otimes_R X_3 \ar[d] \ar[r] & 0 \\
& 0 & 0 & 0 &
}
\]
in which all rows and columns are exact. Applying the snake lemma to the first two lines we obtain that the first map on the third line is injective, which proves the desired claim.

\eqref{it:flatness-2}
Assume that $X_1$ and $X_3$ are flat, and
consider an injective morphism $M_1 \hookrightarrow M_2$ in $\Mofr(R)$. Then using~\eqref{it:flatness-1} and our assumptions we obtain a commutative diagram
\[
\xymatrix@R=0.5cm{
 & 0 \ar[d] & & 0 \ar[d] & \\
 0 \ar[r] & M_1 \otimes_R X_1 \ar[r] \ar[d] & M_1 \otimes_R X_2 \ar[r] \ar[d] & M_1 \otimes_R X_3 \ar[r] \ar[d] & 0 \\
 0 \ar[r] & M_2 \otimes_R X_1 \ar[r] & M_2 \otimes_R X_2 \ar[r] & M_2 \otimes_R X_3 \ar[r] & 0 \\
}
\]
with exact rows and columns. The four-lemma implies that the morphism on the middle column is injective, proving that $X_2$ is flat.
\end{proof}

\section{Complements on the completed category}

\subsection{Statement}

In this section we consider the setting of~\cite[Part~I]{bezr}, in its \'etale variant. Namely, we consider an algebraically closed field $\F$, an $\F$-torus $A$, and a (Zariski locally trivial) $A$-torsor $\pi : X \to Y$ where $X,Y$ are algebraic varieties over $\F$. We assume we are given a finite stratification
\[
Y = \bigsqcup_{s \in \mathcal{S}} Y_s
\]
where each $Y_s$ is isomorphic to an affine space, and the restriction of $\pi$ to $X_s := \pi^{-1}(Y_s)$ is a trivial torsor. For any $s \in \mathcal{S}$ we denote by
\[
j'_s : Y_s \to Y, \quad j_s : X_s \to X
\]
the embeddings. 

We fix an algebraic closure $\bk$ of a finite field of characteristic different from $\mathrm{char}(\F)$.
We will assume that for any $s,t \in \mathcal{S}$ and any $n \in \Z$ the sheaf
\[
\scH^n((j'_t)^* (j'_s)_* \underline{\bk}_{Y_s})
\]
is constant. By base change this implies that each $\scH^n((j_t)^* (j_s)_* \underline{\bk}_{X_s})$ is constant too, and these conditions guarantee that the formalism of perverse sheaves applies in the category $\Db_{\mathcal{S}}(Y,\bk)$ of bounded complexes of $\bk$-sheaves $\scF$ on $Y$ such that $\scH^n((j'_s)^*\scF)$ is constant of finite rank for any $n$, $s$, and in the category $\Db_{\mathcal{S}}(X,\bk)$ of bounded complexes of $\bk$-sheaves $\scF$ on $X$ such that $\scH^n((j_s)^*\scF)$ is constant of finite rank for any $n$, $s$.

We consider the ``completed category'' $\widehat{D}_{\mathcal{S}}(X \quot A, \bk)$ as defined in~\cite[Definition~3.1]{bezr}; this category is a certain full subcategory in the category of pro-objects in $\Db_{\mathcal{S}}(X,\bk)$. (This construction, as well as those considered below, and the proofs of their properties, are initially due to Yun, see~\cite[Appendix~A]{by}.) This category is triangulated, and as explained in~\cite[\S 5.2]{bezr} it admits a (bounded) ``perverse'' t-structure $\bigl( {}^{\mathrm{p}} \hspace{-1pt} \widehat{D}_{\mathcal{S}}(X \quot A, \bk)^{\leq 0}, {}^{\mathrm{p}} \hspace{-1pt} \widehat{D}_{\mathcal{S}}(X \quot A, \bk)^{\geq 0} \bigr)$. We can of course consider the similar constructions for each variety $X_s$ endowed with the trivial stratification; the corresponding completed category will be denoted $\widehat{D}_{\mathcal{S}}(X_s \quot A, \bk)$. For each $s$, the (derived) functors $(j_s)_*$, $(j_s)_!$, $j_s^*$, $j_s^!$ induce triangulated functors
\begin{gather}
\label{eqn:push-functors}
(j_s)_*, (j_s)_! : \widehat{D}_{\mathcal{S}}(X_s \quot A, \bk) \to \widehat{D}_{\mathcal{S}}(X \quot A, \bk), \\
j_s^*, j_s^! : \widehat{D}_{\mathcal{S}}(X \quot A, \bk) \to \widehat{D}_{\mathcal{S}}(X_s \quot A, \bk).
\end{gather}

Our primary goal in this section is to prove the following claim.

\begin{prop}
\label{prop:exactness-pushforward}
The functors $(j_s)_*$, $(j_s)_!$ in~\eqref{eqn:push-functors} are t-exact with respect to the perverse t-structures.
\end{prop}

This proposition will be deduced from the fact that the similar functors
\begin{equation}
\label{eqn:push-functors-noncomp}
(j_s)_*, (j_s)_! : \Db_{\mathcal{S}}(X_s \quot A, \bk) \to \Db_{\mathcal{S}}(X \quot A, \bk)
\end{equation}
are t-exact with respect to the perverse t-structures, since $j_s$ is affine (see~\cite[Corollaire~4.1.3]{bbd}). The other ingredient is a result from~\cite[Appendix~A]{by}; since the proof of this claim is somewhat sketchy, and since the construction of the perverse t-structure in~\cite{bezr} is slightly different from the original construction in~\cite{by}, we provide an explicit proof of this result in our setting.

\subsection{Preliminaries on $R^\wedge_A$-modules}
\label{ss:preliminary-Rmod}

As in~\cite{bezr} we denote by $R_A$ the group algebra of the cocharacter lattice $X_*(A)$ over $\bk$, and by $R_A^\wedge$ the completion of $R_A$ with respect to the natural augmention ideal $\mathfrak{m}_A \subset R_A$. We will denote by 
$\Mod^{\mathrm{nil}}(R_A^\wedge)$ the full subcategory of $\Mod(R^\wedge_A)$ whose objects are the nilpotent $R_A^\wedge$-modules, i.e.~the modules such that any element is annihilated by a power of $\mathfrak{m}_A$. We will also denote 
by $\Mod^{\mathrm{fg},\mathrm{nil}}(R_A^\wedge)$ the full subcategory of $\Mof(R_A^\wedge)$ whose objects are the modules which are both finitely generated and nilpotent. (These modules are necessarily finite-dimensional.)

\begin{lem}
\label{lem:injective-resolutions-Rwedge}
For any $M$ in $\Mod^{\mathrm{nil}}(R_A^\wedge)$, there exists an injective $R_A^\wedge$-module $N$ which belongs to $\Mod^{\mathrm{nil}}(R_A^\wedge)$ and an embedding $M \hookrightarrow N$.
\end{lem}

\begin{proof}
Choose an injective $R_A^\wedge$-module $N'$ and an embedding $M \hookrightarrow N'$. Since $M$ is nilpotent, this embedding necessarily factors through the submodule
\[
N = \{n \in N' \mid (\mathfrak{m}_A)^k \cdot n = 0 \text{ for } k \gg 0 \}.
\]
Now $N$ is an injective $R_A^\wedge$-module by~\cite[\href{https://stacks.math.columbia.edu/tag/08XW}{Tag 08XW}]{stacks-project}.
\end{proof}

As in~\cite[\S 3.1.7]{bmr}, this lemma implies that the canonical functor
\[
\Db \Mod^{\mathrm{nil}}(R_A^\wedge) \to \Db \Mod(R_A^\wedge)
\]
is fully faithful, and that its essential image consists of complexes all of whose cohomology objects belong to $\Mod^{\mathrm{nil}}(R_A^\wedge)$.

Now, we consider a bounded complex $M$ of finitely generated $R_A^\wedge$-modules. Then we can consider for any $n \geq 0$ the complex
\[
\bigl( R^\wedge_A / (\mathfrak{m}_A)^n \cdot R^\wedge_A \bigr) \otimes_{R^\wedge_A} M \quad \in \Db \Mod^{\mathrm{nil}}(R_A^\wedge),
\]
and also the derived tensor product
\[
\bigl( R^\wedge_A / (\mathfrak{m}_A)^n \cdot R^\wedge_A \bigr) \lotimes_{R^\wedge_A} M \quad \in \Db \Mod^{\mathrm{nil}}(R_A^\wedge).
\]
(Note that $R_A^\wedge$ has finite global dimension, so that this complex is indeed bounded.)

\begin{lem}
\label{lem:pro-obj-RA}
The pro-objects in $\Db \Mod^{\mathrm{nil}}(R_A^\wedge)$
\[
``\varprojlim_n" \bigl( R^\wedge_A / (\mathfrak{m}_A)^n \cdot R^\wedge_A \bigr) \otimes_{R^\wedge_A} M \quad \text{and} \quad
``\varprojlim_n" \bigl( R^\wedge_A / (\mathfrak{m}_A)^n \cdot R^\wedge_A \bigr) \lotimes_{R^\wedge_A} M
\]
are canonically isomorphic.
\end{lem}

\begin{proof}
There exists for any $n$ a canonical morphism of complexes
\[
\bigl( R^\wedge_A / (\mathfrak{m}_A)^n \cdot R^\wedge_A \bigr) \lotimes_{R^\wedge_A} M \to \bigl( R^\wedge_A / (\mathfrak{m}_A)^n \cdot R^\wedge_A \bigr) \otimes_{R^\wedge_A} M;
\]
we will prove that these morphisms define an isomorphism between the pro-objects under consideration. For that, by definition we need to show that for any $N \in \Db \Mod^{\mathrm{nil}}(R_A^\wedge)$ the induced morphism
\[
\varinjlim_n \Hom \Bigl( \bigl( R^\wedge_A / (\mathfrak{m}_A)^n \cdot R^\wedge_A \bigr) \otimes_{R^\wedge_A} M, N \Bigr) \to \varinjlim_n \Hom \Bigl( \bigl( R^\wedge_A / (\mathfrak{m}_A)^n \cdot R^\wedge_A \bigr) \lotimes_{R^\wedge_A} M, N \Bigr)
\]
is an isomorphism.

Fix a bounded below complex $N^\bullet$ of injective and nilpotent $R_A^\wedge$-modules whose image in $D^+ \Mod^{\mathrm{nil}}(R_A^\wedge)$ is $N$. (Such a complex exists by Lemma~\ref{lem:injective-resolutions-Rwedge}.) Then for any bounded complex $M'$ of finitely generated $R^\wedge_A$-modules, we observe that we have a canonical isomorphism
\[
\varinjlim_n \Hom^\bullet_{R^\wedge_A} \Bigl( \bigl( R^\wedge_A / (\mathfrak{m}_A)^n \cdot R^\wedge_A \bigr) \otimes_{R^\wedge_A} M', N^\bullet \Bigr) \simto
\Hom^\bullet_{R^\wedge_A} \bigl( M', N^\bullet \bigr),
\]
where $\Hom^\bullet_{R^\wedge_A}$ denotes the complex of morphisms of $R^\wedge_A$-modules between two complexes of modules. Taking $0$-th cohomology, and by exactness of filtrant direct limits, we deduce a canonical isomorphism
\[
\varinjlim_n \Hom_{\Db \Mod(R^\wedge_A)} \Bigl( \bigl( R^\wedge_A / (\mathfrak{m}_A)^n \cdot R^\wedge_A \bigr) \otimes_{R^\wedge_A} M', N \Bigr) \simto
\Hom_{\Db \Mod(R^\wedge_A)} \bigl( M', N \bigr).
\]
Applying this with $M'=M$ we obtain an isomorphism
\[
\varinjlim_n \Hom_{\Db \Mod(R^\wedge_A)} \Bigl( \bigl( R^\wedge_A / (\mathfrak{m}_A)^n \cdot R^\wedge_A \bigr) \otimes_{R^\wedge_A} M, N \Bigr) \simto
\Hom_{\Db \Mod(R^\wedge_A)} \bigl( M, N \bigr).
\]
On the other hand, applying this isomorphism with $M'$ a bounded projective reso\-lution of $M$ we obtain an isomorphism
\[
\varinjlim_n \Hom_{\Db \Mod(R^\wedge_A)} \Bigl( \bigl( R^\wedge_A / (\mathfrak{m}_A)^n \cdot R^\wedge_A \bigr) \lotimes_{R^\wedge_A} M, N \Bigr) \simto
\Hom_{\Db \Mod(R^\wedge_A)} \bigl( M, N \bigr).
\]
This provides the desired identification.
\end{proof}

In~\cite[\S 4.1]{bezr} we consider a certain subcategory $\widehat{D}(R^\wedge_A)$ of the category of pro-objects in $\Db \Mod^{\mathrm{nil}}(R_A^\wedge)$. We show in~\cite[Proposition~4.5]{bezr} that the functor
\[
M \mapsto ``\varprojlim_n" \bigl( R^\wedge_A / (\mathfrak{m}_A)^n \cdot R^\wedge_A \bigr) \lotimes_{R^\wedge_A} M
\]
induces an equivalence of categories
\[
\Db \Mof(R_A^\wedge) \simto \widehat{D}(R^\wedge_A).
\]
Lemma~\ref{lem:pro-obj-RA} shows that the image of a complex $M$ of finitely generated $R^\wedge_A$-modules can also be described as
\[
``\varprojlim_n" \bigl( R^\wedge_A / (\mathfrak{m}_A)^n \cdot R^\wedge_A \bigr) \otimes_{R^\wedge_A} M.
\]
In particular, if for some $m$ the complex $M$ satisfies $\mathsf{H}^i(M)=0$ for $i > m$, resp.~for $i<m$, then its image in $\widehat{D}(R^\wedge_A)$ can be written as $``\varprojlim_n \hspace{-5pt} " M_n$ where each $M_n$ satisfies $\mathsf{H}^i(M)=0$ for $i > m$, resp.~for $i<m$.

\subsection{Proof of Proposition~\ref{prop:exactness-pushforward}}

The following statement is~\cite[Lemma~A.6.2]{by}.

\begin{prop}
\label{prop:perverse-t-str-completed-cat}
Let $\scF \in \widehat{D}_{\mathcal{S}}(X \quot A, \bk)$. Then $\scF$ belongs to ${}^{\mathrm{p}} \hspace{-1pt} \widehat{D}_{\mathcal{S}}(X \quot A, \bk)^{\leq 0}$, resp.~to ${}^{\mathrm{p}} \hspace{-1pt} \widehat{D}_{\mathcal{S}}(X \quot A, \bk)^{\geq 0}$, if and only if there exists a projective system $(\scF_n : n \geq 0)$ of objects in ${}^{\mathrm{p}} \hspace{-1pt} \Db_{\mathcal{S}}(X, \bk)^{\leq 0}$, resp.~in ${}^{\mathrm{p}} \hspace{-1pt} \Db_{\mathcal{S}}(X, \bk)^{\geq 0}$, and an isomorphism $\scF \cong ``\varprojlim_n \hspace{-5pt} " \scF_n$.
\end{prop}

\begin{proof}
The proof of the ``only if'' direction is given in~\cite[Lemma~A.6.2]{by}. It proceeds by induction on the number of strata in the support of $\scF$, the base case (one stratum) being given by the comments at the end of~\S\ref{ss:preliminary-Rmod}. The ``if'' direction can be deduced as follows. Assume given a projective system $(\scF_n : n \geq 0)$ of objects in ${}^{\mathrm{p}} \hspace{-1pt} \Db_{\mathcal{S}}(X, \bk)^{\leq 0}$ such that $``\varprojlim_n \hspace{-5pt} " \scF_n$ belongs to $\widehat{D}_{\mathcal{S}}(X \quot A, \bk)$. To prove that $\scF$ belongs to ${}^{\mathrm{p}} \hspace{-1pt} \widehat{D}_{\mathcal{S}}(X \quot A, \bk)^{\leq 0}$ it suffices to prove that for any $\scG$ in ${}^{\mathrm{p}} \hspace{-1pt} \widehat{D}_{\mathcal{S}}(X \quot A, \bk)^{\geq 1}$ we have $\Hom(\scF,\scG)=0$. Now, by the ``only if'' direction, there exists a projective system $(\scG_n : n \geq 0)$ of objects in ${}^{\mathrm{p}} \hspace{-1pt} \Db_{\mathcal{S}}(X, \bk)^{\geq 1}$ such that $\scG \cong ``\varprojlim_n \hspace{-5pt} " \scG_n$. We then have
\[
\Hom(\scF, \scG) = \varprojlim_n \varinjlim_m \Hom(\scF_m, \scG_n) = 0
\]
since $\Hom(\scF_m, \scG_n)=0$ for any $n,m$. The case of ${}^{\mathrm{p}} \hspace{-1pt} \widehat{D}_{\mathcal{S}}(X \quot A, \bk)^{\geq 0}$ is similar.
\end{proof}

Using this proposition we can give the proof of Proposition~\ref{prop:exactness-pushforward}.

\begin{proof}[Proof of Proposition~\ref{prop:exactness-pushforward}]
Formal properties of the ``recollement'' formalism show that $(j_s)_!$ is right t-exact and $(j_s)_*$ is left t-exact. The fact that $(j_s)_!$ is also left t-exact and that $(j_s)_*$ is also right t-exact follows from the t-exactness of the functors~\eqref{eqn:push-functors-noncomp},
together with Proposition~\ref{prop:perverse-t-str-completed-cat} (applied in the ``only if'' direction on $X_s$, and in the ``if'' direction on $X$).
\end{proof}

\section{Infinitesimal flatness}

Let $k$ be a commutative ring, and let $H$ be an affine group scheme over $\mathrm{Spec}(k)$. Let $I \subset \scO(H)$ be the ideal defining the unit in $H$, i.e.~the kernel of the augmentation morphism $\scO(H) \to k$ in the $k$-Hopf algebra $\scO(H)$. Following~\cite[\S I.7.9]{jantzen}, we will say that $H$ is \emph{infinitesimally flat} if the quotient $\scO(H)/I^n$ is a finite projective module (equivalently, is finitely presented and flat, see~\cite[\href{https://stacks.math.columbia.edu/tag/00NX}{Tag 00NX}]{stacks-project}) over $k$ for any $n \geq 1$. This notion behaves well under flat base change, as explained in the following lemma.

\begin{lem}
\label{lem:infflat-base-change}
If $H$ is infinitesimally flat, then for any flat morphism $k \to k'$ the group scheme $\mathrm{Spec}(k') \times_{\mathrm{Spec}(k)} H$ over $k'$ is infinitesimally flat.
\end{lem}

\begin{proof}
The claim is obvious from the fact that the ideal of the unit in the $k'$-group scheme $\mathrm{Spec}(k') \times_{\mathrm{Spec}(k)} H$ is $k' \otimes_k I \subset k' \otimes_k \scO(H)$.
\end{proof}

The main interest of this notion comes from the following statement, copied from~\cite[Lemma~I.7.16]{jantzen}, where we denote by $\mathrm{Dist}(H)$ the distribution algebra of $H$ (see~\cite[\S I.7.7]{jantzen}).

\begin{lem}
\label{lem:morph-Dist}
Assume that $H$ is infinitesimally flat, noetherian, and integral. Then for any $H$-modules $M,M'$ such that $M'$ is projective over $k$, the natural morphism
\[
\Hom_H(M,M') \to \Hom_{\mathrm{Dist}(H)}(M,M')
\]
is an isomorphism.
\end{lem}

This notion is related to that of \emph{regular immersion} (see~\cite[\href{https://stacks.math.columbia.edu/tag/063J}{Tag 063J}]{stacks-project}) as follows.

\begin{lem}
\label{lem:inf-flat-reg-immersion}
If the embedding of the unit $\mathrm{Spec}(k) \to H$ is a regular immersion, then $H$ is infinitesimally flat.
\end{lem}

\begin{proof}
If the embedding $\mathrm{Spec}(k) \to H$ is a regular immersion, then by~\cite[\href{https://stacks.math.columbia.edu/tag/063K}{Tag 063K}]{stacks-project} and~\cite[\href{https://stacks.math.columbia.edu/tag/063M}{Tag 063M}]{stacks-project} the quotient $I/I^2$ is a finite projective $k$-module, and for any $n \geq 1$ we have an isomorphism $\mathrm{Sym}^n_k(I/I^2) \simto I^n/I^{n+1}$. Now the left-hand side is finite and projective by~\cite[\href{https://stacks.math.columbia.edu/tag/01CK}{Tag 01CK}]{stacks-project}. Hence each $\scO(H)/I^n$ is an extension of finite projective modules, and is therefore finite and projective.
\end{proof}

We deduce the following property, in case $k$ is a $\bk$-algebra for some field $\bk$.

\begin{cor}
\label{cor:smooth-infflat}
Assume that $\Spec(k)$ is smooth over $\bk$, and that $H$ is smooth over $k$. Then $H$ is infinitesimally flat.
\end{cor}

\begin{proof}
The schemes $\mathrm{Spec}(k)$ and $\mathrm{Spec}(H)$ are smooth over $\bk$, hence regular by~\cite[\href{https://stacks.math.columbia.edu/tag/056S}{Tag 056S}]{stacks-project}. Using~\cite[\href{https://stacks.math.columbia.edu/tag/0E9J}{Tag 0E9J}]{stacks-project} this implies that the immersion $\mathrm{Spec}(k) \to H$ is regular, so that $H$ is infinitesimally flat by Lemma~\ref{lem:inf-flat-reg-immersion}.
\end{proof}

\end{document}